\documentclass[fleqn,11pt,a4paper]{article}
\usepackage{pdflscape}
\usepackage{amsfonts}
\usepackage{amssymb}
\usepackage{amsmath}
\usepackage{graphicx}
\usepackage{pdflscape}
\allowdisplaybreaks
\usepackage{amsthm}
\usepackage{rotating}
\usepackage{bbm}
\usepackage{longtable}
\usepackage{array}
\usepackage{enumitem}
\usepackage{tabularx}
\usepackage{csquotes}
\usepackage{multirow}
\usepackage{float}
\usepackage{graphicx}
\usepackage{epstopdf}
\usepackage{color}
\usepackage{amsmath,amsfonts,amssymb}
\usepackage{mathtools, nccmath}
\usepackage{algorithm}
\usepackage{algpseudocode}
\usepackage{geometry}
\definecolor{darkred}{rgb}{0.5,0.2,0.2}
\usepackage[pdftitle={RV CL},pdfauthor={Mikkel Bennedsen},colorlinks=TRUE,allcolors=darkred]{hyperref}
\usepackage{booktabs}
\usepackage{pdflscape}
\usepackage[round]{natbib}
\usepackage{subcaption}
\usepackage{mwe}
\usepackage{adjustbox}
\usepackage{placeins}

\newcommand{\dd}{\mathrm{d}}

\newcommand{\vect}{\operatorname{vec}}
\usepackage{dsfont}
\usepackage{cleveref}
\usepackage{float}
\floatstyle{plaintop}
\restylefloat{table}

\newcommand{\interior}[1]{%
  {\kern0pt#1}^{\mathrm{o}}%
}

\restylefloat{table}
\usepackage[T1]{fontenc}
\theoremstyle{plain}
\newtheorem{theorem}{Theorem}[section]

\newtheorem{assumption}{Assumption}[section]

\newtheorem{corollary}{Corollary}[section]

\newtheorem{lemma}{Lemma}[section]
\newtheorem{proposition}{Proposition}[section]
\newtheorem{definition}{Definition}
\theoremstyle{definition}
\newtheorem{example}{Example}[section]
\theoremstyle{remark}
\newtheorem{remark}{Remark}[section]
\Crefname{assumption}{assumption}{assumptions}
\Crefname{assumption}{Assumption}{Assumptions}

\DeclarePairedDelimiter{\nint}\lceil\rceil

\def\bi{\begin{itemize}}
\def\ei{\end{itemize}}

\allowdisplaybreaks
\graphicspath{{Figures/}}
\numberwithin{equation}{section}
\DeclareMathOperator*{\argmin}{arg\,min}

\DeclareMathOperator{\Tr}{Tr}

\newif\ifi

\usepackage{caption}

\definecolor{darkgreen}{rgb}{0.0,0.5,0.0}

\begin{document}

\renewcommand{\thefootnote}{\fnsymbol{footnote}}

\title{
Consistent community recovery in stochastic block
Ornstein--Uhlenbeck processes}%
% \thanks{
% We benefited from comments by and discussions with Stefán Guðmundsson and Bent Jesper Christensen, as well as from comments by conference participants at the 12th Bachelier World Congress in Rio de Janeiro, and the Bernoulli--IMS 11th World Congress in Probability and Statistics in Bochum. Anders Norlyk gratefully acknowledges financial support from the Center for Research in Energy: Economics and Markets (CoRE).
% \par\vspace{0.8\baselineskip}
% }
% }

\author{
Anders Norlyk\thanks{
%(Corresponding author)
Department of Economics and Business Economics and Center for Research in Energy: Economics and Markets (CoRE),
Aarhus University,
Fuglesangs All\'e 4,
8210 Aarhus V, Denmark.
E-mail:
\href{mailto:amn@econ.au.dk}{\nolinkurl{amn@econ.au.dk}}.
}
\and
Almut E.~D.~Veraart\thanks{
%(Corresponding author)
Department of Mathematics,
Imperial College London,
South Kensington Campus,
London SW7 2AZ, UK.
E-mail:
\href{mailto:a.veraart@imperial.ac.uk}{\nolinkurl{a.veraart@imperial.ac.uk}}.
}
}

\maketitle
\begin{abstract}
We propose the stochastic block Ornstein-Uhlenbeck (SBOU) process, a continuous-time multivariate model in which the drift matrix encodes a latent group structure among its components. Our main contribution is a community-detection algorithm whose misclassification proportion converges to zero in a regime combining infill, long-span, and high-dimensional asymptotics. To our knowledge, this is the first consistency result of this kind for latent group recovery in a discretely observed continuous-time multivariate model. As a key intermediate result, we establish consistency of the discretely observed maximum likelihood estimator of the drift matrix in the same regime, thereby extending the high-dimensional L\'evy-driven Ornstein-Uhlenbeck literature. For practical implementation, we develop a feasible model-selection procedure for estimating the support of the drift matrix, which enables data-driven selection of the number of latent groups. The SBOU framework can be viewed as a continuous-time generalisation of the discrete-time stochastic-block VAR model, allowing for both positive and negative dynamic interactions between groups as opposed to only positive. We illustrate the methodology on the RE-Europe wind-capacity dataset and recover a country-level grouping consistent with the geographic benchmark.
\end{abstract}

\bigskip


\baselineskip=17pt
\newpage
%\tableofcontents

\section{Introduction}

Many multivariate economic systems are high-dimensional, making it difficult to model each interaction separately. This is a manifestation of the curse of dimensionality: without additional structure, a dynamic model for $d$ variables may require estimating on the order of $d^2$ interaction parameters. Classical approaches address this problem by reducing the dimension of the observed data, for example through factor methods such as principal component analysis, or by applying clustering and network-filtering methods to empirical dependence measures such as correlations \citep{Mantegna1999}.

A more recent strand of literature instead reduces dimensionality by imposing structure directly on the dynamic model. In network time-series models, an observed graph is used to restrict the set of possible interactions between variables. For example, in a network VAR for social interactions or electricity prices, the coefficient linking unit $j$ to unit $i$ may be set to zero whenever the corresponding edge is absent from the observed network. The model therefore avoids estimating all $d^2$ possible pairwise effects and instead estimates only those interactions allowed by the graph. This idea is developed in, among others, \citet{Zhu2017,Zhu2020,Knight2016,Knight2019}.

In many applications, however, the relevant network or grouping structure is not observed. One may still expect variables to exhibit similarities in how they interact with the rest of the system, but the groups themselves must then be inferred from the data. This leads to latent-group approaches, where variables are assigned to unobserved classes and the dynamic interaction structure is assumed to depend on these class memberships \citep{Gudhmundsson2021,Brownlees2022,Gudhmundsson2025,Fang2024}. This paper concerns this latent case.

Our approach is related in spirit to PCA and to the correlation-based hierarchical clustering method of \citet{Mantegna1999}, in the sense that all three seek lower-dimensional structure in a high-dimensional system. The difference is in the object being structured. PCA and correlation-based clustering operate on reduced-form summaries of variation or dependence. By contrast, we impose the latent group structure directly on a dynamic interaction matrix. In our setting, the groups are therefore not defined by which variables merely move together, but by how variables dynamically affect one another.

Many systems of economic interest evolve naturally in continuous time and are observed on fine or irregular grids. Continuous-time formulations are especially natural when temporal aggregation is undesirable, when the observation scheme is non-uniform, or when one wants the underlying interaction model to remain meaningful as the mesh of the sampling grid shrinks. However, model-based clustering for discretely observed continuous-time multivariate dynamics remains comparatively underdeveloped. A recent exception is \citet{Fang2024}, who introduce latent group structure into the intensity measure of a Hawkes process. Their work shows how latent communities can be incorporated into continuous-time models for event data. The present paper is complementary, but focuses instead on continuously evolving state variables and places the latent group structure in the drift matrix governing their dynamic interactions.

This paper contributes to this strand of literature by introducing the stochastic block Ornstein-Uhlenbeck (SBOU) process, a multivariate continuous-time model in which the drift matrix encodes latent group structure. Specifically, we study a L{\'e}vy-driven system of the form
\begin{equation}\label{eqn.Intro:1}
dY_t = -QY_{t-} dt + dL_t,
\end{equation}
where the drift matrix $Q$ is assumed to contain a stochastic-block structure. The key modeling idea is simple: if the drift governs who affects whom over time, then clustering should target the structural object that drives the dynamics rather than a reduced-form dependence summary. This is the main contribution of the paper. It moves latent-group recovery from a discrete-time or static-correlation setting to a genuinely high-dimensional, continuous-time environment.

The SBOU framework can be viewed as a continuous-time counterpart to the stochastic-block VAR of \citet{Gudhmundsson2021}, tailored to mean-reverting multivariate systems. This shift is substantively important. A continuous-time specification is natural for irregularly sampled or high-frequency data, and it allows the underlying interaction model to remain meaningful as the sampling mesh shrinks. By placing the latent block structure in the drift of a multivariate Ornstein-Uhlenbeck process, the model also connects directly to the literature on continuous-time inference for mean-reverting systems \citep{Mai2014,Gaiffas2019,Courgeau2022a,Dexheimer2024,MehtaVeraart2026}. A further difference is that the block-level interaction parameters in \citet{Gudhmundsson2021} are restricted to be nonnegative, whereas the SBOU drift specification allows for signed dynamic interactions. This is useful in applications where reinforcement and substitution effects may coexist.

The paper makes four main contributions. First, it proposes a continuous-time latent-group model in which stochastic-block structure is built into the drift matrix of a multivariate L{\'e}vy-driven Ornstein-Uhlenbeck process. Second, it establishes consistency of a feasible drift estimator based on discrete observations in a high-dimensional regime combining increasing dimension, infill asymptotics, and long-span asymptotics. This provides the link from observed sample paths to the structural object used for clustering, and connects the paper to the literature on drift estimation for high-dimensional and L{\'e}vy-driven Ornstein-Uhlenbeck processes \citep{Mai2014,Ciolek2020,Gaiffas2019,Courgeau2022b,Dexheimer2024}. Third, given the number of latent groups, the paper develops a spectral clustering procedure based on the estimated drift matrix and proves consistency of community recovery. Treating the number of groups as known at this stage is standard in spectral-clustering consistency analyses \citep{Qin2013,Lei2015,Gudhmundsson2021}. Fourth, it develops implementation tools for the case where key inputs are unknown. In particular, we propose a support-recovery procedure for the drift matrix, which turns the estimated continuous-time interaction matrix into the network object required for selecting the number of latent groups. We then use the estimator of \citet{Ma2021} to choose the number of groups used as input to the spectral clustering algorithm. These steps make the proposed framework feasible when neither the sparsity pattern of the drift matrix nor the number of communities is known which is arguably always the case in any interesting use-case.

Conceptually, the paper connects two branches of  literature that have largely evolved separately: high-dimensional continuous-time stochastic-process inference and model-based community detection. In the former, the main object of interest is typically estimation of a drift or interaction matrix \citep{Gaiffas2019,Ciolek2020,Courgeau2022a,Courgeau2022b,Dexheimer2024}; in the latter, the graph is often observed and the goal is to recover its latent partition \citep{Qin2013,Lei2015,Ma2021}. Here, by contrast, the graph-like interaction structure is itself latent and must be recovered from discretely observed continuous-time dynamics. In this sense, the paper also contributes to the network time-series literature by developing a continuous-time latent-group framework for community recovery from an unobserved drift structure \citep{Courgeau2022a,Courgeau2022b,Lucchese2023}.

The remainder of the paper is organized as follows. Section~2 introduces the SBOU model and its stochastic-block drift structure. Section~3 studies estimation of the drift matrix from discrete observations. Section~4 develops the clustering algorithm and proves its consistency. Section~5 discusses practical selection of the support and the number of groups. Section~6 presents the empirical application, and Section~7 concludes.

\section{The model, notations and key definitions}\label{Sec:Model}
In this section we introduce our stochastic block Ornstein-Uhlenbeck process. Before doing so, we introduce a few definitions and some notation.
\subsection{Notations}

Throughout this paper, we will use several different norms. The spectral or operator norm is denoted by $\Vert \cdot \Vert$, the Frobenius norm by $\Vert \cdot \Vert_F$, and the Euclidean norm by $\Vert \cdot \Vert_2$. For any matrix \(A\), the notation \(A_{i\bullet}\) and \(A_{\bullet j}\) will refer to the \(i\)-th row and \(j\)-th column of \(A\), respectively. \(A_{ij}\) refers to the \((i,j)\)-th entry of a matrix \(A\). If a matrix is indexed, such as \(A_0\), the \((i,j)\)-th entry will be denoted as \([A_0]_{ij}\). For a real symmetric matrix \(A\), \(\lambda_i(A)\) denotes its eigenvalues, with specific orderings stated where used. Let $\mathcal{M}_{(d,k)}(\mathbb{R})$ denote the set of all $d \times k$ matrices with real-valued entries. If $d=k$, we refer to it simply as $\mathcal{M}_{d}(\mathbb{R})$.

We use standard Landau notation: for sequences \(\left(f(d)\right)_{d \in \mathbb{N}}\) and \(\left(g(d)\right)_{d \in \mathbb{N}}\) of positive reals, \(f(d) = O(g(d))\) means \(\limsup_{d\to\infty} f(d)/g(d) < \infty\), \(f(d) = o(g(d))\) means \(\lim_{d\to\infty} f(d)/g(d) = 0\), and we also adopt the somewhat atypical notation (but useful) \(f(d) = \Omega(g(d))\) as \(d \to \infty\) means there exist a constant \(C > 0\) and \(d' \in \mathbb{N}\) such that \(f(d) \geq Cg(d)\) for all \(d \geq d'\). Finally, \(f(d) = \omega(g(d))\) means \(\lim_{d\to\infty} f(d)/g(d) = \infty\).

The notation \(\vert \cdot \vert\) will depend on context. When applied to scalars, it refers to the absolute value; when applied to sets, it denotes the cardinality. The symbol \(\widehat{\cdot}\) will generally be used for estimated quantities. For a random variable $\bold{X}$, $\mathcal{L}\left(\bold{X}\right)$ refers to the law of $\bold{X}$.

Lastly, we introduce several definitions that will be used throughout the paper.

\begin{definition}[An event occurring with high probability (w.h.p.)] \label{Def:WithHigProb}
Let \(\left( E_d \right)_{d \in \mathbb{N}}\) be a sequence of events (sets). We say that \(E_d\) occurs with high probability if
\begin{equation*}
    \lim_{d \to \infty} \mathbb{P}(E_d) = 1.
\end{equation*}
\end{definition}

\begin{remark}
We note that \Cref{Def:WithHigProb} is sometimes defined more stringently, requiring that there exists some \(c > 0\), independent of \(d\), such that \(\mathbb{P}(E_d) \ge 1 - O(d^{-c})\), meaning there is a rate assumption on the convergence of the probability of the events. All our proofs remain valid under this assumption with slight modifications.
\end{remark}

We also introduce the notion of a random variable being bounded by some fixed but unknown constant w.h.p., which we will use frequently throughout.

\begin{definition}
For a sequence of random variables \(\left( X_d \right)_{d \in \mathbb{N}}\) and a sequence \(\left( a_d \right)_{d \in \mathbb{N}}\), we say \(X_d = O(a_d)\) w.h.p. if there exists a constant \(C\) such that \(\vert X_d \vert \leq C a_d\) w.h.p.
\end{definition}

\begin{remark}
The notion \(X_d = O(a_d)\) w.h.p. implies that \(X_d = O_{\mathbb{P}}(a_d)\), but the converse does not necessarily hold.
\end{remark}

  \subsection{The stochastic block Ornstein-Uhlenbeck process}\label{Sec:ModelSpecific}
  We will now introduce the SBOU-process. As mentioned, this process should encode a \emph{latent} group structure into \eqref{eqn.Intro:1} such that entries in the same group are more likely to affect the dynamics of each other than entries in different groups. We construct such a process by assuming that the drift matrix in \eqref{eqn.Intro:1} is a specific type of random matrix.
\begin{definition}[Generalised stochastic block Model of type $2$]\label{GSBM2}
 The generalised stochastic block model without self-loops, denoted GSBM2, is a directed random graph \(\mathcal{G} = (\mathcal{V}, \mathcal{E}, \mathcal{W})\), where:
\begin{itemize}
    \item \(\mathcal{V} = \{1, 2, \dots, d\}\) is the set of vertices,
    \item \(\mathcal{E} \subseteq \mathcal{V} \times \mathcal{V}\) is the set of directed edges, and
    \item \(\mathcal{W} = \{w_{(i,j)} \in \mathbb{R} \setminus \{0\} : (i,j) \in \mathcal{E}\}\) is the set of edge weights.
\end{itemize}

The edge weights \(w_{ij}\) are drawn from a distribution \(W\) with support on \([\underline{w}, \overline{w}] \setminus \{0\}\), where \(\underline{w}, \overline{w} \in \mathbb{R}\) and \(\underline{w} \leq \overline{w}\), with mean \(\mu\) and where $\bar{\mu} \equiv \mathbb{E}\left(\vert W\vert \right)$. For any pair of vertices \(i\) and \(j\), the edge weight \(w_{ij} \sim W\).

The vertices are fixed, but the existence of an edge between vertices is random. In \( \text{GSBM}2\), the set of \(d\) vertices \(\mathcal{V}\) is partitioned into \(k\) non-empty communities, \(\mathcal{V}_1, \mathcal{V}_2, \dots, \mathcal{V}_k\). The probability of an edge between vertex \(i\) in community \(l\) and vertex \(j\) in community \(m\) is given by:
\[
\mathbb{P}((i,j) \in \mathcal{E}) = \theta_i \theta_j b_{lm}, \quad \text{for } i \neq j,
\]
where \(\theta_i, \theta_j \in (0,1]\) are vertex-specific probability components, and \(b_{lm} \in [0,1]\) is a block-specific probability component for communities \(l\) and \(m\).

For \(i = j\), we set \(\mathbb{P}((i,j) \in \mathcal{E}) = 0\), meaning self-loops are not allowed.
\end{definition}
\begin{remark}\label{Remark:sMapping}
    We underline that the GSBM2 is simply the GSBM introduced \cite{Gudhmundsson2021}, with no self-loops and where edge weights are allowed to be negative.
\end{remark}
\begin{remark}\label{Remark.1}
We introduce the same notation as in \cite{Gudhmundsson2021} and let $\bold{Z}$ be the $d\times k$ matrix such that the matrix entry $\bold{Z}_{il} = 1$ if vertex $i$ belongs to community $l$ and $0$ otherwise. Similarly, we let $\bold{B}$ be the $k\times k$ community pair-specific probability component matrix with entries given by $\bold{B}_{lm}=b_{lm}$. We let $\Theta$ be the $d\times d$ diagonal vertex specific probability matrix with entries given by $\Theta_{ii} = \theta_i$.  Then, we may denote the GSBM2 in short as $\mathcal{G}\sim GSBM2
(\bold{Z},\bold{B},\Theta,W)$. Finally, we introduce the function $s: \mathcal{V} \mapsto \mathcal{K}$ where $\mathcal{K} = \{1,\dots,k\}$ denotes the community set. So that $s$ maps the nodes to their respective communities.
\end{remark}
\begin{remark}
The literature on the (degree) corrected stochastic block model is divided between two approaches: one that explicitly models the distribution of the membership matrix $\bold{Z}$, often assuming a multinomial distribution \citep{Wang2017,yan2014}, and another that treats $\bold{Z}$ as a fixed but unknown matrix \citep{Qin2013,Ma2021,Gudhmundsson2021}. The distinction in how $\bold{Z}$ is interpreted is subtle and largely depends on the type of research questions being addressed. When $\bold{Z}$ is treated as a random matrix, it allows for probabilistic inquiries, such as determining the likelihood that an element belongs to group $j \leq k$. However, this type of question is not the focus of our investigation. Accordingly, we adopt the latter approach, treating $\bold{Z}$ as a fixed but unknown matrix.

\end{remark}
\begin{remark}\label{Rem:IdentificationUpToScale}
    It is somewhat nontraditional to restrict the node-specific probabilities to be less than $1$, i.e. interpret them as probabilities (\citealp{Su2019,Ma2021}). However, we like it for various reason. First of all, since $\theta_i\theta_j b_{s(i)s(j)}$ represents a probability, the parameters in GSBM2 are only identified up to scale. As such, there is no cost of making this restriction in the sense that GSBM2 becomes less general. To see this, let $\mathcal{G}\sim GSBM2
(\bold{Z},\bold{B},\Theta,W)$ and for some $\tau>0$, let $\tilde{\mathcal{G}}\sim GSBM2
(\bold{Z}, \tau\bold{B},\Theta/\sqrt{\tau},W)$. Let $g$ be a possible realisation of the $GSBM2
(\bold{Z},\bold{B},\Theta,W)$. Here $g_{ij}\in\{0,1\}$ records whether an edge is present between nodes $i$ and $j$. Since the likelihood of $\mathcal{G}$ is $d^2$ independent Bernoulli trials, we have that 
    \begin{equation*}
        \begin{aligned}
        \mathbb{P}\left(\mathcal{G} = g\right) &= \prod_{i,j =1}^d \left(\theta_i\theta_j b_{s(i)s(j)}\right)^{g_{ij}}\left(1-\theta_i\theta_j b_{s(i)s(j)}\right)^{1-g_{ij}}\\
        &=\prod_{i,j =1}^d \left(\frac{\theta_i}{\sqrt{\tau}}\frac{\theta_j}{\sqrt{\tau}} \tau b_{s(i)s(j)}\right)^{g_{ij}}\left(1-\frac{\theta_i}{\sqrt{\tau}}\frac{\theta_j}{\sqrt{\tau}} \tau b_{s(i)s(j)}\right)^{1-g_{ij}}\\
        &=  \mathbb{P}\left(\tilde{\mathcal{G}} = g\right).
        \end{aligned}
    \end{equation*}
    Secondly, knowing that $1$ is an upper bound to any $\theta_j$ does simplify some of the proof in \Cref{Sec:Cons}. Thirdly, allowing $\Theta$ to depend on $d$ lets us capture different sparsity regimes while keeping $\bold{B}$ independent of $d$. This differs from much of the existing literature, where the block-probability matrix is allowed to vary with $d$ \citep{Su2019,Ma2021,Gudhmundsson2021}. The fixed-$\bold{B}$ formulation is useful for our analysis: \Cref{prop:eigenvaluesnindependence} shows that an eigenvalue condition imposed in, for example, \citet{Gudhmundsson2021} holds automatically in our setting.

\end{remark}
The matrices \(\bold{B}\) and \(\bold{\theta}\) play a crucial role in determining the strength of the group structure in the \( \text{GSBM}_2 \). The matrix \(\bold{B}\) essentially carries the signal indicating which nodes belong to the same communities. Intuitively, for any \(v \in \{1, \dots, k\}\), we expect \(b_{vv} > b_{vl}\) for \(l \neq v\), meaning that, all else being equal, edges are more likely to form between nodes within the same community than between nodes in different communities. The larger the difference between the probabilities of within-community and between-community links, the stronger the group signal.

While \(\bold{\theta}\) does not contain direct information about the group structure, it significantly influences how clearly the structure is observed. This is because the node-specific probabilities in \(\bold{\theta}\) predominantly determine whether a link is established. Recall that the probability of an edge between nodes \(i\) and \(j\) is given by \(\mathbb{P}((i,j) \in \mathcal{E}) = \theta_i \theta_j b_{s(i)s(j)}\) for \(i \neq j\). If the node-specific probabilities are very small, few links are created. Even if there is a strong distinction between within-community and between-community link probabilities, the lack of edges can obscure this signal, making it difficult for an algorithm to detect the communities.

Moreover, even a single erroneous link between nodes from different communities can disrupt clustering. Conversely, as the node-specific probabilities \(\bold{\theta}\) increase, the signal carried by \(\mathbf{B}\) becomes more prominent. In this way, \(\bold{\theta}\) influences the overall clarity of the group signal, effectively modulating its strength.

This relationship is illustrated in \Cref{fig:combined_signal_illustration}, where we plot several realizations of a \( \text{GSBM}_2 \) using a force-directed graph. The figure demonstrates that both \(\bold{B}\) and \(\bold{\theta}\) play a role in determining the ease of group detection. Although this is only a visual test, we will later see that our clustering algorithm struggles in scenarios where the signal is weak.

\begin{figure}[htbp]
    \centering
    \caption{Illustrations of different signal strength.}
    % First subfigure (top)
    \begin{subfigure}[b]{\textwidth}
        \centering
        \includegraphics[scale=0.35]{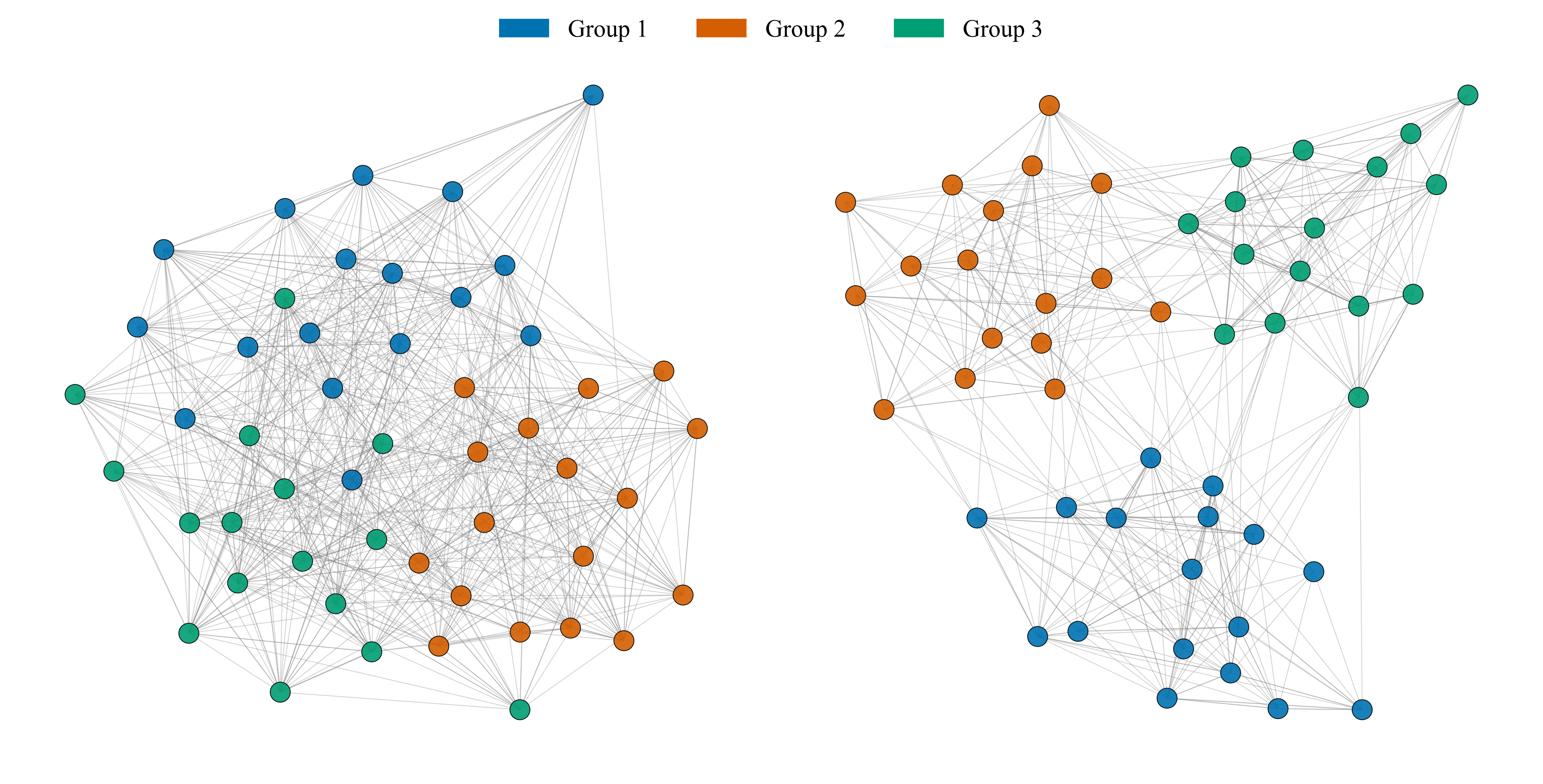}
        \caption{Illustration of consequences of a weak vs. strong signal (carried in $\bold{B}$)}
        \label{fig:signal_vs_nosignal}
    \end{subfigure}

    % Second subfigure (bottom)
    \begin{subfigure}[b]{\textwidth}
        \centering
        \includegraphics[scale=0.35]{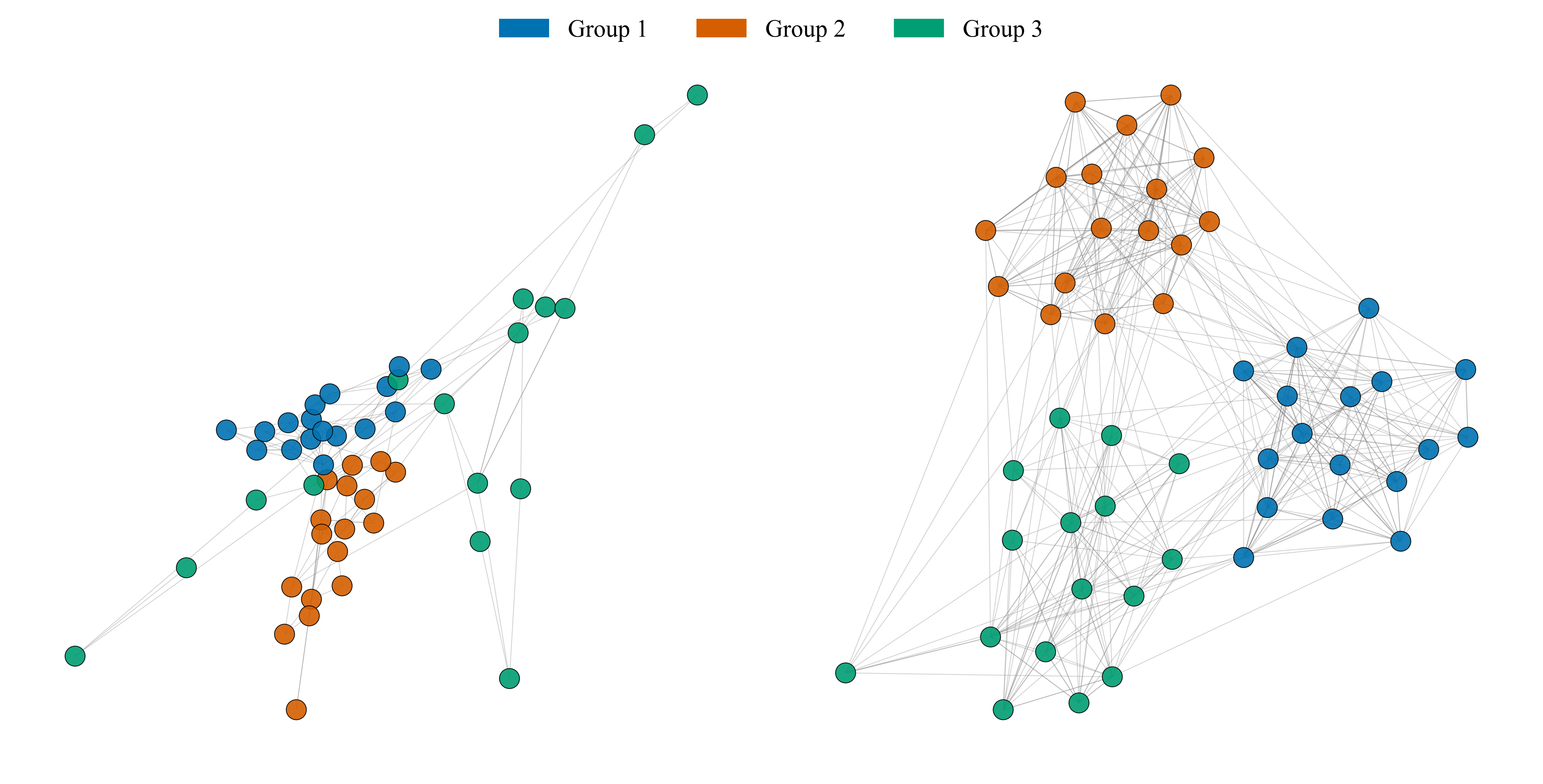}
        \caption{Illustration of a strong signal in conjunction with few links as a consequence of low values in $\bold{\Theta}$}
        \label{fig:weak_vs_strong}
    \end{subfigure}
    \begin{scriptsize}
    \parbox{0.98\textwidth}{\emph{Note: }Figure (a) displays two realisation of a GSBM$2$ with $d=50$ and $k=3$. In both figures $\Theta_{ii}=\theta_i = 1$ for any $i$. In the left hand side figure, $b_{s(i)s(j)}$ = $0.5$ when $s(i)=s(j)$ and $0.2$ otherwise. In the right hand side figure, $b_{s(i)s(j)}$ = $0.5$ when $s(i)=s(j)$ and $0.05$ otherwise. Figure (b) displays two realisations of a GSBM$2$ with $d=50$ and $k=3$. In both figures $b_{s(i)s(j)}$ = $0.5$ when $s(i)=s(j)$ and $0.05$. In the left hand side figure $\Theta_{ii}=\theta_i = 0.5$ for any $i$. In the right hand side figure, $\Theta_{ii}=\theta_i = 1$ for any $i$. The plots are so-called force-directed graphs. Meaning that nodes repulse each-other, but links between nodes pull them together. If there a many links between nodes in the same group this in turn highlights the group structure.} \end{scriptsize}
    \label{fig:combined_signal_illustration}
\end{figure}

With the above in mind, we are now ready to introduce our SBOU process. We let $\mathbb{L}=\left\{\mathbf{L}_t, t \geq 0\right\}$ denote a two-sided, $d$-dimensional Lévy process with characteristics $(\mathbf{b}, \Sigma, F)$, i.e. $\mathbf{b} \in \mathbb{R}^d, \Sigma$ a symmetric positive semi-definite $d \times d$ matrix and $F$ a Lévy measure on $\mathbb{R}^d$, that is $F(\{0\})=0$ and $\int_{\mathbb{R}^d}\left(\|x\|^2 \wedge 1\right) F(d x)<\infty$.  We consider a $d$-dimensional Ornstein-Uhlenbeck (OU) process, $\mathbb{Y}=\{\bold{Y}_t, t \ge 0\}$ where $\bold{Y}_t = \left(\bold{Y}_t^{(1)}\dots ,\bold{Y}_t^{(d)}\right)$ driven by $\mathbb{L}$ with drift matrix $Q \in \mathcal{M}_d(\mathbb{R})$ satisfying the stochastic differential equation (SDE)
\begin{equation}\label{eqn:1}
\mathrm{d} \mathbf{Y}_t=-Q \mathbf{Y}_{t-} \mathrm{d} t+\mathrm{d} \mathbf{L}_t, \quad t \geq 0.
\end{equation}
\begin{definition}[SBOU Process]\label{SBOUProc}
Let \(\mathcal{G}\) be a draw from \(GSBM2(\bold{Z}, \bold{B}, \Theta, W)\), and let \(A\) be the corresponding adjacency matrix, where the \((i,j)\)-th entry is \(w_{ij}\) if there is a directed edge from \(i\) to \(j\). Define \(D\) as the \(d\)-dimensional diagonal matrix with the \(i\)-th diagonal entry given by
\[
D_{ii} = \sum_{j=1}^d |A_{ij}| + \sum_{j=1}^d |A_{ji}|,
\]
which represents the sum of the absolute in-degrees and out-degrees of the \(i\)-th vertex in \(A\).

The SBOU process is defined as the unique strong, \emph{strictly stationary} solution (when it exists) to \eqref{eqn:1} with drift matrix \(Q(\psi) = \psi_1 I + \Psi\), where \(\Psi = \psi_2 D^{-1/2} A^\top D^{-1/2}\), and \(\psi = (\psi_1, \psi_2) \in \mathbb{R}^2\). Throughout, \(\psi_1\) and \(\psi_2\) are fixed constants that do not depend on \(d\).
\end{definition}

\begin{remark}
    Since $D^{-1/2}$ is not guaranteed to be well-defined (for finite $d$, some rows of $D$ may have all $0$ entries), by $D^{-1/2}$ we mean the Moore--Penrose pseudo-inverse of $D^{1/2}$, i.e.\ the diagonal matrix whose $i$th diagonal entry equals $D_{ii}^{-1/2}$ if $D_{ii}>0$ and $0$ otherwise.
\end{remark}
The close relationship between the SBOU process and the graph Ornstein-Uhlenbeck (OU) process from \cite{Courgeau2022a} is evident. As in that model, the components of \(\mathbb{Y}\) represent the nodes of a graph, linked through a collection of edges specified by the matrix \(Q(\psi)\). The dynamics of the \(i\)-th component of the SBOU process are governed by the following equation:
\begin{equation}\label{eqn.Intro:2}
    d \bold{Y}_t^{(i)} = -\psi_1 \bold{Y}_{t-}^{(i)} \, dt -  \sum_{j \neq i} \Psi_{ij} \bold{Y}_{t-}^{(j)} \, dt + d \bold{L}_t^{(i)}, \quad t \ge 0.
\end{equation}
In this equation, the first term is referred to as the momentum effect, while the second term represents the network effect. The dynamics of \(\bold{Y}_t^{(i)}\) are influenced by \(\bold{Y}_t^{(j)}\) if and only if \(\Psi_{ij} \neq 0\), i.e., when there is a directed edge from node \(j\) to node \(i\). The \((i,j)\)-th entry of \(\Psi\) is given by:
\begin{equation}\label{eqn:2}
\Psi_{ij} = \begin{cases}
\psi_2 \dfrac{A_{ji}}{\sqrt{D_{ii}D_{jj}}}, & \text{if there is an edge from } j \text{ to } i, \\
0, & \text{otherwise},
\end{cases}
\end{equation}
where \(A_{ji}\) is the \((j,i)\)-th entry of a realization of a GSBM2. This shows that a group structure has been incorporated into the OU process since the likelihood of \(A_{ji}\) being non-zero is precisely given by \(\theta_i \theta_j b_{s(j)s(i)}\), which is higher when \(s(j) = s(i)\), meaning the nodes belong to the same community.

Since \(A\) is the adjacency matrix corresponding to a realization of a random graph generated by GSBM2, and this realization is unobserved, the group structure itself remains unknown.

\begin{remark}\label{Remark.3}
    We highlight a few things with respect to the construction of the network effect in our SBOU process. The standardisation is important for three reasons. Firstly, it will be important to achieve stationarity of our process. Secondly, as argued in \cite{Gudhmundsson2021}, it plays an important role in establishing the convergence properties of $\Psi +\Psi^\top$ as $d$ tends to infinity. Thirdly, it also has the interpretation that nodes with few edge-connections affect the few nodes they affect more than nodes with many edges all else equal.They also affect the few edges they do affect more. In a similar vein, nodes with many edges are affected less by each edge-connection and also affect each edge-connection less. Depending on the context, this may or may not be an appealing feature. Indeed, one could imagine that high-degree nodes exert less influence on each individual node they are connected to, simply because their influence or resources are distributed among many connections. On the other hand, high-degree nodes can have a significant influence because they act as hubs through which information or influence flows which does not diminish with the number of connections.
\end{remark}
\begin{remark}
    We also underline that we only require the support of $W$ to be compact, not also positive as in \cite{Gudhmundsson2021}. This has the more realistic interpretation of us allowing nodes to affect the value of each other in a positive as well as negative manner. This is aligned with the assumptions in the literature on the drift matrix of a multivariate Ornstein-Uhlenbeck process. See e.g. \citealp{Gaiffas2019,Fasen2013} who use it for pairs-trading. Us working with this somewhat unusual degree matrix as the sum of absolute values of the nodes is a direct consequence of allowing for this. Henceforth, we call it the absolute degree matrix.
\end{remark}
In the following, we give easy to check conditions for when the SBOU-process is well-defined.
\begin{theorem}\label{Thm.1}
Let \( A \) be as defined in \Cref{SBOUProc}, and let \(\bar{D}\) be the diagonal matrix where the \((i,i)\)-th entry is given by \(\bar{D}_{ii} = \sum_{j=1}^d |A_{ij}|\), and \(\tilde{D}\) be the diagonal matrix with \(\tilde{D}_{ii} = \sum_{j=1}^d |A_{ji}|\). Define \(\tau = \min_{i} \left(\dfrac{\bar{D}_{ii}}{\tilde{D}_{ii}} \right)\). Assume that \(\mathcal{L}(\boldsymbol{Y}_0)\) is the invariant distribution of \(\mathbb{Y}\), and that \(\psi = (\psi_1, \psi_2) \in \mathbb{R}^{+} \times \mathbb{R}^{+}\) satisfies \(\psi_1 > \psi_2 / (1 + \tau)\). Further, assume the Lévy measure \(F(\cdot)\) satisfies the moment condition \(\int_{\Vert z \Vert > 1} \ln \Vert z \Vert \, F(dz) < \infty\).

Then, for \(Q = Q(\psi)\) as defined in \Cref{SBOUProc}, there exists a unique, strictly stationary solution to \eqref{eqn:1}.
\end{theorem}

\begin{proof}
    See \Cref{ProofsModelClusteringSelection}.
\end{proof}
    \begin{remark}
        We note that, although $\psi_2>0$, it does not imply that the network effects are only positive since the edge-weights are allowed to be negative.
    \end{remark}

\section{Auxiliary drift estimation result}\label{Sec:EstShort}

Our community detection procedure requires an estimator of the drift matrix \(Q_0\) whose estimation error vanishes sufficiently quickly as \(d \to \infty\). Since drift estimation is auxiliary to the main contribution of the paper, we defer the full setup and most assumptions to \Cref{App:Est}. Here we only sketch the main ideas underlying the proof.

A substantial literature already exists on drift estimation for multivariate OU processes. In the Gaussian setting, that is, when \(\mathbb{L}=\mathbb{W}\), \citet{Ciolek2020,Gaiffas2019} study estimation under continuous-time observation schemes and under the additional assumption that the continuous \(\mathbb{P}^0\)-martingale component \(\mathbb{X}^c\) is observed. The estimators considered are the Dantzig and Lasso estimators in \citet{Ciolek2020}, and the Lasso and adaptive Lasso estimators in \citet{Gaiffas2019}. These results were extended by \citet{Dexheimer2024} to the Lévy-driven case, allowing \(\mathbb{L}\) to be an arbitrary Lévy process, for the Dantzig and Slope estimators. However, all of these results rely on continuous observation schemes and are therefore, in essence, infeasible in practice.

In \Cref{App:Est}, we develop a feasible estimator that can be used in the community detection procedure of \Cref{Sec:Cons}. For our purposes, the maximum likelihood estimator is sufficient, and we therefore focus on this case. The argument proceeds in three steps. We introduce three estimators based on three likelihood functions that become progressively more feasible: we begin with an estimator based on continuous observations and end with one based only on discrete observations. The proof strategy is to show that the relevant consistency property, formulated appropriately, is transferred from the least feasible estimator to the next, and finally to a fully feasible estimator.

Let \(Q_0\) denote the true drift matrix of the OU process defined in \Cref{eqn:1}. It is well known that varying \(Q\) in \Cref{eqn:1} induces a family of probability measures on the canonical path space; see \Cref{Prelim:EstSec}. On this path space, we represent the OU process by \(\mathbb{X}\), and write \(\mathbb{P}^{Q}\) for the probability measure induced by the matrix  \(Q\). We denote by \(\mathbb{X}^c\) the \(\mathbb{P}^0\)-continuous martingale part of \(\mathbb{X}\), i,.e. the case with no drift.

As a first step, we show that the Frobenius norm error of the maximum likelihood estimator \(\widehat Q\) can be controlled with high probability. The arguments follow closely those in \citet{Ciolek2020,Dexheimer2024}, where analogous bounds are derived for related estimators and can be adapted to the MLE with only minor modifications.

A technical contribution of this step is to provide an alternative to the diagonalizability assumption on \(Q_0\) used in both \citet{Ciolek2020} and \citet{Dexheimer2024}. In the cited work, diagonalizability is invoked solely to bound the operator norm of the matrix exponential of \(Q_0\) in terms of the smallest real part of its eigenvalues, with a prefactor equal to the condition number of the eigenvector matrix. We show that the same exponential operator-norm decay, without any condition-number prefactor, holds under the requirement that the smallest eigenvalue of the symmetrized matrix \(Q_0 + Q_0^\top\) is strictly positive. The two conditions are independent: neither implies the other, as we illustrate with explicit two-dimensional counterexamples in \Cref{App:Est}. Our condition is, however, the natural one for the SBOU process introduced in \Cref{SBOUProc}, whose drift matrix is generically non-diagonalizable but for which the symmetrized-positivity condition is easily verified whenever the momentum effect dominates the network effect; see \Cref{App:Est}. 

Having established this first result, we next move to discrete-time observations of \(\mathbb{X}\), while still assuming (for now) that the continuous \(\mathbb{P}^0\)-martingale component \(\mathbb{X}^c\) is separately observed. We then construct an estimator \(\widehat Q^{\mathcal P_T}\) based on possibly non-equidistant discrete observations and show that it essentially inherits the Frobenius norm bound established for \(\widehat Q\). In this step, one additional assumption is needed: the mesh of the observation grid must shrink sufficiently fast relative to \(d\). This ensures that the restricted eigenvalue condition holds, a standard requirement in the theory of \(\ell_1\)-penalized estimators; see \citet{Bickel2009,Buhlmann2011}. In the non-Gaussian case, \citet{Dexheimer2024} similarly require \(T\) to grow sufficiently quickly in order to guarantee this condition.

Finally, we drop the assumption that the continuous \(\mathbb{P}^0\)-martingale component is observed, which is the practically relevant setting. The main difficulty is that, from discrete observations alone, one cannot separate the contribution of the Brownian part from that of the jump part. The standard approach is therefore to filter out jumps in order to recover a good proxy for the continuous martingale part. We adopt this strategy and define \(\bar Q^{\mathcal P_T}\) as the corresponding estimator. This is the first fully feasible estimator in our sequence. The assumptions required at this stage concern the thresholding scheme and the jump behavior of \(\mathbb{L}\). Since these are close to standard conditions in the literature, see \citet{Mai2014,Lucchese2023,Courgeau2022b}, we do not state them here. The main difference is that, in our setting, the dimension \(d\) also tends to infinity, whereas the cited papers keep \(d\) fixed. Accordingly, both the observation mesh and the threshold sequence must be chosen so as to vanish sufficiently quickly with \(d\). Under these conditions, we show that the convergence properties of \(\widehat Q^{\mathcal P_T}\) are transferred to \(\bar Q^{\mathcal P_T}\).

We summarize the result as follows.

\begin{theorem}\label{Thm:AuxiliaryEstimation}
Under the assumptions stated in \Cref{App:Est}, the estimator \(\bar Q^{\mathcal P_T}\) satisfies
\[
\|\bar Q^{\mathcal P_T} - Q_0\|_F \le h(d,\mathcal P_T)
\]
with high probability, where \(h(d,\mathcal P_T)=o(1)\) as \(d\to\infty\). The explicit form of \(h\) is given in \Cref{Rem:functionalFormf}; see \Cref{DiscreteEstimatorConsistencyBothCases} for the full statement.
\end{theorem}
\begin{proof}
    See \Cref{DiscreteEstimatorConsistencyBothCases}.
\end{proof}

\section{Community detection algorithm: Methodology and consistency}\label{Sec:Cons}
The next step is, given an estimate of the drift component in the SBOU process, to cluster the nodes of the underlying GSBM2. We shall assume that we observe an SBOU process (\Cref{SBOUProc}) on a discrete grid $\mathcal{P}_T:=\left\{0=s_0^{N_T}<s_1^{N_T}<\ldots<s_{{N_T}}^{N_T}=T\right\}$.  In the following, we shall introduce our clustering algorithm and prove that the share of correctly clustered nodes will tend to $1$ in what we shall call a high-dimensional, in-fill and long-span setup. That is, as the grid-size tends to $0$, the time-horizon $T$ tends to infinity and as the dimension of our observed SBOU process tends to infinity.  
\subsection{Methodology}
In this section, we introduce our clustering algorithm. Let \(\Psi^s := \Psi + \Psi^\top\) denote the symmetrized version of the network effect in \Cref{SBOUProc}, and let \(\widehat{\Psi}^s\) be an estimate of \(\Psi^s\). We construct \(\widehat{\Psi}^s\) from the MLE \(\widehat{Q}\) of the drift matrix described in \Cref{FeasibleEstimator} by setting \(\widehat{\Psi} := \widehat{Q} - \mathrm{diag}(\widehat{Q})\) and \(\widehat{\Psi}^s := \widehat{\Psi} + \widehat{\Psi}^\top\). This construction matches the structural property that the true \(\Psi\) has zero diagonal, since the underlying GSBM2 has no self-loops, and avoids the need for a separate estimator of \(\psi_1\): the dependence on \(\psi_1\) is absorbed into the discarded diagonal. Define the \(d \times k\) matrix \(\widehat{\mathcal{X}}\), which contains the eigenvectors corresponding to the \(k\) largest eigenvalues (in absolute value) of \(\widehat{\Psi}^s\). Furthermore, let \(\widehat{\mathcal{X}^*}\) be the row-normalized version of \(\widehat{\mathcal{X}}\), with respect to the Euclidean norm. This matrix contains the necessary information for performing clustering.

Following \cite{Gudhmundsson2021}, we present Algorithm \ref{alg:VARBlockbuster} for community detection.

\begin{algorithm}
\caption{The Ornstein-Uhlenbeck clustering algorithm}
\label{alg:VARBlockbuster}
\textbf{Input:} Sample \(\bold{Y}_{s_n^{N_T}}\) for \(n = 1, \dots, N_T\) from SBOU processes, with number of groups \(k\). \\
\textbf{Procedure:}
\begin{algorithmic}[1]
    \State Estimate the matrix \(\widehat{\Psi}\) using the method described in \Cref{FeasibleEstimator}.
    \State Compute the symmetrized matrix \(\widehat{\Psi}^s = \widehat{\Psi} + \widehat{\Psi}^\top\).
    \State Construct the eigenvector matrix \(\widehat{\mathcal{X}} \in \mathbb{R}^{d \times k}\), where the columns correspond to the eigenvectors associated with the \(k\) largest eigenvalues (in absolute value) of \(\widehat{\Psi}^s\).
    \State Normalize each row of \(\widehat{\mathcal{X}}\) by its Euclidean norm to obtain the row-normalized matrix \(\widehat{\mathcal{X}^*}\), i.e., \(\widehat{\mathcal{X}^*}_{ij} = \frac{\widehat{\mathcal{X}}_{ij}}{\Vert \widehat{\mathcal{X}}_{i\bullet} \Vert_2}\).
    \State Apply \(k\)-means clustering to the rows of \(\widehat{\mathcal{X}^*}\), taking a global minimizer of the within-cluster sum of squares.
\end{algorithmic}
\textbf{Output:} The \(k\)-means partition \(\widehat{\mathcal{V}} = \{ \widehat{\mathcal{V}}_1, \dots, \widehat{\mathcal{V}}_k \}\), which serves as an estimate of the group partition.
\end{algorithm}
\begin{remark}
    In general $\Vert \widehat{\mathcal{X}}_{i\bullet} \Vert_2$ is not guaranteed to be positive, so we adopt the convention $0/0 := 0$ to ensure that $\widehat{\mathcal{X}^*}$ is well-defined. We will later introduce its population counterpart $\mathcal{X}^*$, for which no such convention is needed: as shown in the proof of \Cref{Lemma:2}, none of its rows are equal to zero.
\end{remark}
At first glance, it may seem unclear why Algorithm \ref{alg:VARBlockbuster} is effective. However, we will provide the underlying motivation later in this section, where we will also formally demonstrate that the proportion of mis-clustered nodes tends to zero as \(T\) and \(d\) tend to infinity, while the grid size tends to zero in a specific manner.

\subsection{Theory}
  The proof of consistency for our community-detection algorithm requires distinguishing carefully between graphs sampled from a GSBM2 (no self-loops allowed, the model used in this paper) and a regular GSBM (self-loops allowed, as in \cite{Gudhmundsson2021}). Throughout, we use bars to denote quantities derived from the GSBM and no bar for the GSBM2 analogues; calligraphic letters denote expectations of their non-calligraphic counterparts (e.g.\ $\mathcal{A} = \mathbb{E}[A]$):
  \begin{center}
  \begin{tabular}{lll}
  \toprule
  Object & GSBM2 (no self-loops) & GSBM (self-loops) \\
  \midrule
  Adjacency matrix                       & $A$              & $\bar A$ \\
  Absolute degree matrix                 & $D$              & $\bar D$ \\
  Expected adjacency                     & $\mathcal{A}$    & $\bar{\mathcal{A}}$ \\
  Expected absolute degree               & $\mathcal{D}$    & $\bar{\mathcal{D}}$ \\
  Network effect                         & $\Psi$           & $\bar\Psi$ \\
  Symmetrized network effect             & $\Psi^s$         & $\bar\Psi^s$ \\
  Population symmetrized network effect  & $\Psi^s_p$       & $\bar\Psi^s_p$ \\
  \bottomrule
  \end{tabular}
  \end{center}
  The network effect $\Psi$ is defined as in \Cref{SBOUProc}; $\Psi^s := \Psi + \Psi^\top$ is its symmetrized version (also see Remark \ref{Remark.3}); and $\Psi^s_p$ replaces the realisations $A, D$ in $\Psi^s$ with their expectations $\mathcal A, \mathcal D$. The barred versions are defined analogously for the GSBM.

With this established, we impose the following assumption on the sum of node specific probabilities within each community
\begin{assumption}\label{Ass.NodeSum}
    For any $l\in \{1,2,\dots,k\}$, it holds that 
    \begin{equation*}
        \sum_{i\in \mathcal{V}_l}\theta_i = \sigma_d,
    \end{equation*}
    where $\left(\sigma_d\right)_{d\in \mathbb{N}}$ is some sequence.
\end{assumption}
\begin{remark}
    For $\sigma_d=1$, we are in the 'standard' GSBM framework. Further, since we are working with probabilities, clearly, $\sigma_d \le d$ with a strict inequality for $k>1$.
\end{remark}
The assumptions on $\sigma_d$ will be crucial for our consistency result. Indeed, the sparsest regime in which exact recovery is possible for the stochastic block model is when the minimal expected degree grows as $\Omega(\log d)$ as $d \to \infty$ \citep{Abbe2015}. The following propositions finds the proper rate of $\sigma_d$ to ensure we are in this regime.
\begin{proposition}\label{Prop:LBMinNode}
 For a GSBM2, satisfying Assumption \ref{Ass.NodeSum}, it holds that 
 \begin{equation*}
     \min_i \theta_i  \le \dfrac{k\sigma_d}{d}.
 \end{equation*}
\end{proposition}
\begin{proof}
    See \Cref{ProofsModelClusteringSelection}.
\end{proof}
\begin{remark}
   Another implication of \Cref{Prop:LBMinNode} is that the cardinality of each community, $\vert \mathcal{V}_l \vert$, is at least $\nint{\sigma_d}$, since each $\theta_i$ is a probability.
\end{remark}
\begin{proposition}\label{Prop:SumReg}
     Let $\mathcal{\bar{D}}$, $\mathcal{D}$ denote the absolute expected degree matrices associated to respectively the GSBM and GSBM2. Grant \Cref{Ass.NodeSum}. If $\sigma_d = o\left((d \log d)^{1/2} \right)$, then
\begin{equation*}
    \begin{aligned}
        \min_{i\in \{1,2,\dots,d\}}\mathcal{D}_{ii} &= o(\log d), \\
        \min_{i\in \{1,2,\dots,d\}}\bar{\mathcal{D}}_{ii} &= o(\log d), \quad \text{as d} \rightarrow \infty.
    \end{aligned}
\end{equation*}
 \end{proposition}
 \begin{proof}
See \Cref{ProofsModelClusteringSelection}.     
 \end{proof}
\begin{remark}
    The above result tells us that, in the classical setting with $\sigma_d = 1$, we cannot hope to derive any concentration results.
\end{remark}
 With the above in mind, we now impose the following assumptions:
\begin{assumption}\label{Ass.2}
    \begin{itemize}
        \item We assume that $\bold{B}+\bold{B}^\top$ is positive definite and $\bold{B}$ is independent of $d$. 
         \item We assume that  $\sigma_d= \omega\left( (d\log d)^{1/2 }\right)$ as $d \rightarrow \infty$; that is, $\sigma_d/(d\log d)^{1/2}\to\infty$.
        \item We assume that $\min_i \theta_i = \Omega \left(\dfrac{\sigma_d}{d}\right)$ as  $d \rightarrow \infty$.
        \item We assume that $\mu \ne 0$.
    \end{itemize}
\end{assumption}
\begin{remark}
  The first assumption is very intuitive and essentially establishes that nodes in the same group of the GSBM2 should have a higher probability of links between them than nodes in different groups. Indeed, following the Example in \cite{Gudhmundsson2021}, for $k=2$ the assumption is satisfied if $b_{ii}>b_{ij}$ for $i,j \in \{1,2\}$, $i\ne j$. The aware reader will also note that, since $\bold{B}$ is independent of $d$, the number of communities $k$ is held fixed. The third assumption on the behaviour of the minimal node-specific in Assumption \ref{Ass.2} will be key to not only get a dense enough degree matrix to derive concentration results (see e.g. \Cref{Prop:MinimalDegree,Prop:SumReg}), but also to establish Proposition \ref{EigenvectorBound}. It will, however, have important consequences for the behaviour of all other node-specific probabilities. Clearly a consequence of this assumption is that, for any $j$, $\theta_j=\Omega \left(\dfrac{\sigma_d}{d}\right)$. The last assumption simply states that the mean edge-weight should not be $0$.
\end{remark}

\begin{proposition}[Minimal expected degree]\label{Prop:MinimalDegree}
    Under Assumptions \ref{Ass.NodeSum} and \ref{Ass.2}, it holds that 
    \begin{equation*}
        \begin{aligned}
    \bar{\mathcal{D}}_{\min}&= \min_{i\in \{1,2,\dots,d\}} \sum_{j=1}^d\mathbb{E}\vert\bar{A}_{ij} \vert + \mathbb{E}\vert\bar{A}_{ji} \vert = \Omega\left( \dfrac{\sigma_d^2}{d}\right), \\
    \mathcal{D}_{\min}&= \min_{i\in \{1,2,\dots,d\}} \sum_{j=1}^d\mathbb{E}\vert A_{ij} \vert + \mathbb{E}\vert A_{ji} \vert = \Omega\left( \dfrac{\sigma_d^2}{d}\right), \quad \text{ as } d\rightarrow \infty.
    \end{aligned}
    \end{equation*}
\end{proposition}
\begin{proof}
    See \Cref{ProofsModelClusteringSelection}.
\end{proof}
\begin{remark}\label{RemarkNodeProb}
    Proposition \ref{Prop:MinimalDegree} heavily relies on \Cref{Ass.2} and, hence, the assumption on $\min_i \theta_i$ is needed to get a dense enough graph to get any concentration results. The above also clearly implies that $\bar{\mathcal{D}}_{\min} = \Omega\left( \log d\right)$ by choice of $\sigma_d$.
\end{remark}
Knowing that, under our assumptions, there is hope of bounding the terms in \eqref{eqn:1}, we introduce the following lemmas, analogue to Lemmas $3.2$ and $3.3$ in \cite{Qin2013}, slightly adapted to our setting, which motivates our clustering strategy.

\begin{lemma}[Explicit form of $\bar{\Psi}^s_p$]\label{Lemma:1}
Grant \Cref{Ass.NodeSum}. Let $\bar{\Psi}^s_p$ be from a GSBM with parameters $\left(\bold{Z},\bold{B}, \Theta, \bold{W} \right)$. Introduce the matrix $D_B \in \mathcal{M}_k(\mathbb{R})$ defined as a diagonal matrix with $(i,i)$ element given by $[D_B]_{ii} = \sum_{j=1}^k (b_{ij}+b_{ji})$. Here $b_{ij}$ denotes the $(i,j)$th entry of $\bold{B}$. Lastly, define $B_L = D_B^{-1/2}\left( \bold{B}+ \bold{B}^\top\right)D_B^{-1/2} $ and $\Theta_{\sigma_d}=\Theta/\sigma_d$. Then we can write $\bar{\Psi}^s_p$ as 
\begin{equation*}
    \begin{aligned}
    \bar{\Psi}^s_p &=\dfrac{\mu}{\bar{\mu}} \psi_2 \Theta_{\sigma_d}^{1/2}ZB_LZ^\top \Theta_{\sigma_d}^{1/2},
    \end{aligned}
\end{equation*}
where $\bar{\mu}= \mathbb{E}\left(\vert W\vert \right)$ and $\mu= \mathbb{E}\left( W \right)$.
\end{lemma}
\begin{proof}
    See Supplementary material \ref{ProofsModelClusteringSelection}.
\end{proof}
  We let $\widehat{\mathcal{X}}$ and $\mathcal{X}$ be the $d\times k$ matrices containing the eigenvectors corresponding to the $k$ largest eigenvalues (in absolute size) of $\widehat{\Psi}^s$ and $\bar{\Psi}^s_p$, respectively. Further, we let $\mathcal{X}^*$ be the corresponding row-normalised matrices, i.e.~$\mathcal{X}^*=\mathcal{N}\mathcal{X}$, where  $\mathcal{N}$ is a diagonal matrix with $i$th diagonal entry equal to $1/\Vert \mathcal{X}_{i\bullet }\Vert_2$. The next lemma provides the key insight to our clustering algorithm. Indeed, it establishes that $\mathcal{X}^*$ contains all information regarding the community structure of our SBOU process.

\begin{lemma}\label{Lemma:2}
    Grant \Cref{Ass.NodeSum,Ass.2}. Consider the symmetrized population network matrix associated to the GSBM, $\bar{\Psi}^s_p$. Let $\mathcal{X}\in\mathcal{M}_{(d,k)}\left( \mathbb{R}\right)$ be the matrix containing the $k$ eigenvectors associated with the $k$ largest eigenvalues of $\bar{\Psi}^s_p$ in absolute size. Let $\mathcal{X}^*$ be the population row-normalised eigenvector matrix associated with the symmetrized population network matrix, i.e.~$\mathcal{X}^*$ is defined to be the matrix with $(i,j)$th entry equal to $\dfrac{\mathcal{X}_{ij}}{\Vert\mathcal{X}_{i\bullet} \Vert_2}$. Finally, let $\vert\lambda_1(\bar{\Psi}^s_p)\vert\ge \vert\lambda_2(\bar{\Psi}^s_p)\vert \ge \dots \ge \vert\lambda_d(\bar{\Psi}^s_p)\vert$ denote the absolute values of the eigenvalues of $\bar{\Psi}^s_p$. Then, $(i)$ we can factorize $\mathcal{X}^*$ as
    \begin{equation*}
        \mathcal{X}^* =\bold{Z}U,
    \end{equation*}
    where $\bold{Z}$ is the community membership matrix of the GSBM and $U$ is a $k\times k$ orthonormal matrix; $(ii)$ $\lambda_i(\bar{\Psi}^s_p) = 0$ for $i=k+1,\dots, d$ and $\vert \lambda_k(\bar{\Psi}^s_p)\vert>0$. In other words, the eigenvectors in $\mathcal{X}$ are the eigenvectors associated to the only non-zero eigenvalues of $\bar{\Psi}^s_p$.
\end{lemma}
\begin{proof}
    See Supplementary material \ref{ProofsModelClusteringSelection}.
\end{proof}
We now state the result alluded to in the introduction. When using spectral clustering, particularly on the graph Laplacian, it is crucial to bound the $k$-th largest absolute eigenvalue of this Laplacian away from zero. This is a common assumption in the literature \citep{Ma2021,Lei2015,Joseph2016,Gudhmundsson2021,Brownlees2022} and plays an important role in controlling the proportion of misclustered nodes in the proof of \Cref{Thm:ConsistencyCommunitytCluster}. However, under the chosen asymptotic setup, where all $d$ -dependence is captured by $\Theta$, this condition naturally follows from the model setup.
\begin{proposition}\label{prop:eigenvaluesnindependence}
    Grant \Cref{Ass.NodeSum,Ass.2}. The $k$ non-zero eigenvalues of $\bar{\Psi}^s_p$ do not depend on $d$.
\end{proposition}
\begin{proof}
    Let the notation be as in \Cref{Lemma:1,Lemma:2}. In the proof of Lemma \ref{Lemma:2}, we show that for any $d\in \mathbb{N}$, the eigenvalues of $C=\psi_2 \dfrac{\mu}{\bar{\mu}}\mathbf{B}_L$ are equal to the absolute value of the $k$ non-zero eigenvalues of $\bar{\Psi}_p^s$, where $B_L$ is defined as in Lemma \ref{Lemma:1}. Since $B_L$ does not depend on $d$, its eigenvalues will also not depend on $d$ thus concluding the proof.
\end{proof}

   \Cref{Lemma:2} motivates our estimation strategy. Indeed, following this lemma, we see that $\mathcal{X}^*_{i\bullet} = \mathcal{X}^*_{j\bullet}$ if and only if $Z_{i\bullet} = Z_{j\bullet}$. Thus, if we were to observe $\mathcal{X}^*$, applying the $K$-means algorithm on it would perfectly identify all nodes to their correct communities. Unfortunately, we do not observe $\mathcal{X}^*$, so our clustering strategy hinges on showing that our estimate $\widehat{\mathcal{X}^*}$ is 'close' to $\mathcal{X}^*$. Consequently, applying the $K$-means algorithm to $\widehat{\mathcal{X}^*}$, instead of $\mathcal{X}^*$, should yield a large proportion of correctly classified nodes. By 'close', we mean that the spectral norm of their differences becomes negligible. To establish this, we apply a version of the Davis-Kahan theorem, which, combined with \Cref{prop:eigenvaluesnindependence}, bounds this difference essentially in terms of $\Vert \widehat{\Psi}^s - \bar{\Psi}^s_p \Vert_F$. Our consistency result relies on showing that $\Vert \widehat{\Psi}^s - \bar{\Psi}^s_p \Vert$ tends to zero as $d$ and $T$ tend to infinity while $\Delta_{\mathcal{P}_T}$ tends to zero in a certain manner. We will bound this by controlling the four terms below. Define $\bar{\Psi}^s_{p-} = \bar{\Psi}^s_p - \mathrm{diag}(\bar{\Psi}^s_p)$. Then,
\begin{equation*}
    \Vert \widehat{\Psi}^s - \bar{\Psi}^s_p  \Vert \le\Vert \Psi_p^s-\bar{\Psi}^s_{p-} \Vert+\Vert \mathrm{diag}(\bar{\Psi}^s_p) \Vert+ \Vert\Psi^s -\Psi^s_p\Vert  +\Vert \widehat{\Psi}^s - \Psi^s \Vert.
\end{equation*}
The above highlights the sources of error that contribute to the magnitude of $\Vert \widehat{\Psi}^s - \bar{\Psi}^s_p \Vert$. The first term addresses the fact that \Cref{Lemma:2} concerns the population version of $\bar{\Psi}^s$, the symmetrized network effect associated to the adjacency matrix of a realisation of a GSBM, whereas \Cref{SBOUProc} has a network effect defined in terms of the realisation of a GSBM2. The first two terms thus capture the difference in terms of the spectral norm between the population symmetrized network effect of a GSBM and the population symmetrized network effect of a GSBM2 with the same parametrization. The third term captures the difference between a given realisation of a GSBM2 and its population version. Proving that this term is negligible thus amount to showing a concentration result. The final term essentially captures the estimation error and is determined by how good an estimator we have at hand. In order to control the last term, $\Vert \widehat{\Psi}^s - \Psi^s \Vert$, we apply our estimation result, which requires conditions $(1)$ and $(4)$ of \Cref{Ass.H} to hold for $Q_0$. Condition $(1)$ is trivially satisfied. Condition $(4)$ is more substantive; it requires the smallest eigenvalue of $(Q_0 + Q_0^\top)/2$ to be bounded away from zero. We impose the additional restriction $\psi_1 > \psi_2$ on the SBOU process parameters, which guarantees condition $(4)$ via \Cref{Prop.5}.
We now aim to use the bound on the spectral norm from Proposition \ref{Prop.4} to derive a bound on the Frobenius norm of the difference between \( \widehat{\mathcal{X}^*} \) and \( \mathcal{X}^* \), up to a rotation. A crucial aspect of this proof is that the smallest nonzero eigenvalue of \( \bar{\Psi}_p^s \) is independent of \( d \), as established in \Cref{prop:eigenvaluesnindependence}. To achieve the result below, \cite{Gudhmundsson2021, Brownlees2022} implicitly assume that 
$
\min_{i}\Vert \widehat{\mathcal{X}}_{i\bullet}\Vert^{-1}$
does not approach zero too rapidly. \cite{Rohe2011} assumes this explicitly. By contrast, \cite{Qin2013} avoids such assumptions, but in return, their upper bound relies on
$
\min_{i}\Vert \widehat{\mathcal{X}}_{i\bullet}\Vert,
$
an empirical term that remains poorly understood. This lack of control over the empirical term introduces uncertainty into their upper bound on the share of misclustered nodes, which is undesirable.

In contrast, using \Cref{RowNormalisedBound}, we obtain an upper bound that depends solely on population values, which are more tractable and controllable. We do so without imposing any additional assumption. This approach also carries over to our consistency result in \Cref{Thm:ConsistencyCommunitytCluster} which does not depend on any empirical terms or assumptions regarding the decay rate on the minimal eigenvalue. This stands in contrast to most of the existing literature.

To control the estimation term $\Vert \widehat{\Psi}^s - \Psi^s\Vert$ in the four-term decomposition above, we impose the following high-level condition on the estimator.

\begin{assumption}\label{Ass.EstimationBound}
There exists a deterministic function $f(d,\mathcal{P}_T)$ with $f(d,\mathcal{P}_T) = o(1)$ as $d \to \infty$ such that, with high probability,
\begin{equation*}
\Vert \widehat{\Psi}^s - \Psi^s \Vert \le f(d,\mathcal{P}_T).
\end{equation*}
\end{assumption}

\begin{remark}\label{Rem:AssEstSatisfiedByMLE}
\Cref{Ass.EstimationBound} is satisfied by the feasible estimator constructed from the MLE 
$\bar{Q}^{\mathcal{P}_T}$ in \Cref{FeasibleEstimator}. Indeed, under the assumptions of 
\Cref{DiscreteEstimatorConsistencyBothCases}, there exists a deterministic function 
$h(d,\mathcal{P}_T)=o(1)$ such that, with high probability,
\[
\Vert \bar{Q}^{\mathcal{P}_T}-Q_0\Vert_F \le h(d,\mathcal{P}_T).
\]
Since $\widehat{\Psi}=\widehat{Q}-\mathrm{diag}(\widehat{Q})$, this implies
\[
\Vert \widehat{\Psi}^s-\Psi^s\Vert
\le 
\Vert \widehat{\Psi}^s-\Psi^s\Vert_F
\le 
2\Vert \bar{Q}^{\mathcal{P}_T}-Q_0\Vert_F
\le 
2h(d,\mathcal{P}_T).
\]
Thus \Cref{Ass.EstimationBound} holds with 
$f(d,\mathcal{P}_T)=2h(d,\mathcal{P}_T)$.
\end{remark}

\begin{proposition}\label{EigenvectorBound}
Grant \Cref{Ass.NodeSum,Ass.2,Ass.EstimationBound} and suppose $\psi_1 > \psi_2$. Let $\widehat{\mathcal{X}^*}$ and $\mathcal{X}^*$ be the row-normalised matrices corresponding to the $k$ eigenvectors associated with the $k$ largest eigenvalues of $\widehat{\Psi}^s$ and $\bar{\Psi}^s_p$, respectively. Then, with high probability,
\begin{equation*}
\Vert \widehat{\mathcal{X}^*} - \mathcal{X}^*\mathcal{O}_d \Vert_F = O\!\left( \sqrt{d}\,f(d,\mathcal{P}_T) + d\sqrt{\dfrac{\log d}{\sigma_d^2}} \right),
\end{equation*}
where $\mathcal{O}_d$ is a sequence of orthonormal rotation matrices.
\end{proposition}
\begin{proof}
See \Cref{ProofsModelClusteringSelection}.    
\end{proof}
 \begin{remark}
     Henceforth, we drop the $d$ dependence of the orthonormal matrix above and simply write it as $\mathcal{O}$.
 \end{remark}

Building on the previous result, we now proceed to establish the consistency of our approach. Recall that the $K$-means algorithm aims to partition a given set of observations $\left(\mathbf{x}_1, \mathbf{x}_2, \ldots, \mathbf{x}_n\right)$, where each $\mathbf{x}_i$ is a $p$-dimensional vector (with both the sample size and dimensionality allowed to vary), into $k\le n$ distinct clusters, $\mathbf{S} = { S_1, S_2, \ldots, S_k }$. The goal is to achieve this partitioning in a way that minimizes the within-cluster sum of squares (WCSS). Formally, this optimization problem can be expressed as:
\begin{equation*}
    \underset{\mathbf{S}}{\arg \min} \sum_{i=1}^{k} \sum_{\mathbf{x} \in S_i} \left\| \mathbf{x} - \boldsymbol{\mu}_i \right\|^2,\quad \boldsymbol{\mu}_i = \frac{1}{\left| S_i \right|} \sum_{\mathbf{x} \in S_i} \mathbf{x},
\end{equation*}
where $\boldsymbol{\mu}_i$ is the centroid (or mean) of cluster $S_i$.

In \Cref{alg:VARBlockbuster}, we apply $K$-means to the rows of $\widehat{\mathcal{X}^*}$, hence in our case $n=d$ and the dimension of the vectors is of size $p=k$, the same size as the number of clusters. Let $\mathcal{O}$ be the orthonormal matrix from \Cref{EigenvectorBound}. With this objective function in hand, it becomes clear why \Cref{Lemma:2} indicates that, if we were to observe $\mathcal{X}^*\mathcal{O}$, \Cref{alg:VARBlockbuster} would yield perfect clustering. Indeed, for $i,j$ in the same group, it holds by the lemma that $\mathcal{X}^*_{i\bullet}\mathcal{O} = \mathcal{X}^*_{j\bullet}\mathcal{O}$. By partitioning the nodes into their true clusters, i.e, $\bold{S} = \{\mathcal{V}_1,\mathcal{V}_2,\dots,\mathcal{V}_k\}$, the $l$-th centroid, denoted $\mathcal{C}_{l \bullet }$, when computed from $K$-means on $\mathcal{X}^*\mathcal{O}$, would equal $\boldsymbol{Z}_{i\bullet} U \mathcal{O}$ for any $i$ such that $s(i) = l$. Therefore, this partition would lead to a WCSS of $0$. Let $\mathcal{C}:= \left(\mathcal{C}_{1\bullet}^\top, \dots, \mathcal{C}_{k\bullet}^\top\right)^\top$.
The slight complication of working with $\mathcal{X}^*\mathcal{O}$, as opposed to $\mathcal{X}^*$, arises from the fact that $\widehat{\mathcal{X}^*}$ only converges to $\mathcal{X}^*$ up to a rotation, following \Cref{EigenvectorBound}. Similarly, let $\widehat{\mathbf{C}} := \left(\widehat{\mathbf{C}}_{1\bullet}^\top, \dots, \widehat{\mathbf{C}}_{k\bullet}^\top\right)^\top$ be the matrix of estimated centroids from applying \Cref{alg:VARBlockbuster} to $\widehat{\mathcal{X}^*}$, and let $\hat{\bold{V}} = \{\widehat{\mathcal{V}_1}, \dots, \widehat{\mathcal{V}_k}\}$ be the estimated groups.

To show that our algorithm consistently clusters the observations, we need to define mis-clustering in a mathematically rigorous way. Before giving a precise definition, we introduce a map $\hat{s}: \mathcal{V} \to \mathcal{K}$, where $\mathcal{K} = \{1, \dots, k\}$, that assigns each node to its estimated group. With this in mind, we now present the following definition.
\begin{definition}\label{Def:Miscluster}
    We say that a node $i \in \{1,2,\dots,d\}$ is correctly specified, if $\widehat{\mathbf{C}}_{\hat{s}(i) \bullet}$ is closer to the centroid $\mathcal{C}_{s(i)\bullet}$ than it is to any other centroid $\mathcal{C}_{s(j)\bullet}$ of  any other $j$ such that $s(i)\ne s(j)$. With this in mind, we let the set of mis-clustered nodes, denoted by $\mathcal{M}$ be defined as 
\begin{equation*}
    \mathcal{M}=\left\{i: \exists j  \text { s.t. } s(i)\ne s(j), \text { and }\left\|\widehat{\mathbf{C}}_{\hat{s}(i) \bullet} -\mathcal{C}_{s(i)\bullet}\right\|_2 \ge \left\|\widehat{\mathbf{C}}_{\hat{s}(i) \bullet}-\mathcal{C}_{s(j)\bullet}\right\|_2\right\}.
\end{equation*}
\end{definition}
The above definition is quite intuitive. Indeed, as $\widehat{\mathcal{X}^*}$ approaches $\mathcal{X}^*\mathcal{O}$, we would expect to cluster more and more correctly, as previously argued. Consequently, for any $l = 1, \dots, k$, $\widehat{\mathbf{C}}_{l \bullet}$ should become increasingly closer to $\mathcal{C}_{l\bullet}$. Thus, if node $i$ in reality belongs to group $l^\prime$, but is estimated to be in group $l$, then the distance from $\widehat{\mathbf{C}}_{\hat{s}(i) \bullet} = \widehat{\mathbf{C}}_{l \bullet}$ to $\mathcal{C}_{l\bullet}$ should be smaller than the distance to $\mathcal{C}_{s(i)\bullet} = \mathcal{C}_{l^\prime\bullet}$, leading to misclustering.

Finally, Theorem \ref{Thm:ConsistencyCommunitytCluster} establishes the consistency of our community detection algorithm:
\begin{theorem}\label{Thm:ConsistencyCommunitytCluster}
    Consider an SBOU process as in \Cref{SBOUProc} with the set of misclustered nodes, $\mathcal{M}$ defined as in \Cref{Def:Miscluster}. Grant \Cref{Ass.NodeSum,Ass.2,Ass.EstimationBound} and suppose $\psi_1 > \psi_2$. Then, with high probability,
    \begin{equation*}
        \dfrac{\vert \mathcal{M}\vert}{d} =  O \left( f(d,\mathcal{P}_T) ^2+\dfrac{d\log d}{\sigma_d^2} \right).
    \end{equation*}
    In particular, since $f(d,\mathcal{P}_T)\to 0$ (by \Cref{DiscreteEstimatorConsistencyBothCases}) and $d\log d/\sigma_d^2\to 0$ under \Cref{Ass.2}, the algorithm is consistent:
    \begin{equation*}
        \dfrac{\vert \mathcal{M}\vert}{d} \overset{\mathbb{P}}{\longrightarrow} 0, \quad \text{as } d\to\infty.
    \end{equation*}
\end{theorem}
\begin{proof}
See \Cref{ProofsModelClusteringSelection}.    
\end{proof}
\begin{remark}
    The condition $\sigma_d=\omega((d\log d)^{1/2})$ in \Cref{Ass.2} is precisely what forces $d\log d/\sigma_d^2\to 0$ and hence, together with $f(d,\mathcal{P}_T)\to 0$, delivers the consistency stated above. The rate $(d\log d)^{1/2}$ is the sparsest regime in which exact recovery is possible for the stochastic block model \citep{Abbe2015}. Since $\sum_{i=1}^d\theta_i = k\sigma_d \le d$ forces $\sigma_d\le d/k$, the admissible range for the community scale is $\omega((d\log d)^{1/2}) \le \sigma_d \le d/k$.
\end{remark}

In \Cref{App:SimStudyClusteringAlgo}, we test the algorithm in an extensive finite-sample simulation study and show that it performs very well in detecting the underlying groups. The study clearly shows that the clustering algorithm improves as both $T$ and $d$ grow, as the theory would suggest. It also highlights the somewhat expected finding that the clustering algorithm performs best when the group signal is strong, as determined by $\boldsymbol{B}$ and $\Theta$, as discussed in \Cref{Sec:ModelSpecific}. One might ask why we do not take the arguably simpler approach of clustering directly on the correlation matrix of $\mathbb{Y}$, as proposed by \cite{Mantegna1999} and followed by much of the subsequent literature, given that $Q$ enters its covariance. In \Cref{SimCorrelation}, we show that the correlation-based matrix is much more susceptible to noise than our proposed method, which is considerably better at distinguishing between a \emph{true underlying group signal} and \emph{noise}.

\section{Choosing the number of communities and assessing signal strength}\label{SecModelSelection}

In the previous section, we established consistency of the community detection algorithm under one key requirement: the true number of communities \(k\) is known. Also it was pointed out that in finite samples, for the algorithm to work, the group signal should be sufficiently strong; see also \Cref{fig:signal_vs_nosignal,fig:combined_signal_illustration,table:SimTable1}. We now address both issues.

We begin with estimation of \(k\). A common approach is to use model selection procedures such as cross-validation, as in \citet{Gudhmundsson2021}. As noted by \citet{Ma2021}, however, adapting such methods to stochastic block models and providing rigorous theoretical guarantees is difficult because of the dependence structure inherent in network data. One could instead pursue a likelihood-based approach, in the spirit of \citet{Wang2017}; see also \citet{Gudhmundsson2025}. In the present paper, however, we work with the estimator of \citet{Ma2021}\footnote{\cite{Ma2021} strictly speaking work with $2$ estimators $K_1$ and $K_2$. They can only show that $K_1$ does not underestimate the true number of groups asymptotically, whereas they show that $K_2$ is consistent at the expense of introducing a tuning parameter. In our simulation studies the two always agree}, which consistently estimates the number of communities in a degree-corrected stochastic block model. The procedure evaluates a pseudo-likelihood ratio over a finite candidate set \(\{1,2,\dots,K_{\max}\}\) and selects the maximizer. Naturally, \(K_{\max}\) must satisfy \(K_{\max}\geq k\).

The procedure in \citet{Ma2021} assumes that the realized stochastic block model is observed. In our framework, this is equivalent to observing the support of the drift matrix of the SBOU process, which is not realistic in applications. We therefore estimate the drift matrix; see \Cref{Sec:EstShort}. Since the maximum likelihood estimator does not set entries exactly equal to zero, its support cannot be recovered directly from the estimated matrix, as this would lead to a fully dense support estimate. Although one could instead use an estimator with support recovery guarantees, such as the adaptive Lasso, we work with the MLE because it is the only estimator for which we establish consistency in an infill, long-span, high-dimensional regime. We therefore recover the support through a separate model selection step; see \Cref{SecSupportRec}. The resulting novel procedure performs well in the simulation study reported in \Cref{tab:SupportRecovery}.

We next turn to the strength of the community signal. In practice, this is best assessed visually: sorting the MLE estimate $\widehat{Q}$ of the drift matrix by the estimated groups and inspecting the resulting heat map for clear block structure is the standard diagnostic in the SB-VAR literature \citep{Gudhmundsson2021}, and we adopt it here. Obtaining better, quantitative diagnostics appears difficult. Although we show in \Cref{EstimatorsBAndTheta} that consistent plug-in estimators of the edge-probability matrix $\mathcal{A}$ are available, they require knowledge of $k$ and the partition, which renders any quantitative signal-strength assessment circular: it would presuppose the very structure we aim to assess. We demonstrate the strong-signal regime in which the plug-in estimators recover $\mathcal{A}$ accurately in \Cref{Sim2}.

\section{Empirical application}\label{Sec:Empirical}

To illustrate the performance of the clustering algorithm in \Cref{alg:VARBlockbuster} and the group-number estimator developed in \Cref{SecModelSelection}, we apply both procedures to an empirical dataset. Our objective is to assess whether these tools can provide reliable insight in a real-world setting. Ideally, we would like to evaluate the methods on a dataset with a known group structure, so that we can examine whether the clustering algorithm and the estimator for the number of groups are capable of recovering the underlying ground truth.

In practice, however, it is difficult to find a dataset that both satisfies the modeling assumptions of a multivariate Ornstein-Uhlenbeck process and comes with a genuinely known group structure. We therefore instead consider a dataset for which there is a strong and natural benchmark hypothesis regarding both the number of groups and the partition itself, even though the true grouping is not formally observed. This still provides an informative empirical exercise: if the proposed procedure recovers the benchmark structure, this supports its usefulness in practice, while systematic deviations may reveal additional latent structure in the data.

\subsection{Data description}

Our empirical analysis is based on the RE-Europe dataset of \citet{Jensen2017}. The dataset provides hourly wind-capacity factors, measured as fractions of full capacity, for approximately 1,500 locations in Europe corresponding to the node locations (buses) in \citet{Hutcheon2013}. The average distance between neighboring nodes is approximately 50 km. Wind speeds from the COSMO-REA6 dataset on a \(7\times 7\) km\(^2\) spatial grid over the period 2012--2014 \citep{Bollmeyer2015} are converted into wind-capacity factors on the same grid. These capacity factors are then assigned to nearby nodes, and if a grid cell is associated with multiple nodes, its contribution is divided evenly across them. At the node level, the resulting values are averaged to obtain wind capacity as a fraction of full capacity.

We focus on a subset consisting of 125 nodes located in France, Germany, Spain, Italy, and Poland, with 25 nodes selected from each country; see \Cref{fig:Clustering5countries}. The resulting panel is observed on an equidistant hourly grid with \(N_T=26{,}304\). We interpret \(T=1\) as one day, so that the sample covers \(T=1{,}096\) days with \(\Delta_{\mathcal P_T}=1/24\).

Inspection of the raw series reveals pronounced daily and yearly seasonality. We therefore deseasonalize the data using a LOESS-based MSTL procedure implemented in the Python library \texttt{statsmodels}. The resulting series have mean close to zero and display no evidence of a unit root. In \Cref{fig:EmpiricalFitOU}, we show a representative cleaned series together with its sample autocorrelation function. The series exhibits mean reversion and an approximately exponentially decaying autocorrelation pattern, which supports the use of a multivariate OU process as a working model for the cleaned data.

\begin{figure}[htbp]
    \caption{Illustration of a cleaned time series and its sample autocorrelation function for bus node F-99, following the ENTSO-E naming convention.}
    \label{fig:EmpiricalFitOU}
    \centering
    \includegraphics[scale=0.35]{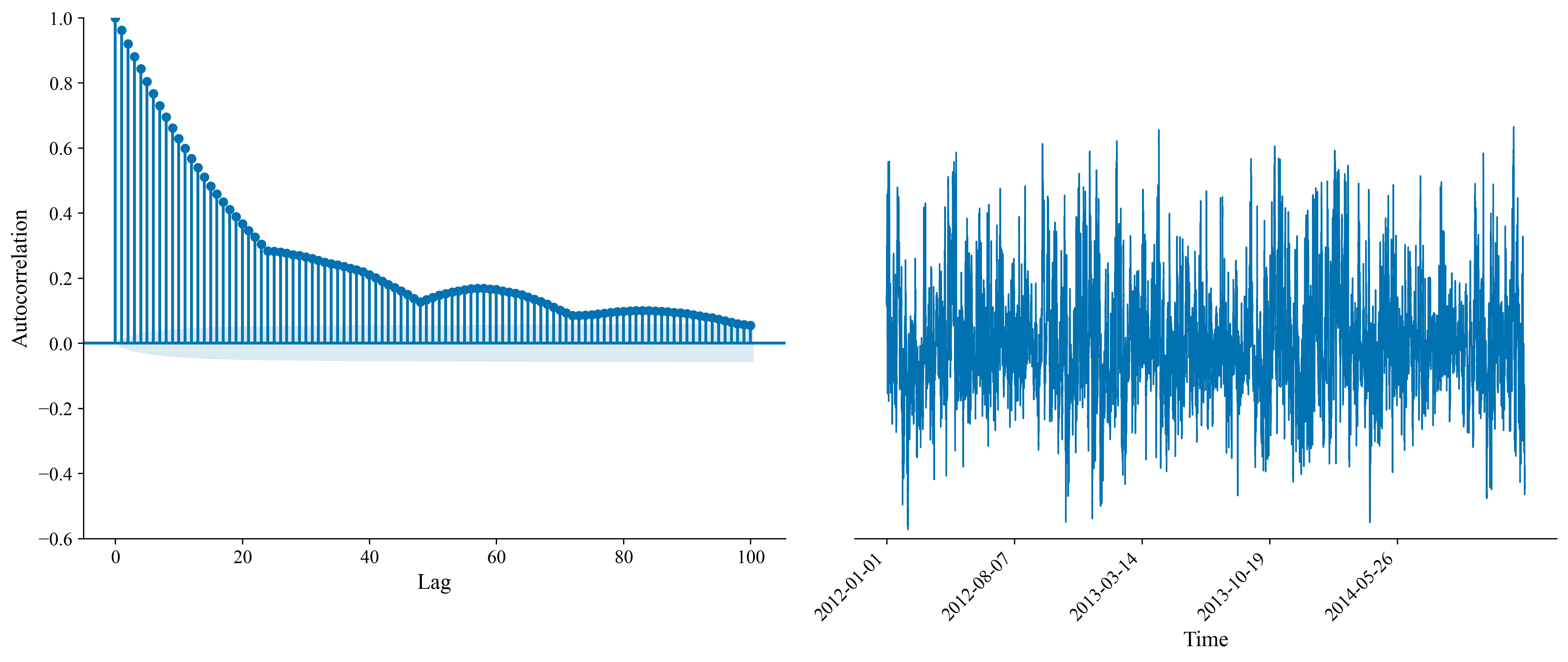}
\end{figure}

\subsection{Estimated group structure and benchmark partition}

In \Cref{fig:HeatMapDrift}, we display the estimated drift matrix after ordering the nodes by country. The figure reveals a somewhat pronounced block structure aligned with national boundaries. This is consistent with the construction of the sample: within each country, the selected nodes are geographically close to one another and are relatively far from nodes in neighboring countries. Since wind speeds evolve spatially and nearby locations are subject to similar meteorological conditions, it is natural to expect stronger dependence among geographically proximate nodes. The estimated drift matrix reflects exactly this pattern.

Accordingly, the country decomposition provides a natural benchmark partition for the empirical application. For the full dataset, this benchmark suggests five groups. At the same time, one may expect finer local structure within countries, since geographically close nodes within the same country may still form smaller subgroups and so the nationalities may not reflect all apparent group structure. To examine this, we therefore study both the full sample of \(d=125\) nodes and all ten decompositions of the five countries into a two-country subset and the complementary three-country subset. For the two-country subsets, the benchmark number of groups is two, while for the three-country subsets it is three.

\begin{figure}[htbp]
\caption{Heatmap of estimated drift matrix}
    \label{fig:HeatMapDrift}
    \centering
    \includegraphics[scale=0.35]{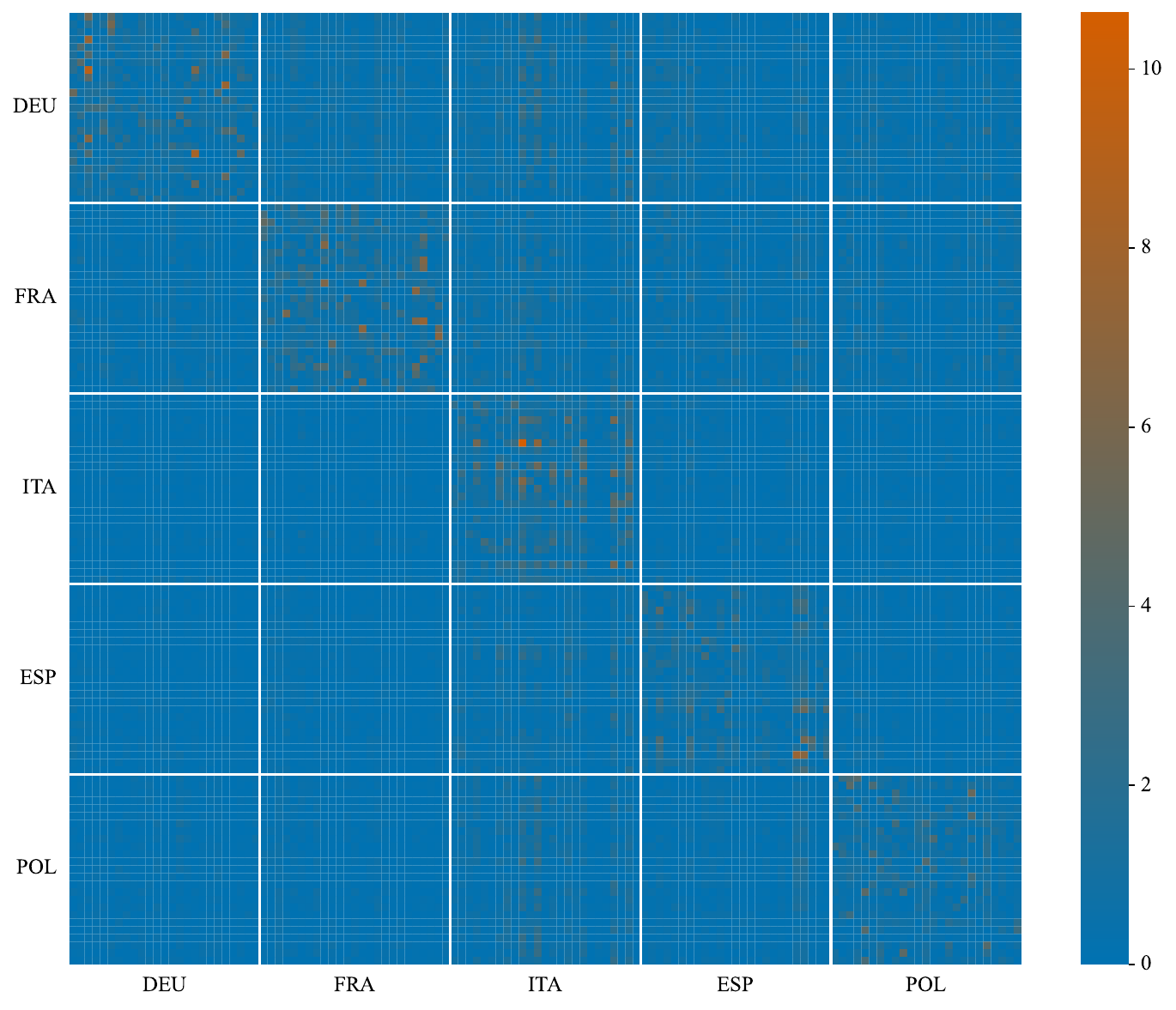}
    \begin{scriptsize}
    \parbox{0.98\textwidth}{\emph{Note:} The figure shows, the heat map of the maximum-likelihood estimate of the drift matrix with time series sorted by country.}
    \end{scriptsize}
\end{figure}

\subsection{Results}\label{Sec:EmpiricalResults}

Before presenting the empirical results, we briefly summarize the implementation choices, deferring full detail to \Cref{App:EmpiricalImpl}. We use the \(K_1\) estimator of \citet{Ma2021} with two modifications. First, we restrict the candidate set to \(k\geq 2\) to address its tendency to favour a single community in our setting. Second, we estimate $k$ over a grid of admissible minimum-group-size requirements and report the value of $k$ that remains selected over the longest contiguous range of the grid; ties are broken in favour of the smallest $k$. The minimum-size requirement prevents the procedure from selecting partitions with very small groups, such as singleton or pairwise clusters, which are difficult to interpret as genuine latent communities. Varying this requirement over a grid therefore allows us to choose a value of $k$ that is stable across reasonable lower bounds on group size. The choice of \(K_{\max}\) and a comparison with the \(K_2\) estimator of \cite{Ma2021} are discussed in \Cref{App:EmpiricalImpl}.

Applying this framework to the full dataset, we estimate five groups. The resulting partition aligns closely with national boundaries, which serve as our benchmark. To examine whether the procedure continues to recover meaningful community structures in smaller samples, we additionally consider twenty country-based subsets: ten containing observations from two countries and ten containing observations from three countries. For example, one of the two-country subsets contains only the observations from France and Germany, resulting in a dataset of $50$ nodes. For each subset, we re-estimate the number of communities using the procedure described above which is based on the content of \Cref{SecModelSelection} and then partition the nodes by applying the clustering algorithm of \Cref{Sec:Cons}.

\begin{figure}[htbp]
\caption{Estimated group partition for the full RE-Europe-$125$ sample based on \Cref{alg:VARBlockbuster}.}
    \label{fig:Clustering5countries}
    \centering
    \includegraphics[scale=0.5]{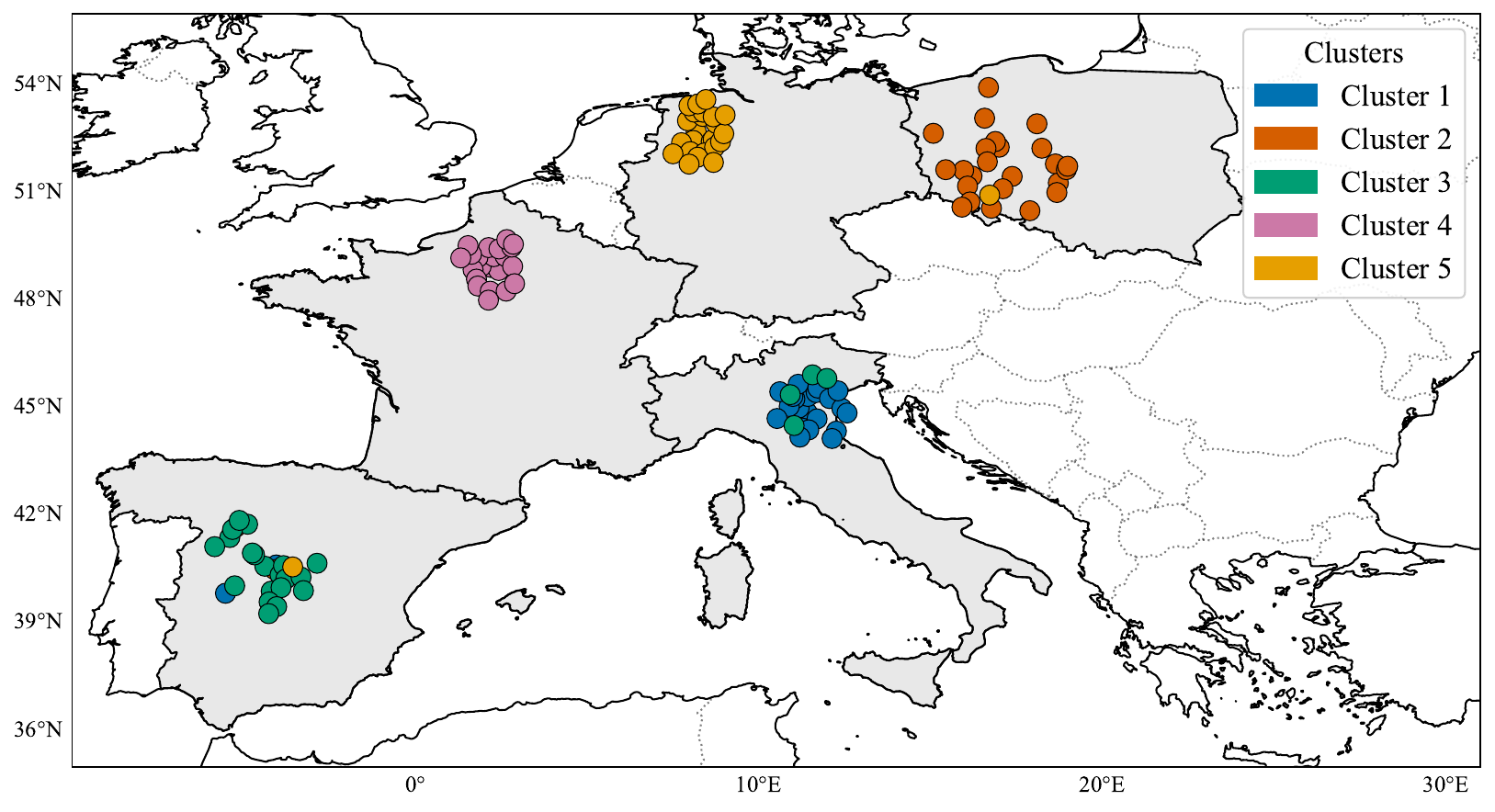}
\begin{scriptsize}
\parbox{0.98\textwidth}{\emph{Note:} The figure shows the estimated group partition for the full RE-Europe sample with \(d=125\) nodes when \(k=5\). The resulting clusters align closely with national boundaries, with one estimated group for each of France, Germany, Spain, Italy, and Poland.}
\end{scriptsize}
\end{figure}
\begin{table}[ht]
\centering
\small
\caption{Estimated number of groups for two-country and three-country subsets.}
\label{tab:EstimatedGroupsSubsets}
\begin{tabular}{lccc|lccc}
\toprule
2-subset & $\widehat K$ & switches & plateau share & 3-subset & $\widehat K$ & switches & plateau share \\
\midrule
(DEU, FRA) & 2 & 1 & 0.67 & (DEU, FRA, ITA) & 3 & 0 & 1.00 \\
(DEU, ITA) & 2 & 0 & 1.00 & (DEU, FRA, ESP) & 2 & 0 & 1.00 \\
(DEU, ESP) & 2 & 1 & 0.67 & (DEU, FRA, POL) & 3 & 0 & 1.00 \\
(DEU, POL) & 3 & 0 & 1.00 & (DEU, ITA, ESP) & 2 & 2 & 0.50 \\
(FRA, ITA) & 2 & 0 & 1.00 & (DEU, ITA, POL) & 3 & 2 & 0.50 \\
(FRA, ESP) & 3 & 1 & 0.83 & (DEU, ESP, POL) & 2 & 0 & 1.00 \\
(FRA, POL) & 3 & 0 & 1.00 & (FRA, ITA, ESP) & 3 & 1 & 0.83 \\
(ITA, ESP) & 2 & 1 & 0.67 & (FRA, ITA, POL) & 3 & 1 & 0.83 \\
(ITA, POL) & 2 & 1 & 0.67 & (FRA, ESP, POL) & 3 & 0 & 1.00 \\
(ESP, POL) & 2 & 2 & 0.50 & (ITA, ESP, POL) & 4 & 1 & 0.50 \\
\bottomrule
\end{tabular}
\vspace{0.3em}
\begin{minipage}{0.97\linewidth}
\footnotesize
\textit{Notes:} \(\widehat K\) is the value of \(k\) that remains selected over the longest contiguous range of admissible minimum-group-size requirements, with ties broken in favour of the smallest \(k\); see \Cref{App:EmpiricalImpl} for the formal definition. The switch count is the number of times \(\widehat K(g_i)\) changes as \(g_i\) varies over the grid; the plateau share is the longest plateau length normalised by the grid size. Lower switch counts and higher plateau shares indicate robustness to the choice of minimum-group-size requirement.
\end{minipage}
\end{table}

The results for the subset analysis are reported in \Cref{tab:EstimatedGroupsSubsets} as well as two additional measures of the robustness of the estimated value of $k$. For the two-country subsets, the estimated number of groups is often equal to two, while for the three-country subsets it is often equal to three. Moreover, these findings are robust to the minimal group size requirement. There are a few subsets for which the estimated number of groups deviates from the benchmark suggested by the country decomposition. Such deviations could be informative rather than merely problematic: they may reveal finer spatial heterogeneity within countries, or regional dependence patterns that cut across national classifications. Taken together, the empirical results provide encouraging evidence for the proposed methodology. The full-sample partition is closely aligned with the country-level benchmark, and the subset analyses generally select the expected number of groups. Moreover, the partitions reported in \Cref{app:EmpiricalClusteringSubGroups2and3} indicate that the method also recovers geographically meaningful structure when applied to smaller two-country and three-country systems.

\FloatBarrier
\section{Conclusion}

This paper develops a continuous-time model for detecting latent group structure in high-dimensional multivariate time series. A large strand of the literature focuses on model-free clustering methods based on different measures of dependence, often taken to be correlations between components. We take a different approach. We introduce the stochastic block Ornstein-Uhlenbeck (SBOU) process, a L{\'e}vy-driven multivariate Ornstein-Uhlenbeck process in which the drift matrix depends on unobserved group memberships. In doing so, we further develop the emerging literature on continuous-time network time series and graph-based continuous-time models by introducing a latent interaction graph with a group structure to be inferred, as opposed to models where the graph structure is known \citep{Courgeau2022a,Courgeau2022b,Lucchese2023}. The model can also be viewed as a continuous-time analogue of the stochastic-block VAR framework of \citet{Gudhmundsson2021}, adapted to mean-reverting systems and extended to allow for both positive and negative dynamic interactions between components.

The paper establishes three main findings in addition to introducing the model. First, it establishes consistency of a feasible drift estimator based on discrete observations in a joint asymptotic regime combining increasing dimension, infill asymptotics, and long-span asymptotics. This step is essential because the latent groups are not observed directly; they must be recovered from the estimated drift matrix. Second, taking the number of latent groups as known, the paper adapts spectral clustering to the estimated drift matrix and proves consistency of community recovery in the SBOU setting. Third, since the number of groups is not known in practice, we propose a support-recovery procedure for the drift matrix and use the resulting estimated network structure to apply the group-number selection estimator of \citet{Ma2021} to the continuous-time setting. This allows us to estimate the number of groups, $k$, and makes the clustering algorithm practically usable.

The theoretical findings are complemented by finite-sample simulation studies, reported in the supplementary material, and empirical evidence. The simulation results support the asymptotic theory and show that the clustering procedure improves as the dimension, time span, and signal strength increase. Similarly, the ability to recover the correct number of groups, $k$, is demonstrated in extensive simulation studies.

The empirical application to the RE-Europe wind-capacity data of \citet{Jensen2017} further illustrates the usefulness of the method beyond the simulated setting. We construct a five-country panel consisting of $125$ nodes, with $25$ nodes selected from each of France, Germany, Spain, Italy, and Poland. For each node, we observe hourly wind-capacity factors over the period 2012--2014, measured as fractions of full capacity. This gives a high-dimensional time-series dataset for which the interaction network is unobserved, but where the country labels provide a natural geographic benchmark.

For the full sample, the group-number selection procedure estimates $k=5$, and the resulting partition aligns closely with the country-level benchmark. This provides evidence that the procedure can recover meaningful groups outside controlled simulation studies, in a real high-dimensional dataset where the interaction network is unobserved. To further assess the stability of the method, we also consider all decompositions of the five-country dataset into two-country and three-country subsets. In these subset analyses, the estimated number of groups is often consistent with the corresponding geographic decomposition, and the recovered partitions remain close to the benchmark classification. These findings provide evidence that the SBOU framework can recover interpretable latent structure in settings where the interaction network is not observed directly, but must instead be inferred from high-dimensional time-series dynamics.

The paper also points to several directions for further work. On the theoretical side, the rate conditions used for feasible drift estimation are sufficient but likely conservative. In particular, the current proof requires the observation mesh $\Delta_{\mathcal P_T}$ to shrink sufficiently fast relative to the dimension and the time span. Sharper bounds for the discretization error in discretely observed L{\'e}vy-driven systems could weaken this mesh requirement and allow the theory to cover less frequent sampling schemes or faster-growing dimensions. It would also be natural to develop a full theory for the support-recovery and group-number selection steps, which are used here as implementation tools. On the modeling side, an interesting direction is to allow for more general forms of group structure. The present framework is based on communities in which nodes in the same group are more likely to interact with one another. In some applications, however, group structure is directional rather than purely community-based. Examples include trade flows, migration patterns, or input-output relations, where one group of units may systematically affect another group without the reverse effect being of the same form. Extending the SBOU framework to such sender-receiver or directional block structures would be a natural next step for future research, in a similar spirit to \citet{Gudhmundsson2025}.

%\pagebreak 
{\small 
\section*{Acknowledgement}
The authors benefited from comments by and discussions with Stefán Guðmundsson and Bent Jesper Christensen, as well as from comments by conference participants at the 12th Bachelier World Congress in Rio de Janeiro, and the Bernoulli--IMS 11th World Congress in Probability and Statistics in Bochum. AN gratefully acknowledges financial support from the Center for Research in Energy: Economics and Markets (CoRE).
AEDV's work has been supported through the 
EPSRC NeST Programme grant EP/X002195/1. 
\bibliographystyle{chicago}
\bibliography{mybib-v3}
}
\pagebreak
\appendix
\section*{Supplementary material}
\section{Estimation theory}\label{App:Est}
In this section, we provide a rigorous treatment of the estimator introduced in \Cref{Sec:EstShort}. We begin by specifying the estimation framework and notation used throughout.

\subsection{Estimation setup and notation}\label{PDynOfYandX}

For each \(Q \in \mathbb{R}^{d\times d}\), let \(\mathbb{Y}^Q\) denote the unique strong solution to \eqref{eqn:1} with drift matrix \(Q\). We write \(\Omega\) for the space of \(\mathbb{R}^d\)-valued càdlàg functions on \([0,\infty)\), equipped with the Borel \(\sigma\)-algebra \(\mathcal{F}\) induced by the Skorokhod topology, and let \((\mathcal{F}_t)_{t\ge 0}\) denote the canonical filtration. The law of \(\mathbb{Y}^Q\) on \((\Omega,\mathcal{F})\) is denoted by \(\mathbb{P}^Q\). Varying \(Q\) in \eqref{eqn:1} therefore induces a family of probability measures \(\{\mathbb{P}^Q : Q \in \mathbb{R}^{d\times d}\}\) on \((\Omega,(\mathcal{F}_t)_{t\ge 0},\mathcal{F})\). A comprehensive treatment of this framework can be found in \cite[§16]{Billingsley2013}; for our purposes, however, the material in \Cref{Prelim:EstSec} suffices. We write \(\mathcal{M}^+(\mathbb{R}^d)\) for the set of real \(d\times d\) matrices whose eigenvalues have strictly positive real parts; equivalently, \(A\in\mathcal{M}^+(\mathbb{R}^d)\) if and only if \(\Vert e^{-tA}\Vert\to 0\) as \(t\to\infty\). For a real symmetric matrix \(A\), \(\lambda_{\max}(A)\) and \(\lambda_{\min}(A)\) denote its largest and smallest eigenvalues.

We now consider estimation of the true drift matrix \(Q_0\) in \eqref{eqn:1}. Throughout this section, we impose the following assumption.

\begin{assumption}\label{Ass:Estimation}
We assume that \(Q_0 \in \mathcal{M}^+(\mathbb{R}^d)\), that \(\Sigma\) is strictly positive definite, that \(F\) admits a second moment and satisfies the moment condition in \Cref{Thm.1}, and that \(\mathbb{Y}\) is stationary. That is, the law of \(\boldsymbol{Y}_0\) coincides with the invariant distribution, which exists by \citet[Proposition 2.2]{Masuda2004}.
\end{assumption}

Let \(\mathbb{X}=(\boldsymbol{X}_t)_{t\ge 0}\) denote the canonical process on \(\Omega\). Under \Cref{Ass:Estimation}, the log-likelihood is well defined for each \(T>0\); see, for example, \citealp{Courgeau2022a,Dexheimer2024}. It is given by
\begin{equation}\label{eqn:LikelihoodCont}
\frac{d\mathbb{P}_T^Q}{d\mathbb{P}_T^{\mathbf 0}}
=
\exp\left(
-\int_0^T (\Sigma^{-1}Q\boldsymbol{X}_{s-})^\top\, d\boldsymbol{X}_s^{\mathrm c}
-\frac12\int_0^T
(\Sigma^{-1/2}Q\boldsymbol{X}_{s-})^\top \Sigma^{-1/2}Q\boldsymbol{X}_{s-}\, ds
\right).
\end{equation}
where \(\mathbb{P}_T^Q\) and \(\mathbb{P}_T^{\mathbf 0}\) denote the restrictions of \(\mathbb{P}^Q\) and \(\mathbb{P}^{\mathbf 0}\) to \(\mathcal{F}_T\), respectively, and \(\boldsymbol{X}^{\mathrm c}\) denotes the continuous \(\mathbb{P}^{\mathbf 0}\)-martingale part of the canonical process \(\mathbb{X}\); see \Cref{Prelim:EstSec}. We define the normalized negative log-likelihood by
\[
\mathcal{L}_T(Q)
=
-\frac{1}{T}\log\left(\frac{d\mathbb{P}_T^Q}{d\mathbb{P}_T^{\mathbf 0}}\right),
\qquad T>0,
\]
and let
\[
\widehat Q_T \in \argmin_{Q\in\mathbb{R}^{d\times d}} \mathcal{L}_T(Q).
\]

Under \Cref{Ass:Estimation}, \citealp{Courgeau2022a} show that the maximum-likelihood estimator \(\widehat Q_T\) exists and is unique \(\mathbb{P}_T^{Q_0}\)-almost surely, and that
\begin{equation}\label{eqn:TAsymEst}
T^{1/2}\bigl(\operatorname{vec}(\widehat Q_T)-\operatorname{vec}(Q_0)\bigr)
\stackrel{\mathcal D}{\longrightarrow}
\mathcal N\bigl(\mathbf 0_{d^2},\, C_\infty^{-1}\otimes \Sigma\bigr),
\qquad \text{as } T\to\infty,
\end{equation}
where
\[
C_\infty
=
\int xx^\top\,\mu(dx)
=
\mathbb E\!\left[\mathbb{Y}_\infty\mathbb{Y}_\infty^\top\right]
=
\int_0^\infty e^{-sQ_0}\Sigma e^{-sQ_0^\top}\,ds.
\]
Here, \(\mu\) denotes the invariant distribution. We refer to Theorems 3.4.4 and 4.2.3 of \cite{Courgeau2022a} for the proofs.

A consequence of \eqref{eqn:TAsymEst} is that
\[
\|\operatorname{vec}(\widehat Q_T)-\operatorname{vec}(Q_0)\|_2
=
\|\widehat Q_T-Q_0\|_F
=
O_{\mathbb{P}_T^{Q_0}}(T^{-1/2}),
\]
and hence, for every \(\varepsilon>0\),
\[
\|\widehat Q_T-Q_0\|_F
=
o_{\mathbb{P}_T^{Q_0}}(T^{-1/2+\varepsilon}),
\qquad \text{as } T\to\infty.
\]

Our focus differs from that of \cite{Courgeau2022a} and related work. Rather than studying \(\|\widehat Q_T-Q_0\|_F\) in the classical regime where \(T\to\infty\) with fixed \(d\), we investigate its behaviour in a joint asymptotic setting in which both \(T\) and \(d\) diverge. To this end, we rely on ideas from \cite{Dexheimer2024,Ciolek2020,Gaiffas2019}, and now introduce notation consistent with theirs.

Throughout this section, \(\otimes\) denotes the Kronecker product, while \(\odot\) denotes the Hadamard, or elementwise, product. The operator \(\operatorname{vec}\) denotes the vectorization map, which stacks the columns of a matrix into a single vector.

The Frobenius norm is associated with the scalar product
\begin{equation*}
    \left\langle A_1, A_2 \right\rangle_F := \operatorname{Tr}(A_1^\top A_2), \quad A_1, A_2 \in \mathbb{R}^{d_1 \times d_2},
\end{equation*}
where \(\operatorname{Tr}(\cdot)\) denotes the trace. Additionally, for \(r > 0\), we define the set
\begin{equation*}
    \mathbb{B}(r) := \left\{ \mathbf{B} \in \mathbb{R}^{d \times d} : \|\mathbf{B}\|_F \leq r \right\}.
\end{equation*}

Given \(\beta = (\beta_1, \ldots, \beta_d) \in \mathbb{R}^d\), let \((\beta_1^{\#}, \ldots, \beta_d^{\#})\) be the non-increasing rearrangement of \(|\beta_1|, \ldots, |\beta_d|\). For a vector of tuning parameters \(\lambda = (\lambda_1, \ldots, \lambda_d) \in \mathbb{R}^d\) with \(\lambda_1 \geq \lambda_2 \geq \ldots \geq \lambda_d \geq 0\), we define
\begin{equation*}
    \|\beta\|_* := \sum_{j=1}^d \lambda_j \beta_j^{\#}, \quad \beta \in \mathbb{R}^d.
\end{equation*}
It is known that \(\|\cdot\|_*\) defines a norm on \(\mathbb{R}^d\) (see Proposition 1.2 in \cite{Bogdan2015}). In the following, the weights will always be given by
\begin{equation*}
    \lambda_j = \sqrt{\log\left(\frac{2d}{j}\right)}, \quad j \in \{1, \ldots, d\}.
\end{equation*}
For \(A \in \mathbb{R}^{d_1 \times d_2}\), we set (with slight abuse of notation)
\begin{equation*}
    \|A\|_* := \|\operatorname{vec}(A)\|_*,
\end{equation*}
which gives
\begin{equation*}
    \|A\|_* = \sum_{j=1}^{d_1 d_2} \operatorname{vec}(A)_j^{\#} \sqrt{\log\left(\frac{2d_1 d_2}{j}\right)}.
\end{equation*}

For stochastic processes \(\left( \bold{X}_t \right)_{t \in [0, T]}, \left( \bold{Y}_t \right)_{t \in [0, T]} \in L^2([0, T], \mathrm{d}t)\), we introduce the scalar product
\begin{equation*}
    \langle X, Y \rangle_{L^2} := \frac{1}{T} \int_0^T X_s^\top Y_s \, \mathrm{d}s.
\end{equation*}
When working with a family of probability measures $\left\{\mathbb{P}^\theta: \theta \in \Theta\right\}$, defined on a suitable space, we denote taking expectations under $\mathbb{P}^{\theta}$ by $\mathbb{E}_{\mathbb{P}^{\theta}}\left( \cdot \right)$. $\mathbb{E}\left(\cdot \right)$ always refers to expectations under $\mathbb{P}$, i.e. the baseline probability measure.

\subsection{Estimation in the continuous setting}
In the following, we assume that we observe a continuous record of observations up to time $T>0$ of the Ornstein Uhlenbeck process from \eqref{eqn:1} with drift matrix $Q_0$. Additionally, we assume that the path of the continuous $\mathbb{P}^0$ martingale part $\mathbb{X}^c = \left(\bold{X}_{t}^c\right)_{t\ge0}$ is also available and that $\Sigma$ is known. Under these assumptions, it is feasible to calculate \eqref{eqn:LikelihoodCont}. We are interested in bounding $\Vert \widehat{Q} -Q_0 \Vert_F^2$ as $T$ \emph{and} $d$ tend to infinity. 
To achieve this, we will need to impose the following assumptions.
\begin{assumption}\label{Ass.H}
    \begin{enumerate}
    \item $\Vert Q_0 \Vert_F,\Vert Q^{-1}_0 \Vert_F = O(d)$ as $d\rightarrow \infty$
        \item $\Vert \Sigma \Vert_F = O(d)$, $\max_{k\in\{1,\dots,d\}}\int \left(z^{(k)}\right)^2 F(dz) = O(1)$ and $\sqrt{\int \Vert z \Vert_2^4 F(dz)} = O(d)$ as $d\rightarrow \infty$
        \item $\Vert b \Vert_2^2 = O(d)$ as $d\rightarrow \infty$
        \item $\lambda_{\min} \left(\dfrac{Q_0 + Q_0^\top}{2}\right) > \epsilon$ for some $\epsilon >0$.
        \item $\kappa_{\max} = \max_{i}\lambda_{i}(C_\infty) = O(d)$ as $d\rightarrow \infty$  and there exists a $C>0$ such that $\kappa_{\min} = \min_{i}\lambda_i(C_\infty)\ge C$ for all $d \in \mathbb{N}$.
    \end{enumerate}
\end{assumption}
\begin{remark}
We note that all assumptions above are mild and fairly easy to check. The above bounds  can be relaxed as desired at the cost of a slightly slower convergence of our estimator. However, all proofs still go through if these are modified. We underline that for any $d \in \mathbb{N}$, $\kappa_{\min} > 0$ following e.g. the remarks in Appendix A of \cite{Dexheimer2022} and hence the condition on the minimal eigenvalue of $C_{\infty}$ is not strict.
\end{remark}
\begin{remark}
    When $F$ is finite $\sqrt{\int \Vert z \Vert_2^4 F(dz)} = O(d)$ implies $\int \Vert z \Vert_2^2 F(dz) = O(d)$, but in general this does not hold.
\end{remark}

\begin{assumption}\label{Ass:LevyH}
    Grant \Cref{Ass:Estimation,Ass.H}. There exists a function $H: \mathbb{R}^{+} \times  \mathbb{R}^{+} \rightarrow \mathbb{R}^{+}$such that
(i) for any $T, r>0$, the functions $H(T,\cdot)$ and $H(\cdot, r)$ are non-increasing, with $\lim_{T\rightarrow \infty} H(T,r) = 0$ and
(ii) for any vector $u \in \mathbb{R}^d$ with $\|u\| \leq 1$, it holds
\begin{equation*}   
\forall T, r>0, \quad \mathbb{P}^{Q_0}\left(\left|u^{\top}\left(\widehat{\mathbf{C}}_T-\mathbf{C}_{\infty}\right) u\right| \geq r\right) \leq H(T, r),
\end{equation*}
where
\begin{equation*}  
\widehat{\mathbf{C}}_T:=\frac{1}{T} \int_0^T X_s X_s^{\top} \mathrm{d} s \quad \text { and } \quad \mathbf{C}_{\infty}:=\int x x^{\top} \mu(\mathrm{d} x).
\end{equation*}
\end{assumption}
\begin{remark}
When $\mathbb{L} = \mathbb{W}$ and $Q_0$ is diagonalizable with a minimal eigenvalue bounded away from $0$ uniformly in $d$, \citet{Ciolek2020} show that under \Cref{Ass:Estimation}, this assumption is satisfied, with
\begin{equation*}
H(T, r) = 2\exp\!\left(-T \cdot \frac{r_0}{8\, p_0\, K_\infty} \cdot \frac{r^2}{r + K_\infty}\right),
\end{equation*}
where $r_0 := \min_{1 \le j \le d}\operatorname{Re}(\theta_j)$ is the minimal real part of the eigenvalues of $Q_0$, $K_\infty = \lambda_{\max}(C_\infty)$, and $p_0 = \Vert P_0\Vert_{op}\Vert P_0^{-1}\Vert_{op}$ is the condition number of the eigenvector matrix from the diagonalization $Q_0 = P_0 \mathrm{diag}(\theta_1,\ldots,\theta_d)P_0^{-1}$; see \citet[Proposition $3.2$]{Ciolek2020}. The constants $r_0, K_\infty, p_0$ may depend on $d$ through model parameters, so $H$ decays exponentially in $T$ for any fixed $d$, but the rate may degrade with $d$. We note that the assumption of diagonalizability of $Q_0$ may be replaced by  \Cref{Ass.H} $(4)$. The reason for this is that the assumptions in \cite{Ciolek2020} are used to derive an exponential bound on the operator norm of the matrix exponential of $Q_0$ in terms of its smallest real part of the eigenvalues \citep[See proof of Proposition $3.2$]{Ciolek2020}. Using \Cref{ExponentialBound1} we may get the same under \Cref{Ass.H} $4$. As shown in the example below neither assumption implies the other, but our alternative assumption is more suited for estimation of the drift matrix of a SBOU process (\Cref{SBOUProc}) where in general diagonalizability is not ensured, whereas the new condition is easily satisfied.  \begin{example}
Let 
$$
A=\left(\begin{array}{ll}
1 & 1 \\
0 & 1
\end{array}\right).
$$
Then
$$
A+A^\top=\left(\begin{array}{ll}
2 & 1 \\
1 & 2
\end{array}\right).
$$
Then $A$ has eigenvalue $\lambda = 1$ with algebraic multiplicity $2$. Whereas the symmetric version has eigenvalues $1$ and $3$. However, as is evident, $A$ is not diagonalizable. Indeed $\dim \text{ker} (A - I )= \dim \{ \left(\begin{array}{l}
x_1  \\
0 
\end{array}\right) : x_1 \in \mathbb{R}\} = 1 <2$. As for the reverse, let 
$$A = \left(\begin{array}{ll}
1 & 10 \\
0 & 2
\end{array}\right).$$
Then,
$$A + A^\top = \left(\begin{array}{ll}
2 & 10 \\
10 & 4
\end{array}\right).$$
Clearly, 
$A$ has eigenvalues $1$ and $2$ and hence is diagonalizable but it is easily found that $A+A^\top$ has eigenvalue $3-\sqrt{101}<0$. 
\end{example}
\end{remark}
\begin{remark}
    In the literature \citep{Bickel2009,Gaiffas2019,Dexheimer2024,Li2023} \Cref{Ass:LevyH} is closely related to what is referred to as a restricted eigenvalue condition. Assumption \ref{Ass:LevyH} gives a uniform probabilistic bound on how far the empirical covariance matrix can deviate from its population counterpart when evaluated in any fixed direction \(u \in \mathbb{R}^d\) with \(\|u\|_2 \leq 1\).
\end{remark}
We now show that the maximum likelihood estimator (MLE) of $Q_0$ satisfies an upper bound similar to that derived in Propositions $2.2$ and $2.4$ on the Lasso and Slope estimator in \cite{Dexheimer2022}.\footnote{We cite the arXiv preprint version \cite{Dexheimer2022} since the published version \cite{Dexheimer2024} contains a slightly modified formulation of these propositions; the preprint statements are the ones used in our derivation.}
\begin{assumption}\label{Ass:WHPT}
   Define
    $$T_d := \inf\left\{ T^\prime > 0: \left(21 (d \wedge e)\right)^d H(T^\prime,\dfrac{\kappa_{\min}}{3d})\le \dfrac{1}{d}\right\}$$
    where $H$ is the function from \Cref{Ass:LevyH}. We assume $T\ge T_d$.
\end{assumption}
\begin{remark}
When $\mathbb{L} = \mathbb{W}$, Proposition 3.2 of \cite{Ciolek2020} suggests that the assumption is satisfied if  $d^4\log d = o(T_d)$ as $d \rightarrow \infty$, and we may get the same bound when using \Cref{ExponentialBound1}.
\end{remark}
\begin{proposition}\label{DexHeimerNormBound}
    Grant \Cref{Ass:Estimation,Ass.H,Ass:WHPT}. If $\widehat{Q}$ is a solution of the minimization problem $\min_{Q\in\mathbb{R}^{d \times d}}\mathcal{L}_T(Q)$ where $\mathcal{L}_T$ is defined as in \eqref{eqn:LikelihoodCont}, then $\widehat{Q}$ satisfies for all $Q \in \mathbb{R}^{d \times d}$ that 
    \begin{equation*}
\left\|\Sigma^{-1/2}\left(\widehat{Q}-Q_0\right) X\right\|_{L^2}^2-\left\|\Sigma^{-1/2}\left(Q-Q_0\right) X\right\|_{L^2}^2 \leq 2c_*^2 \dfrac{\kappa_{\max}}{\kappa_{\min} T}\left( \log 4 d \vee 2d^2\right) ,
\end{equation*} 
with high probability under $\mathbb{P}^{Q_0}$.
\end{proposition}
\begin{proof}
     See \Cref{Proofs:ContSetting}. 
\end{proof}
\begin{corollary}\label{CorrollaryBoundMaxLike}
    Let everything be as in the setup of \Cref{DexHeimerNormBound}. Then $\widehat{Q}$ satisfies 
    \begin{equation*}
        \Vert \widehat{Q}-Q_0 \Vert_F^2 \le  \lambda_{\max}(\Sigma)\, c_*^2 \dfrac{4\kappa_{\max}}{\kappa_{\min}^2 T}\left( \log 4 d\vee 2d^2\right).
    \end{equation*}
   with high probability under $\mathbb{P}^{Q_0}$.
\end{corollary}
\begin{proof}
    See \Cref{Proofs:ContSetting}.
\end{proof}
\begin{remark}\label{ConvergenceContCase}
By \Cref{Ass:Estimation,Ass:WHPT}, the above bound tends to $0$ as $d \rightarrow \infty$.
\end{remark}

\subsection{Discrete observations}\label{DiscreteSetting}
Until now, we have assumed that our process is observed in continuous time. In practice, however, this is unattainable, as real-world data, regardless of its frequency, is always recorded at discrete intervals due to practical constraints. We now turn to the more realistic setting where observations are on a discrete time grid, which may or may not be irregular. Let \(T > 0\) and \(N_T \in \mathbb{N}\) denote the time horizon and the number of observations, respectively. Define the discrete time grid as:

\begin{equation*}
\mathcal{P}_{T}:=\left\{0=s_0^{N_T}<s_1^{N_T}<\ldots<s_{N_T}^{N_T}=T\right\},
\end{equation*}
and let \(\Delta_{\mathcal{P}_T} := \sup_{0 \leq i < N_T - 1} \left(s_{i+1}^{N_T} - s_i^{N_T}\right)\) denote the maximum distance between consecutive observation times. We follow the notation that for some stochastic process $\mathbb{Z}$, $\Delta_{\mathcal{P}_T}^k \mathbb{Z} := \bold{Z}_{s_{k+1}^{N_T}}-\bold{Z}_{s_{k}^{N_T}}$, that is we work with forward-looking increments; this distinction will be important later. As in the continuous setting, our convergence results will depend on the observation scheme satisfying certain assumptions as \(d \to \infty\) and \(T \to \infty\) rather than $d$ fixed and $T\rightarrow \infty$. We impose the following assumption:

\begin{assumption}\label{Ass.High-FrequencySampling}
    Let \(T = T(d)\). The observation scheme \(\mathcal{P}_T\) satisfies:
    \begin{equation*}
        \begin{aligned}
            \Delta_{\mathcal{P}_T} &= o(1), \quad \text{ as } d \to \infty, \\
            N_T \Delta_{\mathcal{P}_T} &= O(T), \quad \text{ as } d \to \infty, \\
            \Delta_{\mathcal{P}_T} T &= o(1), \quad \text{ as } d \to \infty,\\
            N_T &> d.
        \end{aligned}
    \end{equation*}
\end{assumption}

\begin{remark}
    The conditions in \Cref{Ass.High-FrequencySampling} are standard high-frequency sampling assumptions, as discussed in \citep{Courgeau2022b,Lucchese2023,Mai2014}, except that here we consider the limit as \(d \to \infty\). As noted in \cite{Mai2014}, the second condition in \Cref{Ass.High-FrequencySampling} can always be satisfied by discarding some observations. Additionally, if the process is observed on an equidistant grid, this condition is trivially satisfied. We still think of \(T\) as a function of \(d\). The last assumption is required for \Cref{StrictConvexityOfC_THat}.
\end{remark}
We denote the discretely observed process by
\begin{equation*}
\mathbb{X}_{\mathcal{P}_{T}} := \left\{\mathbf{X}_s, s \in \mathcal{P}_{T}\right\}.
\end{equation*}
As before, we assume that we observe the continuous \(\mathbb{P}^0\)-martingale part of \(\mathbb{X}_{\mathcal{P}_T}\). Specifically, the observed process is
\begin{equation*}
\mathbb{X}^c_{\mathcal{P}_T} := \left\{\mathbf{X}^c_s, s \in \mathcal{P}_T\right\}.
\end{equation*} 

In this setting, \eqref{eqn:LikelihoodCont} is not observed and must be approximated by the following discrete log-likelihood:
\begin{equation}\label{DiscreteLikelihood}
    \begin{aligned}
\mathcal{L}_T^{\mathcal{P}_T}(Q) &= \dfrac{1}{T} \sum_{n=0}^{N_T-1} \left(\Sigma^{-1}Q \boldsymbol{X}_{s_n^{N_T}}\right)^{\top} \left(\boldsymbol{X}_{s_{n+1}^{N_T}}^c - \boldsymbol{X}_{s_n^{N_T}}^c \right) 
\\&+ \dfrac{1}{2T} \sum_{n=0}^{N_T-1} \left(\Sigma^{-1/2} Q \boldsymbol{X}_{s_n^{N_T}}\right)^{\top} \Sigma^{-1/2} Q \boldsymbol{X}_{s_n^{N_T}} \left(s_{n+1}^{N_T} - s_n^{N_T}\right).
\end{aligned}
\end{equation}

For stochastic processes \(\left(\boldsymbol{X}_t\right)_{t \in[0, T]}\) and \(\left(\boldsymbol{Y}_t\right)_{t \in[0, T]} \in L^2\left([0, T], \mathrm{d}t\right)\), we introduce the \emph{semi} inner product:
\begin{equation*}    
\langle \mathbb{X}, \mathbb{Y}\rangle_{L^2_{\mathcal{P}_T}} := \dfrac{1}{T} \sum_{n=0}^{N_T-1} \boldsymbol{X}_{s_n^{N_T}}^\top \boldsymbol{Y}_{s_n^{N_T}} \left(s_{n+1}^{N_T} - s_n^{N_T}\right).
\end{equation*}

\begin{lemma}\label{StrictConvexityOfC_THat}
Suppose \Cref{Ass.High-FrequencySampling} is satisfied. Let \(\widehat{\mathbf{C}}_{\mathcal{P}_T} = \dfrac{1}{T} \sum_{n=0}^{N_T-1} \boldsymbol{X}_{s_n^{N_T}} \boldsymbol{X}_{s_n^{N_T}}^\top \left(s_{n+1}^{N_T} - s_n^{N_T}\right)\). Then \(\widehat{\mathbf{C}}_{\mathcal{P}_T}\) is almost surely positive definite under \(\mathbb{P}^Q\) for all \(Q \in \mathbb{R}^{d \times d}\).
\end{lemma}
\begin{proof}
    See \Cref{Proofs:InfeasibleCase}.
\end{proof}
Now, we extend \Cref{Ass:LevyH} to the discrete setting.
\begin{proposition}\label{Prop:LevyH_discrete} 
    Grant \Cref{Ass:Estimation,Ass.H,Ass:LevyH}. There exists a function $H_2: \mathbb{R}^{+} \times \mathbb{R}^{+} \rightarrow \mathbb{R}^{+}$ such that
(i) for any $\Delta, r>0$, the function $H_2(\Delta,\cdot)$ is non-increasing, whereas $H_2(\cdot,r)$ is non-decreasing with $\lim_{\Delta \rightarrow 0} H_2(\Delta,r) = 0$ for any $r>0$ and
(ii) for any vector $u \in \mathbb{R}^d$ with $\|u\| \leq 1$, it holds for any observation scheme $\mathcal{P}_T$ and
\begin{equation*}   
 \forall r>0, \quad \mathbb{P}^{Q_0}\left(\left|u^{\top}\left(\hat{\mathbf{C}}_{\mathcal{P}_T}-\mathbf{C}_{\infty}\right) u\right| \geq r\right) \leq H(T, r/2) + H_2(\Delta_{\mathcal{P}_T}, r/2),
\end{equation*}
with $H_2\left(\Delta,r\right) = \dfrac{2\sqrt{\mathbb{E}_{\mathbb{P}^Q}\left(\Vert \mathbf{X}_{0}\Vert_2^2\right)}\sqrt{\sup_{u\in [0,\Delta]}\mathbb{E}_{\mathbb{P}^Q}\left(\Vert\bold{X}_u-\mathbf{X}_{0}\Vert_2^2\right)}}{r}$ and $H(\cdot, \cdot)$ as in \Cref{Ass:LevyH}.
\end{proposition}
\subsubsection{Infeasible case}
Although \eqref{DiscreteLikelihood} represents the observed likelihood function under this setting, we introduce the following unobserved likelihood function, as it allows us to establish several key properties. We will subsequently demonstrate that the likelihood function in \eqref{DiscreteLikelihood} inherits these properties.

\begin{equation}\label{egn:PseudoLikelihood}
    \begin{aligned}
    \tilde{\mathcal{L}}_T^{\mathcal{P}_T}(Q) &= \dfrac{1}{T}\sum_{n=0}^{N_T-1} \left( \Sigma^{-1} Q \boldsymbol{X}_{s_n^{N_T}} \right)^{\top} \left( \boldsymbol{X}_{s_{n+1}^{N_T}}^c(Q_0) - \boldsymbol{X}_{s_n^{N_T}}^c(Q_0) \right) \\
    &\quad + \dfrac{1}{2T} \sum_{n=0}^{N_T-1} \left( \Sigma^{-1/2} Q \boldsymbol{X}_{s_n^{N_T}} \right)^{\top} \Sigma^{-1/2} Q \boldsymbol{X}_{s_n^{N_T}} \left( s_{n+1}^{N_T} - s_n^{N_T} \right) \\
    &\quad - \dfrac{1}{T} \sum_{n=0}^{N_T-1} \left( \Sigma^{-1/2} Q \boldsymbol{X}_{s_n^{N_T}} \right)^{\top} \Sigma^{-1/2} Q_0 \boldsymbol{X}_{s_n^{N_T}} \left( s_{n+1}^{N_T} - s_n^{N_T} \right),
    \end{aligned}
\end{equation}
where \(\boldsymbol{X}_s^c(Q_0)\) is the \(\mathbb{P}^{Q_0}\)-continuous martingale as defined in \Cref{Prelim:EstSec}. We refer to this as the pseudo-likelihood function of \(\mathcal{L}_T^{\mathcal{P}_T}\). In the following sections, we will present a series of results related to this function. 
\begin{assumption}\label{Ass:WHPDelta}
   $\Delta_{\mathcal{P}_T}$ satisfies $\Delta_{\mathcal{P}_T} \le \Delta^*$, where
    $$\Delta^* := \inf\left\{ \Delta > 0: \left(21 (d \wedge e)\right)^d H_2(\Delta,\dfrac{\kappa_{\min}}{3d})\le \dfrac{1}{d}\right\}$$
    and $H_2$ is the function from \Cref{Prop:LevyH_discrete}.
\end{assumption}
\begin{remark}
    By inspecting the functional form of $H_2(\cdot, \cdot)$ and using \Cref{BoundOnIncrements,BoundExpectationNormInvariant}, we see that the above assumption is satisfied for large $d$ if $\Delta_{\mathcal{P}_T} = o\left(\left(21d\right)^{-(d+5)}\right)$ as $d\rightarrow \infty$. This in turn requires $\Delta_{\mathcal{P}_T}$ to decay to $0$ faster than exponentially as $d$ tends to infinity.
\end{remark}

\begin{corollary}\label{PropertiesDiscretizedEstimator}
    Grant \Cref{Ass.H,Ass:Estimation,Ass:LevyH,Ass:WHPT,Ass:WHPDelta,Ass.High-FrequencySampling}. Let $\tilde{Q}$ be a solution of the minimization problem $\min _{Q \in \mathbb{R}^{d\times d}}\tilde{\mathcal{L}}_T^{\mathcal{P}_T}$. Then $\tilde{Q}$ satisfies 
    \begin{equation*}
        \begin{aligned}
        \Vert \tilde{Q}-Q_0 \Vert_F^2 &\le  \lambda_{\max}(\Sigma)\, c_*^2 \dfrac{4\kappa_{\max}}{\kappa_{\min}^2 T}\left( \log 4 d \vee 2d^2\right),
        \end{aligned}
    \end{equation*}
    with high probability under $\mathbb{P}^{Q_0}$. Above, $c_*$ is the universal constant from \Cref{Prop3.3OfDexheimerDiscrete}.
\end{corollary}
\begin{proof}
    See \Cref{Proofs:InfeasibleCase}.
\end{proof}
Although the estimator derived from minimizing the pseudo-likelihood function \(\tilde{\mathcal{L}}_T^{\mathcal{P}_T}(Q)\) exhibits desirable properties, it cannot be directly applied in practice since \(\tilde{\mathcal{L}}_T^{\mathcal{P}_T}(Q)\) is not observed under the assumptions of this section. In the following, we demonstrate that the observed likelihood function \eqref{DiscreteLikelihood} is sufficiently close to the pseudo-likelihood function to inherit these properties. Additionally, the proof of \Cref{LikelihoodEquality} clarifies why the pseudo-likelihood function is unobservable in this setting.

\begin{lemma}\label{LikelihoodEquality}
Grant \Cref{Ass.H,Ass:Estimation,Ass.High-FrequencySampling,Ass:LevyH,Ass:WHPT,Ass:WHPDelta}. Let $\widehat{Q}$ be a solution of the minimization problem $\min _{Q \in \mathbb{R}^{d\times d}}\mathcal{L}_T^{\mathcal{P}_T}(Q)$ and $\tilde{Q}$ be the solution of the minimization problem $\min _{Q \in \mathbb{R}^{d\times d}}\mathcal{\tilde{L}}_T^{\mathcal{P}_T}(Q)$ where $\mathcal{L}_T^{\mathcal{P}_T}$ is defined as in \eqref{DiscreteLikelihood}  and $\mathcal{\tilde{L}}_T^{\mathcal{P}_T}$ as in \eqref{egn:PseudoLikelihood}. Then,
\begin{equation*}
\Vert \widehat{Q}-\tilde{Q}\Vert_F \le \Vert Q_0\Vert_F  \dfrac{2\Delta_{\mathcal{P}_T}^{1/4}}{\kappa_{\min}}
\end{equation*}
with high probability under $\mathbb{P}^{Q_0}$.
\end{lemma}
\begin{proof}
    See \Cref{Proofs:InfeasibleCase}.
\end{proof}
\subsubsection{Feasible case}\label{FeasibleEstimator}
Despite the nice properties of our estimators related to the discretized likelihood functions, these estimators are infeasible since we do not observe $\bold{X}^c$. To address this issue, we propose to work under a new approximate likelihood function. Given a collection $\left(\nu_{\mathcal{P}_T}^{(i)}, i\in\{1,2,\dots,d\} \right)$. For $\bold{x} = \left(\bold{x}_1,\bold{x}_2,\dots,\bold{x}_d \right)$, we define the $d$-valued vector 
\begin{equation*}
    \vec{\bold{1}}(x) = \begin{pmatrix}
\bold{1}_{[0,\nu_{\mathcal{P}_T}^{(1)}]}(\vert\bold{x}_1\vert) \\
\bold{1}_{[0,\nu_{\mathcal{P}_T}^{(2)}]}(\vert\bold{x}_2\vert) \\
\vdots \\
\bold{1}_{[0,\nu_{\mathcal{P}_T}^{(d)}]}(\vert\bold{x}_d\vert)
\end{pmatrix}.
\end{equation*}
We now introduce our discretized, feasible likelihood function
\begin{equation}\label{FeasibleDiscLike}
    \begin{aligned}
\bar{\mathcal{L}}_T^{\mathcal{P}_T}(Q)&=\frac{1}{T} \sum_{n=0}^{N_T-1}\left(\Sigma^{-1} Q \mathbf{X}_{s_n^{N_T}}\right)^{\top}\left( \Delta_{\mathcal{P}_T}^n \bold{\bold{X}}\odot\vec{\bold{1}}(\Delta_{\mathcal{P}_T}^n \bold{\bold{X}})\right)\\&+\frac{1}{2 T} \sum_{n=0}^{N_T-1}\left(\Sigma^{-1/2} Q \mathbf{X}_{s_n^{N_T}}\right)^{\top} \Sigma^{-1/2} Q \mathbf{X}_{s_n^{N_T}}\left(s^{N_T}_{n+1}-s^{N_T}_n\right).
    \end{aligned}
\end{equation}
By comparing it to \eqref{DiscreteLikelihood}, we observe that the primary difference lies in the replacement of the unobserved continuous \(\mathbb{P}^0\)-martingale with the increments of the observed process \(\mathbb{X}\), which are thresholded with respect to the collection \(\left(\nu_{\mathcal{P}_T}^{(i)}, i \in \{1,2,\dots,d\} \right)\). If these thresholded increments were equal to the unobserved continuous martingale part, \(\bar{\mathcal{L}}_T^{\mathcal{P}_T}(Q)\) would inherit the properties of \(\mathcal{L}_T^{\mathcal{P}_T}(Q)\). However, this is not the case. 

An inherent challenge with the threshold-crossing approach is that, in addition to the true jumps generated by the compound Poisson process, large diffusive fluctuations are often misidentified as jumps. The key challenge, therefore, is to control the extent of these misidentifications. This ultimately depends on appropriately scaling the thresholds \(\left(\nu_{\mathcal{P}_T}^{(i)}, i \in \{1,2,\dots,d\}\right)\) and imposing additional assumptions.

\begin{assumption}\label{Ass.GeneralBothCases}
    We assume that
    \begin{equation*}
        \begin{aligned}
        \max_{j\in\{1,\dots,d\}}\mathbb{E}\left( (\bold{Y}_s^{(j)})^2 \right)=:\mathbb{E}\left( (\bold{Y}_s^{(i^*)})^2 \right) &= O\left(d^2 \right),\quad \text{ as } d \rightarrow \infty.
        \end{aligned}
    \end{equation*}
\end{assumption}
\begin{remark}
    We note that \Cref{Ass.H} implies that $$
        \max_{k,j\in\{1,2,\dots,d\}} \vert [Q_0]_{kj} \vert = O(d)\text{ as } d \rightarrow \infty\quad ,
        \max_{k,j\in\{1,2,\dots,d\}} \vert \Sigma_{kj} \vert = O(d)\text{ as } d \rightarrow \infty.$$
\end{remark}

\begin{assumption}\label{Ass.AssFiniteAct}
    \begin{enumerate}
        \item $F\left( \mathbb{R}^d \right)<\infty$ and in fact $F\left( \mathbb{R}^d \right) = O(d)$ as $d \rightarrow \infty$.
        \item We assume that $\nu_{\mathcal{P}_T^{(j)}} = \Delta_{\mathcal{P}_T}^{\beta^{(j)}}$. Further, for any $j$, $\beta^{(j)}\in (0,1/4]$.
        \item $\max_{i\in\{1,\dots, d\}}\tilde{F}^{(i)}\left(-2 \Delta_{\mathcal{P}_T}^{\beta^{(i)}},2 \Delta_{\mathcal{P}_T}^{\beta^{(i)}} \right) = o\left( (d^2T)^{-1}\right), \quad \text{ as } d \rightarrow \infty$. Here $\tilde{F}^{(i)}$ denotes the $i$th marginal of $\tilde{F}$, the jump distribution in the compound Poisson representation of $\mathbb{\tilde{U}}$ in \Cref{LevyItoImpl}.
        \item Let $\beta^* := \max_{i\in\{1,\dots,d\}}\beta^{(i)}$. $\Delta_{\mathcal{P}_T}^{1-2\beta^*} = o\left( (d^8T)^{-1}\right)$, $\quad \text{ as } d \rightarrow \infty$.
        \item $\Delta_{\mathcal{P}_T}^{1/2} = o\left( d^{-10}\right)$, $\quad \text{ as } d \rightarrow \infty$.
        \item Let  $( \mathbf{b},  \Sigma,  F)$ be the characteristic triplet of $\mathbb{L}$. We assume that the modified drift $\bold{\tilde{b}}:= \bold{b} - \int_{\Vert z\Vert \le 1} z F(dz) =0$.
        \item There exists $\delta \in (0,1/4)$ such that $d^{10}\Delta_{\mathcal{P}_T}^{1/2-\delta} = o(1)$ as $d\rightarrow \infty$.
    \end{enumerate}
\end{assumption}
\begin{remark}
  A few remarks are in order. First, the interval in \Cref{Ass.AssFiniteAct}(2) is smaller than those considered in \cite{Lucchese2023, Courgeau2022b, Ma2021}, and this results from the fact that our problem differs slightly from those in the aforementioned papers, as we allow \(d\) to tend to infinity. \Cref{Ass.AssFiniteAct}(3) and (4) address the two types of errors that may arise when applying a threshold: (i) the failure to discard jumps that should have been discarded, and (ii) the incorrect removal of large diffusive fluctuations, mistaking them for jumps. 

\Cref{Ass.AssFiniteAct}(3) controls the first type of error by governing the rate at which the probability of observing jump increments smaller than twice the threshold decreases. \Cref{Ass.AssFiniteAct}(4) controls the second type of error by bounding the probability of observing large diffusive moves that exceed the threshold. While the impact of \Cref{Ass.AssFiniteAct}(4) may not be immediately apparent, it follows directly from \Cref{Lemma:CountBound}.
\end{remark}
\begin{remark}
    Since we will be working under \Cref{Ass:WHPT}, the rates in \Cref{Ass.AssFiniteAct} $3$ and $4$ are essentially the same. Following the remark in \cite{Mai2014}, if $\tilde{F}$ has a bounded Lebesgue density,then \Cref{Ass.AssFiniteAct} $4$ becomes an assumption about $\Delta_{\mathcal{P}_T}^{\beta^{(i)}} = o((d^2T)^{-1})$. An optimal choice of $\beta^{(i)}$ would then be $\beta^{(i)}= 1/4$. 
\end{remark}
\begin{remark}
Since the Lévy component in \eqref{eqn:1} is often regarded as "noise", \Cref{Ass.AssFiniteAct}(6) may seem somewhat unsatisfactory, as it implies that the Lévy process has a non-zero mean. Alternatively, we could impose the assumption that the drift term \(\tilde{\bold{b}}\) is known, as done in \citep{Lucchese2023}. This would allow \(\mathbb{L}\) in \eqref{eqn:1} to be a martingale and thus more suitable to interpret as "noise".
\end{remark}
\begin{lemma}\label{LikelihoodEqualityFeasibleCase}
Grant \Cref{Ass.H,Ass:Estimation,Ass.High-FrequencySampling,Ass:LevyH,Ass:WHPT,Ass:WHPDelta,Ass.GeneralBothCases,Ass.AssFiniteAct}. Let $\bar{Q}$ be a solution of the minimization problem $\min _{Q \in \mathbb{R}^{d\times d}}\mathcal{\bar{L}}_T^{\mathcal{P}_T}(Q)$ and $\widehat{Q}$ be the solution of the minimization problem $\min _{Q \in \mathbb{R}^{d\times d}}\mathcal{L}_T^{\mathcal{P}_T}(Q)$ where $\mathcal{\bar{L}}_T^{\mathcal{P}_T}$ is defined as in \eqref{FeasibleDiscLike}  and $\mathcal{L}_T^{\mathcal{P}_T}$ as in \eqref{DiscreteLikelihood}. Then,
\begin{equation*}
\Vert \bar{Q}-\widehat{Q}\Vert_F \le \Delta_{\mathcal{P}_T}^{1/4} \dfrac{2}{\kappa_{\min}}
\end{equation*}
with high probability under $\mathbb{P}^{Q_0}$.
\end{lemma}
\begin{proof}
    See \Cref{ProofsFiniteAct}.
\end{proof}
\begin{assumption}\label{Ass.InfiniteAct}
    \begin{enumerate}
     \item $F\left( \mathbb{R}^d \right)=\infty$.
        \item  $F_{\vert \{x \in \mathbb{R}^d : \Vert x \Vert_2 > 1 \}}\left( \mathbb{R}^d \right)= O(d)$ as $d\rightarrow \infty$. That is, the intensity of big jumps satisfies $\lambda := F\left(\{x : \Vert x \Vert_2 > 1\}\right) = O(d)$.
        \item We assume that $\nu_{\mathcal{P}_T^{(j)}} = \Delta_{\mathcal{P}_T}^{\beta^{(j)}}$ for $\beta^{(j)} \in (0,1/4)$ and $j \in \{1,\dots,d \}$.
        \item Let $\bar{F}$ denote the probability measure of jumps related to the compound Poisson process $\mathbb{U}$ in \Cref{LevyIto}, and let $\bar{F}^{(i)}$ denote its $i$-th marginal. Then $\max_{i\in\{1,\dots,d\}}\bar{F}^{(i)}\left(-4\Delta_{\mathcal{P}_T}^{\beta^{(i)}},4\Delta_{\mathcal{P}_T}^{\beta^{(i)}}\right) = o\left((d^2T)^{-1}\right)$ as $d\rightarrow \infty$.
        \item  Let $\beta^* := \max_{i\in\{1,\dots,d\}}\beta^{(i)}$. Then $\Delta_{\mathcal{P}_T}^{1-2\beta^*} = o\left( (d^8T)^{-1}\right)$, $\quad \text{ as } d \rightarrow \infty$.
        \item For any $i\in \{1,\dots,d \}$,    $\Delta_{\mathcal{P}_T}^{1/2} = o\left( d^{-10}\right)$, $\quad \text{ as } d \rightarrow \infty$.
        \item Let  $( \mathbf{b},  \Sigma,  F)$ be the characteristic triplet of $\mathbb{L}$. We assume that $\bold{b}=0$.
        \item The fourth moment of $\mathbb{L}$ exists and $\max_{j\in \{1,\dots,d\}}\mathbb{E}\left(\vert\bold{Y}_0^{j}\vert^4\right) = O(d^4)$ as $d\rightarrow \infty$.
        \item $\exists \delta \in (0,1/4)$ and $\varepsilon>0$ s.t. $d^{1/2}\Delta_{\mathcal{P}_T}^{\delta} \rightarrow 0$ as $d\rightarrow \infty$. Further,
        \begin{enumerate}
            \item   $\left(d^{10}\Delta_{\mathcal{P}_T}^{1/2-\beta^*-\delta}\vee d^{6}\Delta_{\mathcal{P}_T}^{1/2-2\beta^*-\delta}\right) \rightarrow 0$ as $d\rightarrow \infty$.
            \item $d^3\Delta_{\mathcal{P}_T}^{\varepsilon -\delta}\rightarrow 0$ as $d\rightarrow \infty$.
            \item For  $0\le s \le \bar{s}$ it holds that $\max_{i\in\{1,\dots,d\}}\mathbb{E}\left( \vert \bold{M}_s^{i}\vert\bold{1}_{\vert \bold{M}_s^{i}\vert \le \bar{s}^{\beta^{(i)}}}\right)= O(\bar{s}^{1+\varepsilon})$ as $\bar{s}\rightarrow 0$.
        \end{enumerate}
    \end{enumerate}
\end{assumption}
    \begin{remark}
         Assumptions $2$--$7$ are equivalent to the assumptions in \Cref{Ass.AssFiniteAct} and ensure that the finite activity part (consisting of small jumps) of $\mathbb{Y}$, denoted by $\tilde{\mathbb{Y}}$ (see \Cref{LevyItoImpl}), satisfies \Cref{Ass.AssFiniteAct}. This distinction is important in the proof of \Cref{LikelihoodEqualityFeasibleCaseInfiniteActCase}, where we control the error arising from thresholding separately for the infinite activity and finite activity components (resp. small and large jumps). Assumption $8$ is similar to assumptions in \cite{Mai2014,Courgeau2022b,Lucchese2023}.
            Assumption $9$ is closely related to Assumption $7.(iii)$ in \cite{Lucchese2023}. The difference comes from the fact that we are bounding slightly different quantities. Suppose that $\delta$ is chosen so that $\Delta_{\mathcal{P}_T}^\delta \sim T^{1/4}\Delta_{\mathcal{P}_T}^{\frac{\varepsilon}{2}}$; then the assumptions are essentially (ignoring $d$) equivalent. By Jensen's inequality, \Cref{Ass.InfiniteAct} also implies $\max_{j\in\{1,\dots,d\}}\mathbb{E}\left(\vert\bold{Y}_0^{(j)}\vert^2\right) = O(d^2)$ and $\max_{j\in\{1,\dots,d\}}\mathbb{E}\left(\vert\bold{Y}_0^{(j)}\vert\right) = O(d)$ as $d\rightarrow\infty$.
    \end{remark}

\begin{remark}
Both \Cref{Ass.AssFiniteAct}(6) and \Cref{Ass.InfiniteAct}(7) serve the same purpose: to remove any deterministic linear drift from $\mathbb{L}$. The assumption takes a different form in the two regimes because of how the L\'evy-It\^o decomposition is written (with or without compensator on the small jumps), but the role is identical - and it is needed because thresholding cannot filter out a deterministic drift: drift increments are of order $\Delta_{\mathcal{P}_T}$ while Brownian increments are of order $\sqrt{\Delta_{\mathcal{P}_T}}$, so any threshold that retains the diffusion necessarily retains the drift, which would otherwise accumulate into a non-vanishing bias in $\bar{Q}$.
\end{remark}

\begin{remark}
By \citet[Example 4.1]{Lucchese2023}, \Cref{Ass.InfiniteAct}(c) is non-vacuous. For standard Gamma marginals, the condition holds componentwise with exponent $\epsilon_i=\beta^{(i)}$, and the same argument extends directly to the two-parameter range $0\leq s\leq \bar{s}$ considered here. Since $\bar{s}\to0$, the maximum over components is governed by the smallest exponent. Consequently, \Cref{Ass.InfiniteAct}(c) holds with
$\epsilon=\beta_{\min}$,
with $\beta_{\min}:=\inf_{d\geq1}\min_{1\leq i\leq d}\beta^{(i)}>0$.
Thus, it is sufficient that the thresholding powers be uniformly bounded away from zero as $d\to\infty$.
\end{remark}

\begin{lemma}\label{LikelihoodEqualityFeasibleCaseInfiniteActCase}
Grant \Cref{Ass.H,Ass:Estimation,Ass.High-FrequencySampling,Ass:LevyH,Ass:WHPT,Ass:WHPDelta,Ass.GeneralBothCases,Ass.InfiniteAct}. Let $\bar{Q}$ be a solution of the minimization problem $\min _{Q \in \mathbb{R}^{d\times d}}\mathcal{\bar{L}}_T^{\mathcal{P}_T}(Q)$ and $\widehat{Q}$ be the solution of the minimization problem $\min _{Q \in \mathbb{R}^{d\times d}}\mathcal{L}_T^{\mathcal{P}_T}(Q)$ where $\mathcal{\bar{L}}_T^{\mathcal{P}_T}$ is defined as in \eqref{FeasibleDiscLike}  and $\mathcal{L}_T^{\mathcal{P}_T}$ as in \eqref{DiscreteLikelihood}. Then,
\begin{equation*}
\Vert \bar{Q}-\widehat{Q}\Vert_F \le 3\Delta_{\mathcal{P}_T}^{\delta} \dfrac{2}{\kappa_{\min}}
\end{equation*}
with high probability under $\mathbb{P}^{Q_0}$, where $\delta$ is as given in \Cref{Ass.InfiniteAct}.
\end{lemma}
\begin{proof}
    See \Cref{ProofsInfAct}.
\end{proof}
\begin{theorem}\label{DiscreteEstimatorConsistencyBothCases}
Suppose \Cref{Ass.H,Ass:Estimation,Ass.High-FrequencySampling,Ass:LevyH,Ass:WHPT,Ass:WHPDelta,Ass.GeneralBothCases} and \Cref{Ass.AssFiniteAct} or \Cref{Ass.InfiniteAct} are satisfied. Let $\bar{Q}$ be a solution of the minimization problem $\min _{Q \in \mathbb{R}^{d\times d}}\mathcal{\bar{L}}_T^{\mathcal{P}_T}(Q)$ and $\widehat{Q}$ be the solution of the minimization problem $\min _{Q \in \mathbb{R}^{d\times d}}\mathcal{L}_T^{\mathcal{P}_T}(Q)$ where $\mathcal{\bar{L}}_T^{\mathcal{P}_T}$ is defined as in \eqref{FeasibleDiscLike}  and $\mathcal{L}_T^{\mathcal{P}_T}$ as in \eqref{DiscreteLikelihood}. Then
\begin{equation*}
\Vert \bar{Q}-Q_0\Vert_F \le 3\Delta_{\mathcal{P}_T}^{\delta} \dfrac{2}{\kappa_{\min}}+ \left\|Q_0\right\|_F \frac{2 \Delta_{\mathcal{P}_T}^{1 / 4}}{\kappa_{\min }}+ c_*\sqrt{\lambda_{\max}(\Sigma)\, \frac{16 \kappa_{\max }}{\kappa_{\min }^2 T}\left(\log 4 d \vee 2 d^2\right)}
\end{equation*}
with high probability under $\mathbb{P}^{Q_0}$, where $\delta$ is as given in \Cref{Ass.InfiniteAct} under that assumption and equal to $1/4$ under \Cref{Ass.AssFiniteAct}.
\end{theorem}
\begin{proof}
    The proof follows trivially from \Cref{PropertiesDiscretizedEstimator} and \Cref{LikelihoodEquality} combined with either \Cref{LikelihoodEqualityFeasibleCase} or \Cref{LikelihoodEqualityFeasibleCaseInfiniteActCase} and the triangle inequality.
\end{proof}
\begin{remark}\label{Rem:functionalFormf}
    We let $h(d,\mathcal{P}_T) := 3\Delta_{\mathcal{P}_T}^{\delta} \dfrac{2}{\kappa_{\min}}+ \left\|Q_0\right\|_F \frac{2 \Delta_{\mathcal{P}_T}^{1 / 4}}{\kappa_{\min }}+ c_*\sqrt{\lambda_{\max}(\Sigma)\, \frac{16 \kappa_{\max }}{\kappa_{\min }^2 T}\left(\log 4 d \vee 2 d^2\right)}$.
\end{remark}
We note that under the assumptions of the Theorem, the right hand side tends to zero as $d$ tends to infinity.

\begin{corollary}\label{corr:cons}
    Suppose \Cref{Ass.H,Ass:Estimation,Ass.High-FrequencySampling,Ass:LevyH,Ass:WHPT,Ass:WHPDelta,Ass.GeneralBothCases} and \Cref{Ass.AssFiniteAct} or \Cref{Ass.InfiniteAct} are satisfied. Then,
    \begin{equation*}
        \Vert \bar{Q}-Q_0\Vert_F = o_\mathbb{P}(1), \quad \text{ as } d\rightarrow \infty. 
    \end{equation*}
\end{corollary}
\begin{proof}
    By \Cref{DiscreteEstimatorConsistencyBothCases}, $\Vert \bar{Q}-Q_0\Vert_F$ is bounded by a quantity tending to $0$ on an event whose probability tends to one; hence $\Vert \bar{Q}-Q_0\Vert_F = o_\mathbb{P}(1)$.
\end{proof}
Up until now, we have assumed $\Sigma$ known. In practice this is not the case. The following corollary shows that when we estimate $Q_0$ by maximum likelihood we do not need to know it.
\begin{corollary}\label{SpecificFormQEst}
    Let $\bar{Q}$ be a solution of the minimization problem $\min _{Q \in \Theta}\mathcal{\bar{L}}_T^{\mathcal{P}_T}(Q)$. Then it holds that 
    \begin{equation*}
        \bar{Q} =  -\sum_{n=0}^{N_T-1}\left( \Delta_{\mathcal{P}_T}^n \bold{\bold{X}}\odot\vec{\bold{1}}(\Delta_{\mathcal{P}_T}^n \bold{\bold{X}})\right)\mathbf{X}_{s_n^{N_T}}^{\top}\left( \sum_{n=0}^{N_T-1}\mathbf{X}_{s_n^{N_T}} \mathbf{X}_{s_n^{N_T}}^{\top}\left(s_{n+1}^{N_T}-s_n^{N_T}\right)\right)^{-1}
    \end{equation*}
\end{corollary}
\begin{proof}
    Since \eqref{FeasibleDiscLike} is convex and differentiable, the proof follows trivially by finding the gradient and setting it equal to $0$.
\end{proof}
\section{Proofs of the estimation theory}
\subsection{Preliminaries and notation}\label{Prelim:EstSec}
We begin by recalling two foundational results on Lévy processes, the Lévy-Khintchine representation and the Lévy-Itô decomposition, which underpin the canonical-space construction and the proofs that follow.

\paragraph{Lévy-Itô decomposition and related properties}\label{LevyIto}
We revisit the filtered probability space $\left(\Omega^{\prime}, \mathcal{F}^{\prime},\left\{\mathcal{F}_t^{\prime}, t \in \mathbb{R}\right\}, \mathbb{P}\right)$ to which all stochastic processes are adapted. We now state two important results regarding Lévy processes which we will use in this section. Namely, the Lévy-Khintchine representation and the Lévy-Itô decomposition. We take the following versions as stated in \cite{Lucchese2023}. Consider a Lévy process $\mathbb{L}$. The Lévy-Khintchine theorems state:

\begin{theorem}[Lévy-Khintchine representation, Sato (1999, Theorem 8.1)] For each $t>0$ the law $\mathcal{L}\left(\mathbf{L}_t\right)$ of the Lévy process $\mathbb{L}$ is infinitely divisible with characteristic triplet $(t \mathbf{b}, t \Sigma, t F)$ where $\mathbf{b} \in \mathbb{R}^d, \Sigma$ is a symmetric, positive semi-definite $d \times d$ matrix and finally $F$ is a Lévy measure on $\mathbb{R}^d$, satisfying that
$$
\mathbb{E}^{\prime}\left[\exp \left\{i z^{\mathrm{T}} \mathbf{L}_t\right\}\right]=\exp \left\{t\left[i \mathbf{b}^{\mathrm{T}} z-\frac{1}{2} z^{\mathrm{T}} \Sigma z+\int_{\mathbb{R}^d}\left(e^{i y^{\mathrm{T}} z}-1-i y^{\mathrm{T}} z \tau(y)\right) F(d y)\right]\right\}, \quad z \in \mathbb{R}^d,
$$
where $\tau: \mathbb{R}^d \rightarrow \mathbb{R}$ is a fixed truncation function, i.e. a bounded measurable function such that $\tau(x)=1+o(\|x\|)$ as $\|x\| \rightarrow 0$ and $\tau(x)=O\left(\|x\|^{-1}\right)$ as $\|x\| \rightarrow \infty$. $(\mathbf{b}, \Sigma, F)$ is then referred to as the characteristic triplet of $\mathbb{L}$.
\end{theorem}
\begin{remark}
  We will always let $\tau(x) := 1_{D^c}$ where $D = \{x \in \mathbb{R}^d: \Vert x \Vert \le 1 \}$ and in turn ignore the $\tau$-dependence of $b$.
\end{remark}
Further, we have
\begin{theorem}[Lévy-Itô decomposition, \cite{Sato} (Theorem 19.2)]\label{Levyito} One has the following representation for the Lévy process $\mathbb{L}$ with characteristics $(\mathbf{b}, \Sigma, F)$ :
\begin{equation}\label{eqn:LevyIto}
\begin{aligned}
\mathbf{L}_t=\mathbf{b} t+\Sigma^{1 / 2} \mathbf{W}_t+\int_0^t \int_{\{x:\|x\|>1\}} & x\, \mu(d s, d x) \\
& +\int_0^t \int_{\{x:\|x\| \leq 1\}} x\,\{\mu(d s, d x)-F(d x)\,ds\}, \quad t \geq 0, \mathbb{P}-\text { a.s. }
\end{aligned}
\end{equation}
where $\mathbb{W}=\left(\mathbf{W}_t\right)_{t \geq 0}$ is a standard Brownian motion on $\left(\Omega^{\prime}, \mathcal{F}^{\prime},\left(\mathcal{F}_t^{\prime}\right)_{t \geq 0}, \mathbb{P}\right)$ and $\Sigma^{1 / 2}$ is a $d \times d$ matrix satisfying $\Sigma=\Sigma^{1 / 2} \Sigma^{1 / 2, T}$ while $\left\{\mu(B): B \in \mathcal{B}\left((0, \infty) \times \mathbb{R}^d \backslash\{0\}\right)\right\}$ is a Poisson random measure independent of $\mathbb{W}$ such that
$$
\mu(d s, d x)=\sum_{r \geq 0: \Delta \mathbf{L}_r \neq 0} \delta_{\left(r, \Delta \mathbf{L}_r\right)}(d s, d x),
$$
and $\Delta \mathbb{L}=\left(\Delta \mathbf{L}_t\right)_{t \geq 0}$ is the jump process of $\mathbb{L}$, i.e. $\Delta \mathbf{L}_t=\mathbf{L}_t-\mathbf{L}_{t-}$. We define the integral
$$
\int_0^t \int_{\{x:\|x\| \leq 1\}} x\,\{\mu(d s, d x)-F(d x)\,ds\}:=\lim _{\epsilon \downarrow 0} \int_0^t \int_{\{x: \epsilon<\|x\| \leq 1\}} x\,\{\mu(d s, d x)-F(d x)\,ds\} .
$$
\end{theorem}

\begin{remark}
    We underline that the above limit should be interpreted in the $\mathcal{L}^2(\mathbb{P})$ sense.
\end{remark}

\paragraph{Canonical space and the continuous martingale part}
With $(\mathbf{b}, \Sigma, F)$ now fixed, we set up the canonical probability space on which the rest of this section operates.
For any $Q \in \mathbb{R}^{d\times d}$, we denote by $\mathbb{Y}^Q$ the unique, strong solution to \eqref{eqn:1} with drift matrix $Q$ (such a solution always exists following e.g. Theorem $6.2.9$ of \cite{Applebaum2009}, and it is c\'adl\'ag). We let $\Omega$ be the set of all c\'adl\'ag functions $\omega: \mathbb{R}_+ \rightarrow \mathbb{R}^d$. We then consider the filtration $\left(\mathcal{F}_t\right)_{t\ge0}$ where for each $t\ge 0$, $\mathcal{F}_t = \bigcap_{s>t}\sigma\left( \omega(u): u\le s \right)$. Then $\mathcal{F}_t$ satisfy the usual hypothesis. Finally, let $\mathcal{F}$ be the smallest $\sigma$-algebra containing $\left(\mathcal{F}_t\right)_{t\ge0}$. We denote $\left(\Omega, \mathcal{F}\right)$ our canonical space. $\left(\Omega^\prime, \mathcal{F}^\prime,\mathbb{P}\right)$ still serves as our stochastic basis. It is then clear that the solution $\mathbb{Y}^Q$ induces a measure on the canonical space through the mapping $\Omega^\prime \ni \omega^\prime \mapsto Y^Q(\omega^\prime)\in \Omega$. Henceforth, we denote this measure $\mathbb{P}^Q$. Such a measure exists for any $Q \in  \mathbb{R}^{d\times d}$ and in turn, we have a family of measures on the canonical space, $\{ \mathbb{P}^Q: Q\in  \mathbb{R}^{d\times d}\}$. We define the canonical process $\mathbb{X}:= \{ \bold{X}_t(\omega) = \omega(t), t\ge0 \}$. Clearly $\mathbb{P}^Q(\mathbb{X}\in A)=\mathbb{P}(\mathbb{Y}^Q\in A)$ for $A\in \mathcal{F}$ and $\mathbb{P}^Q(\bold{X}_t\in B)=\mathbb{P}(\bold{Y}_t^Q\in B)$ for $B\in \mathcal{B}\left( \mathbb{R}^d\right)$. Hence, under $\mathbb{P}^Q$, $\mathbb{X}$ can be viewed as the sample paths of $\mathbb{Y}^Q$.  We denote by $\mathbb{P}_t^Q$ the restriction of the measure $\mathbb{P}^Q$  to $\mathcal{F}_t$. Following \cite{Sorensen2020}, under $\mathbb{P}^Q$, the process 
\begin{equation*}
    \bold{X}_t^c(Q) = \bold{X}_t -\bold{y}_0 - \mathbf{b}\,t - \sum_{s\le t} \Delta \bold{X}_s \bold{1}_{\Vert \Delta \bold{X}_s \Vert >1}  + \int_0^t Q \bold{X}_s ds  - \int_{0}^t \int_{\Vert x \Vert \le 1} x \left( \mu^X -  F\cdot Leb^1 \right)(dx,ds),
\end{equation*}
where $Leb^1$ is the one-dimensional Lebesgue measure and
\begin{equation*}
    \mu^X(dx,ds)= \sum_{r} 1_{\{\Delta \bold{X}_r \ne 0\}} \delta_{(r,\Delta \bold{X}_r)}(dx,ds)
\end{equation*}
is a continuous local martingale on $\left(\Omega, \mathcal{F},\left(\mathcal{F}_t\right)_{t\ge0},\mathbb{P}^Q\right)$. We note, that for $A \in \mathcal{B}\left(\mathbb{R}^d\right)$, it holds that 
\begin{equation}\label{ContMartPart}
    \begin{aligned}
        &\mathbb{P}^Q\left(\bold{X}_t^c(Q) \in A \right) \\&= \mathbb{P}^Q\left(\bold{X}_t -\bold{x}_0 - \sum_{s\le t} \Delta \bold{X}_s \bold{1}_{\Vert \Delta \bold{X}_s \Vert >1}  + \int_0^t Q \bold{X}_s ds  - \int_{0}^t \int_{\Vert x \Vert \le 1} x \left( \mu^X - F\cdot Leb^1 \right)(dx,ds) \in A \right) \\
        &=\mathbb{P}\left(\bold{Y}_t -\bold{y}_0 - \sum_{s\le t} \Delta \bold{Y}_s \bold{1}_{\Vert \Delta \bold{Y}_s \Vert >1}  + \int_0^t Q \bold{Y}_s ds  - \int_{0}^t \int_{\Vert y \Vert \le 1} y \left( \mu^Y - F\cdot Leb^1 \right)(dy,ds) \in A \right) \\
        &=\mathbb{P}\left(\bold{Y}_t -\bold{y}_0 - \int_{0}^t \int_{\Vert y \Vert > 1} y\mu^Y(dy,ds)  + \int_0^t Q \bold{Y}_s ds  - \int_{0}^t \int_{\Vert y \Vert \le 1} y \left( \mu^Y - F\cdot Leb^1 \right)(dy,ds) \in A \right) \\
        &=\mathbb{P}\left(\Sigma^{1/2}\bold{W}_t \in A \right),
    \end{aligned}
\end{equation}
so that $\bold{X}_t^c(Q)$ is a $\mathbb{P}^Q$ Gaussian process. Here, $F$ denotes the Lévy measure associated to the driving Lévy process $\mathbb{L}$ in \eqref{eqn:1}, as defined in \Cref{LevyIto}, and $\mu^Y$ is the jump measure of $\mathbb{Y}$, which coincides with $\mu$ in \Cref{Levyito} since $\Delta \bold{Y}_s = \Delta\bold{L}_s$ for all $s>0$ (the drift integral $-\int_0^s Q\bold{Y}_r\,dr$ is absolutely continuous and therefore jump-free); we write $\mu^X$ and $\mu^Y$ rather than $\mu$ to make explicit whether the jump measure is computed on the canonical process $\mathbb{X}$ or on $\mathbb{Y}$. We used the Lévy Itô decomposition from \Cref{LevyIto} in the second last inequality, and that $\sum_{s\le t} \Delta \bold{Y}_s \bold{1}_{\Vert \Delta \bold{Y}_s \Vert >1}  = \int_{0}^t \int_{\Vert y \Vert > 1} y\mu^Y(dy,ds)$ following e.g.~Theorem $5.1.(4)$ of \cite{JanP2020}. Henceforth, we will drop the $0$ dependence, and  whenever we write $\bold{X}^c_t$ , we refer to the $\mathbb{P}^0$ continuous local martingale. Finally, we introduce a bit of notation:
Let $\mu_{\mathcal{P}_T}$ be the measure on $[0,T]$ defined by 
\begin{equation}\label{dDimMeasure}
    \mu_{\mathcal{P}_T}(dr) = \sum_{n=0}^{N_T-1} \left( s_{n+1}^{N_T} -s_n^{N_T} \right)\delta_{s_n^{N_T}}(dr).
\end{equation}
Further, $\mu_{\mathcal{P}_T}^{\tilde{W}}$ is the random measure on $[0,T]$ defined by 
\begin{equation}\label{stocMeas}
\mu_{\mathcal{P}_T}^{\mathbb{\tilde{W}}}(dr) = \sum_{n=0}^{N_T-1} \Delta_{\mathcal{P}_T}^n \bold{\tilde{W}}\delta_{s_n^{N_T}}(dr),
\end{equation}
where $\mathbb{\tilde{W}}$ is the $\mathbb{P}^{Q_0}$ Wiener process given by $\mathbb{\tilde{W}}:= \Sigma^{-1/2}\mathbb{X}^c(Q_0)$.

Suppose now that $Q_0 \in \mathbb{R}^{d \times d}$ is the drift matrix of an Ornstein-Uhlenbeck process as in \Cref{eqn:1} that we wish to estimate.
 \subsection{Helpful lemmas}
\begin{lemma}\label{IntMatrixExp}
 Let $A$ be a $\mathbb{R}^{d\times d}$ matrix satisfying that its spectrum lies in the left half-plane, i.e.
\begin{equation*}
    r_{\max} = \max\{ \mathrm{Re}\,\lambda: \lambda \in \sigma(A)\}<0.
\end{equation*} Then
\begin{equation*}
    \int_0^\infty e^{sA}ds =-A^{-1}.
\end{equation*} 
\end{lemma}
\begin{proof}
    See e.g. Lemma $15.10.2$ of \cite{Bernstein2009}.
\end{proof}
\begin{lemma}\label{ExponentialBound1}
    Let $A$ be a $\mathbb{R}^{d\times d}$ matrix satisfying that its spectrum lies in the left half-plane, i.e.
\begin{equation*}
    r_{\max} = \max\{ \mathrm{Re}\,\lambda: \lambda \in \sigma(A)\}<0
\end{equation*}
and satisfying that $\max_i\lambda_i\left(\dfrac{A +A^\top}{2} \right) = -\min_i\lambda_i\left(\dfrac{-A +(-A)^\top}{2} \right) \le -\epsilon$.
Then,
\begin{equation}\label{eqn:exponentialBound}
    \Vert e^{tA} \Vert_F^2 \le d\, e^{-2\epsilon t},\quad \text{for } t\ge0.
\end{equation}
\end{lemma}
\begin{proof}
    Under the assumptions of the Lemma, it holds that 
    \begin{equation}
    \Vert e^{tA} \Vert \le e^{-\epsilon t},\quad \text{for } t\ge0,
\end{equation}
see \citet[Lemma $3.1$]{Hu2004}. Using now that for any matrix $B \in \mathbb{R}^{d\times d}$ it holds that 
\begin{equation}\label{eqn:FrobSpecRel}
\Vert B \Vert_F \le \sqrt{\mathrm{rank}(B)}\Vert B \Vert,
\end{equation}
we get the desired result.
\end{proof}

\begin{lemma}\label{ExponentialBound2}
Let $A$ be a $d\times d$ matrix. Then for any sub-multiplicative norm, $\Vert \cdot \Vert$
\begin{equation*}
     \Vert I_d - e^{-sA}\Vert \le (\vert s \vert \Vert A \Vert)e^{\vert s \vert \Vert A \Vert }
\end{equation*}
where $I_d$ denotes the identity. 
\end{lemma}
\begin{proof}
    Using the triangle inequality and the submultiplicativity of the norm, we have that 
    \begin{equation*}
        \begin{aligned}
        \Vert I_d - e^{-sA}\Vert &= \Vert I_d - \sum_{k=0}^\infty \dfrac{(-sA)^k}{k!}\Vert\\
        &\le\Vert \sum_{k=1}^\infty \dfrac{(-sA)^k}{k!}\Vert \\
        &\le  \sum_{k=1}^\infty \dfrac{(\vert s\vert \Vert A\Vert )^k}{k!}\\
        &=  (\vert s \vert \Vert A \Vert )\sum_{k=0}^\infty \dfrac{(\vert s\vert \Vert A\Vert )^k}{(k+1)!} \\
         &\le  (\vert s \vert \Vert A \Vert)e^{\vert s \vert \Vert A \Vert }.
        \end{aligned}
    \end{equation*}
\end{proof}

\begin{lemma}\label{BoundExpectationNormInvariant}
      Let $\mathbb{Y}$ be a diffusion as in \eqref{eqn:1} with drift matrix $Q_0$ satisfying \Cref{Ass:Estimation,Ass.H}. Then, for any $s>0$,
    \begin{equation*}
        \mathbb{E}\left( \Vert \bold{Y}_s\Vert_2^2 \right) = O (d^4), \text{ as } d \rightarrow \infty.
    \end{equation*}
\end{lemma}
\begin{proof}
    By \cref{Ass:Estimation}, $\mathbb{Y}$ is stationary and, hence, the law of $\bold{Y}_0$ is that of its invariant distribution. This allows us to write $\bold{Y}_s$, for any $s\ge0$, compactly as
\begin{equation*}
\bold{Y}_s = \int_{-\infty}^s e^{-(s-r)Q_0}d\bold{L}_r.
\end{equation*}
Using the Lévy-Itô decomposition (\Cref{LevyIto}), this means that equivalently,
\begin{equation*}
    \begin{aligned}
 \bold{Y}_s&=\int_{-\infty}^s \mathrm{e}^{-(s-r) Q_0} b \mathrm{d} r+\int_{-\infty}^s \mathrm{e}^{-(s-r) Q_0} \Sigma^{1/2} \mathrm{d} W_r+\int_{-\infty}^s \int_{\Vert z \Vert_2 \geq 1} \mathrm{e}^{-(s-r) Q_0} z\, \mu(\mathrm{d} r, \mathrm{d} z) \\
&+\int_{-\infty}^s \int_{\Vert z \Vert_2<1} \mathrm{e}^{-(s-r) Q_0} z\, \{\mu(\mathrm{d} r, \mathrm{d} z) - F(dz)\,dr\}\\
&= \int_{-\infty}^s \mathrm{e}^{-(s-r) Q_0} b^* \mathrm{d} r+\int_{-\infty}^s \mathrm{e}^{-(s-r) Q_0} \Sigma^{1/2} \mathrm{d} W_r+\int_{-\infty}^s \int \mathrm{e}^{-(s-r) Q_0} z\, \{\mu(\mathrm{d} r, \mathrm{d} z) - F(dz)\,dr\},
\end{aligned}
\end{equation*}
where $b^* = b + \int_{\Vert z \Vert \ge 1} z F(dz)$.
By the triangle inequality and the Itô Isometry, and using that for real numbers $a,b$, $(a+b)^2 \le 2(a^2+b^2)$, we have under \Cref{Ass:Estimation,Ass.H} that, for any $s\ge0$,
\begin{equation*}
    \begin{aligned}
\dfrac{1}{4}\mathbb{E}\left(\Vert\bold{Y}_s\Vert_2^2\right)&\le \mathbb{E}\left(\Vert\int_{-\infty}^s \mathrm{e}^{-(s-r) Q_0} b^* \mathrm{d} r\Vert_2^2\right)+\mathbb{E}\left(\Vert\int_{-\infty}^s \mathrm{e}^{-(s-r) Q_0} \Sigma^{1/2} \mathrm{d} W_r\Vert_2^2\right)\\
&+\mathbb{E}\left(\Vert\int_{-\infty}^s \int \mathrm{e}^{-(s-r) Q_0} z\, \{\mu(\mathrm{d} r, \mathrm{d} z) - F(dz)\,dr\} \Vert_2^2\right) \\
& \le
\Vert\int_{-\infty}^s \mathrm{e}^{-(s-r) Q_0} b^* \mathrm{d} r\Vert_2^2+\int_{-\infty}^s \Vert\mathrm{e}^{-(s-r) Q_0} \Sigma^{1/2}\Vert_F^2 \mathrm{d} r\\
&+\int_{-\infty}^s \int\Vert\ \mathrm{e}^{-(s-r) Q_0} z \Vert_2^2 F(dz)dr \\
\le
&\Vert Q_0^{-1}b^*\Vert_2^2+\int_{0}^{\infty} \Vert \mathrm{e}^{r(-Q_0)} \Vert_F^2  \mathrm{d} r\Vert\Sigma\Vert\\
&+\int_{0}^\infty \Vert \mathrm{e}^{r(-Q_0)} \Vert_F^2 dr\int\Vert z \Vert_2^2 F(dz) \\
\le
&\Vert Q_0^{-1}\Vert_F^2 \Vert b^*\Vert_2^2+\int_{0}^{\infty} d \mathrm{e}^{-2\epsilon r}   \mathrm{d} r\Vert\Sigma\Vert\\
&+\int_{0}^\infty  d\mathrm{e}^{-2\epsilon r}   dr\int\Vert z \Vert_2^2 F(dz)
\\
\le
&\Vert Q_0^{-1}\Vert_F^2 \Vert b^*\Vert_2^2+\dfrac{d}{2 \epsilon}\Vert\Sigma\Vert\\
&+\dfrac{d}{2 \epsilon}\int\Vert z \Vert_2^2 F(dz) =: f_1(d), 
    \end{aligned}
\end{equation*}
where we used \Cref{IntMatrixExp,ExponentialBound1} together with \Cref{Ass.H}. By \Cref{Ass.H}, the term $\Vert Q_0^{-1}\Vert_F^2\Vert b^*\Vert_2^2 = O(d^2)\cdot O(d^2) = O(d^4)$ dominates, where $\Vert b^*\Vert_2^2 = O(d^2)$ follows since $\Vert b\Vert_2^2 = O(d)$ and $\Vert \int_{\Vert z\Vert \ge 1} z F(dz)\Vert_2^2 \le (\int \Vert z\Vert_2^2 F(dz))^2 = O(d^2)$. The remaining terms are $O(d^2)$. Hence $f_1(d) = O(d^4)$, and the desired result follows.
\end{proof}
\begin{lemma}\label{BoundOnIncrements}
    Let $\mathbb{Y}$ be a diffusion as in \eqref{eqn:1} with drift matrix $Q_0 \in \mathbb{R}^{d\times d}$ satisfying \Cref{Ass:Estimation,Ass.H}. Then
    \begin{equation*}
        \mathbb{E}\left( \Vert \bold{Y}_s -\bold{Y}_0 \Vert_2^2 \right) = O (d^3s), \text{ as } d^3s \rightarrow 0. 
    \end{equation*}
\end{lemma}
\begin{proof}
     By \Cref{Ass:Estimation}, $\mathbb{Y}$ is stationary and, hence, the law of $\bold{Y}_0$ is that of its invariant distribution. This allows us to for any $s\ge0$ write $\bold{Y}_s$ compactly as
\begin{equation*}
\bold{Y}_s = \int_{-\infty}^s e^{-(s-r)Q_0}d\bold{L}_r.
\end{equation*}
Let $\bold{Z}_s := \bold{Y}_s - \bold{Y}_0 $.
Now, by the Lévy-Itô decomposition, for all $s>0$,
\begin{equation*}
\begin{aligned}
& \bold{Z}_s=\int_0^s \mathrm{e}^{-(s-r) Q_0} b \mathrm{d} r+\int_0^s \mathrm{e}^{-(s-r) Q_0} \Sigma^{1/2} \mathrm{d} W_r+\int_0^s \int_{\Vert z\Vert \geq 1} \mathrm{e}^{-(s-r) Q_0} z\, \mu(\mathrm{d} r, \mathrm{d} z) \\
&+\int_0^s \int_{\Vert z\Vert<1} \mathrm{e}^{-(s-r) Q_0} z\, \{\mu(\mathrm{d} r, \mathrm{d} z) - F(dz)\,dr\}.
\end{aligned}
\end{equation*}
Let $b^*: = b + \int_{\Vert z \Vert \ge 1} z F(dz)$. By the Itô isometry and using again that for real numbers $a,b$, $(a+b)^2 \le 2(a^2+b^2)$, we get
\begin{equation*}
\begin{aligned}
\dfrac{1}{4}\mathbb{E}\left(\Vert\bold{Z}_s\Vert_2^2\right) &\le\Vert\int_{0}^s \mathrm{e}^{-(s-r) Q_0} b^* \mathrm{d} r\Vert_2^2+\int_{0}^s \Vert\mathrm{e}^{-(s-r) Q_0} \Sigma^{1/2}\Vert_F^2 \mathrm{d} r\\
&+\int_{0}^s \int\Vert\ \mathrm{e}^{-(s-r) Q_0} z \Vert_2^2 F(dz)dr \\
&\le\Vert Q_0^{-1} \left(I_d - e^{-sQ_0} \right)b^* \Vert_2^2+\int_{0}^s \Vert \mathrm{e}^{-(s-r)  Q_0}\Vert_F^2  \mathrm{d} r\Vert\Sigma\Vert\\
&+\int_{0}^s \Vert \mathrm{e}^{-(s-r)  Q_0}\Vert_F^2  \int\Vert  z \Vert_2^2 F(dz)dr \\
&\le \Vert Q_0^{-1} \Vert_F^2 d \left((\vert s \vert \Vert Q_0 \Vert)^2 e^{2\vert s \vert \Vert Q_0 \Vert } \right) \Vert b^* \Vert_2^2+ \int_{0}^s \Vert \mathrm{e}^{r(-  Q_0)}\Vert_F^2  \mathrm{d} r\Vert\Sigma\Vert\\
&+\int_{0}^s \Vert \mathrm{e}^{r  (-Q_0)}\Vert_F^2  \int\Vert  z \Vert_2^2 F(dz)dr\\
&\le \Vert Q_0^{-1} \Vert_F^2 \left((\vert s \vert \Vert Q_0\Vert )^2 e^{2\vert s \vert \Vert Q_0 \Vert } \right) \Vert b^* \Vert_2^2+ \int_{0}^s \Vert \mathrm{e}^{r(-  Q_0)}\Vert_F^2  \mathrm{d} r\Vert\Sigma\Vert\\
&+\int_{0}^s \Vert \mathrm{e}^{r  (-Q_0)}\Vert_F^2  \int\Vert  z \Vert_2^2 F(dz)dr
\\
&\le \Vert Q_0^{-1} \Vert_F^2 \left((\vert s \vert \Vert Q_0 \Vert_F)^2 e^{2\vert s \vert \Vert Q_0 \Vert_F } \right) \Vert b^* \Vert_2^2+ \dfrac{e^1 d}{2\epsilon}(1-e^{-2\epsilon s})\Vert\Sigma\Vert\\
&+\dfrac{e^1 d}{2\epsilon}(1-e^{-2\epsilon s})\int\Vert  z \Vert_2^2 F(dz)\\
&\le \Vert Q_0^{-1} \Vert_F^2 \left((\vert s \vert \Vert Q_0 \Vert_F)^2 e^{2\vert s \vert \Vert Q_0 \Vert_F } \right) \Vert b^* \Vert_2^2+ e^1 d s\Vert\Sigma\Vert\\
&+e^1 d s \int\Vert  z \Vert_2^2 F(dz) =: f_2(d,s),
\end{aligned}
\end{equation*}
where we used \Cref{IntMatrixExp,ExponentialBound1,ExponentialBound2}, the inequality $1-e^{-x} \le x$ for $x \ge 0$, and so $\epsilon$ is the constant from \Cref{Ass.H}. Following \Cref{Ass.H}, we have the desired result.
\end{proof}
\begin{lemma}\label{MatrixCalc3}
    For $T>0$, consider the measurable space $\left([0,T],\mathcal{B}\left([0,T]\right), Leb \right)$. For any function $f: \mathbb{R} \mapsto \mathbb{R}^d$, $d \in \mathbb{N}$ satisfying $f \in L^1\left([0,T], \mathbb{R}^d\right)$, and $0\le a <b \le T$, it holds that
    \begin{equation*}
        \Vert \int_{a}^b f(x) \dd x \Vert_2 \le   \int_{a}^b \Vert f(x)\Vert_2  \dd x. 
    \end{equation*}
\end{lemma}
\begin{proof}
Let $g := \int_a^b f(x) \dd x$. Then by Cauchy-Schwarz, we find
\begin{equation*}
    \begin{aligned}
        \Vert g \Vert_2^2 &= g^\top g\\
        &= \int_a^b g^\top f(x) \dd x \\
         &\le \int_a^b \Vert g \Vert_2 \Vert f(x) \Vert_2 \dd x \\
         &=  \Vert g \Vert_2 \int_a^b  \Vert f(x) \Vert_2 \dd x. 
    \end{aligned}
\end{equation*}
Dividing by $\Vert g \Vert_2$ now yields the result.
\end{proof}
\begin{lemma}\label{MatrixCalc1}
    Let $A \in \mathbb{R}^{d\times d}$. Further, let $\mathbb{X}$ and $\mathbb{Y}$ be stochastic processes on $\mathbb{R}^d$ and $\mu$ a measure on $[0,T]$. Then 
\begin{equation*}
\int_{0}^T\bold{X}_s^\top A \bold{Y}_s \mu(\dd s) =     \Tr\left(A^\top\int_{0}^T\bold{X}_s \bold{Y}_s^\top \mu(\dd s) \right).
\end{equation*}
\end{lemma}
\begin{proof}
    We note that 
    \begin{equation*}
        \begin{aligned}
        \int_{0}^T\bold{X}_s^\top A \bold{Y}_s \mu(\dd s) &= \int_0^T \sum_{k=1}^d\sum_{l=1}^d \bold{X}_s^{(k)}A_{kl}\bold{Y}_s^{(l)} \mu(\dd s)\\
        &= \sum_{k=1}^d\sum_{l=1}^d A_{kl}\int_0^T \bold{X}_s^{(k)} \bold{Y}_s^{(l)} \mu(\dd s).
        \end{aligned}
    \end{equation*}
        Further,
    \begin{equation*}
        \left[A^\top \int_{0}^T\bold{X}_s \bold{Y}_s^\top \mu(\dd s)\right]_{ij} =\sum_{k=1}^d A_{ki} \int_0^T\bold{X}_s^{(k)}\bold{Y}_s^{(j)} \mu(\dd s),
    \end{equation*}
    and so
    \begin{equation*}
    \Tr\left(\left[A^\top \int_{0}^T\bold{X}_s \bold{Y}_s^\top \mu(\dd s)\right]\right) =\sum_{l=1}^d\sum_{k=1}^d A_{kl} \int_0^T\bold{X}_s^{(k)}\bold{Y}_s^{(l)}\mu(\dd s),
    \end{equation*}
    yielding the desired result.
\end{proof}
\begin{lemma}\label{MatrixCalc2}
    Let $A \in \mathbb{R}^{d\times d}$. Further, let $\mathbb{X}$ be stochastic processes on $\mathbb{R}^d$ and $\mu$ a d-dimensional measure on $[0,T]$. Then 
\begin{equation*}
\int_{0}^T\bold{X}_s^\top A \mu( \dd s) =     \Tr\left(A^\top\int_{0}^T\bold{X}_s \mu^\top(\dd s) \right).
\end{equation*}
\end{lemma}
\begin{proof}
Identical to the proof of \Cref{MatrixCalc1}, replacing the product measure $\mathbb{Y}_s\,\mu(ds)$ by the vector-valued measure $\mu(ds)$ component-wise.
\end{proof}

\begin{lemma}\label{Matrixcalc4}
    Consider $\bold{A}, \bold{B} \in \mathbb{R}^{d\times d}$. It holds that 
    \begin{equation*}
        \Tr \left(\mathbf{B}\bold{A}\bold{B}^\top \right) =\vect\left(\mathbf{B}^\top \right)^\top \left(I_{d\times d} \otimes \bold{A} \right)\vect\left(\mathbf{B}^\top\right).
    \end{equation*}
\end{lemma}
\begin{proof}
    We first note that $\left( \bold{B} \bold{A} \right)_{ij} = \sum_{k=1}^d \bold{B}_{ik}\bold{A}_{kj}$. Hence $\left(\bold{B}\bold{A}\bold{B}^\top \right)_{ij} = \sum_{l=1}^d \left(\sum_{k=1}^d \bold{B}_{ik}\bold{A}_{kl}\right)\bold{B}^\top_{lj}$. In turn, we conclude that 
    \begin{equation*}
        \Tr \left(\bold{B}\bold{A}\bold{B}^\top \right) = \sum_{j=1}^d\sum_{l=1}^d \left(\sum_{k=1}^d \bold{B}_{jk}\bold{A}_{kl}\right)\bold{B}^\top_{lj}.
    \end{equation*}
    We then note that 
    \begin{equation*}
        vec\left(\bold{B}^\top\right)^\top \left(I_{d\times d}\otimes \bold{A}\right) = \left(\bold{B}_{1\bullet}\bold{A},\dots,\bold{B}_{d\bullet }\bold{A}\right) 
    \end{equation*}
    Hence,
    \begin{equation*}
        \left( vec\left(\bold{B}^\top\right)^\top \left(I_{d\times d}\otimes \bold{A}\right) \right)_{i + (j-1)d}= \left(\bold{B}_{j\bullet} \bold{A}\right)_i = \sum_{k=1}^d \bold{B}_{jk}\bold{A}_{ki}, \quad i,j=1,2,\dots,d.
    \end{equation*}
    We thus conclude that 
    \begin{equation*}
        \begin{aligned}
            vec\left(\bold{B}^\top \right)^\top \left(I_{d\times d} \otimes \bold{A} \right)vec\left(\bold{B}^\top\right) &= \sum_{j=1}^d\sum_{l=1}^d\left( vec\left(\bold{B}^\top\right)^\top \left(I_{d\times d}\otimes \bold{A}\right) \right)_{l + (j-1)d} \left(vec\left(\bold{B}^\top\right)\right)_{l + (j-1)d} \\&=
            \sum_{j=1}^d\sum_{l=1}^d\left(\sum_{k=1}^d \bold{B}_{jk}\bold{A}_{kl}\right)\bold{B}^\top_{lj},
        \end{aligned}
    \end{equation*}
    as desired.
\end{proof}

\subsection{Proofs in the continuous setting}\label{Proofs:ContSetting}
\begin{proof}[Proof of \Cref{DexHeimerNormBound}]
We follow \citet{Dexheimer2022}. Hence, we assume we are on the event %\newline 
$$E=E_1\cap E_2 :=\left\{ \inf _{\mathbf{B} \in \mathbb{R}^{d \times d} \backslash\{0\}} \frac{\|\mathbf{B} X\|_{L^2}^2}{\|\mathbf{B}\|_F^2} \geq \frac{\kappa_{\min }}{2}\right\} \bigcap \left\{\sup _{\mathbf{B} \in \mathbb{R}^{d \times d}, \mathbf{B} \neq 0} \frac{\left\langle\varepsilon_T, \mathbf{B}\right\rangle_F}{\|\mathbf{B}\|_S} \leq c_* \sqrt{\frac{\kappa_{\max }}{T}} \right\},$$ where $c_*$ is the constant defined in Proposition $3.3$ of \citet{Dexheimer2022}, which does not depend on $d$ or $T$. By the same argument as in the proof of Lemma $2.1$ of \citet{Dexheimer2022} (applied with no penalty), we have that for any $Q \in \mathcal{M}_d(\mathbb{R})$
\begin{multline}\label{eqn:DexFirstBound}
\left\|\Sigma^{-1/2}\left(\widehat{Q}-Q_0\right) X\right\|_{L^2}^2-\left\|\Sigma^{-1/2}\left(Q-Q_0\right) X\right\|_{L^2}^2 \\
\leq 2\left\langle\varepsilon_T, \Sigma^{-1/2}(Q-\widehat{Q})\right\rangle_F-\left\|\Sigma^{-1/2}(\widehat{Q}-Q) X\right\|_{L^2}^2,
\end{multline}
where
\begin{equation*}
\varepsilon_T^{\top}:=\frac{1}{T} \int_0^T \bold{X}_s \mathrm{d} \widetilde{\bold{X}}^c(Q_0)_s^{\top},
\end{equation*}
with $\widetilde{\bold{X}}^c(Q_0) = \Sigma^{-1/2}\bold{X}^c(Q_0)$ being a $\mathbb{P}^{Q_0}$-Wiener process (see \Cref{Prelim:EstSec}). We now note that, for $0<\varepsilon_0<1$,
$
\|\mathbf{B}\|_S:=\|\mathbf{B}\|_* \vee \sqrt{\log \left(4 \varepsilon_0^{-1}\right)}\|\mathbf{B}\|_F
$. Further, $\|\mathbf{B}\|_* \le \sqrt{2}d\Vert \mathbf{B} \Vert_F $, for all $ \mathbf{B} \in \mathbb{R}^{d \times d}$ following the arguments in \cite{Dexheimer2022} $(2.7)$ . We thus notice, that on $E$, 
\begin{equation*}
    \begin{aligned}
  2\left\langle\varepsilon_T, \Sigma^{-1/2}(Q-\widehat{Q})\right\rangle_F&\le 2c_* \sqrt{\dfrac{\kappa_{\max}}{T}} \Vert \Sigma^{-1/2}(Q-\widehat{Q}) \Vert_S\\&\le2 c_* \sqrt{\dfrac{\kappa_{\max}}{T}}\left( \sqrt{\log 4 \epsilon_0^{-1}} \vee \sqrt{2}d\right) \Vert \Sigma^{-1/2}(Q-\widehat{Q}) \Vert_F \\
    & \le 2c_* \sqrt{\dfrac{2 \kappa_{\max}}{\kappa_{\min} T}}\left( \sqrt{\log 4 \epsilon_0^{-1}} \vee \sqrt{2}d\right) \Vert \Sigma^{-1/2}(Q-\widehat{Q}) X \Vert _{L_2}\\
     & \le c_*^2 \dfrac{2\kappa_{\max}}{\kappa_{\min} T}\left( \log 4 \epsilon_0^{-1} \vee 2d^2\right) + \Vert \Sigma^{-1/2}(Q-\widehat{Q}) X \Vert _{L_2}^2,
    \end{aligned}
\end{equation*}
where we used the AM-GM inequality. Inserting this into \eqref{eqn:DexFirstBound} gives us that
\begin{equation*}
\left\|\Sigma^{-1/2}\left(\widehat{Q}-Q_0\right) X\right\|_{L^2}^2-\left\|\Sigma^{-1/2}\left(Q-Q_0\right) X\right\|_{L^2}^2 \leq c_*^2 \dfrac{2\kappa_{\max}}{\kappa_{\min} T}\left( \log 4 \epsilon_0^{-1} \vee 2d^2\right) ,
\end{equation*}
as desired. Now we just need to provide a lower bound on the probability of $E$. Note that 
\begin{equation*}
    \begin{aligned}
        \mathbb{P}^{Q_0}\left( E \right) &= 1 - \mathbb{P}^{Q_0}(E^c) \\
        &=1 - \mathbb{P}^{Q_0}(E_1^c\cup E_2^c) \\
        & \ge 1 - \dfrac{\epsilon_0}{2} - \left(21(d \wedge e)\right)^d H(T,\dfrac{\kappa_{\min}}{3d}),
    \end{aligned}
\end{equation*}
where we used De Morgan's law in the second equality and Propositions $3.1$ and $3.3$ of \cite{Dexheimer2022} in providing the lower bound.
\end{proof}
\begin{proof}[Proof of \Cref{CorrollaryBoundMaxLike}]
    First we apply \Cref{DexHeimerNormBound} with $Q=Q_0$ and use that we are on $E$ to get, \begin{equation*}
\dfrac{\kappa_{\min}}{2} \Vert \Sigma^{-1/2}\left(\widehat{Q}-Q_0\right)\Vert_F^2 \le \left\|\Sigma^{-1/2}\left(\widehat{Q}-Q_0\right) X\right\|_{L^2}^2 \leq c_*^2 \dfrac{2\kappa_{\max}}{\kappa_{\min} T}\left( \log 4 \epsilon_0^{-1} \vee 2d^2\right),
\end{equation*} 
with probability  $1 - \dfrac{\epsilon_0}{2} - \left(21(d \wedge e)\right)^d H(T,\dfrac{\kappa_{\min}}{3d})$. Now using the submultiplicativity of the Frobenius norm, we have that
\begin{equation*}
    \Vert \widehat{Q}-Q_0 \Vert_F^2=\Vert\Sigma^{1/2} \Sigma^{-1/2} \left(\widehat{Q}-Q_0\right)\Vert_F^2 \le \Vert \Sigma^{1/2} \Vert^2\Vert \Sigma^{-1/2} \left(\widehat{Q}-Q_0\right) \Vert_F^2 = \lambda_{\max}(\Sigma)\,\Vert \Sigma^{-1/2}(\widehat{Q}-Q_0)\Vert_F^2,
\end{equation*}
and finally letting $\epsilon_0 = 1/d$ yields the result.
\end{proof}
\subsection{Proofs: Infeasible case}\label{Proofs:InfeasibleCase}
\begin{proof}[Proof of \Cref{StrictConvexityOfC_THat}]
We follow the arguments in \cite{Gaiffas2019} and observe that, for $u\in\mathbb{R}^d$ such that $\Vert u \Vert_2=1$, we have that 
\begin{equation*}
    u^\top \widehat{\mathbf{C}}_{\mathcal{P}_T} u = \dfrac{1}{T}\sum_{n=0}^{N_T-1} \left(u^\top\bold{X}_{s_n^{N_T}}\right)^2 \left( s_{n+1}^{N_T} - s_n^{N_T} \right), 
\end{equation*}
which is equal to $0$ only if $\mathbb{X}_{\mathcal{P}_T}$ is included in a hyperplane of $\mathbb{R}^d$. Since by assumption $\Sigma$ is positive definite, this event is a null set.
\end{proof}
\begin{proof}[Proof of \Cref{Prop:LevyH_discrete}]
    We note that 
    \begin{equation*}
        \begin{aligned}
            \widehat{\mathbf{C}}_{\mathcal{P}_T} - \widehat{\mathbf{C}}_T &=\dfrac{1}{T} \sum_{n=0}^{N_T-1}\bold{X}_{s_n^{N_T}}\bold{X}_{s_n^{N_T}}^\top (s^{N_T}_{n+1}-s^{N_T}_n)-\dfrac{1}{T} \int_0^T \bold{X}_s \ \bold{X}_s^\top \dd s\\
            &=\dfrac{1}{T} \sum_{n=0}^{N_T-1}\bold{X}_{s_n^{N_T}}\bold{X}_{s_n^{N_T}}^\top \int_{s^{N_T}_n}^{s^{N_T}_{n+1}}\dd s-\dfrac{1}{T} \int_0^T \bold{X}_s  \bold{X}_s^\top \dd s\\
            &=\dfrac{1}{T} \sum_{n=0}^{N_T-1}\int_{s^{N_T}_n}^{s^{N_T}_{n+1}}\bold{X}_{s_n^{N_T}}\bold{X}_{s_n^{N_T}}^\top \dd s-\dfrac{1}{T} \int_0^T \bold{X}_s  \bold{X}_s^\top \dd s\\
            &=\dfrac{1}{T} \int_{0}^{T}\bold{X}_{\lfloor s \rfloor_{\mathcal{P}_T} } \bold{X}_{\lfloor s \rfloor_{\mathcal{P}_T} }^\top  \dd s-\dfrac{1}{T} \int_0^T \bold{X}_s  \bold{X}_s^\top \dd s
            \\
          &= \frac{1}{T} \int_0^T\left(\mathbf{X}_{\lfloor s\rfloor_{\mathcal{P}_T}}-\mathbf{X}_s\right) \mathbf{X}_{\lfloor s\rfloor_{\mathcal{P}_T}}^{\top} \mathrm{d} s\\&+\frac{1}{T} \int_0^T \mathbf{X}_s\left(\mathbf{X}_{\lfloor s\rfloor_{\mathcal{P}_T}}^{\top}-\mathbf{X}_s^{\top}\right) \mathrm{d} s.
        \end{aligned}
    \end{equation*}
    Hence, following the submultiplicativity of the Frobenius norm, \Cref{MatrixCalc3}, Cauchy-Schwarz and stationarity of $\bold X$ under $\mathbb{P}^{Q_0}$ (\Cref{Ass:Estimation}), which gives $\mathbb{E}_{\mathbb{P}^{Q_0}}\Vert\bold{X}_t\Vert_2^2=\mathbb{E}_{\mathbb{P}^{Q_0}}\Vert\bold{X}_0\Vert_2^2$ for all $t\ge 0$, we find that
    \begin{equation*}
        \begin{aligned}
            \mathbb{E}_{\mathbb{P}^{Q_0}}\left(\Vert\widehat{\mathbf{C}}_{\mathcal{P}_T} - \widehat{\mathbf{C}}_T \Vert_F\right) &\le \dfrac{1}{T} \int_{0}^{T}\mathbb{E}_{\mathbb{P}^{Q_0}}\left( \Vert\bold{X}_{\lfloor s \rfloor_{\mathcal{P}_T} }- \bold{X}_s\Vert_2 \Vert \bold{X}_{\lfloor s \rfloor_{\mathcal{P}_T} }\Vert _2 \right) \dd s\\
            &+\dfrac{1}{T} \int_{0}^{T}\mathbb{E}_{\mathbb{P}^{Q_0}}\left( \Vert\bold{X}_{\lfloor s \rfloor_{\mathcal{P}_T} }- \bold{X}_s\Vert_2 \Vert \bold{X}_{s }\Vert _2 \right) \dd s\\
            &\le \dfrac{2}{T}\int_0^T\sqrt{\mathbb{E}_{\mathbb{P}^{Q_0}}\left(\Vert \mathbf{X}_{0}\Vert_2^2\right)}\sqrt{\mathbb{E}_{\mathbb{P}^{Q_0}}\left(\Vert\bold{X}_s-\mathbf{X}_{\lfloor s\rfloor_{\mathcal{P}_T}}\Vert_2^2\right)}\dd s\\
            &\le 2\sqrt{\mathbb{E}_{\mathbb{P}^{Q_0}}\left(\Vert \mathbf{X}_{0}\Vert_2^2\right)}\sqrt{\sup_{u\in [0,\Delta_{\mathcal{P}_T}]}\mathbb{E}_{\mathbb{P}^{Q_0}}\left(\Vert\bold{X}_u-\mathbf{X}_{0}\Vert_2^2\right)}\\
            &= O\left(d^{7/2} \Delta_{\mathcal{P}_T}^{1/2}\right), \text{ as } d \rightarrow \infty.
        \end{aligned}
    \end{equation*}
    We used \Cref{BoundExpectationNormInvariant,BoundOnIncrements} in the last equality. Since $\widehat{\mathbf{C}}_{\mathcal{P}_T} - \widehat{\mathbf{C}}_T$ is a symmetric, real-valued matrix, by the Rayleigh-Ritz Theorem, for any $u \in \mathbb{R}^d$,
    \begin{equation*}
        \left|u^{\top}\left(\widehat{\mathbf{C}}_{\mathcal{P}_T}-\widehat{\mathbf{C}}_T\right) u\right| \le \Vert \widehat{\mathbf{C}}_{\mathcal{P}_T}-\widehat{\mathbf{C}}_T \Vert \le \Vert \widehat{\mathbf{C}}_{\mathcal{P}_T}-\widehat{\mathbf{C}}_T \Vert _F.
    \end{equation*}
    By Markov's inequality and the fact that 
    \begin{equation*}
        \begin{aligned}
           \mathbb{P}^{Q_0}\left(\left|u^{\top}\left(\widehat{\mathbf{C}}_{\mathcal{P}_T}-\mathbf{C}_{\infty}\right) u\right| \geq r\right)&=\mathbb{P}^{Q_0}\left(\left|u^{\top}\left(\widehat{\mathbf{C}}_{\mathcal{P}_T}-\widehat{\mathbf{C}}_T+\widehat{\mathbf{C}}_T-\mathbf{C}_{\infty}\right) u\right| \geq r\right)\\ &\le    \mathbb{P}^{Q_0}\left(\left|u^{\top}\left(\widehat{\mathbf{C}}_{\mathcal{P}_T}-\widehat{\mathbf{C}}_T\right) u\right| \geq r/2\right)\\&+\mathbb{P}^{Q_0}\left(\left|u^{\top}\left(\widehat{\mathbf{C}}_T-\mathbf{C}_{\infty}\right) u\right| \geq r/2\right),
        \end{aligned}
    \end{equation*} we get the desired result. 
\end{proof}
\begin{lemma}\label{Lemma1DexheimerDiscrete}
Suppose \Cref{Ass:Estimation} holds and recall the definition of $\tilde{\mathcal{L}}_T^{\mathcal{P}_T}(\cdot)$ in \eqref{egn:PseudoLikelihood}. If $\tilde{Q}$ is a solution of the minimization problem $\min _{Q \in \mathbb{R}^{d \times d}}\tilde{\mathcal{L}}_T^{\mathcal{P}_T}(Q)$, then $\tilde{Q}$ satisfies, for all $Q \in \mathbb{R}^{d \times d}$,
\begin{equation*}
    \begin{aligned}
\left\|\Sigma^{-1/2}\left(\tilde{Q}-Q_0\right) X\right\|_{L_{\mathcal{P}_T}^2}^2-\left\|\Sigma^{-1/2}\left(Q-Q_0\right) X\right\|_{L_{\mathcal{P}_T}^2}^2 &\leq 2\left\langle\varepsilon_{\mathcal{P}_T}^{N_T}, \Sigma^{-1/2}(Q-\tilde{Q})\right\rangle_F\\&-\left\|\Sigma^{-1/2}(\tilde{Q}-Q) X\right\|_{L_{\mathcal{P}_T}^2}^2,\quad  \mathbb{P}^{Q_0}\text{-almost surely },
\end{aligned}
\end{equation*}
where
\begin{equation}\label{epsilon_Tisc}
(\varepsilon_{\mathcal{P}_T}^{n})^{\top}:=\frac{1}{T} \sum_{k=0}^{(N_T-1)\wedge n} X_k \left(\Delta_{\mathcal{P}_T}^k \Sigma^{-1/2}\bold{X}^c(Q_0)\right)^\top,
\end{equation}
with $\left(\Sigma^{-1/2}\bold{X}^c(Q_0)\right)$ being the $\mathbb{P}^{Q_0}$-Wiener process introduced above \Cref{Prelim:EstSec}. 
\end{lemma}
\begin{proof}
We follow \cite{Dexheimer2022}. For $Q\in \mathbb{R}^{d\times d}$, we notice that we can rewrite $\tilde{\mathcal{L}}_T^{\mathcal{P}_T}(Q)$ using the rules for the trace operator and \Cref{MatrixCalc1,MatrixCalc2} in the following manner:
\begin{equation}\label{TrInfeasDisc}
    \begin{aligned} \tilde{\mathcal{L}}_T^{\mathcal{P}_T}(Q)&=\frac{1}{T} \int_0^T\left(\Sigma^{-1/2} Q \mathbf{X}_s\right)^{\top}  \mu_{\mathcal{P}_T}^{\mathbb{\tilde{W}}}(ds)+\frac{1}{2 T} \int_0^T\left(\Sigma^{-1/2} Q \mathbf{X}_s\right)^{\top} \Sigma^{-1/2} Q \mathbf{X}_s\mu_{\mathcal{P}_T}(ds)\\
    &-\frac{1}{T} \int_0^T\left(\Sigma^{-1/2} Q \mathbf{X}_s\right)^{\top} \Sigma^{-1/2} Q_0 \mathbf{X}_s\mu_{\mathcal{P}_T}(ds)\\ 
    &= \frac{1}{T} \int_0^T\left(\Sigma^{-1/2} Q \mathbf{X}_s\right)^{\top}  \mu_{\mathcal{P}_T}^{\mathbb{\tilde{W}}}(ds)+\frac{1}{2 T} \int_0^T\left(\Sigma^{-1/2} Q \mathbf{X}_s\right)^{\top} \Sigma^{-1/2} (Q-2Q_0) \mathbf{X}_s\mu_{\mathcal{P}_T}(ds)\\
    &= \frac{1}{T} \int_0^T\mathbf{X}_s^\top Q^\top(\Sigma^{-1/2} )^\top \mu_{\mathcal{P}_T}^{\mathbb{\tilde{W}}}(ds)+\frac{1}{2 T} \int_0^T\mathbf{X}_s^\top Q^\top\Sigma^{-1} (Q-2Q_0) \mathbf{X}_s\mu_{\mathcal{P}_T}(ds)\\
    &= \Tr\biggl(\frac{1}{T} \Sigma^{-1/2} Q\int_0^T\mathbf{X}_s \left(\mu_{\mathcal{P}_T}^{\mathbb{\tilde{W}}}\right)^\top (ds)+\frac{1}{2 T}  (Q-2Q_0)^\top(\Sigma^{-1})^\top Q\int_0^T\mathbf{X}_s  \mathbf{X}_s^\top\mu_{\mathcal{P}_T}(ds)\biggr)\\
    &= \Tr\biggl( \Sigma^{-1/2} Q \left(\epsilon_{\mathcal{P}_T}^{N_T}\right)^\top+\frac{1}{2}  Q\widehat{C}_{{\mathcal{P}_T}}(Q-2Q_0)^\top\Sigma^{-1}\biggr)\\
    &= \Tr\biggl( \Sigma^{-1/2} Q \left(\epsilon_{\mathcal{P}_T}^{N_T}\right)^\top+\frac{1}{2}  Q\widehat{C}_{{\mathcal{P}_T}}Q^\top\Sigma^{-1}-Q\widehat{C}_{{\mathcal{P}_T}}Q_0^\top\Sigma^{-1}\biggr).
    \end{aligned}
\end{equation}
We notice a few things. First of all, following the above calculations, the gradient of $\tilde{\mathcal{L}}_T^{\mathcal{P}_T}$ is $(\Sigma^{-1/2})^\top\epsilon_{\mathcal{P}_T}^{N_T}+\Sigma^{-1}\left(Q-Q_0\right)\widehat{C}_{{\mathcal{P}_T}}$. Further, the Hessian of $\tilde{\mathcal{L}}_T^{\mathcal{P}_T}$ is $\widehat{C}_{\mathcal{P}_T}\otimes\Sigma^{-1}$, which is positive definite since $\Sigma$ is positive definite and, following \Cref{StrictConvexityOfC_THat}, $\widehat{C}_{{\mathcal{P}_T}}$ is $\mathbb{P}^Q$-almost surely positive definite for all $Q$ (in particular for $Q_0$). This implies $\mathbb{P}^{Q_0}$-almost sure strict convexity of $\tilde{\mathcal{L}}_T^{\mathcal{P}_T}(Q)$, hence $\tilde{Q}$ is unique, and $\tilde{\mathcal{L}}_T^{\mathcal{P}_T}$ is also differentiable in $Q$. Since $\tilde{Q}$ is the unique minimizer of a differentiable strictly convex function, the first-order condition $\nabla \tilde{\mathcal{L}}_T^{\mathcal{P}_T}(\tilde{Q}) = \bold{0}$ holds, i.e.\
$\bold{0}=(\Sigma^{-1/2})^\top\epsilon_{\mathcal{P}_T}^{N_T}+\Sigma^{-1}\left(\tilde{Q}-Q_0\right)\widehat{C}_{{\mathcal{P}_T}}$. Hence, following the derivations in \cite{Dexheimer2022} and using \Cref{MatrixCalc2}, we note that
\begin{equation*}
\begin{aligned}
&\left\|\Sigma^{-1/2}\left(\tilde{Q}-Q_0\right) X\right\|_{L_{\mathcal{P}_T}^2}^2-\left\|\Sigma^{-1/2}\left(Q-Q_0\right) X\right\|_{L_{\mathcal{P}_T}^2}^2+\left\|\Sigma^{-1/2}(\tilde{Q}-Q) X\right\|_{L_{\mathcal{P}_T}^2}^2\\
\quad&=\left\langle\left(\tilde{Q}-Q_0\right)^{\top} \Sigma^{-1}\left(\tilde{Q}-Q_0\right)-\left(Q-Q_0\right)^{\top} \Sigma^{-1}\left(Q-Q_0\right)+(\tilde{Q}-Q)^{\top} \Sigma^{-1}(\tilde{Q}-Q), \widehat{C}_{{\mathcal{P}_T}}\right\rangle_F\\
& =2\left\langle\left(\tilde{Q}-Q_0\right)^{\top} \Sigma^{-1}(\tilde{Q}-Q), \widehat{C}_{{\mathcal{P}_T}}\right\rangle_F \\
& =2\left\langle(\tilde{Q}-Q)^{\top}, \widehat{C}_{{\mathcal{P}_T}}\left(\tilde{Q}-Q_0\right)^{\top} \Sigma^{-1}\right\rangle_F \\
& =2\left\langle\Sigma^{-1}\left(\tilde{Q}-Q_0\right)\widehat{C}_{{\mathcal{P}_T}},(\tilde{Q}-Q)\right\rangle_F \\
& =2\left\langle\left(\Sigma^{-1/2}\right)^{\top} \varepsilon_{\mathcal{P}_T}^{N_T},(Q-\tilde{Q})\right\rangle_F \\
& = 2\left\langle\varepsilon_{\mathcal{P}_T}^{N_T},\Sigma^{-1/2}(Q-\tilde{Q}) \right\rangle_F.
\end{aligned}
\end{equation*}
In the first equality, we used that, following \Cref{MatrixCalc2},
\begin{equation*}
\begin{aligned}
\left\|\Sigma^{-1/2}\left(\tilde{Q}-Q_0\right) X\right\|_{L_{\mathcal{P}_T}^2}^2&=\dfrac{1}{T}\int_0^T\bold{X}_s^\top\left(\Sigma^{-1/2}(\tilde{Q}-Q_0)\right)^\top\left(\Sigma^{-1/2}(\tilde{Q}-Q_0)\right) \bold{X}_s \mu_{\mathcal{P}_T}(\dd s) \\
&=\Tr\left( \left(\Sigma^{-1/2}(\tilde{Q}-Q_0)\right)^\top\left(\Sigma^{-1/2}(\tilde{Q}-Q_0)\right) \dfrac{1}{T}\int_0^T\bold{X}_s\bold{X}_s^\top \mu_{\mathcal{P}_T}(\dd s)\right),
\end{aligned}
\end{equation*} 
while in the third equality, we used that 
\begin{equation*}
    \begin{aligned}
    \langle \tilde{Q}^\top \Sigma^{-1}Q_0,\widehat{C}_{{\mathcal{P}_T}}\rangle_F &=  \Tr \left(Q_0^T \Sigma^{-1}\tilde{Q}\widehat{C}_{{\mathcal{P}_T}}\right) \\&=\Tr \left(\widehat{C}_{{\mathcal{P}_T}}^\top \tilde{Q}^\top \Sigma^{-1}Q_0\right)
    \\&=\Tr \left( \tilde{Q}^\top \Sigma^{-1}Q_0\widehat{C}_{{\mathcal{P}_T}}^\top\right)\\
    &=\langle Q_0^\top \Sigma^{-1}\tilde{Q},\widehat{C}_{{\mathcal{P}_T}}\rangle_F.
    \end{aligned}
\end{equation*}
\end{proof}

\begin{proposition}\label{Prop3.1OfDexheimerDiscrete}
Grant \Cref{Ass.H,Ass:Estimation,Ass:LevyH}. For any $T>0$ and sampling scheme $\mathcal{P}_T$, define
\begin{equation}\label{Q_def}
    Q_{\mathcal{P}_T}:=\left\{\sup _{\mathbf{B} \in \mathbb{B}_2(1)}\left|\operatorname{Tr}\left(\mathbf{B}\left(\widehat{C}_{\mathcal{P}_T} -\mathbf{C}_{\infty}\right) \mathbf{B}^{\boldsymbol{T}}\right)\right| \leq \frac{\kappa_{\min}}{2}\right\}.
\end{equation}
It holds that
\begin{equation*}
Q_{\mathcal{P}_T} \subseteq\left\{\lambda_{\min}\left(\widehat{C}_{\mathcal{P}_T}\right) \geq \frac{\kappa_{\min }}{2}\right\} \cap \left\{\lambda_{\max}\left(\widehat{C}_{\mathcal{P}_T}\right) \leq \frac{\kappa_{\min}}{2}+\kappa_{\max}\right\}
\end{equation*}
and
\begin{equation*}
\mathbb{P}^{Q_0}\left(Q_{\mathcal{P}_T}\right) \geq 1-(21(d \wedge \mathrm{e}))^d \left(H\left(T,\frac{\kappa_{\min}}{3d}\right) +H_2\left(\Delta_{\mathcal{P}_T},\frac{\kappa_{\min}}{3d}\right)\right),
\end{equation*}
where $H(\cdot,\cdot),H_2(\cdot,\cdot)$ are as in \Cref{Prop:LevyH_discrete}.
\end{proposition}
\begin{proof}
    The proof is identical to the proof of Proposition $3.1$ in \cite{Dexheimer2022}. Only now we use that, by \Cref{MatrixCalc1},
    \begin{equation*}
        \Vert \bold{B}X \Vert_{L_{\mathcal{P}_T}^2}^2 = \Tr\left(\bold{B}\widehat{C}_{\mathcal{P}_T} \bold{B}^T\right),
    \end{equation*}
    combined with \Cref{Matrixcalc4} and \Cref{Prop:LevyH_discrete}.
\end{proof}

\begin{lemma}\label{Prop3.3OfDexheimerDiscrete_first}
Grant \Cref{Ass:Estimation}. There exists a universal constant $c_0>0$ such that, for any $\mathcal{D} \subset \mathbb{R}^{d \times d}$, it holds for all $u>0$,
\begin{equation*}
\mathbb{P}^{Q_0}\left(\sup _{\mathbf{B} \in \mathcal{D}}\left\langle\varepsilon_{\mathcal{P}_T}^{N_T}, \mathbf{B}\right\rangle_F \mathbf{1}_{Q_{\mathcal{P}_T}\left(\kappa_{\max }\right)} \leq c_0 \sqrt{\frac{6 \kappa_{\max }}{T}}(w(\mathcal{D})+u \operatorname{rad}(\mathcal{D}))\right) \geq 1-2 \exp \left(-u^2\right),
\end{equation*}
where  the Gaussian width $w(\mathcal{D})$ and radius $\operatorname{rad}(\mathcal{D})$ of a set $\mathcal{D} \subset \mathbb{R}^{d \times d}$ are defined by setting
\begin{equation*}
w(\mathcal{D}):=\mathbf{E}\left[\sup _{\mathbf{B} \in \mathcal{D}}\langle\operatorname{vec}(\mathbf{B}), Z\rangle_F\right] \quad \text { and } \quad \operatorname{rad}(\mathcal{D}):=\sup _{\mathbf{B} \in \mathcal{D}}\|\mathbf{B}\|_F,
\end{equation*}
where $Z \sim \mathcal{N}\left(0, \operatorname{Id}_{d^2 \times d^2}\right)$. Here, $\varepsilon_{\mathcal{P}_T}^n$ is defined as in \eqref{epsilon_Tisc} and $Q_{\mathcal{P}_T}$ as in \eqref{Q_def}.
\end{lemma}
\begin{proof}
We follow \cite{Dexheimer2022} and their proof of Lemma $3.2$. From their derivations, we have
\begin{equation}\label{eqn:Importantquant}
\begin{aligned}
    Q_{\mathcal{P}_T}\left(\kappa_{\max }\right) & =\left\{\sup _{\mathbf{B} \in \mathbb{B}(1)} \operatorname{Tr}\left(\mathbf{B}\left(\hat{\mathbf{C}}_{\mathcal{P}_T}-\mathbf{C}_{\infty}\right) \mathbf{B}^{\top}\right) \leq \kappa_{\max }\right\} \\
    & \subseteq\left\{\forall \mathbf{B}_1, \mathbf{B}_2 \in \mathcal{D}: \operatorname{Tr}\left(\left(\mathbf{B}_1-\mathbf{B}_2\right) \hat{\mathbf{C}}_{\mathcal{P}_T}\left(\mathbf{B}_1-\mathbf{B}_2\right)^{\top}\right)\le 2 \kappa_{\max }\left\|\mathbf{B}_1-\mathbf{B}_2\right\|_F^2\right\} .
\end{aligned}
\end{equation}
Let $\bold{B}_1 \neq \bold{B}_2 \in \mathcal{D} $ be given. Recall the definitions of $\mu_{\mathcal{P}_T}^{\mathbb{\tilde{W}}}(ds)$ and $\tilde{\mathbb{W}}$ from \eqref{stocMeas}: $\mu_{\mathcal{P}_T}^{\tilde{W}}$ is the random measure on $[0,T]$ defined by 
\begin{equation}
\mu_{\mathcal{P}_T}^{\mathbb{\tilde{W}}}(dr) = \sum_{n=0}^{N_T-1} \Delta_{\mathcal{P}_T}^n \bold{\tilde{W}}\delta_{s_n^{N_T}}(dr),
\end{equation}
where $\mathbb{\tilde{W}}$ is the $\mathbb{P}^{Q_0}$ Wiener process given by $\mathbb{\tilde{W}}:= \Sigma^{-1/2}\mathbb{X}^c(Q_0)$. Let $\lfloor s\rfloor_{\mathcal{P}_T}:=s_k^{N_T}$ for $s_k^{N_T}$  such that $s_k^{N_T} \leq s<s_{k+1}^{N_T}$ for $k \in\{1, \ldots, N\}$. It holds that 
\begin{equation*}
\begin{aligned}
    \int_0^T \left( (\bold{B}_1-\bold{B}_2)\bold{X}_s\right)^\top \mu_{\mathcal{P}_T}^{\mathbb{\tilde{\bold{W}}}}(ds)&=\sum_{k=0}^{N_T-1}\left( (\bold{B}_1-\bold{B}_2)\bold{X}_{s_k^{N_T}}\right)^\top\left(\tilde{\bold{W}}_{s_{k+1}^{N_T}}-\tilde{\bold{W}}_{s_k^{N_T}}\right) \\
    &=\sum_{k=0}^{N_T-1}\left( (\bold{B}_1-\bold{B}_2)\bold{X}_{s_k^{N_T}}\right)^\top \int_{s_k^{N_T}}^{s_{k+1}^{N_T}} d\tilde{\bold{W}}_s \\
    &=\sum_{k=0}^{N_T-1}\int_{s_k^{N_T}}^{s_{k+1}^{N_T}} \left( (\bold{B}_1-\bold{B}_2)\bold{X}_{s_k^{N_T}}\right)^\top d\tilde{\bold{W}}_s\\
&=\int_0^T\left( (\bold{B}_1-\bold{B}_2)\bold{X}_{\lfloor s\rfloor_{\mathcal{P}_T}}\right)^\top d\tilde{\bold{W}}_s, \quad \mathbb{P}^{Q_0}-a.s.,
\end{aligned}
\end{equation*}
where we used Exercise $III.2.30$ of \cite{Karatzas2012}, the $L^2$ integrability condition of $\bold{X}$, that we are working with forward looking increments and that $\bold{X}_{s_k^{N_T}}$ is $\mathcal{F}_{s_k^{N_T}}$-adapted, in the third equality.
In other words, at time $T$ the $\mathbb{P}^{Q_0}$ discrete martingale $\int_0^T \left( (\bold{B}_1-\bold{B}_2)\bold{X}_s\right)^\top \mu_{\mathcal{P}_T}^{\mathbb{\tilde{W}}}(ds)$ is almost surely equal to that of the $\mathbb{P}^{Q_0}$ continuous martingale  $\int_0^T\left( (\bold{B}_1-\bold{B}_2)\bold{X}_{\lfloor s \rfloor}\right)^\top d\tilde{\bold{W}}_s$ with quadratic variation given by 
\begin{equation*}
    \begin{aligned}
    \int_0^T\Vert\left( (\bold{B}_1-\bold{B}_2)\bold{X}_{\lfloor s \rfloor}\right)\Vert_2^2 ds &= \sum_{k=0}^{N_T-1}\int_{s_k^{N_T}}^{s_{k+1}^{N_T}}\Vert\left( (\bold{B}_1-\bold{B}_2)\bold{X}_{s_{k}^{N_T}}\right)\Vert_2^2 ds
\\&=\sum_{k=0}^{N_T-1}\Vert\left( (\bold{B}_1-\bold{B}_2)\bold{X}_{s_{k}^{N_T}}\right)\Vert_2^2 (s_{k+1}^{N_T}-s_k^{N_T}) \\
&= T \Tr\left((\bold{B}_1-\bold{B}_2)\hat{\bold{C}}_{\mathcal{P}_T}(\bold{B}_1-\bold{B}_2)^\top\right),
\end{aligned}
\end{equation*}
where we used \Cref{MatrixCalc1} in the last equality.
Using this insight, we can follow the derivation in \cite{Dexheimer2022} (e.g. we may use Bernstein's inequality for continuous martingales) to get that 
\begin{equation*}
    \begin{aligned}
\mathbb{E}_{Q_0}\biggl(\exp &\biggl( \frac{\left\langle\varepsilon_{\mathcal{P}_T}^{N_T}, \mathbf{B}_1-\mathbf{B}_2\right\rangle_F^2 \mathbf{1}_{Q_{\mathcal{P}_T}\left(\kappa_{\max }\right)}}{12 T^{-1} \kappa_{\max }\left\|\mathbf{B}_1-\mathbf{B}_2\right\|_F^2} \biggr)\biggr)\\
& \leq \int_0^{\infty} \mathbb{P}^{Q_0}\left(\exp \left(\frac{\left\langle\varepsilon_{\mathcal{P}_T}^{N_T}, \mathbf{B}_1-\mathbf{B}_2\right\rangle_F^2 \mathbf{1}_{Q_{\mathcal{P}_T}\left(\kappa_{\max }\right)}}{12 T^{-1} \kappa_{\max }\left\|\mathbf{B}_1-\mathbf{B}_2\right\|_F^2}\right)>u\right) \mathrm{d} u \\
& \leq 1+\int_1^{\infty} \mathbb{P}^{Q_0}\left(\left|\left\langle\varepsilon_{\mathcal{P}_T}^{N_T}, \mathbf{B}_1-\mathbf{B}_2\right\rangle_F\right| \mathbf{1}_{Q_{\mathcal{P}_T}\left(\kappa_{\max }\right)}>\sqrt{12 T^{-1} \kappa_{\max } \log (u)}\left\|\mathbf{B}_1-\mathbf{B}_2\right\|_F\right) \mathrm{d} u \\
& =1+\int_1^{\infty} \mathbb{P}^{Q_0}\left(\left|\int_0^T\left(\left(\mathbf{B}_1-\mathbf{B}_2\right) X_s\right)^{\top} \mu_{\mathcal{P}_T}^{\tilde{W}}(\mathrm{d}s) \right|>\sqrt{12 T \kappa_{\max } \log (u)}\left\|\mathbf{B}_1-\mathbf{B}_2\right\|_F, Q_{\mathcal{P}_T}\left(\kappa_{\max }\right)\right) \mathrm{d} u\\
&\le1+\int_1^{\infty} \mathbb{P}^{Q_0}\biggl(\left|\int_0^T\left(\left(\mathbf{B}_1-\mathbf{B}_2\right) X_s\right)^{\top} \mu_{\mathcal{P}_T}^{\tilde{W}}(\mathrm{d}s) \right|>\sqrt{12 T \kappa_{\max } \log (u)}\left\|\mathbf{B}_1-\mathbf{B}_2\right\|_F,\\&\quad \quad \int_0^T\Vert\left( (\bold{B}_1-\bold{B}_2)\bold{X}_{\lfloor s \rfloor}\right)\Vert_2^2 ds  \le 2T\kappa_{\max }\left\|\mathbf{B}_1-\mathbf{B}_2\right\|_F^2 \biggr) \mathrm{d} u \\
& \le 1+ 2\int_1^\infty u^{-3}\mathrm{d} u =2,
\end{aligned}
\end{equation*}
where we used that by \Cref{MatrixCalc2}, it holds that $\langle\varepsilon_{\mathcal{P}_T}^{N_T}, \mathbf{B}_1-\mathbf{B}_2\rangle_F = T^{-1}\int_0^T\left(\left(\mathbf{B}_1-\mathbf{B}_2\right) X_s\right)^{\top} \mu_{\mathcal{P}_T}^{\tilde{W}}(\mathrm{d}s)$, where once more, $\mu_{\mathcal{P}_T}^{\tilde{W}}(\mathrm{d}s)$ is as in \eqref{stocMeas}.
By the arguments for Lemma $3.2$ in \cite{Dexheimer2022}, we have our result.
\end{proof}
\begin{lemma}\label{Prop3.3OfDexheimerDiscrete}
Grant \Cref{Ass:Estimation}. Let $0<\varepsilon_0<1$.
Then it holds for all samplings scheme $\mathcal{P}_T$, that
\begin{equation*}
\mathbb{P}^{Q_0}\left(\sup _{\mathbf{B} \in \mathbb{R}^{d \times d}, \mathbf{B} \neq 0} \frac{\left\langle\varepsilon_{\mathcal{P}_T}^{N_T}, \mathbf{B}\right\rangle_F}{\|\mathbf{B}\|_S} \bold{1}_{Q_{\mathcal{P}_T}\left(\kappa_{\max }\right)} \leq c_* \sqrt{\frac{\kappa_{\max }}{T}}\right) \geq 1-\frac{\varepsilon_0}{2},
\end{equation*}
where
\begin{equation*}
\|\mathbf{B}\|_S:=\|\mathbf{B}\|_* \vee \sqrt{\log \left(4 \varepsilon_0^{-1}\right)}\|\mathbf{B}\|_F, \quad \forall \mathbf{B} \in \mathbb{R}^{d \times d},
\end{equation*} and
$c_*:=c_0\left(\sqrt{\frac{3 \pi}{\log (2)}}+\sqrt{300}\right)$ with $c_0$ being the constant from \Cref{Prop3.3OfDexheimerDiscrete_first}.
\end{lemma}
\begin{proof}
    The proof follows by the arguments to Proposition $3.3$ in \cite{Dexheimer2022}, invoking our \Cref{Prop3.3OfDexheimerDiscrete_first} in place of their Lemma $3.2$.
\end{proof}
\begin{proof}[Proof of \Cref{PropertiesDiscretizedEstimator}]
    We assume we are on the set 
    \begin{equation*}
        E=E_1\cap E_2 :=\left\{ \inf _{\mathbf{B} \in \mathbb{R}^{d \times d} \backslash\{0\}} \frac{\|\mathbf{B} X\|_{L_{\mathcal{P}_T}^2}^2}{\|\mathbf{B}\|_2^2} \geq \frac{\kappa_{\min }}{2}\right\} \bigcap \left\{\sup _{\mathbf{B} \in \mathbb{R}^{d \times d}, \mathbf{B} \neq 0} \frac{\left\langle\varepsilon_{\mathcal{P}_T}^{N_T}, \mathbf{B}\right\rangle_F}{\|\mathbf{B}\|_S} \leq c_* \sqrt{\frac{\kappa_{\max }}{T}} \right\}.
    \end{equation*}
    Here $\|\mathbf{B}\|_S$ is defined as in \Cref{Prop3.3OfDexheimerDiscrete} and thus depends on $\varepsilon_0$.
    Following \Cref{Prop3.1OfDexheimerDiscrete} and \Cref{Prop3.3OfDexheimerDiscrete} it holds that 
    \begin{equation*}
    \begin{aligned}
        \mathbb{P}^{Q_0}\left( E \right) &= 1 - \mathbb{P}^{Q_0}(E^c) \\
        &=1 - \mathbb{P}^{Q_0}(E_1^c\cup E_2^c) \\
        & \ge 1 - \dfrac{\varepsilon_0}{2} - (21(d \wedge \mathrm{e}))^d \left(H\left(T,\frac{\kappa_{\min }}{3d}\right) +H_2\left(\Delta_{\mathcal{P}_T},\frac{\kappa_{\min }}{3d}\right)\right).
    \end{aligned}
\end{equation*}
By repeating the arguments of \Cref{DexHeimerNormBound} and \Cref{CorrollaryBoundMaxLike}, now invoking \Cref{Lemma1DexheimerDiscrete} in place of the (continuous-time) Lemma~$2.1$ argument of \cite{Dexheimer2022}, we obtain the bound on the norm difference for the maximum likelihood estimate of $Q_0$.
\end{proof}

\begin{proposition}\label{Prop:BoundEntryWisel1Norm}
    Suppose \Cref{Ass:Estimation,Ass.H,Ass.High-FrequencySampling} are satisfied. Let $ \tilde{C}_{{\mathcal{P}_T}} =\dfrac{1}{T} \int_0^T \bold{X}_{\lfloor s \rfloor_{\mathcal{P}_T} } \bold{X}_s^\top \dd s$. Then,
    \begin{equation*}
        \Vert\widehat{C}_{{\mathcal{P}_T}} - \tilde{C}_{{\mathcal{P}_T}}^\top\Vert_F \le \Delta_{\mathcal{P}_T}^{1/4},
    \end{equation*}
    with probability greater than $1-\dfrac{\sqrt{\mathbb{E}_{\mathbb{P}^{Q_0}}\left(\Vert \mathbf{X}_{0}\Vert_2^2\right)} \sqrt{f_2(d,\Delta_{\mathcal{P}_T})}}{\Delta_{\mathcal{P}_T}^{1/4}} = 1 - O(d^{7/2}\Delta^{1/4}_{\mathcal{P}_T})$, as $d\rightarrow \infty$, under $\mathbb{P}^{Q_{0}}$ with $f_2$ defined as in \Cref{BoundOnIncrements}.
\end{proposition}
\begin{proof}
    We note that 
    \begin{equation*}
        \begin{aligned}
            \widehat{C}_{\mathcal{P}_T} - \tilde{C}_{{\mathcal{P}_T}}^\top &=\dfrac{1}{T} \sum_{n=0}^{N_T-1}\bold{X}_{s_n^{N_T}}\bold{X}_{s_n^{N_T}}^\top (s^{N_T}_{n+1}-s^{N_T}_n)-\dfrac{1}{T} \int_0^T \bold{X}_s \bold{X}_{\lfloor s \rfloor_{\mathcal{P}_T} }^\top \dd s\\
            &=\dfrac{1}{T} \sum_{n=0}^{N_T-1}\bold{X}_{s_n^{N_T}}\bold{X}_{s_n^{N_T}}^\top \int_{s^{N_T}_n}^{s^{N_T}_{n+1}}\dd s-\dfrac{1}{T} \int_0^T \bold{X}_s \bold{X}_{\lfloor s \rfloor_{\mathcal{P}_T} }^\top \dd s\\
            &=\dfrac{1}{T} \sum_{n=0}^{N_T-1}\int_{s^{N_T}_n}^{s^{N_T}_{n+1}}\bold{X}_{s_n^{N_T}}\bold{X}_{s_n^{N_T}}^\top \dd s-\dfrac{1}{T} \int_0^T \bold{X}_s \bold{X}_{\lfloor s \rfloor_{\mathcal{P}_T} }^\top \dd s\\
            &=\dfrac{1}{T} \int_{0}^{T}\bold{X}_{\lfloor s \rfloor_{\mathcal{P}_T} } \bold{X}_{\lfloor s \rfloor_{\mathcal{P}_T} }^\top  \dd s-\dfrac{1}{T} \int_0^T \bold{X}_s \bold{X}_{\lfloor s \rfloor_{\mathcal{P}_T} }^\top \dd s
            \\
            &=\dfrac{1}{T} \int_{0}^{T}\left(\bold{X}_{\lfloor s \rfloor_{\mathcal{P}_T} }- \bold{X}_s\right) \bold{X}_{\lfloor s \rfloor_{\mathcal{P}_T} }^\top  \dd s.
        \end{aligned}
    \end{equation*}
    Hence, following the submultiplicativity of the Frobenius norm, \Cref{MatrixCalc3}, Cauchy-Schwarz and stationarity of $\bold X$ under $\mathbb P^{Q_0}$ (\Cref{Ass:Estimation}), which gives $\mathbb E_{\mathbb P^{Q_0}}\Vert\bold X_t\Vert_2^2 = \mathbb E_{\mathbb P^{Q_0}}\Vert\bold X_0\Vert_2^2$ for all $t \ge 0$, we find that
    \begin{equation*}
        \begin{aligned}
            \mathbb{E}_{\mathbb{P}^{Q_0}}\left(\Vert \widehat{C}_{\mathcal{P}_T} - \tilde{C}_{{\mathcal{P}_T}}^\top \Vert _F\right) &\le \dfrac{1}{T} \int_{0}^{T}\mathbb{E}_{\mathbb{P}^{Q_0}}\left( \Vert\bold{X}_{\lfloor s \rfloor_{\mathcal{P}_T} }- \bold{X}_s\Vert_2 \Vert \bold{X}_{\lfloor s \rfloor_{\mathcal{P}_T} }\Vert _2 \right) \dd s\\
            &\le \dfrac{1}{T}\int_0^T\sqrt{\mathbb{E}_{\mathbb{P}^{Q_0}}\left(\Vert \mathbf{X}_{0}\Vert_2^2\right)}\sqrt{\mathbb{E}_{\mathbb{P}^{Q_0}}\left(\Vert\bold{X}_s-\mathbf{X}_{\lfloor s\rfloor_{\mathcal{P}_T}}\Vert_2^2\right)}\dd s\\
            &\le \sqrt{\mathbb{E}_{\mathbb{P}^{Q_0}}\left(\Vert \mathbf{X}_{0}\Vert_2^2\right)}\sqrt{\sup_{u\in [0,\Delta_{\mathcal{P}_T}]}\mathbb{E}_{\mathbb{P}^{Q_0}}\left(\Vert\bold{X}_u-\mathbf{X}_{0}\Vert_2^2\right)}\\
            &\le \sqrt{\mathbb{E}_{\mathbb{P}^{Q_0}}\left(\Vert \mathbf{X}_{0}\Vert_2^2\right)} \sqrt{f_2(d,\Delta_{\mathcal{P}_T})}\\
            &= O\left(d^{7/2} \Delta_{\mathcal{P}_T}^{1/2}\right),\quad \text{as } d\rightarrow \infty.
        \end{aligned}
    \end{equation*}
    We used \Cref{BoundOnIncrements} in the last inequality and $f_2(\cdot,\cdot)$ is the function from there. By Markov's inequality, we thus get the desired result.
\end{proof}
\begin{proposition}\label{Prop:UpperBoundC_TInverse}
    Grant \Cref{Ass.H,Ass:Estimation,Ass.High-FrequencySampling}. Then,
    \begin{equation*}
        \Vert \widehat{C}_{\mathcal{P}_T}^{-1} \Vert \le \dfrac{2}{\kappa_{\min}},
    \end{equation*}
    with probability greater than $1-(21(d \wedge \mathrm{e}))^d \left(H\left(T,\frac{\kappa_{\min}}{3d}\right) +H_2\left(\Delta_{\mathcal{P}_T},\frac{\kappa_{\min}}{3d}\right)\right) $ under $\mathbb{P}^{Q_0}$.
\end{proposition}
\begin{proof}
    By \Cref{Prop3.1OfDexheimerDiscrete}, the event $Q_{\mathcal{P}_T}$ defined in \eqref{Q_def} satisfies $Q_{\mathcal{P}_T} \subseteq \{\lambda_{\min}(\widehat{C}_{\mathcal{P}_T}) \ge \kappa_{\min}/2\}$ and has probability at least $1-(21(d \wedge \mathrm{e}))^d \left(H\left(T,\frac{\kappa_{\min}}{3d}\right) +H_2\left(\Delta_{\mathcal{P}_T},\frac{\kappa_{\min}}{3d}\right)\right)$ under $\mathbb{P}^{Q_0}$. Hence on $Q_{\mathcal{P}_T}$,
    \begin{equation*}
        \Vert \widehat{C}_{\mathcal{P}_T}^{-1} \Vert = \lambda_{\max}\left(\widehat{C}_{\mathcal{P}_T}^{-1} \right) = \dfrac{1}{\lambda_{\min}\left(\widehat{C}_{\mathcal{P}_T} \right)} \le \dfrac{2}{\kappa_{\min}},
    \end{equation*}
    where the first two equalities follow from the fact that $\widehat{C}_{\mathcal{P}_T}$ is a normal, positive definite matrix (so its spectral norm equals its maximal eigenvalue, and it is diagonalizable, so the eigenvalues of its inverse are the inverses of the eigenvalues of $\widehat{C}_{\mathcal{P}_T}$), and the final inequality is by the event we are on. This yields the desired result.
\end{proof}
\begin{proof}[Proof of \Cref{LikelihoodEquality}]
We assume throughout that we are on the set $$\left\{ \Vert\widehat{C}_{{\mathcal{P}_T}} - \tilde{C}_{{\mathcal{P}_T}}^\top\Vert_F \le \Delta_{\mathcal{P}_T}^{1/4}, \Vert\widehat{C}_{{\mathcal{P}_T}}^{-1}\Vert\le \dfrac{2}{\kappa_{\min}}\right\}.$$ This set is of high probability under $\mathbb{P}^{Q_0}$ following the results in \Cref{Prop:BoundEntryWisel1Norm,Prop:UpperBoundC_TInverse} combined with \Cref{Ass:WHPT,Ass:WHPDelta}.
We remember that $\Sigma^{-1/2}\bold{X}_t^c$ is a $\mathbb{P}^0$ Wiener process following the derivations in \Cref{PDynOfYandX}. Further, by the same results,
\begin{equation*}
\bold{X}_t^c = \bold{X}_t^c([Q_0]) - \int_0^tQ_0\bold{X}_s ds, 
\end{equation*}
where
\[
    \widetilde{\boldsymbol W}_t
    :=
    \Sigma^{-1/2}\boldsymbol X_t^c[Q_0]
\]
is a $\mathbb P^{Q_0}$-Wiener process. Hence, we have that 
\begin{equation*}
\Sigma^{-1/2}\Delta_{\mathcal{P}_T}^n\bold{X}^c = \Delta_{\mathcal{P}_T}^n\bold{\tilde{W}}-\Sigma^{-1/2}\int_{s_{n}^{N_T}}^{s_{n+1}^{N_T}}Q_0\bold{X}_s ds, 
\end{equation*}
We get that \eqref{DiscreteLikelihood} can be written as
\begin{equation}\label{DiscreteWrittenOut}
    \begin{aligned}
    \mathcal{L}_T^{\mathcal{P}_T}(Q)&=\frac{1}{T} \sum_{n=0}^{N_T-1}\left(\Sigma^{-1} Q \mathbf{X}_{s_n^{N_T}}\right)^{\top}    \Delta_{\mathcal{P}_T}^n \bold{X}^c+\frac{1}{2 T} \sum_{n=0}^{N_T-1}\left(\Sigma^{-1/2} Q \mathbf{X}_{s_n^{N_T}}\right)^{\top} \Sigma^{-1/2} Q \mathbf{X}_{s_n^{N_T}}\left(s_{n+1}^{N_T}-s_n^{N_T}\right) \\
    &= \frac{1}{T} \int_0^T\left(\Sigma^{-1/2} Q \mathbf{X}_s\right)^{\top}    \mu_{\mathcal{P}_T}^{\mathbb{\tilde{W}}}(ds)+\frac{1}{2 T} \int_0^T\left(\Sigma^{-1/2} Q \mathbf{X}_s\right)^{\top} \Sigma^{-1/2} Q \mathbf{X}_s\mu_{\mathcal{P}_T}(ds)\\
    &-\frac{1}{T} \sum_{n=0}^{N_T-1}\left(\Sigma^{-1/2} Q \mathbf{X}_{s_n^{N_T}}\right)^{\top}\Sigma^{-1/2}Q_0 \int_{s_n^{N_T}}^{s_{n+1}^{N_T}}\bold{X}_sds\\
    &= \frac{1}{T} \int_0^T\left(\Sigma^{-1/2} Q \mathbf{X}_s\right)^{\top}    \mu_{\mathcal{P}_T}^{\mathbb{\tilde{W}}}(ds)+\frac{1}{2 T} \int_0^T\left(\Sigma^{-1/2} Q\mathbf{X}_s\right)^{\top} \Sigma^{-1/2} Q \mathbf{X}_s\mu_{\mathcal{P}_T}(ds)\\
    &-\frac{1}{T} \int_0^T\left(\Sigma^{-1/2} Q \mathbf{X}_{\lfloor s\rfloor_{\mathcal{P}_T}}\right)^{\top}\Sigma^{-1/2}Q_0 \bold{X}_sds
    \\
    &= \frac{1}{T} \int_0^T\left(\Sigma^{-1/2} Q \mathbf{X}_s\right)^{\top}    \mu_{\mathcal{P}_T}^{\mathbb{\tilde{W}}}(ds)+\frac{1}{2 T} \int_0^T\left(\Sigma^{-1/2} Q \mathbf{X}_s\right)^{\top} \Sigma^{-1/2} Q \mathbf{X}_s\mu_{\mathcal{P}_T}(ds)\\
    &-\frac{1}{T} \int_0^T\left(\Sigma^{-1/2} Q \mathbf{X}_{\lfloor s\rfloor_{\mathcal{P}_T}}\right)^{\top}\Sigma^{-1/2}Q_0 \bold{X}_sds
    \\&= \Tr\biggl(\frac{1}{T} \Sigma^{-1/2} Q\int_0^T\mathbf{X}_s  \mu_{\mathcal{P}_T}^{\mathbb{\tilde{W}}}(ds)^\top+\frac{1}{2 T}  Q^\top(\Sigma^{-1})^\top Q\int_0^T\mathbf{X}_s \mathbf{X}_s^\top \mu_{\mathcal{P}_T}(ds) \\&-
    \frac{1}{ T}  Q_0^\top(\Sigma^{-1})^\top Q\int_0^T\mathbf{X}_{\lfloor s\rfloor_{\mathcal{P}_T}} \mathbf{X}_s^\top ds\biggr)\\
    &= \Tr\biggl( \Sigma^{-1/2} Q (\varepsilon_{\mathcal{P}_T}^{N_T})^\top+\frac{1}{2}  Q\widehat{C}_{{\mathcal{P}_T}}Q^\top\Sigma^{-1}-Q\tilde{C}_{{\mathcal{P}_T}}Q_0^\top\Sigma^{-1}\biggr),
    \end{aligned}
\end{equation}
where $\lfloor s\rfloor_{\mathcal{P}_T}$ is as previously defined, and we used \Cref{MatrixCalc1,MatrixCalc2} in the second to last equality. Finally $\tilde{C}_{{\mathcal{P}_T}}:=\dfrac{1}{T}\int_0^T\mathbf{X}_{\lfloor s\rfloor_{\mathcal{P}_T}} \mathbf{X}_s^\top ds$. Comparing \eqref{DiscreteWrittenOut} with \eqref{TrInfeasDisc}, we note that they share the same Hessian matrix, $\widehat{C}_{\mathcal{P}_T}\otimes\Sigma^{-1}$, which is positive definite following \Cref{StrictConvexityOfC_THat}. Since the minimizations are over $\mathbb R^{d\times d}$ and the objective
functions are differentiable and strictly convex on the event under
consideration, the minimizers are characterized by the first-order conditions
\[
    \nabla_Q \mathcal L_T^{\mathcal P_T}(\widehat Q)=0,
    \qquad
    \nabla_Q \widetilde{\mathcal L}_T^{\mathcal P_T}(\widetilde Q)=0.
\]
\begin{equation*}
    \begin{aligned}(\Sigma^{-1/2})^\top\varepsilon_{\mathcal{P}_T}^{N_T}+\Sigma^{-1}\tilde{Q}\widehat{C}_{{\mathcal{P}_T}}-\Sigma^{-1}Q_0\widehat{C}_{{\mathcal{P}_T}} &= (\Sigma^{-1/2})^\top\varepsilon_{\mathcal{P}_T}^{N_T}+\Sigma^{-1}\widehat{Q}\widehat{C}_{{\mathcal{P}_T}}-\Sigma^{-1}Q_0\tilde{C}_{{\mathcal{P}_T}}^\top
    \\ &\iff \\\left(\tilde{Q}-\widehat{Q}\right)\widehat{C}_{{\mathcal{P}_T}} &= Q_0\left(\widehat{C}_{{\mathcal{P}_T}} - \tilde{C}_{{\mathcal{P}_T}}^\top\right)\\
    &\iff \\\left(\tilde{Q}-\widehat{Q}\right) &= Q_0\left(\widehat{C}_{{\mathcal{P}_T}} - \tilde{C}_{{\mathcal{P}_T}}^\top\right)\widehat{C}_{{\mathcal{P}_T}}^{-1}.
    \end{aligned}
\end{equation*}
We thus conclude that
\begin{equation*}
    \begin{aligned}
    \Vert \tilde{Q}-\widehat{Q}\Vert_F &\le \Vert Q_0\Vert_F \Vert\widehat{C}_{{\mathcal{P}_T}} - \tilde{C}_{{\mathcal{P}_T}}^\top\Vert_F \Vert\widehat{C}_{{\mathcal{P}_T}}^{-1}\Vert
    \\&\le \Vert Q_0\Vert_F  \dfrac{2\Delta_{\mathcal{P}_T}^{1/4}}{\kappa_{\min}},
    \end{aligned}
\end{equation*}
where we used that for any matrices $A,B$,
\begin{equation*}
    \begin{aligned}
    \Vert AB\Vert_F &\le \Vert A\Vert_F \Vert B \Vert,
    \end{aligned}
\end{equation*}
together with $\Vert Q_0(\widehat{C}_{\mathcal{P}_T} - \tilde{C}_{\mathcal{P}_T}^\top)\Vert_F \le \Vert Q_0\Vert_F \Vert \widehat{C}_{\mathcal{P}_T} - \tilde{C}_{\mathcal{P}_T}^\top\Vert_F$ by submultiplicativity of the Frobenius norm, and the operator-norm bound on $\widehat{C}_{\mathcal{P}_T}^{-1}$ from \Cref{Prop:UpperBoundC_TInverse}.
\end{proof}
\subsection{Proofs: Feasible case}
In this section, we have all proofs regarding the discretization of the likelihood and the magnitude of the error associated to this discretization as a function of $d$, $\Delta_{\mathcal{P}_T}$, $N$ and $T$. Throughout this entire section, we will often refer to the $(i,j)$th entry of $Q_0$. We denote this $[Q_0]_{ij}$. In this section, we will aim to bound 
\begin{equation}\label{L1BoundRef}
    \Vert\left(\dfrac{1}{T}\sum_{n=0}^{N_T-1} \left( \Delta_n\bold{X}^c-\Delta_{\mathcal{P}_T}^n \bold{\bold{X}}\odot\vec{\bold{1}}(\Delta_{\mathcal{P}_T}^n \bold{\bold{X}})\right)\mathbf{X}_{s_n^{N_T}}^\top\right)\Vert_1
\end{equation} in probability. We distinguish between whether or not $\mathbb{L}$ is of finite activity or not as this will have implications on the Lévy-Itô decomposition $\mathbb{L}$. We cover that next.
\subsubsection{Preliminaries}
We will use the two important theorems recalled in \Cref{LevyIto} to bound \eqref{L1BoundRef}. With this in mind, we underline a few things these theorems tell us.
\paragraph{A closer look on the implications of the Lévy-Itô Decomposition}\label{LevyItoImpl}
Following \eqref{eqn:LevyIto} using \eqref{eqn:1}, we have that 
\begin{equation*}
    \Delta_{\mathcal{P}_T}^k \bold{Y}= -\int_{s^{\mathcal{P}_T }_{k-1}}^{s^{N_T}_{k}}Q\bold{Y}_s ds+\Delta_{\mathcal{P}_T}^k\bold{L}.
\end{equation*}
From the Lévy-Itô decomposition, we have that
\begin{equation}\label{LevyProcessDecomInfAct}
    \Delta_{\mathcal{P}_T}^k\bold{L} = \bold{b}(s^{N_T}_{k}-s^{N_T }_{k-1}) + \Sigma^{1/2}\Delta_{\mathcal{P}_T}^k \bold{W} + \Delta_{\mathcal{P}_T}^k\bold{U} + \Delta_{\mathcal{P}_T}^k \bold{M}, \quad  \mathbb{P}-a.s.,
\end{equation}
where $\bold{U}_t := \int_0^t \int_{\{x:\|x\|>1\}}  x \mu(d t, d x)$ is a finite activity Lévy process, since $F\left(\{x: \Vert x \Vert > 1\}\right) < \infty$ as our Lévy measure satisfies $\int \left(\Vert x \Vert^2 \wedge 1 \right) F(dx) < \infty$, and $\bold{M}_t := \int_0^t \int_{\{x:\|x\| \leq 1\}} x\{\mu(d t, d x)- F(d x)d t\}$. Let $\bold{D}_t :=-\int_{0}^{t}Q\bold{Y}_s ds$, then we have the following decomposition of the increments of $ \mathbb{Y}$ from \eqref{eqn:1}: 
\begin{equation}\label{eqn:DecompositionInfiniteAct}
    \Delta_{\mathcal{P}_T}^k \bold{Y} = \Delta_{\mathcal{P}_T}^k \bold{D} + \bold{b}(s^{N_T}_{k}-s^{N_T }_{k-1}) + \Sigma^{1/2}\Delta_{\mathcal{P}_T}^k \bold{W} + \Delta_{\mathcal{P}_T}^k\bold{U} + \Delta_{\mathcal{P}_T}^k \bold{M},\quad \mathbb{P}-a.s..
\end{equation}
Let $\tilde{\bold{Y}} =\bold{Y} -  \bold{M}$ and hence  $\Delta_{\mathcal{P}_T}^k \tilde{\bold{Y}} := \Delta_{\mathcal{P}_T}^k \bold{Y} - \Delta_{\mathcal{P}_T}^k \bold{M}$. Effectively, $\tilde{\bold{Y}} $ under the infinite activity case behaves as $\bold{Y}$ under the finite activity case, only with slightly different jump rate ($\lambda$). This will be used in the following proofs.
Using that $\mu$ is a Poisson random measure, we have that $\bold{M}_t, \bold{U}_t$ and $\bold{W}_t$ are all independent from each other. 
Using Propositions $3.3$ and $3.5$ of \cite{Tankov2003}, we underline that since $\mathbb{U}$ is a finite activity Lévy process, it is a compound Poisson process, that is, for any $t>0$,
    \begin{equation*}
        \bold{U}_t = \sum_{j=1}^{N_t}\bold{Z}_j,\quad \mathbb{P}-a.s.,
    \end{equation*}
    where $N_t$ is a univariate Poisson process with rate $\lambda = F(\{x:\Vert x \Vert>1 \})$ and $\bold{Z}_j$ have probability measure $\bar{F}(dx)=F|_{x:\Vert x \Vert>1}(dx)/\lambda$. We only work with $\mathbb{U}$ in the infinite activity case.
    In the finite activity case, we may simplify \eqref{LevyProcessDecomInfAct} as 
\begin{equation}\label{LevyProcessDecomFiniteAct}
    \Delta_{\mathcal{P}_T}^k\bold{L} = \bold{\tilde{b}}(s^{N_T}_{k}-s^{N_T }_{k-1}) + \Sigma^{1/2}\Delta_{\mathcal{P}_T}^k \bold{W} + \Delta_{\mathcal{P}_T}^k\bold{\tilde{U}}, \quad \mathbb{P}-a.s,
\end{equation}
where $\mathbf{\tilde{b}} = \mathbf{b}-\int_{x:\Vert x \Vert \le 1} x F(dx) < \infty$ and $\tilde{\mathbb{U}}$ is a finite activity Lévy process. Using Propositions $3.3$ and $3.5$ of \cite{Tankov2003}, we again have that since $\tilde{\mathbb{U}}$ is a finite activity Lévy process, it can also be represented as a compound Poisson process, that is, for any $t>0$,
    \begin{equation*}
        \tilde{\bold{U}}_t = \sum_{j=1}^{\tilde{N}_t}\tilde{\bold{Z}}_j,\quad \mathbb{P}-a.s.,
    \end{equation*}
    where $\tilde{N}_t$ is a univariate Poisson process with rate $\tilde{\lambda} = F(\mathbb{R}^d)$ and $\bold{\tilde{Z}}_j$ have probability measure $\tilde{F}(dx)=F(dx)/\tilde{\lambda}$.
We now introduce the last stochastic process, we shall work with, $\mathbb{J}$, which we define as
\begin{equation*}
    \mathbb{J} = \mathbb{L} - \Sigma^{1/2}\mathbb{W},
\end{equation*}
i.e.\ it is the non-Brownian part of the Lévy process. We note that the exact definition of $\mathbb{J}$ thus varies depending on whether or not we are in the finite or infinite activity case. In the finite activity case, $\mathbb{J} = \mathbb{\tilde{U}}$ is simply the compound Poisson process from \eqref{LevyProcessDecomFiniteAct} whereas it equals $\mathbb{U} + \mathbb{M}$ from \eqref{eqn:DecompositionInfiniteAct} in the infinite activity case.
We now state and prove the following Lemma which holds regardless of whether or not we are in the finite or infinite activity case.

\begin{lemma}\label{Lemma:CountBound}
     Grant \Cref{Ass:Estimation,Ass.High-FrequencySampling}. Let $\delta \in [0,1/2)$. Assume that, for some $l\ge 1$, $\max_{k\in\{1,\dots,d\}}\mathbb{E}\left(\vert \bold{Y}_0^{(k)}\vert^l \right)= O(d^l)$ as $d\rightarrow \infty$, and that $\Delta_{\mathcal{P}_T}^{1-2\delta} = o(d^{-4})$ as $d\rightarrow \infty$. Then, for any $\alpha\in \mathbb{N}$, $c \in \mathbb{R}^+$, $k \in \{0,1,\dots,N_T-1\}$ and $i ,j\in \{1,\dots,d\}$, there exists a constant $C$ \emph{independent} of $k$, $i$ and $j$ such that
    \begin{multline*}
     %   \begin{aligned}
        \max_{i \in \{1,\dots,d\}}\mathbb{P}\left(\vert-\sum_{j=1}^d [Q_0]_{ij}\int_{s_k^{N_T}}^{s_{k+1}^{N_T}} \bold{Y}_s^{(j)} ds + \sum_{j=1}^d\Sigma_{ij}^{1/2}\Delta_{\mathcal{P}_T}^k\bold{W}^{(j)}\vert >c \Delta_{\mathcal{P}_T}^{\delta}\right) \\
        =O \left(d^{3l+1} \Delta_{\mathcal{P}_T}^{(l(1-\delta))}\vee d\left(d^4\Delta_{\mathcal{P}_T}^{1-2\delta}\right)^\alpha\right),\quad \text{as } d \rightarrow \infty.
      %  \end{aligned}
    \end{multline*}
    Moreover, if $\delta \in (0,1/4)$ and there exist $l^\prime \in \mathbb{N}$ and $\alpha^\prime > 0$ with $\max_{k\in\{1,\dots,d\}}\mathbb{E}\left(\vert \bold{Y}_0^{(k)}\vert^{l^\prime} \right)= O(d^{l^\prime})$ and $d^{4\alpha^\prime+l^\prime}\Delta_{\mathcal{P}_T}^{\frac{1}{2}\left(\alpha^\prime  - \frac{1}{2}l^\prime\right)}= o(1)$ as $d\rightarrow \infty$, then
    \begin{multline*}
        %\begin{aligned}
        \max_{i \in \{1,\dots,d\}} \mathbb{P}\left(\vert-\sum_{j=1}^d [Q_0]_{ij}\int_{s_k^{N_T}}^{s_{k+1}^{N_T}} \bold{Y}_s^{(j)} ds + \sum_{j=1}^d\Sigma_{ij}^{1/2}\Delta_{\mathcal{P}_T}^k\bold{W}^{(j)}\vert >c \Delta_{\mathcal{P}_T}^{\delta}\right) \\=O \left(d^{3l^\prime+1} \Delta_{\mathcal{P}_T}^{(l^\prime(1-\delta))}\right),\quad \text{as }d \rightarrow \infty.
        %\end{aligned}
    \end{multline*}
\end{lemma}
\begin{proof}
    We first slightly modify Lemma $3.7$ in \cite{Mai2014} while noting that the proof strategy is based on their proof. We then note, that for $\delta \in (0,1/2)$, then, for any $i \in \{1,\dots,d\}$ and $k \in \{0,\dots, N_T-1 \}$, it holds that
    \begin{equation*}
    \begin{aligned}
        &\mathbb{P}\left(\vert-\sum_{j=1}^d [Q_0]_{ij}\int_{s_k^{N_T}}^{s_{k+1}^{N_T}} \bold{Y}_s^{(j)} ds + \sum_{j=1}^d\Sigma_{ij}^{1/2}\Delta_{\mathcal{P}_T}^k\bold{W}^{(j)}\vert > c\Delta_{\mathcal{P}_T}^{\delta}\right) \\ &\le  \mathbb{P}\left(\vert\sum_{j=1}^d- [Q_0]_{ij}\int_{s_k^{N_T}}^{s_{k+1}^{N_T}} \bold{Y}_s^{(j)} ds\vert + \vert \sum_{j=1}^d\Sigma_{ij}^{1/2}\Delta_{\mathcal{P}_T}^k\bold{W}^{(j)}\vert > c\Delta_{\mathcal{P}_T}^{\delta}\right) \\ &\le  \mathbb{P}\left(\sum_{j=1}^d \vert- [Q_0]_{ij}\int_{s_k^{N_T}}^{s_{k+1}^{N_T}} \bold{Y}_s^{(j)} ds\vert > c\Delta_{\mathcal{P}_T}^{\delta}/2\right)+\mathbb{P}\left(\vert \sum_{j=1}^d\Sigma_{ij}^{1/2}\Delta_{\mathcal{P}_T}^k\bold{W}^{(j)}\vert > c\Delta_{\mathcal{P}_T}^{\delta}/2\right) \\
        &\le  \sum_{j=1}^d\mathbb{P}\left(\vert- [Q_0]_{ij}\int_{s_k^{N_T}}^{s_{k+1}^{N_T}} \bold{Y}_s^{(j)} ds\vert > c\Delta_{\mathcal{P}_T}^{\delta}/(2d)\right)+\sum_{j=1}^d \mathbb{P}\left(\vert \Sigma_{ij}^{1/2}\Delta_{\mathcal{P}_T}^k\bold{W}^{(j)}\vert > c\Delta_{\mathcal{P}_T}^{\delta}/(2d)\right).  
    \end{aligned}
    \end{equation*}
    We first note that, following Lemma $22.2$, in \cite{Klenke2013}, for any $i,j \in \{ 1,\dots,d\}$ such that $\Sigma^{1/2}_{ij} \ne 0$, it follows that 
    \begin{equation*}
        \begin{aligned}
        \mathbb{P}\left(\vert \Sigma_{ij}^{1/2}\Delta_{\mathcal{P}_T}^k\bold{W}^{(j)}\vert > c\Delta_{\mathcal{P}_T}^{\delta}/(2d)\right)&=\mathbb{P}\left(\vert \vert\Sigma_{ij}^{1/2}\vert(s_{k+1}^{N_T}-s_k^{N_T})^{1/2}\bold{Z}\vert > c\Delta_{\mathcal{P}_T}^{\delta}/(2d)\right)\\&\le 2c^{-1}\sqrt{\dfrac{2}{\pi}}2d\vert\Sigma_{ij}^{1/2}\vert\Delta_{P_{T}}^{1/2-\delta}e^{-\dfrac{\left(c\,\Delta_{P_{T}}^{\delta-1/2}/(2d\Sigma_{ij}^{1/2})\right)^{2}}{2}}\\
        &= 2c^{-1}\sqrt{\dfrac{2}{\pi}}2d\max_{ij}\vert\Sigma_{ij}^{1/2}\vert\Delta_{P_{T}}^{1/2-\delta}e^{-\dfrac{c^{2}}{8d^2\left(\max_{ij}\vert\Sigma_{ij}^{1/2}\vert \right)^{2}\Delta_{P_{T}}^{1-2\delta}}}\\
          &= O\left(\sqrt{d^2 \Delta_{\mathcal{P}_T}^{1-2\delta}}\, e^{-\dfrac{c^{2}}{d^4\Delta_{\mathcal{P}_T}^{1-2\delta}}}\right),\quad \text{as } d\rightarrow \infty,
        \end{aligned}
    \end{equation*}
    where $\bold{Z} \sim N(0,1)$. Now by the arguments in Lemma $3.7$ of \cite{Mai2014} (in particular the use of Jensen's inequality), we have that 
    \begin{equation*}
        \begin{aligned}
            \mathbb{E}\left(\vert- [Q_0]_{ij}\int_{s_k^{N_T}}^{s_{k+1}^{N_T}} \bold{Y}_s^{(j)} ds\vert^l\right) \le \max_{ij}\vert [Q_0]_{ij}\vert^l\Delta_{\mathcal{P}_T}^{l}\mathbb{E}\left(\vert\bold{Y}_0^{(j)}\vert^l\right),\quad \text{as } d \rightarrow \infty.
        \end{aligned}
    \end{equation*}
    Hence, we have that 
    \begin{equation}\label{eqn:ContIncreaseBound}
    \begin{aligned}
         \mathbb{P}\left(\vert- [Q_0]_{ij}\int_{s_k^{N_T}}^{s_{k+1}^{N_T}} \bold{Y}_s^{(j)} ds\vert > c\Delta_{\mathcal{P}_T}^{\delta}/(2d)\right)&= \mathbb{P}\left(\vert- [Q_0]_{ij}\int_{s_k^{N_T}}^{s_{k+1}^{N_T}} \bold{Y}_s^{(j)} ds\vert^l > c^l\Delta_{\mathcal{P}_T}^{l\delta}/(2d)^l\right) \\& \le(\frac{2d}{c})^l\max_{ij}\vert [Q_0]_{ij}\vert^l\Delta_{\mathcal{P}_T}^{l(1-\delta)}\mathbb{E}\left(\vert\bold{Y}_0^{(j)}\vert^l\right) \\&=O \left(d^{3l} \Delta_{\mathcal{P}_T}^{(l(1-\delta))}\right) ,\quad \text{as } d \rightarrow \infty.
    \end{aligned}
    \end{equation}
    We conclude that, for $\delta \in (0,1/2)$, $i \in \{1,\dots,d \}$, $k\in \{0,\dots, N_T-1 \}$ and $\alpha\in \mathbb{N}$,
    \begin{multline*}
       % \begin{aligned}
        \mathbb{P}\left(\vert-\sum_{j=1}^d [Q_0]_{ij}\int_{s_k^{N_T}}^{s_{k+1}^{N_T}} \bold{Y}_s^{(j)} ds + \sum_{j=1}^d\Sigma_{ij}^{1/2}\Delta_{\mathcal{P}_T}^k\bold{W}^{(j)}\vert > c\Delta_{\mathcal{P}_T}^{\delta}\right) \\
        = O \left(dd^{3l} \Delta_{\mathcal{P}_T}^{(l(1-\delta))}\vee d\left(d^4\Delta_{\mathcal{P}_T}^{1-2\delta}\right)^\alpha\right),\quad \text{as } d \rightarrow \infty.
      %  \end{aligned}
    \end{multline*}
    Here we used that $e^{-1/x}$ tends faster to $0$ than $x^\alpha$ for any $\alpha>0$ as $x\rightarrow 0$.
    For the second assertion, we now choose $\alpha = l^\prime + \alpha^{\prime}$. 
    For any $l$, 
    \begin{equation*}
        d^{3l}\Delta_{\mathcal{P}_T}^{(l(1-\delta))} =\left(d^{3}\Delta_{\mathcal{P}_T}^{((1-2\delta))}\right)^l \Delta_{\mathcal{P}_T}^{l\delta}.
    \end{equation*}
    Hence, we have that 
    \begin{multline*}
      %  \begin{aligned}
         \dfrac{\left(d^4\Delta_{\mathcal{P}_T}^{1-2\delta}\right)^\alpha}{d^{3l^\prime} \Delta_{\mathcal{P}_T}^{(l^\prime(1-\delta))}}= d^{\alpha}\dfrac{\left(d^3\Delta_{\mathcal{P}_T}^{1-2\delta}\right)^{\alpha^\prime}}{\Delta_{\mathcal{P}_T}^{l^{\prime}\delta}} = d^{4\alpha^\prime +l^\prime}\Delta_{\mathcal{P}_T}^{\alpha^\prime -2\alpha^{\prime}\delta - l^\prime/4}\Delta_{\mathcal{P}_T}^{l^\prime/4 - l^\prime\delta} \\
         \le d^{4\alpha^\prime+l^\prime}\Delta_{\mathcal{P}_T}^{\frac{1}{2}\left(\alpha^\prime  - \frac{1}{2}l^\prime\right)}\Delta_{\mathcal{P}_T}^{l^\prime/4 - l^\prime\delta}  = o(1), \text{ as } d \rightarrow \infty,
        %\end{aligned}
    \end{multline*}
    where the inequality uses $\alpha^\prime(1-2\delta) > \alpha^\prime/2$ for $\delta < 1/4$ (so $\Delta_{\mathcal{P}_T}^{\alpha^\prime-2\alpha^\prime\delta - l^\prime/4} \le \Delta_{\mathcal{P}_T}^{(\alpha^\prime - l^\prime/2)/2}$ since $\Delta_{\mathcal{P}_T} < 1$), and the final $o(1)$ follows from the hypothesis $d^{4\alpha^\prime+l^\prime}\Delta_{\mathcal{P}_T}^{(\alpha^\prime - l^\prime/2)/2}=o(1)$ together with $\Delta_{\mathcal{P}_T}^{l^\prime/4 - l^\prime\delta} = O(1)$ (which holds since $\delta < 1/4$).
\end{proof}
We will now establish some properties related to the stochastic processes introduced in \Cref{LevyIto}.
\begin{proposition}\label{PropertiesOfIntroducedStochasticProcesses}
    Grant \Cref{Ass:Estimation,Ass.H,Ass.High-FrequencySampling,Ass.GeneralBothCases}, and assume that either \Cref{Ass.AssFiniteAct} or \Cref{Ass.InfiniteAct} holds. For any function $f: \mathbb{R}^+ \rightarrow \mathbb{R}^+$, $n\in \{1,2,\dots, N_T \}$ and $k\in \{1,\dots,d\}$, and $\alpha \in \mathbb{N}$, it holds that
    \begin{equation*}
        \begin{aligned}
        \max_{k\in \{1,\dots,d\}}\mathbb{P}\left( \vert (\Sigma^{1/2}\Delta_{\mathcal{P}_T}^n \bold{W})^{(k)}\vert>f(\Delta_{\mathcal{P}_T})\right)&= O\left(d\left(d^4\frac{\Delta_{\mathcal{P}_T}}{f(\Delta_{\mathcal{P}_T})^2}\right)^\alpha\right), \quad \text{as } d\rightarrow \infty. \ \\
        \max_{k\in \{1,\dots,d\}}\mathbb{P}\left( \vert\Delta_{\mathcal{P}_T}^n \bold{D}^{(k)}\vert>f(\Delta_{\mathcal{P}_T})\right)&= O\left(d^7\Delta_{\mathcal{P}_T}^2f(\Delta_{\mathcal{P}_T})^{-2}\right), \quad \text{as } d\rightarrow \infty. \\ 
        \max_{k\in \{1,\dots,d\}} \mathbb{P}\left( \vert\Delta_{\mathcal{P}_T}^n \bold{U}^{(k)}\vert>0\right)&= O\left(d\Delta_{\mathcal{P}_T}\right), \quad \text{as } d\rightarrow \infty. \\
        \max_{k\in \{1,\dots,d\}} \mathbb{P}\left( \vert\Delta_{\mathcal{P}_T}^n \bold{M}^{(k)}\vert >f(\Delta_{\mathcal{P}_T})\right)&= O\left(\Delta_{\mathcal{P}_T}f(\Delta_{\mathcal{P}_T})^{-2}\right), \quad \text{as } d  \rightarrow \infty.
        \end{aligned}
    \end{equation*}
    Here $\mathbb{M}$, $\mathbb{U}$, $\mathbb{D}$ and $\mathbb{W}$ are the stochastic processes defined in \Cref{LevyItoImpl}.
\end{proposition}
\begin{proof}
We establish the four bounds in the order they appear in the statement.

\emph{(i) Brownian increment.} The claim follows directly from the Gaussian-tail derivation in the proof of \Cref{Lemma:CountBound} applied per-coordinate to $(\Sigma^{1/2}\Delta_{\mathcal{P}_T}^n \bold{W})^{(k)}$.

\emph{(ii) Drift.} We note that
\begin{equation*}
    \Delta_{\mathcal{P}_T}^n \bold{D}^{(k)} =-\sum_{j=1}^d [Q_0]_{kj}\int_{s_n^{N_T}}^{s_{n+1}^{N_T}} \bold{Y}_s^{(j)} ds.
\end{equation*}
The claim follows by the derivations in \Cref{Lemma:CountBound} with $l=2$.

\emph{(iii) Big jumps.} Since $\bold{U}$ is a compound Poisson process, its increment over the $n$-th interval is $\Delta_{\mathcal{P}_T}^n \bold{U}^{(k)} = \sum_{j=1}^{\Delta_{\mathcal{P}_T}^n \bold{N}}\bold{Z}_j^{(k)}$, where $\Delta_{\mathcal{P}_T}^n \bold{N}$ is the number of jumps of $\mathbb{U}$ in that interval; this vanishes on $\{\Delta_{\mathcal{P}_T}^n \bold{N}=0\}$, so the indicator-level relation
\begin{equation*}
    \begin{aligned}
        \bold{1}_{\{\vert\Delta_{\mathcal{P}_T}^n \bold{U}^{(k)}\vert\ne0\}} \le \bold{1}_{\{\Delta_{\mathcal{P}_T}^n \bold{N}\ne0\}}
    \end{aligned}
\end{equation*}
holds. Here, $\mathbb{N}$ is the Poisson counting process associated to $\mathbb{U}$.
Further,
\begin{equation}\label{PoissonProcessDecomp}
    \begin{aligned}
        \mathbb{P}\left(\bold{N}_{T}\ne0 \right) &= 1 - e^{-T\lambda} \\
        &\le T\lambda.
     \end{aligned}
\end{equation}
In turn,
\begin{equation*}
 \mathbb{P}\left(\vert\Delta_{\mathcal{P}_T}^n \bold{U}^{(k)} \vert> 0\right) \le \mathbb{P}\left( \Delta_{\mathcal{P}_T}^n \bold{N} \ne 0\right) = O \left(d \Delta_{\mathcal{P}_T}\right),\quad \text{ as } d\rightarrow \infty,
\end{equation*}
using the bound $\lambda = F(\{x:\Vert x\Vert>1\}) = O(d)$ from either \Cref{Ass.AssFiniteAct} or \Cref{Ass.InfiniteAct}.

\emph{(iv) Compensated small jumps.} By It\^o isometry, for any $\epsilon >0$,
\begin{equation*}
    \begin{aligned}
    \mathbb{E}\left( \Vert\int_0^T \int_{\{x: \epsilon<\|x\| \leq 1\}} x\{\mu(d t, d x)- F(d x)d t\}\Vert_2^2\right) &= \left(\int_0^T \int_{\{x: \epsilon<\|x\| \leq 1\}} \Vert x \Vert_2^2 F(d x)d t\right).
    \end{aligned}
\end{equation*}
Using that $\mathbb{M}$ is the $L^2(\mathbb{P})$ limit of $\int_0^T \int_{\{x: \epsilon<\|x\| \leq 1\}} x\{\mu(d t, d x)- F(d x)d t\}$, we thus have that
\begin{equation*}
    \begin{aligned}
        \mathbb{E}\left(\Vert \bold{M}_T\Vert_2^2\right) &= \lim_{\epsilon \rightarrow 0}\mathbb{E}\left(\Vert  \int_0^T \int_{\{x: \epsilon<\|x\|_2 \leq 1\}} x\{\mu(d t, d x)- F(d x)d t\}\Vert^2_2\right) = T\lim_{\epsilon \rightarrow 0}\int_{\{x: \epsilon<\|x\| \leq 1\}} \Vert x\Vert^2_2 F(d x)\\
        &=T\int_{\{x: \|x\|_2 \leq 1\}} \Vert x\Vert_2^2 F(d x)
        \\&= O(Td),\quad \text{as }d \rightarrow \infty,
    \end{aligned}
\end{equation*}
using the per-coordinate bound $\max_k\int(z^{(k)})^2 F(dz) = O(1)$ from \Cref{Ass.H}.2, summed over $d$ coordinates.
Applying It\^o isometry per coordinate together with stationarity of $\mathbb{M}$, it holds that
\begin{equation}\label{MomentBoundOnM}
    \begin{aligned}
        \max_{k\in \{1,\dots,d\}} \mathbb{E}\left( (\Delta_{\mathcal{P}_T}^n  \bold{M}^{(k)})^2\right) &=\max_{k\in \{1,\dots,d\}} \mathbb{E}\left( (  \bold{M}^{(k)}_{(s^{N_T}_n-s^{N_T}_{n-1})})^2\right)\\&= \Delta_{\mathcal{P}_T}\max_{k\in \{1,\dots,d\}}\int_{\{x:\|x\|_2\le 1\}} (x^{(k)})^2 F(dx) = O\left( \Delta_{\mathcal{P}_T} \right),\quad \text{as } d \rightarrow \infty,
    \end{aligned}
\end{equation}
using the same per-coordinate bound from \Cref{Ass.H}.
Following Markov's inequality, we thus get
\begin{equation}\label{eqn:IntActBoundM}
    \begin{aligned}
        \max_{k\in \{1,\dots,d\}}\mathbb{P}\left( \vert (\Delta_{\mathcal{P}_T}^n  \bold{M}^{(k)})\vert >f(\Delta_{\mathcal{P}_T})\right) \le  f(\Delta_{\mathcal{P}_T})^{-2} \max_{k\in \{1,\dots,d\}}\mathbb{E}\left( (\Delta_{\mathcal{P}_T}^n  \bold{M}^{(k)})^2\right) &= O\left(\Delta_{\mathcal{P}_T}f(\Delta_{\mathcal{P}_T})^{-2}\right),\\&\quad \text{as } d \rightarrow \infty.
    \end{aligned}
\end{equation}
\end{proof}
With these established, we concern ourselves with the proof in the finite activity case.
\subsubsection{Proofs: Finite activity}\label{ProofsFiniteAct}

\begin{lemma}\label{Lemma:SetBound}
Grant \Cref{Ass:Estimation,Ass.H,Ass.High-FrequencySampling,Ass.GeneralBothCases}, and assume that either \Cref{Ass.AssFiniteAct} or \Cref{Ass.InfiniteAct} holds. For each $i\in \{1,\dots, d\}$ and $k \in \{0,\dots, N_T-1 \}$, we define the set where a small increment of our process is equivalent to the absence of jumps, that is
\begin{equation*}
    A_{N_T} = \bigcap_{i=1}^d \bigcap_{k=1}^{N_T-1}  A_{k,N_T}^{(i)},
\end{equation*}
where
\begin{equation*}
    A_{k,N_T}^{(i)}:= \{\omega \in \Omega: \bold{1}_{\vert \Delta_{\mathcal{P}_T}^{k}\bold{Y}^{(i)}\vert \le \nu_{\mathcal{P}_T}^{(i)}} = \bold{1}_{\Delta_{\mathcal{P}_T}^{k} \bold{N}=0} \}.
\end{equation*}
Then it holds, that 
\begin{equation*}
      \mathbb{P}\left( A_{N_T} \right)  \rightarrow 1 \text{ as } d \rightarrow \infty .
\end{equation*}
\end{lemma}
\begin{proof}
We follow the proof strategy of \cite{Mai2014} adapting it slightly to our framework. In a similar vein as the proof of Lemma $3.8$ of said paper, we introduce the events
\begin{equation*}
    \begin{aligned}
    K_{k,N_T}^{(i)} &= \{\vert \Delta_{\mathcal{P}_T}^k \bold{Y}^{(i)} \vert \le \nu_{\mathcal{P}_T}^{(i)}\},\\
    M_{k,N_T} &= \{\Delta_{\mathcal{P}_T}^k \bold{N} = 0\}.
    \end{aligned}
\end{equation*}
We may write the complement of $A_{k,N_T}^{(i)}$ as 
\begin{equation*}
    (A_{k,N_T}^{(i)})^c = \{\bold{1}_{K_{k,N_T}^{(i)}} \ne \bold{1}_{M_{k,N_T}}\} =(K_{k,N_T}^{(i)}\backslash M_{k,N_T})\cup(M_{k,N_T}\backslash K_{k,N_T}^{(i)})  .
\end{equation*}
    It follows straightforwardly from the proof of Lemma $3.8$ in \cite{Mai2014} by replacing the use of the bound from Lemma $3.7$ in Equation $(10)$ of said paper with the bound in \Cref{Lemma:CountBound} with $l^\prime  = 2$ and $\delta = \beta^{(i)}$ that
\begin{multline*}
    %\begin{aligned}
     \mathbb{P}\left((A_{k,N_T}^{(i)})^c \right) = O\biggl(d  \Delta_{\mathcal{P}_T}\left(\tilde{F}^{(i)}\left(2 \Delta_{\mathcal{P}_T}^{\beta^{(i)}} \right)-\tilde{F}^{(i)}\left(-2 \Delta_{\mathcal{P}_T}^{\beta^{(i)}} \right)\right) \\+\left(d^{7} \Delta_{\mathcal{P}_T}^{(2(1-\beta^{(i)}))}\right)\biggr),\quad \text{as } d \rightarrow \infty.
     %\end{aligned}
\end{multline*}
Hence,
\begin{equation*}
    \begin{aligned}
    \mathbb{P}\left( A_{N_T} \right) &= 1- \mathbb{P}\left( A_{N_T}^c \right) \\&\ge 1- \sum_{i=1}^d \sum_{k=0}^{N_T-1}\mathbb{P}\left((A_{k,N_T}^{(i)})^c \right) \\
    &\ge 1- \sum_{i=1}^d \sum_{k=0}^{N_T-1}O\biggl(d \Delta_{\mathcal{P}_T}\left(\tilde{F}^{(i)}\left(2 \Delta_{\mathcal{P}_T}^{\beta^{(i)}} \right)-\tilde{F}^{(i)}\left(-2 \Delta_{\mathcal{P}_T}^{\beta^{(i)}} \right)\right) \\&+\left(d^{7} \Delta_{\mathcal{P}_T}^{(2(1-\beta^{(i)}))}\right)\biggr)\\
    &= 1- O\biggl(d^2T\max_{i\in\{1,2,\dots,d\}}\left(\left(\tilde{F}^{(i)}\left(2 \Delta_{\mathcal{P}_T}^{\beta^{(i)}} \right) - \tilde{F}^{(i)}\left(-2 \Delta_{\mathcal{P}_T}^{\beta^{(i)}} \right)\right)\right) \\&+ \left(Td^{8} \Delta_{\mathcal{P}_T}^{1-2\beta^*}\right)\biggr)\\&=1- o(1),\quad \text{as } d \rightarrow \infty.
    \end{aligned}
\end{equation*}
\end{proof}
\begin{lemma}\label{L_1BoundFInite}
    Grant \Cref{Ass:Estimation,Ass.H,Ass.High-FrequencySampling,Ass.GeneralBothCases,Ass.AssFiniteAct}. Then, for $\delta \in (0,1/4)$
    \begin{equation}
    \Vert\left(\dfrac{1}{T}\sum_{n=0}^{N_T-1} \left( \Delta_n\bold{X}^c-\Delta_{\mathcal{P}_T}^n \bold{\bold{X}}\odot\vec{\bold{1}}(\Delta_{\mathcal{P}_T}^n \bold{\bold{X}})\right)\mathbf{X}_{s_n^{N_T}}^\top\right)\Vert_1 \le \Delta_{\mathcal{P}_T}^{\delta}, 
    \end{equation}
    with high probability under $\mathbb{P}^{{Q_0}}$.
\end{lemma}
\begin{proof}
    We note that 
    \begin{equation*}
        \begin{aligned}
            &\Vert\left(\dfrac{1}{T}\sum_{n=0}^{N_T-1} \left( \Delta_n\bold{X}^c-\Delta_{\mathcal{P}_T}^n \bold{\bold{X}}\odot\vec{\bold{1}}(\Delta_{\mathcal{P}_T}^n \bold{\bold{X}})\right)\mathbf{X}_{s_n^{N_T}}^\top\right)\Vert_1 \\&\le \dfrac{1}{T}\sum_{n=0}^{N_T-1} \Vert\left(\left( \Delta_n\bold{X}^c-\Delta_{\mathcal{P}_T}^n \bold{\bold{X}}\odot\vec{\bold{1}}(\Delta_{\mathcal{P}_T}^n \bold{\bold{X}})\right)\mathbf{X}_{s_n^{N_T}}^\top\right)\Vert_1 \\
            &=\dfrac{1}{T}\sum_{n=0}^{N_T-1} \sum_{k=1}^d \sum_{j=1}^d \vert\bold{X}_{s_n^{N_T}}^{(j)} \left(\Delta_{\mathcal{P}_T}^n \bold{\bold{X}}^{(k)}\bold{1}_{\{\vert \Delta_{\mathcal{P}_T}^n \bold{\bold{X}}^{(k)} \vert \le \nu_{\mathcal{P}_T}^{(k)}\}}-(\Delta_{\mathcal{P}_T}^n \bold{X}^c)^{(k)} \right) \vert.
        \end{aligned}
    \end{equation*}
    Since
    \begin{equation*}
            \begin{aligned}
            &\mathbb{E}_{\mathbb{P}^{Q_0}}\left(\dfrac{1}{T}\sum_{n=0}^{N_T-1} \sum_{k=1}^d \sum_{j=1}^d \vert\bold{X}_{s_n^{N_T}}^{(j)} \left(\Delta_{\mathcal{P}_T}^n \bold{\bold{X}}^{(k)}\bold{1}_{\{\vert \Delta_{\mathcal{P}_T}^n \bold{\bold{X}}^{(k)} \vert \le \nu_{\mathcal{P}_T}^{(k)}\}}-(\Delta_{\mathcal{P}_T}^n \bold{X}^c)^{(k)} \right) \vert \right) \\&= 
            \mathbb{E}_{\mathbb{P}^{Q_0}}\left(\dfrac{1}{T}\sum_{n=0}^{N_T-1} \sum_{k=1}^d \sum_{j=1}^d \vert\bold{X}_{s_n^{N_T}}^{(j)} \left(\Delta_{\mathcal{P}_T}^n \bold{\bold{X}}^{(k)}\bold{1}_{\{\vert \Delta_{\mathcal{P}_T}^n \bold{\bold{X}}^{(k)} \vert \le \nu_{\mathcal{P}_T}^{(k)}\}}-(\Delta_{\mathcal{P}_T}^n \bold{X}^c(Q_0))^{(k)}+\left(\int_{s_{n}^{N_T}}^{s_{n+1}^{N_T}} Q_0 \bold{X_s}ds\right)^{(k)} \right) \vert \right) \\&= 
             \mathbb{E}\left(\dfrac{1}{T}\sum_{n=0}^{N_T-1} \sum_{k=1}^d \sum_{j=1}^d \vert\bold{Y}_{s_n^{N_T}}^{(j)} \left(\Delta_{\mathcal{P}_T}^n \bold{\bold{Y}}^{(k)}\bold{1}_{\{\vert \Delta_{\mathcal{P}_T}^n \bold{\bold{Y}}^{(k)} \vert \le \nu_{\mathcal{P}_T}^{(k)}\}}-(\Delta_{\mathcal{P}_T}^n \Sigma^{1/2}\bold{W})^{(k)}+\left(\int_{s_{n}^{N_T}}^{s_{n+1}^{N_T}} Q_0 \bold{Y}_sds\right)^{(k)} \right) \vert \right).
            \end{aligned}
    \end{equation*}
    We work under the right hand side instead. In the above calculations we used theory of \Cref{PDynOfYandX}, i.e.\ that the relationship between $P_0$ martingale and $\mathbb{P}^{Q_0}$ continuous martingale on the canonical space is given by $\Delta_{\mathcal{P}_T}^n \bold{X}^c = -\int_{s_{n-1}}^{s_n} Q_0 \bold{X_s}ds + \Delta_{\mathcal{P}_T}^n \bold{X}^c(Q_0)$. We also note, that by \eqref{eqn:1} and \Cref{Levyito} under the assumptions,  $\Delta_{\mathcal{P}_T}^n \bold{Y} = -\int_{s_{n}^{N_T}}^{s_{n+1}^{N_T}} Q_0 \bold{Y}_sds + \Sigma^{1/2}\Delta_{\mathcal{P}_T}^n\bold{W} + \Delta_{\mathcal{P}_T}^n \bold{J}$, where $\mathbb{W}$ and $\mathbb{J}$ is a standard Brownian motion and a compound Poisson process. On $A_{N_T}$, we have that, for any $k \in \{1,\dots, d\}$ and $n\in \{1,\dots, N_T-1 \}$,
    \begin{equation*}
        \begin{aligned}
            &\Delta_{\mathcal{P}_T}^n \bold{Y}^{(k)}\bold{1}_{\{\vert \Delta_{\mathcal{P}_T}^n \bold{\bold{Y}}^{(k)} \vert \le \nu_{\mathcal{P}_T}^{(k)}\}} - \left(\Sigma^{1/2}\Delta_{\mathcal{P}_T}^n \bold{W}\right)^{(k)} +\left(\int_{s_{n}^{N_T}}^{s_{n+1}^{N_T}} Q_0 \bold{Y}_sds\right)^{(k)} \\&=\Delta_{\mathcal{P}_T}^n \bold{Y}^{(k)}\bold{1}_{\Delta_{\mathcal{P}_T}^{n}N =0} - \left(\Sigma^{1/2}\Delta_{\mathcal{P}_T}^n \bold{W}\right)^{(k)} +\left(\int_{s_{n}^{N_T}}^{s_{n+1}^{N_T}} Q_0 \bold{Y}_sds\right)^{(k)}
            \\&=\begin{cases}
\begin{array}{l}
0,\\
-\left(\Sigma^{1/2}\Delta_{\mathcal{P}_{T}}^{n}\mathbf{W}\right)^{(k)}+\left(\int_{s_{n}^{N_T}}^{s_{n+1}^{N_T}} Q_0 \bold{Y}_sds\right)^{(k)},
\end{array} & \begin{array}{l}
\text{on }{\Delta_{\mathcal{P}_{T}}^{n}N=0}\\
\text{on }{\Delta_{\mathcal{P}_{T}}^{n}N>0}.
\end{array}\end{cases}
        \end{aligned}
    \end{equation*}
    As in \cite{Courgeau2022a}, we define $B_{n,N_T}:=\{\Delta_{\mathcal{P}_T}^{n}N >0\}$.
    In the finite activity case, $N$ is the Poisson counting process of the compound Poisson representation of $\tilde{\mathbb{U}}$ from \Cref{LevyItoImpl}, with rate $\tilde{\lambda} = F(\mathbb{R}^d)$. We therefore have $\mathbb{P}\left(B_{n,N_T} \right) \le \Delta_{\mathcal{P}_T}\, F(\mathbb{R}^d) = O\left(d\Delta_{\mathcal{P}_T}\right)$ as $d\rightarrow \infty$, using $F(\mathbb{R}^d) = O(d)$ from \Cref{Ass.AssFiniteAct}(1). Then,
    \begin{equation*}
        \begin{aligned}
        &\bold{1}_{A_{N_T}}\dfrac{1}{T}\sum_{n=0}^{N_T-1}\sum_{k=1}^d\sum_{j=1}^d\vert \bold{Y}_{s_n^{N_T}}^{(j)}\left(\Delta_{\mathcal{P}_T}^n \bold{Y}^{(k)}\bold{1}_{\{\Delta_{\mathcal{P}_T}^{n}N =0\}} - \left(\Sigma^{1/2}\Delta_{\mathcal{P}_T}^n \bold{W}\right)^{(k)} +\left(\int_{s_{n}^{N_T}}^{s_{n+1}^{N_T}} Q_0 \bold{Y}_sds\right)^{(k)} \right)\vert \\&= \sum_{n=0}^{N_T-1}\sum_{k=1}^d\sum_{j=1}^d\vert \bold{Y}_{s_n^{N_T}}^{(j)} \left( -\left(\Sigma^{1/2}\Delta_{\mathcal{P}_T}^n \bold{W}\right)^{(k)} +\left(\int_{s_{n}^{N_T}}^{s_{n+1}^{N_T}} Q_0 \bold{Y}_sds\right)^{(k)}\right) \vert\bold{1}_{A_{N_T}\cap B_{n,N_T}}.
        \end{aligned}
    \end{equation*}
    We then note, that 
    \begin{equation*}
        \begin{aligned}
       &\vert \bold{Y}_{s_n^{N_T}}^{(j)} \left( -\left(\Sigma^{1/2}\Delta_{\mathcal{P}_T}^n \bold{W}\right)^{(k)} +\left(\int_{s_{n}^{N_T}}^{s_{n+1}^{N_T}} Q_0 \bold{Y}_sds\right)^{(k)}\right) \vert\\ &=\vert\bold{Y}_{s_n^{N_T}}^{(j)}\left(\sum_{l=1}^d [Q_0]_{kl}\int_{s_{n}^{N_T}}^{s_{n+1}^{N_T}} \bold{Y}_s^{(l)} ds - \sum_{l=1}^d\Sigma_{kl}^{1/2}\Delta_{\mathcal{P}_T}^n\bold{W}^{(l)}\right)\vert 
        \\ &\le\sum_{l=1}^d\vert\bold{Y}_{s_n^{N_T}}^{(j)} [Q_0]_{kl}\int_{s_{n}^{N_T}}^{s_{n+1}^{N_T}} \bold{Y}_s^{(l)} ds\vert + \sum_{l=1}^d\vert\bold{Y}_{s_n^{N_T}}^{(j)}\Sigma_{kl}^{1/2}\Delta_{\mathcal{P}_T}^n\bold{W}^{(l)}\vert.
        \end{aligned}
    \end{equation*}
    Following Hölders and Jensen's inequality, the independence of $\mathbb{W}$ and $\mathbb{N}$ as well as $\bold{Y}_{s_n^{N_T}}$ and $\Delta_{{\mathcal{P}_T}}^n \bold{N}$ since we are working with forward-looking increments, it holds that  
    \begin{equation*}
        \begin{aligned}
        &\mathbb{E}\left(\bold{1}_{A_{N_T}\cap B_{n,N_T}} \sum_{l=1}^d\vert\bold{Y}_{s_n^{N_T}}^{(j)} [Q_0]_{kl}\int_{s_{n}^{N_T}}^{s_{n+1}^{N_T}} \bold{Y}_s^{(l)} ds\vert\right) \\&\le \sum_{l=1}^d \vert [Q_0]_{kl} \vert \mathbb{E}\left( \bold{1}_{A_{N_T}\cap B_{n,N_T}} \vert\bold{Y}_{s_n^{N_T}}^{(j)} \int_{s_{n}^{N_T}}^{s_{n+1}^{N_T}} \bold{Y}_s^{(l)} ds\vert\right) \\
        &\le \sum_{l=1}^d \vert [Q_0]_{kl} \vert \mathbb{E}\left(\bold{1}_{B_{n,N_T}} \vert\bold{Y}_{s_n^{N_T}}^{(j)} \int_{s_{n}^{N_T}}^{s_{n+1}^{N_T}} \bold{Y}_s^{(l)} ds\vert\right) \\
        &\le \sum_{l=1}^d \vert [Q_0]_{kl} \vert \mathbb{E}\left( (\vert\bold{Y}_{s_n^{N_T}}^{(j)}\bold{1}_{B_{n,N_T}} )^2\right)^{1/2} \mathbb{E}\left( \left(\int_{s_{n}^{N_T}}^{s_{n+1}^{N_T}} \bold{Y}_s^{(l)} ds\right)^2\right)^{1/2} \\
        &= \sum_{l=1}^d \vert [Q_0]_{kl} \vert \mathbb{E}\left( (\bold{Y}_0^{(j)})^2\right)^{1/2} \mathbb{E}\left( \left(\int_{0}^{s_{n+1}^{N_T}-s_{n}^{N_T}} \bold{Y}_s^{(l)} ds\right)^2\right)^{1/2} \mathbb{P}\left(B_{n,N_T} \right)^{1/2}\\
        \\
        &\le \sum_{l=1}^d \vert [Q_0]_{kl} \vert \mathbb{E}\left( (\bold{Y}_0^{(j)})^2\right)^{1/2} \mathbb{E}\left( \Delta_{\mathcal{P}_T}\int_{0}^{s_{n+1}^{N_T}-s_{n}^{N_T}} \left(\bold{Y}_s^{(l)}\right)^2 ds\right)^{1/2} \mathbb{P}\left(B_{n,N_T} \right)^{1/2}\\
        &\le \sum_{l=1}^d \vert [Q_0]_{kl} \vert \mathbb{E}\left( (\bold{Y}_0^{(j)})^2\right)^{1/2} \mathbb{E}\left( (\bold{Y}_0^{(l)})^2\right)^{1/2} \Delta_{\mathcal{P}_T}^{3/2}F\left(\mathbb{R}^d\right)^{1/2}.
        \end{aligned}
    \end{equation*}
Again, by independence of $\mathbb{W}$ and $\mathbb{N}$ as well as $\bold{Y}_{s_n^{N_T}}$ and $\Delta_{{\mathcal{P}_T}}^n \bold{N}$ and the Cauchy-Schwarz inequality, we have
\begin{equation*}
\begin{aligned}
    &\mathbb{E}\left(\bold{1}_{A_{N_T}\cap B_{n,N_T}}\sum_{l=1}^d\vert\bold{Y}_{s_n^{N_T}}^{(j)}\Sigma_{kl}^{1/2}\Delta_{\mathcal{P}_T}^n\bold{W}^{(l)}\vert\right)\\
    &\le \mathbb{P}\left( B_{n,N_T}\right)\sum_{l=1}^d\mathbb{E}\left(\vert\bold{Y}_{s_n^{N_T}}^{(j)}\vert\right)\Sigma_{kl}^{1/2}\mathbb{E}\left(\vert\Delta_{\mathcal{P}_T}^n\bold{W}^{(l)}\vert\right)\\ &\le \Delta_{\mathcal{P}_T}F\left(\mathbb{R}^d\right)\sum_{l=1}^d\sqrt{\mathbb{E}\left(\left(\bold{Y}_{s_n^{N_T}}^{(j)}\right)^2\right)}\vert\Sigma_{kl}^{1/2}\vert\sqrt{\mathbb{E}\left(\left(\Delta_{\mathcal{P}_T}^n\bold{W}^{(l)}\right)^2\right)}\\&= \sum_{l=1}^d \vert\Sigma_{kl} \vert  \Delta_{\mathcal{P}_T}^{3/2} \mathbb{E}\left( (\bold{Y}_0^{(j)})^2\right) F\left(\mathbb{R}^d\right).
\end{aligned}
\end{equation*}
We thus conclude that 
\begin{equation}\label{FiniteActBound1}
    \begin{aligned}
         &\mathbb{E}\left(1_{A_{N_T}}\dfrac{1}{T}\sum_{n=0}^{N_T-1} \sum_{k=1}^d \sum_{j=1}^d \vert\bold{Y}_{s_n^{N_T}}^{(j)} \left(\Delta_{\mathcal{P}_T}^n \bold{\bold{Y}}^{(k)}\bold{1}_{\{\vert \Delta_{\mathcal{P}_T}^n \bold{\bold{Y}}^{(k)} \vert \le \nu_{\mathcal{P}_T}^{(k)}\}}-(\Delta_{\mathcal{P}_T}^n \Sigma^{1/2}\bold{W})^{(k)}+\left(\int_{s_{n}^{N_T}}^{s_{n+1}^{N_T}} Q_0 \bold{Y}_sds\right)^{(k)} \right) \vert \right) \\&= 
             \mathbb{E}\biggl(\dfrac{1}{T}\sum_{n=0}^{N_T-1} \sum_{k=1}^d \sum_{j=1}^d \vert\bold{Y}_{s_n^{N_T}}^{(j)} (\Delta_{\mathcal{P}_T}^n \bold{\bold{Y}}^{(k)}\bold{1}_{\{\vert \Delta_{\mathcal{P}_T}^n \bold{\bold{Y}}^{(k)} \vert \le \nu_{\mathcal{P}_T}^{(k)}\}}-(\Delta_{\mathcal{P}_T}^n \Sigma^{1/2}\bold{W})^{(k)}+(\int_{s_{n}^{N_T}}^{s_{n+1}^{N_T}} Q_0 \bold{Y}_sds)^{(k)} )\vert 1_{A_{N_T}\cap B_{n,N_T}}\biggr)
             \\&\le  \mathbb{E}\left(\dfrac{1}{T}\sum_{n=0}^{N_T-1} \sum_{k=1}^d \sum_{j=1}^d\sum_{l=1}^d\vert\bold{Y}_{s_n^{N_T}}^{(j)} [Q_0]_{kl}\int_{s_{n}^{N_T}}^{s_{n+1}^{N_T}} \bold{Y}_s^{(l)} ds\vert 1_{ B_{n,N_T}} + \sum_{l=1}^d\vert\bold{Y}_{s_n^{N_T}}^{(j)}\Sigma_{kl}^{1/2}\Delta_{\mathcal{P}_T}^n\bold{W}^{(l)}\vert1_{ B_{n,N_T}}\right) \\&\le
             \dfrac{1}{T}\sum_{n=0}^{N_T-1} \sum_{k=1}^d \sum_{j=1}^d \biggl( \sum_{l=1}^d \vert [Q_0]_{kl} \vert \mathbb{E}\left( (\bold{Y}_0^{(j)})^2\right)^{1/2} \mathbb{E}\left( (\bold{Y}_0^{(l)})^2\right)^{1/2} \Delta_{\mathcal{P}_T}^{3/2}F\left(\mathbb{R}^d\right)^{1/2} \\& \quad + \sum_{l=1}^d \vert\Sigma_{kl} \vert  \Delta_{\mathcal{P}_T}^{3/2} \mathbb{E}\left( (\bold{Y}_0^{(j)})^2\right) F\left(\mathbb{R}^d\right)\biggr)\\&\le
             \dfrac{N_T}{T} \sum_{k=1}^d \sum_{j=1}^d \biggl( \sum_{l=1}^d \vert [Q_0]_{kl} \vert \mathbb{E}\left( (\bold{Y}_0^{(j)})^2\right)^{1/2} \mathbb{E}\left( (\bold{Y}_0^{(l)})^2\right)^{1/2} \Delta_{\mathcal{P}_T}^{3/2}F\left(\mathbb{R}^d\right)^{1/2} \\& \quad + \sum_{l=1}^d \vert\Sigma_{kl} \vert  \Delta_{\mathcal{P}_T}^{3/2} \mathbb{E}\left( (\bold{Y}_0^{(j)})^2\right) F\left(\mathbb{R}^d\right)\biggr)
             \\&\le
             \dfrac{N_T}{T} \sum_{k=1}^d \sum_{j=1}^d \biggl( \sum_{l=1}^d \max_{k,l}\vert [Q_0]_{kl} \vert \mathbb{E}\left( (\bold{Y}_0^{(j)})^2\right)^{1/2} \mathbb{E}\left( (\bold{Y}_0^{(l)})^2\right)^{1/2} \Delta_{\mathcal{P}_T}^{3/2}F\left(\mathbb{R}^d\right)^{1/2} \\& \quad + \sum_{l=1}^d \max_{k,l}\vert\Sigma_{kl} \vert  \Delta_{\mathcal{P}_T}^{3/2} \mathbb{E}\left( (\bold{Y}_0^{(j)})^2\right) F\left(\mathbb{R}^d\right)\biggr)    \\&\le
             \dfrac{N_T}{T}\max_{k,l}\vert [Q_0]_{kl}\vert d^3\biggl(  \vert \mathbb{E}\left( (\bold{Y}_0^{(i^*)})^2\right)^{1/2} \mathbb{E}\left( (\bold{Y}_0^{(i^*)})^2\right)^{1/2} \Delta_{\mathcal{P}_T}^{3/2}F\left(\mathbb{R}^d\right)^{1/2} \\& \quad +\max_{k,l}\vert\Sigma_{kl} \vert  \Delta_{\mathcal{P}_T}^{3/2} \mathbb{E}\left( (\bold{Y}_0^{(i^*)})^2\right) F\left(\mathbb{R}^d\right)\biggr) \\
             & = O\left( \Delta_{\mathcal{P}_T}^{1/2} d^{8}\right), \quad \text{ as } d\rightarrow \infty .
    \end{aligned}
\end{equation}
We now note that, for $\delta \in (0,1/4)$,
\small
\begin{equation*}
    \begin{aligned}
        &\mathbb{P}^{Q_0}\left(\left\Vert\left(\dfrac{1}{T}\sum_{n=0}^{N_T-1} \left( \Delta_n\bold{X}^c-\Delta_{\mathcal{P}_T}^n \bold{\bold{X}}\odot\vec{\bold{1}}(\Delta_{\mathcal{P}_T}^n \bold{\bold{X}})\right)\mathbf{X}_{s_n^{N_T}}^\top\right)\right\Vert_1 > \Delta_{\mathcal{P}_T}^{\delta}\right)\\
        & \le \mathbb{P}\left(\dfrac{1}{T}\sum_{n=0}^{N_T-1} \sum_{k=1}^d \sum_{j=1}^d \left(\vert\bold{Y}_{s_n^{N_T}}^{(j)} \left(\Delta_{\mathcal{P}_T}^n \bold{\bold{Y}}^{(k)}\bold{1}_{\{\vert \Delta_{\mathcal{P}_T}^n \bold{\bold{Y}}^{(k)} \vert \le \nu_{\mathcal{P}_T}^{(k)}\}}-(\Delta_{\mathcal{P}_T}^n \Sigma^{1/2}\bold{W})^{(k)}+\left(\int_{s_{n}^{N_T}}^{s_{n+1}^{N_T}} Q_0 \bold{Y}_sds\right)^{(k)} \right) \vert\right) >\Delta_{\mathcal{P}_T}^{\delta} \right)\\
        &= \mathbb{P}\biggl(1_{A_{N_T}}\dfrac{1}{T}\sum_{n=0}^{N_T-1} \sum_{k=1}^d \sum_{j=1}^d \left\vert\bold{Y}_{s_n^{N_T}}^{(j)} \left(\Delta_{\mathcal{P}_T}^n \bold{\bold{Y}}^{(k)}\bold{1}_{\{\vert \Delta_{\mathcal{P}_T}^n \bold{\bold{Y}}^{(k)} \vert \le \nu_{\mathcal{P}_T}^{(k)}\}}-(\Delta_{\mathcal{P}_T}^n \Sigma^{1/2}\bold{W})^{(k)}+\left(\int_{s_{n}^{N_T}}^{s_{n+1}^{N_T}} Q_0 \bold{Y}_sds\right)^{(k)} \right) \right\vert \\&+1_{A_{N_T}^c}\dfrac{1}{T}\sum_{n=0}^{N_T-1} \sum_{k=1}^d \sum_{j=1}^d \left\vert\bold{Y}_{s_n^{N_T}}^{(j)} \left(\Delta_{\mathcal{P}_T}^n \bold{\bold{Y}}^{(k)}\bold{1}_{\{\vert \Delta_{\mathcal{P}_T}^n \bold{\bold{Y}}^{(k)} \vert \le \nu_{\mathcal{P}_T}^{(k)}\}}-(\Delta_{\mathcal{P}_T}^n \Sigma^{1/2}\bold{W})^{(k)}+\left(\int_{s_{n}^{N_T}}^{s_{n+1}^{N_T}} Q_0 \bold{Y}_sds\right)^{(k)} \right) \right\vert >\Delta_{\mathcal{P}_T}^{\delta}\biggr)\\
      &\le\mathbb{P}\biggl(1_{A_{N_T}}\dfrac{1}{T}\sum_{n=0}^{N_T-1} \sum_{k=1}^d \sum_{j=1}^d \left\vert\bold{Y}_{s_n^{N_T}}^{(j)} \left(\Delta_{\mathcal{P}_T}^n \bold{\bold{Y}}^{(k)}\bold{1}_{\{\vert \Delta_{\mathcal{P}_T}^n \bold{\bold{Y}}^{(k)} \vert \le \nu_{\mathcal{P}_T}^{(k)}\}}-(\Delta_{\mathcal{P}_T}^n \Sigma^{1/2}\bold{W})^{(k)}+\left(\int_{s_{n}^{N_T}}^{s_{n+1}^{N_T}} Q_0 \bold{Y}_sds\right)^{(k)} \right) \right\vert >\Delta_{\mathcal{P}_T}^{\delta}/2\biggr) \\
      &+\mathbb{P}\biggl(1_{A_{N_T}^c}\dfrac{1}{T}\sum_{n=0}^{N_T-1} \sum_{k=1}^d \sum_{j=1}^d \left\vert\bold{Y}_{s_n^{N_T}}^{(j)} \left(\Delta_{\mathcal{P}_T}^n \bold{\bold{Y}}^{(k)}\bold{1}_{\{\vert \Delta_{\mathcal{P}_T}^n \bold{\bold{Y}}^{(k)} \vert \le \nu_{\mathcal{P}_T}^{(k)}\}}-(\Delta_{\mathcal{P}_T}^n \Sigma^{1/2}\bold{W})^{(k)}+\left(\int_{s_{n}^{N_T}}^{s_{n+1}^{N_T}} Q_0 \bold{Y}_sds\right)^{(k)} \right) \right\vert >\Delta_{\mathcal{P}_T}^{\delta}/2\biggr)\\ 
      &\le\mathbb{P}\biggl(1_{A_{N_T}}\dfrac{1}{T}\sum_{n=0}^{N_T-1} \sum_{k=1}^d \sum_{j=1}^d \left\vert\bold{Y}_{s_n^{N_T}}^{(j)} \left(\Delta_{\mathcal{P}_T}^n \bold{\bold{Y}}^{(k)}\bold{1}_{\{\vert \Delta_{\mathcal{P}_T}^n \bold{\bold{Y}}^{(k)} \vert \le \nu_{\mathcal{P}_T}^{(k)}\}}-(\Delta_{\mathcal{P}_T}^n \Sigma^{1/2}\bold{W})^{(k)}+\left(\int_{s_{n}^{N_T}}^{s_{n+1}^{N_T}} Q_0 \bold{Y}_sds\right)^{(k)} \right) \right\vert >\Delta_{\mathcal{P}_T}^{\delta}/2\biggr) \\
      &+ \mathbb{P}\left(A_{N_T}^c \right).
    \end{aligned}
\end{equation*}
Following \eqref{FiniteActBound1}, \Cref{Lemma:SetBound} and Markov's inequality, we thus conclude that under the assumptions,
\begin{equation*}
    \begin{aligned}
          &\mathbb{P}^{Q}\left(\Vert\left(\dfrac{1}{T}\sum_{n=0}^{N_T-1} \left( \Delta_n\bold{X}^c-\Delta_{\mathcal{P}_T}^n \bold{\bold{X}}\odot\vec{\bold{1}}(\Delta_{\mathcal{P}_T}^n \bold{\bold{X}})\right)\mathbf{X}_{s_n^{N_T}}^\top\right)\Vert_1 > \Delta_{\mathcal{P}_T}^{\delta}\right) \\&\le             2\dfrac{N_T\Delta_{\mathcal{P}_T}^{1+(1/2-\delta)}}{T}\max_{k,l}\vert [Q_0]_{kl}\vert d^3\mathbb{E}\left( (\bold{Y}_0^{(i^*)})^2\right)\biggl(  F\left(\mathbb{R}^d\right)^{1/2} \\& \quad +\max_{k,l}\vert\Sigma_{kl} \vert    F\left(\mathbb{R}^d\right)\biggr) + \mathbb{P}\left(A_{N_T}^c \right) \\
          &= O\left(\Delta_{\mathcal{P}_T}^{1/2- \delta}d^8 \right) + o(1)\\&= o(1), \quad \text{ as } d \rightarrow \infty,
    \end{aligned}
\end{equation*}
which gives us the desired result.
\end{proof}
\begin{proof}[Proof of \Cref{LikelihoodEqualityFeasibleCase}]
    We first remember that 
    \begin{equation*}
        \begin{aligned}
        \mathcal{L}_T^{\mathcal{P}_T}\left(Q\right) &= \dfrac{1}{T}\sum_{n=0}^{N_T-1}\left(\Sigma^{-1}Q \bold{X}_{s_n^{N_T}}\right)^{\top}\left(\bold{X}_{s_{n+1}^{N_T}}^c-\bold{X}_{s_n^{N_T}}^c\right) +\dfrac{1}{2T}\sum_{n=0}^{N_T-1}\left(\Sigma^{-1/2}Q \bold{X}_{s_{n}^{N_T}}\right)^{\top}\Sigma^{-1/2}Q \bold{X}_{s_{n}^{N_T}}\left(s_{n+1}^{N_T}-s_{n}^{N_T}\right) \\
        &=\Tr \biggl(\dfrac{1}{T}\sum_{n=0}^{N_T-1}\left(\Sigma^{-1}Q \bold{X}_{s_n^{N_T}}\right)^{\top}\left(\bold{X}_{s_{n+1}^{N_T}}^c-\bold{X}_{s_n^{N_T}}^c\right) \\&+\dfrac{1}{2T}\sum_{n=0}^{N_T-1}\left(\Sigma^{-1/2}Q \bold{X}_{s_{n}^{N_T}}\right)^{\top}\Sigma^{-1/2}Q \bold{X}_{s_{n}^{N_T}}\left(s_{n+1}^{N_T}-s_{n}^{N_T}\right)  \biggr) \\
        &=\Tr\biggl(\frac{1}{T} \Sigma^{-1} Q\sum_{n=0}^{N_T-1}\mathbf{X}_{s_n^{N_T}} \left(\Delta_n\bold{X}^c\right)^\top +\frac{1}{2 T}  Q^\top(\Sigma^{-1})^\top Q\int_0^T\mathbf{X}_s \mathbf{X}_s^\top \mu_{\mathcal{P}_T}(ds)\biggr),
         \end{aligned}
    \end{equation*}
    where we used \Cref{MatrixCalc1,MatrixCalc2}. Hence, another expression for its gradient is
    \begin{equation*}
        \nabla \mathcal{L}_T^{\mathcal{P}_T}\left(Q\right) = \Sigma^{-1}\dfrac{1}{T}\sum_{n=0}^{N_T-1}\left(\Delta_n\bold{X}^c\right) \mathbf{X}_{s_n^{N_T}}^\top   + \Sigma^{-1}Q\widehat{C}_{\mathcal{P}_T}.
    \end{equation*}
Similarly, for \eqref{FeasibleDiscLike}, we find that 
\begin{equation*}
    \nabla\bar{\mathcal{L}}_T^{\mathcal{P}_T}(Q) =\Sigma^{-1}\dfrac{1}{T}\sum_{n=0}^{N_T-1} \left( \Delta_{\mathcal{P}_T}^n \bold{\bold{X}}\odot\vec{\bold{1}}(\Delta_{\mathcal{P}_T}^n \bold{\bold{X}})\right)\mathbf{X}_{s_n^{N_T}}^\top   + \Sigma^{-1}Q\widehat{C}_{\mathcal{P}_T}.
\end{equation*}
Following differentiability and strict convexity of \eqref{FeasibleDiscLike} and \eqref{DiscreteLikelihood}, we know that their minimizers are the unique matrices at which the respective gradients equal $0$. We thus have that
\begin{equation*}
    \begin{aligned}
        \nabla \mathcal{L}_T^{\mathcal{P}_T}\left(\widehat{Q}\right) &= \nabla\bar{\mathcal{L}}_T^{\mathcal{P}_T}(\bar{Q}) \\
        &\iff\\
        \left(\bar{Q}-\widehat{Q}\right) &= \left(\dfrac{1}{T}\sum_{n=0}^{N_T-1} \Delta_n\bold{X}^c\mathbf{X}_{s_n^{N_T}}^\top -\dfrac{1}{T}\sum_{n=0}^{N_T-1} \left( \Delta_{\mathcal{P}_T}^n \bold{\bold{X}}\odot\vec{\bold{1}}(\Delta_{\mathcal{P}_T}^n \bold{\bold{X}})\right)\mathbf{X}_{s_n^{N_T}}^\top \right)\widehat{C}_{\mathcal{P}_T}^{-1} \\
        &=
        \left(\dfrac{1}{T}\sum_{n=0}^{N_T-1} \left( \Delta_n\bold{X}^c-\Delta_{\mathcal{P}_T}^n \bold{\bold{X}}\odot\vec{\bold{1}}(\Delta_{\mathcal{P}_T}^n \bold{\bold{X}})\right)\mathbf{X}_{s_n^{N_T}}^\top\right)\widehat{C}_{\mathcal{P}_T}^{-1}.
    \end{aligned}
\end{equation*}
We thus have that
\begin{equation*}
    \begin{aligned}
    \Vert \bar{Q}-\widehat{Q} \Vert_F &= \Vert \left(\dfrac{1}{T}\sum_{n=0}^{N_T-1} \left( \Delta_n\bold{X}^c-\Delta_{\mathcal{P}_T}^n \bold{\bold{X}}\odot\vec{\bold{1}}(\Delta_{\mathcal{P}_T}^n \bold{\bold{X}})\right)\mathbf{X}_{s_n^{N_T}}^\top\right)\widehat{C}_{\mathcal{P}_T}^{-1}\Vert_F \\&
    \le \Vert\left(\dfrac{1}{T}\sum_{n=0}^{N_T-1} \left( \Delta_n\bold{X}^c-\Delta_{\mathcal{P}_T}^n \bold{\bold{X}}\odot\vec{\bold{1}}(\Delta_{\mathcal{P}_T}^n \bold{\bold{X}})\right)\mathbf{X}_{s_n^{N_T}}^\top\right)\Vert_1 \Vert\widehat{C}_{\mathcal{P}_T}^{-1}\Vert,
    \end{aligned}
\end{equation*}
where we used that $\Vert AB\Vert_F \le \Vert A\Vert_F\Vert B\Vert$ for any matrices $A,B$, and that the entry-wise $\ell_1$ norm is always an upper bound of the Frobenius norm by Cauchy-Schwarz. By \Cref{L_1BoundFInite} and \Cref{Prop:UpperBoundC_TInverse} we then have the desired bound.
\end{proof}
\subsubsection{Proofs: Infinite activity}\label{ProofsInfAct}
The main result in this Section will be \Cref{L_1BoundInFinite}. However, prior to that we establish a few helpful lemmas.
\begin{lemma}\label{S_2:FirstBound}
    Grant \Cref{Ass:Estimation,Ass.H,Ass.High-FrequencySampling,Ass.GeneralBothCases,Ass.InfiniteAct}, and let the notation be as in \Cref{LevyItoImpl}. Then
    \begin{equation*}
        \mathbb{P}\left(\dfrac{1}{T}\sum_{n=0}^{N_T-1} \sum_{k=1}^d\sum_{j=1}^d \vert\bold{Y}_{s_n^{N_T}}^{(j)} \Delta_{\mathcal{P}_T}^n \bold{\bold{\tilde{Y}}}^{(k)}\left(\bold{1}_{\{\vert \Delta_{\mathcal{P}_T}^n \bold{Y}^{(k)} \vert \le \nu_{\mathcal{P}_T}^{(k)}, \vert\Delta_{\mathcal{P}_T}^n \bold{\bold{\tilde{Y}}}^{(k)} \vert > 2\nu_{\mathcal{P}_T}^{(k)}\}}\right) \vert >0 \right) = o(1), \quad \text{ as } d\rightarrow \infty.
    \end{equation*}
\end{lemma}
\begin{proof}
    We note that 
\begin{equation}\label{BoundFirstTerm}
    \begin{aligned}
    &\mathbb{P}\left(\dfrac{1}{T}\sum_{n=0}^{N_T-1} \sum_{k=1}^d\sum_{j=1}^d \vert\bold{Y}_{s_n^{N_T}}^{(j)} \Delta_{\mathcal{P}_T}^n \bold{\bold{\tilde{Y}}}^{(k)}\left(\bold{1}_{\{\vert \Delta_{\mathcal{P}_T}^n \bold{Y}^{(k)} \vert \le \nu_{\mathcal{P}_T}^{(k)}, \vert\Delta_{\mathcal{P}_T}^n \bold{\bold{\tilde{Y}}}^{(k)} \vert > 2\nu_{\mathcal{P}_T}^{(k)}\}}\right) \vert >0 \right) \\& =\mathbb{P}\left(\bigcup_{n=0}^{N_T-1} \bigcup_{k=1}^d \{ \vert \Delta_{\mathcal{P}_T}^n \bold{Y}^{(k)} \vert \le \nu_{\mathcal{P}_T}^{(k)},\vert \Delta_{\mathcal{P}_T}^n \bold{\bold{\tilde{Y}}}^{(k)} \vert > 2\nu_{\mathcal{P}_T}^{(k)}\} \right)
    \\& \le \sum_{n=0}^{N_T-1} \sum_{k=1}^d  \mathbb{P}\left(\{ \vert \Delta_{\mathcal{P}_T}^n \bold{Y}^{(k)} \vert \le \nu_{\mathcal{P}_T}^{(k)},\vert \Delta_{\mathcal{P}_T}^n \bold{\bold{\tilde{Y}}}^{(k)} \vert > 2\nu_{\mathcal{P}_T}^{(k)}\} \right).
    \end{aligned}
\end{equation}
We now note, that when $\left|\Delta_{\mathcal{P}_T}^n \tilde{\bold{Y}}^{(k)}\right|>2 \nu_{\mathcal{P}_T}^{(k)}$ then with high probability $\vert\Delta_{\mathcal{P}_T}^n \bold{U}^{(k)}\vert >0$. Indeed, we remember that $\Delta_{\mathcal{P}_T}^n\tilde{\bold{Y}}^{(k)} = -\sum_{l=1}^d\int_{s^{\mathcal{P}_T }_{n}}^{s^{\mathcal{P}_T }_{n+1}}[Q_0]_{kl}\bold{Y}_s^{(l)} ds + \sum_{l=1}^d \Sigma^{1/2}_{kl}\Delta_{\mathcal{P}_T}^n \bold{W}^{(l)} +\Delta_{\mathcal{P}_T}^n \bold{U}^{(k)} $. Hence, 
\begin{equation*}
    \begin{aligned}
        &\sum_{n=0}^{N_T-1} \sum_{k=1}^d\mathbb{P}\left(\left|\Delta_{\mathcal{P}_T}^n  \tilde{\bold{Y}}^{(k)}\right|>2 \nu_{\mathcal{P}_T}^{(k)}, \vert\Delta_{\mathcal{P}_T}^n \bold{U}^{(k)}\vert=0\right) \\&= \sum_{n=0}^{N_T-1} \sum_{k=1}^d\mathbb{P}\left(\left|-\sum_{l=1}^d\int_{s^{N_T}_{n}}^{s^{N_T}_{n+1}}[Q_0]_{kl}\bold{Y}_s^{(l)} ds + \sum_{l=1}^d \Sigma^{1/2}_{kl}\Delta_{\mathcal{P}_T}^n \bold{W}^{(l)}\right|>2 \nu_{\mathcal{P}_T}^{(k)}, \vert\Delta_{\mathcal{P}_T}^n \bold{U}^{(k)}\vert=0\right) \\
        &\le  \sum_{n=0}^{N_T-1} \sum_{k=1}^d\mathbb{P}\left(\left|-\sum_{l=1}^d\int_{s^{N_T}_{n}}^{s^{N_T}_{n+1}}[Q_0]_{kl}\bold{Y}_s^{(l)} ds + \sum_{l=1}^d \Sigma^{1/2}_{kl}\Delta_{\mathcal{P}_T}^n \bold{W}^{(l)}\right|>2 \nu_{\mathcal{P}_T}^{(k)}\right).
    \end{aligned}
\end{equation*}
Now, following the second assertion in \Cref{Lemma:CountBound} with $c=2$, $\delta = \beta^*$, $l^\prime=\alpha^\prime=2$, we conclude that
\begin{equation}\label{eqn:InfAct1}
    \begin{aligned}
    &\sum_{n=0}^{N_T-1}\sum_{k=1}^d\mathbb{P}\left(\left|\Delta_{\mathcal{P}_T}^n  \tilde{\bold{Y}}^{(k)}\right|>2 \nu_{\mathcal{P}_T}^{(k)}, \vert\Delta_{\mathcal{P}_T}^n \bold{U}^{(k)}\vert=0\right) \\
    &\le \sum_{n=0}^{N_T-1}\sum_{k=1}^d\mathbb{P}\left(\left|\Delta_{\mathcal{P}_T}^n  \tilde{\bold{Y}}^{(k)}\right|>2 \Delta_{\mathcal{P}_T}^{\beta^*}, \vert\Delta_{\mathcal{P}_T}^n \bold{U}^{(k)}\vert=0\right) \\
    &\le N_T d \max_{k\in\{1,\dots,d\}}\mathbb{P}\left(\left|\Delta_{\mathcal{P}_T}^1  \tilde{\bold{Y}}^{(k)}\right|>2 \Delta_{\mathcal{P}_T}^{\beta^*}, \vert\Delta_{\mathcal{P}_T}^1 \bold{U}^{(k)}\vert=0\right)\\
    &= N_T d\, O \left(d^7\Delta_{\mathcal{P}_T}^{2(1-\beta^*)}\right)\\
    &= O\left(Td^8\Delta_{\mathcal{P}_T}^{1-2\beta^*}\right)
    \\&= o(1),\quad \text{as } d \rightarrow \infty,
    \end{aligned}
\end{equation}
where we used $\nu_{\mathcal{P}_T}^{(k)}\ge\Delta_{\mathcal{P}_T}^{\beta^*}$ (since $\beta^*=\max_k\beta^{(k)}$ and $\Delta_{\mathcal{P}_T}<1$ for large enough $d$) in the first inequality, stationarity of the increments in the second, $N_T=T/\Delta_{\mathcal{P}_T}$ in the second to last equality, and \Cref{Ass.InfiniteAct} in the last.
We now note, that on $\left\{\left|\Delta_{\mathcal{P}_T}^n \bold{Y}^{(k)}\right| \leq \nu_{\mathcal{P}_T}^{(k)},\left|\Delta_{\mathcal{P}_T}^n \tilde{\bold{Y}}^{(k)}\right|>2 \nu_{\mathcal{P}_T}^{(k)}\right\}$, it holds that 
\begin{equation*}
    \begin{aligned}
    \left|\Delta_{\mathcal{P}_T}^n  \tilde{\bold{Y}}^{(k)}\right| &= \left|\Delta_{\mathcal{P}_T}^n  \bold{Y}^{(k)} -\Delta_{\mathcal{P}_T}^n  \bold{M}^{(k)} \right| \\
    \Rightarrow \left|\Delta_{\mathcal{P}_T}^n  \bold{Y}^{(k)} \right| +\left|\Delta_{\mathcal{P}_T}^n  \bold{M}^{(k)} \right| &> 2 \nu_{\mathcal{P}_T}^{(k)}
    \\
     \iff 
  \left|\Delta_{\mathcal{P}_T}^n  \bold{M}^{(k)} \right| &> 2 \nu_{\mathcal{P}_T}^{(k)} -  \left|\Delta_{\mathcal{P}_T}^n  \bold{Y}^{(k)} \right| \ge \nu_{\mathcal{P}_T}^{(k)}.
    \end{aligned}
\end{equation*}
And so 
\begin{equation}
    \left\{\left|\Delta_{\mathcal{P}_T}^n \bold{Y}^{(k)}\right| \leq \nu_{\mathcal{P}_T}^{(k)},\left|\Delta_{\mathcal{P}_T}^n \tilde{\bold{Y}}^{(k)}\right|>2 \nu_{\mathcal{P}_T}^{(k)}\right\} \subset \left\{\left|\Delta_{\mathcal{P}_T}^n  \bold{M}^{(k)} \right|>\nu_{\mathcal{P}_T}^{(k)}  \right\}.
\end{equation}
We thus have, that
\begin{equation}\label{eqn:InfAct2}
    \begin{aligned}
    &\left\{\left|\Delta_{\mathcal{P}_T}^n \bold{Y}^{(k)}\right| \leq \nu_{\mathcal{P}_T}^{(k)},\left|\Delta_{\mathcal{P}_T}^n \tilde{\bold{Y}}^{(k)}\right|>2 \nu_{\mathcal{P}_T}^{(k)}\right\}  = \left\{\left|\Delta_{\mathcal{P}_T}^n \bold{Y}^{(k)}\right| \leq \nu_{\mathcal{P}_T}^{(k)},\left|\Delta_{\mathcal{P}_T}^n \tilde{\bold{Y}}^{(k)}\right|>2 \nu_{\mathcal{P}_T}^{(k)},\left|\Delta_{\mathcal{P}_T}^n  \bold{M}^{(k)}\right| >\nu_{\mathcal{P}_T}^{(k)} \right\} \\
 & = \left\{\left|\Delta_{\mathcal{P}_T}^n \bold{Y}^{(k)}\right| \leq \nu_{\mathcal{P}_T}^{(k)},\left|\Delta_{\mathcal{P}_T}^n \tilde{\bold{Y}}^{(k)}\right|>2 \nu_{\mathcal{P}_T}^{(k)},\left|\Delta_{\mathcal{P}_T}^n  \bold{M}^{(k)}\right| >\nu_{\mathcal{P}_T}^{(k)} ,\vert\Delta_{\mathcal{P}_T}^n \bold{U}^{(k)}\vert=0 \right\}\\& \bigcup
\left\{\left|\Delta_{\mathcal{P}_T}^n \bold{Y}^{(k)}\right| \leq \nu_{\mathcal{P}_T}^{(k)},\left|\Delta_{\mathcal{P}_T}^n \tilde{\bold{Y}}^{(k)}\right|>2 \nu_{\mathcal{P}_T}^{(k)},\left|\Delta_{\mathcal{P}_T}^n  \bold{M}^{(k)} \right|>\nu_{\mathcal{P}_T}^{(k)},\vert\Delta_{\mathcal{P}_T}^n \bold{U}^{(k)}\vert\ne 0 \right\}.
\end{aligned}
\end{equation}
We have following \eqref{eqn:InfAct2} that 
\begin{equation}\label{eqn:InfAct3}
    \begin{aligned}
    &\mathbb{P}\left(\{ \vert \Delta_{\mathcal{P}_T}^n \bold{Y}^{(k)} \vert \le \nu_{\mathcal{P}_T}^{(k)},\vert \Delta_{\mathcal{P}_T}^n \bold{\bold{\tilde{Y}}}^{(k)} \vert > 2\nu_{\mathcal{P}_T}^{(k)}\} \right)\\ &\le \mathbb{P}\left(\left\{\left|\Delta_{\mathcal{P}_T}^n \bold{Y}^{(k)}\right| \leq \nu_{\mathcal{P}_T}^{(k)},\left|\Delta_{\mathcal{P}_T}^n \tilde{\bold{Y}}^{(k)}\right|>2 \nu_{\mathcal{P}_T}^{(k)},\left|\Delta_{\mathcal{P}_T}^n  \bold{M}^{(k)}\right| >\nu_{\mathcal{P}_T}^{(k)} ,\vert\Delta_{\mathcal{P}_T}^n \bold{U}^{(k)}\vert=0\right\}\right)
    \\ &+\mathbb{P}\left(\left\{\left|\Delta_{\mathcal{P}_T}^n \bold{Y}^{(k)}\right| \leq \nu_{\mathcal{P}_T}^{(k)},\left|\Delta_{\mathcal{P}_T}^n \tilde{\bold{Y}}^{(k)}\right|>2 \nu_{\mathcal{P}_T}^{(k)},\left|\Delta_{\mathcal{P}_T}^n  \bold{M}^{(k)}\right| >\nu_{\mathcal{P}_T}^{(k)} ,\vert\Delta_{\mathcal{P}_T}^n \bold{U}^{(k)}\vert \ne 0\right\}\right)\\
     &\le \mathbb{P}\left(\left|\Delta_{\mathcal{P}_T}^n  \bold{M}^{(k)}\right| >\nu_{\mathcal{P}_T}^{(k)} ,\vert\Delta_{\mathcal{P}_T}^n \bold{U}^{(k)}\vert\ne0\right) +\mathbb{P}\left(\left|\Delta_{\mathcal{P}_T}^n  \tilde{\bold{Y}}^{(k)}\right|>2 \nu_{\mathcal{P}_T}^{(k)}, \vert\Delta_{\mathcal{P}_T}^n \bold{U}^{(k)}\vert=0\right) \\
     &= \mathbb{P}\left(\left|\Delta_{\mathcal{P}_T}^n  \bold{M}^{(k)}\right| >\nu_{\mathcal{P}_T}^{(k)} \right)\mathbb{P}\left(\vert\Delta_{\mathcal{P}_T}^n \bold{U}^{(k)}\vert\ne0\right) + \mathbb{P}\left(\left|\Delta_{\mathcal{P}_T}^n  \tilde{\bold{Y}}^{(k)}\right|>2 \nu_{\mathcal{P}_T}^{(k)}, \vert\Delta_{\mathcal{P}_T}^n \bold{U}^{(k)}\vert=0\right) ,
    \end{aligned}
\end{equation}
where we used independence of $\mathbb{M}$ and $\mathbb{U}$ in the last equality. 
Following \Cref{PropertiesOfIntroducedStochasticProcesses} with $f(\Delta_{\mathcal{P}_T}) = \Delta_{\mathcal{P}_T}^{\beta^*}$, we thus conclude that 
\begin{equation}\label{eqn:InfAct4}
    \begin{aligned}
&\max_{k}\mathbb{P}\left(\left|\Delta_{\mathcal{P}_T}^n  \bold{M}^{(k)}\right| >\nu_{\mathcal{P}_T}^{(k)} \right)\max_{k}\mathbb{P}\left(\vert\Delta_{\mathcal{P}_T}^n \bold{U}^{(k)}\vert\ne0\right) \\
&\le \max_{k}\mathbb{P}\left(\left|\Delta_{\mathcal{P}_T}^n  \bold{M}^{(k)}\right| >\Delta_{\mathcal{P}_T}^{\beta^*} \right)\max_{k}\mathbb{P}\left(\vert\Delta_{\mathcal{P}_T}^n \bold{U}^{(k)}\vert\ne0\right) \\
&=O\left(d\Delta_{\mathcal{P}_T}^{2-2\beta^*} \right),\quad \text{as } d\rightarrow\infty,
    \end{aligned}
\end{equation}
where the inequality uses $\nu_{\mathcal{P}_T}^{(k)}=\Delta_{\mathcal{P}_T}^{\beta^{(k)}}\ge\Delta_{\mathcal{P}_T}^{\beta^*}$ (since $\beta^*=\max_k\beta^{(k)}$ and $\Delta_{\mathcal{P}_T}<1$).
As such, using \eqref{BoundFirstTerm},\eqref{eqn:InfAct1} and \eqref{eqn:InfAct3}, we have that
\begin{equation}\label{eqn:InfAct5}
    \begin{aligned}
           &\mathbb{P}\biggl(  \dfrac{1}{T}\sum_{n=0}^{N_T-1} \sum_{k=1}^d\sum_{j=1}^d \vert\bold{Y}_{s_n^{N_T}}^{(j)} \Delta_{\mathcal{P}_T}^n \bold{\bold{\tilde{Y}}}^{(k)}\left(\bold{1}_{\{\vert \Delta_{\mathcal{P}_T}^n \bold{Y}^{(k)} \vert \le \nu_{\mathcal{P}_T}^{(k)}, \vert\Delta_{\mathcal{P}_T}^n \bold{\bold{\tilde{Y}}}^{(k)} \vert > 2\nu_{\mathcal{P}_T}^{(k)}\}}\right) \vert >\Delta_{\mathcal{P}_T}^\delta/2 \biggr)\\
    &\le\sum_{n=0}^{N_T-1} \sum_{k=1}^d  \mathbb{P}\left(\{ \vert \Delta_{\mathcal{P}_T}^n \bold{Y}^{(k)} \vert \le \nu_{\mathcal{P}_T}^{(k)},\vert \Delta_{\mathcal{P}_T}^n \bold{\bold{\tilde{Y}}}^{(k)} \vert > 2\nu_{\mathcal{P}_T}^{(k)}\} \right)\\
    &\le\sum_{n=0}^{N_T-1} \sum_{k=1}^d \biggl(\max_{k\in\{1,\dots,d\}}\mathbb{P}\left(\left|\Delta_{\mathcal{P}_T}^n  \bold{M}^{(k)}\right| >\nu_{\mathcal{P}_T}^{(k)} \right)\mathbb{P}\left(\vert\Delta_{\mathcal{P}_T}^n \bold{U}^{(k)}\vert\ne0\right) \\&+ \max_{k\in\{1,\dots,d\}}\mathbb{P}\left(\left|\Delta_{\mathcal{P}_T}^n  \tilde{\bold{Y}}^{(k)}\right|>2 \nu_{\mathcal{P}_T}^{(k)}, \vert\Delta_{\mathcal{P}_T}^n \bold{U}^{(k)}\vert=0\right)\biggr)\\
    &=O \left(Td^8\Delta_{\mathcal{P}_T}^{1-2\beta^*}\right)
    \\&= o(1),\quad \text{as }d \rightarrow \infty.
    \end{aligned}
\end{equation}
\end{proof}
\begin{lemma}\label{SecondBoundFirstLemma}
   Grant \Cref{Ass:Estimation,Ass.H,Ass.High-FrequencySampling,Ass.GeneralBothCases,Ass.InfiniteAct}, and let the notation be as in \Cref{LevyItoImpl}. Then
    \begin{equation*}
        \mathbb{P}\biggl(\bigcup_{n=0}^{N_T-1}\bigcup_{k=1}^d \{\vert\Delta_{\mathcal{P}_T}^n \bold{\bold{\tilde{Y}}}^{(k)} \vert \le 2\nu_{\mathcal{P}_T}^{(k)},\vert\Delta_{\mathcal{P}_T}^n \bold{U}^{(k)}\vert>0\}\biggr) = o(1),\quad \text{as } d\rightarrow \infty.
    \end{equation*}
\end{lemma}
\begin{proof}
    We first note that
\begin{equation*}
    \begin{aligned}
    &\mathbb{P}\left( \{\left|\Delta_{\mathcal{P}_T}^n  \tilde{\bold{Y}}^{(k)}\right|\le2 \nu_{\mathcal{P}_T}^{(k)}, \vert\Delta_{\mathcal{P}_T}^n \bold{U}^{(k)}\vert>0\}\right) \\&\le  \mathbb{P}\left( \{\left|\Delta_{\mathcal{P}_T}^n  \tilde{\bold{Y}}^{(k)}\right|\le2 \nu_{\mathcal{P}_T}^{(k)}, \vert\Delta_{\mathcal{P}_T}^n \bold{N}= 1\}\right) \\&+ \mathbb{P}\left( \{\left|\Delta_{\mathcal{P}_T}^n  \tilde{\bold{Y}}^{(k)}\right|\le2 \nu_{\mathcal{P}_T}^{(k)}, \vert\Delta_{\mathcal{P}_T}^n \bold{N}> 1\}\right)\\
    &\le  \mathbb{P}\left( \{\left|\Delta_{\mathcal{P}_T}^n  \tilde{\bold{Y}}^{(k)}\right|\le2 \nu_{\mathcal{P}_T}^{(k)}, \vert\Delta_{\mathcal{P}_T}^n \bold{N}= 1\}\right) \\&+ \mathbb{P}\left(\Delta_{\mathcal{P}_T}^n \bold{N} > 1\right) \\
    &= \mathbb{P}\left( \{\left|\Delta_{\mathcal{P}_T}^n  \tilde{\bold{Y}}^{(k)}\right|\le2 \nu_{\mathcal{P}_T}^{(k)}, \vert\Delta_{\mathcal{P}_T}^n \bold{N} = 1\}\right) + O(d^2\Delta_{\mathcal{P}_T}^2),
    \end{aligned}
\end{equation*}
where $\mathbb{N}$ is the Poisson process associated to the compound Poisson Process $\mathbb{U}$ and in the last equality, we used that the intensity $\lambda$ of $\mathbb{N}$ is given by $F\left(\{x:\Vert x \Vert_2>1\}\right)$, which is $O(d)$ by assumption. Thus 
\begin{equation*}
    \begin{aligned}
     \mathbb{P}\left(\Delta_{\mathcal{P}_T}^n \bold{N} > 1\right) &= \sum_{j=2}^{\infty} \dfrac{\left(\lambda\Delta_{\mathcal{P}_T}\right)^j e^{-\lambda\Delta_{\mathcal{P}_T}} }{j!}\\
     &=\left(\lambda\Delta_{\mathcal{P}_T}\right)^2e^{-\lambda\Delta_{\mathcal{P}_T}}\sum_{j=2}^{\infty} \dfrac{\left(\lambda\Delta_{\mathcal{P}_T}\right)^{j-2}  }{j!}\\ &\le \left(\lambda\Delta_{\mathcal{P}_T}\right)^2e^{-\lambda\Delta_{\mathcal{P}_T}}\sum_{j=2}^{\infty} \dfrac{\left(\lambda\Delta_{\mathcal{P}_T}\right)^{j-2}  }{(j-2)!}\\
     &= \left(\lambda\Delta_{\mathcal{P}_T}\right)^2 \\
     &=O\left(\left(d\Delta_{\mathcal{P}_T}\right)^2\right),\quad \text{ as } d \rightarrow \infty.
     \end{aligned}
\end{equation*}
Since $\mathbb{U}$ is a compound Poisson process, with jumps with probability measure $\bar{F}(dx)=F|_{x:\Vert x \Vert>1}(dx)/\lambda$ where $\lambda = F(\{x:\Vert x \Vert_2>1 \})$, we have that on $\{\Delta_{\mathcal{P}_T}^n\bold{N} = 1 \}$, it holds that $\Delta_{\mathcal{P}_T}^n \bold{U}^{(k)} = \sum_{j=1}^{\Delta_{\mathcal{P}_T}^n \bold{N}} \bold{Z}_i^{(k)} =\bold{Z}_1^{(k)}$. Clearly,
\begin{equation*}
    \begin{aligned}
    &\{\left|\Delta_{\mathcal{P}_T}^n  \tilde{\bold{Y}}^{(k)}\right|\le2 \nu_{\mathcal{P}_T}^{(k)}, \Delta_{\mathcal{P}_T}^n \bold{N}= 1\} \\&=\{\left|\Delta_{\mathcal{P}_T}^n  \tilde{\bold{Y}}^{(k)}\right|\le2 \nu_{\mathcal{P}_T}^{(k)}, \vert \bold{Z}_1^{(k)} \vert> 4\nu_{\mathcal{P}_T}^{(k)} ,\Delta_{\mathcal{P}_T}^n \bold{N}= 1\}\\&\bigcup \{\left|\Delta_{\mathcal{P}_T}^n  \tilde{\bold{Y}}^{(k)}\right|\le2 \nu_{\mathcal{P}_T}^{(k)}, \vert\bold{Z}_1^{(k)}\vert\le 4\nu_{\mathcal{P}_T}^{(k)} ,\Delta_{\mathcal{P}_T}^n \bold{N}= 1\}.
    \end{aligned}
\end{equation*}
On $\{\left|\Delta_{\mathcal{P}_T}^n  \tilde{\bold{Y}}^{(k)}\right|\le2 \nu_{\mathcal{P}_T}^{(k)}, \vert\bold{Z}_1^{(k)}\vert> 4\nu_{\mathcal{P}_T}^{(k)} ,\Delta_{\mathcal{P}_T}^n \bold{N}= 1\}$, we claim that necessarily, $\{ \vert\Delta_{\mathcal{P}_T}^n \bold{D}^{(k)} +\left(\Sigma^{1/2}\Delta_{\mathcal{P}_T}^n \bold{W}\right)^{(k)} \vert > 2 \nu_{\mathcal{P}_T}^{(k)} \}$.  Indeed, it follows that
\begin{equation*}
    \begin{aligned}
    \left|\Delta_{\mathcal{P}_T}^n  \tilde{\bold{Y}}^{(k)}\right| &= \left|\Delta_{\mathcal{P}_T}^n \bold{D}^{(k)}+\left(\Sigma^{1/2}\Delta_{\mathcal{P}_T}^n \bold{W}\right)^{(k)} + (\Delta_{\mathcal{P}_T}^n\bold{U})^{(k)} \right|\\ &= \left|\Delta_{\mathcal{P}_T}^n \bold{D}^{(k)}+\left(\Sigma^{1/2}\Delta_{\mathcal{P}_T}^n \bold{W}\right)^{(k)} + \bold{Z}_1^{(k)}\right|\le2 \nu_{\mathcal{P}_T}^{(k)} \\&\iff \Delta_{\mathcal{P}_T}^n \bold{D}^{(k)}+\left(\Sigma^{1/2}\Delta_{\mathcal{P}_T}^n \bold{W}\right)^{(k)} \in \left[-\left(\bold{Z}_1^{(k)} +2 \nu_{\mathcal{P}_T}^{(k)} \right);-\left(\bold{Z}_1^{(k)} -2 \nu_{\mathcal{P}_T}^{(k)} \right)\right].
    \end{aligned}
\end{equation*}
Since $\vert\bold{Z}_1^{(k)}\vert>4\nu_{\mathcal{P}_T}^{(k)}$, every point of this interval has absolute value at least $\vert\bold{Z}_1^{(k)}\vert-2\nu_{\mathcal{P}_T}^{(k)}>2\nu_{\mathcal{P}_T}^{(k)}$, so $\vert\Delta_{\mathcal{P}_T}^n \bold{D}^{(k)}+(\Sigma^{1/2}\Delta_{\mathcal{P}_T}^n \bold{W})^{(k)}\vert>2\nu_{\mathcal{P}_T}^{(k)}$, and the claim follows.
In turn, we find that
\begin{equation*}
    \begin{aligned}
        &\max_{k\in\{1,\dots,d\}}\mathbb{P}\left(\{\left|\Delta_{\mathcal{P}_T}^n  \tilde{\bold{Y}}^{(k)}\right|\le2 \nu_{\mathcal{P}_T}^{(k)}, \vert \bold{Z}_1^{(k)} \vert> 4\nu_{\mathcal{P}_T}^{(k)} ,\Delta_{\mathcal{P}_T}^n \bold{N}= 1\}\right) \\ 
        &\le \mathbb{P}\left(\{ \vert\Delta_{\mathcal{P}_T}^n \bold{D}^{(k)} +\left(\Sigma^{1/2}\Delta_{\mathcal{P}_T}^n \bold{W}\right)^{(k)} \vert > 2 \nu_{\mathcal{P}_T}^{(k)}\}\right) \\
        &\le \mathbb{P}\left(\{ \vert\Delta_{\mathcal{P}_T}^n \bold{D}^{(k)} +\left(\Sigma^{1/2}\Delta_{\mathcal{P}_T}^n \bold{W}\right)^{(k)} \vert > 2 \Delta_{\mathcal{P}_T}^{\beta^*}\}\right) \\
        &= O\left(d^7 \Delta_{\mathcal{P}_T}^{2(1-\beta^*)}\right), \quad \text{ as } d\rightarrow \infty,
    \end{aligned}
\end{equation*}
where we used $\nu_{\mathcal{P}_T}^{(k)}\ge\Delta_{\mathcal{P}_T}^{\beta^*}$ (since $\beta^*=\max_k\beta^{(k)}$ and $\Delta_{\mathcal{P}_T}<1$) and \Cref{Lemma:CountBound} with $c=2$, $l^{\prime}=2$ and $\delta = \beta^*$ in the last inequality.
Further,
\begin{equation*}
    \begin{aligned}
        &\max_{k\in\{1,\dots,d\}}\mathbb{P}\left(\{\left|\Delta_{\mathcal{P}_T}^n  \tilde{\bold{Y}}^{(k)}\right|\le2 \nu_{\mathcal{P}_T}^{(k)}, \vert\bold{Z}_1^{(k)}\vert\le 4\nu_{\mathcal{P}_T}^{(k)} ,\Delta_{\mathcal{P}_T}^n \bold{N}= 1\}\right)
        \\&\le\mathbb{P}\left(\{ \vert\bold{Z}_1^{(k)}\vert\le 4\nu_{\mathcal{P}_T}^{(k)} ,\Delta_{\mathcal{P}_T}^n \bold{N}= 1\}\right) \\
        &\le \lambda\Delta_{\mathcal{P}_T}\bar{F}^{(k)}\left(- 4\Delta_{\mathcal{P}_T}^{\beta^{(k)}}, 4\Delta_{\mathcal{P}_T}^{\beta^{(k)}}\right).
    \end{aligned}
\end{equation*}
Using the independence of $\Delta_{\mathcal{P}_T}^n \bold{N}$ and $\bold{Z}_1^{(k)}$.
Consequently,
\begin{equation*}
    \begin{aligned}
      &\mathbb{P}\biggl(\bigcup_{n=0}^{N_T-1}\bigcup_{k=1}^d \{\vert\Delta_{\mathcal{P}_T}^n \bold{\bold{\tilde{Y}}}^{(k)} \vert \le 2\nu_{\mathcal{P}_T}^{(k)},\vert\Delta_{\mathcal{P}_T}^n \bold{U}^{(k)}\vert>0\}\biggr) \\&\le   \sum_{n=0}^{N_T-1}\sum_{k=1}^d\max_{k}\mathbb{P}\biggl(\{\vert\Delta_{\mathcal{P}_T}^n \bold{\bold{\tilde{Y}}}^{(k)} \vert \le 2\nu_{\mathcal{P}_T}^{(k)},\vert\Delta_{\mathcal{P}_T}^n \bold{U}^{(k)}\vert>0\}\biggr) \\
      &= TdO\left(d^7 \Delta_{\mathcal{P}_T}^{1-2\beta^*} \vee d\max_{k}\bar{F}^{(k)}\left(- 4\Delta_{\mathcal{P}_T}^{\beta^{(k)}}, 4\Delta_{\mathcal{P}_T}^{\beta^{(k)}}\right) \right) \\
      &= o(1),\quad \text{as } d \rightarrow \infty. 
      \end{aligned}
\end{equation*}
\end{proof}
\begin{lemma}\label{SecondBoundS_2D}
   Grant \Cref{Ass:Estimation,Ass.H,Ass.High-FrequencySampling,Ass.GeneralBothCases,Ass.InfiniteAct}, and let the notation be as in \Cref{LevyItoImpl}. Then for $C>0$ and $\delta$ satisfying \Cref{Ass.InfiniteAct}
    \begin{multline*}
        %\begin{aligned}
        \mathbb{P}\biggl(\dfrac{1}{T}\sum_{n=0}^{N_T-1} \sum_{k=1}^d\sum_{j=1}^d \vert\bold{Y}_{s_n^{N_T}}^{(j)}\Delta_{\mathcal{P}_T}^n \bold{\bold{D}}^{(k)}\vert\biggl(\bold{1}_{\{\vert \Delta_{\mathcal{P}_T}^n \bold{Y}^{(k)} \vert > \nu_{\mathcal{P}_T}^{(k)}, \vert\Delta_{\mathcal{P}_T}^n \bold{\bold{\tilde{Y}}}^{(k)} \vert \le 2\nu_{\mathcal{P}_T}^{(k)},\vert\Delta_{\mathcal{P}_T}^n \bold{U}^{(k)}\vert=0\}}\biggr) \vert>C\Delta_{\mathcal{P}_T}^\delta\biggr) \\= o(1) \text{as } d \rightarrow \infty.
        %\end{aligned}
    \end{multline*}
\end{lemma}
\begin{proof}
    We notice that 
\begin{equation*}
     \begin{aligned}
     &\mathbb{E}\left(\dfrac{1}{T}\sum_{n=0}^{N_T-1} \sum_{k=1}^d\sum_{j=1}^d \vert\bold{Y}_{s_n^{N_T}}^{(j)} \left(\Delta_{\mathcal{P}_T}^n \bold{\bold{D}}^{(k)}\right)\left(\bold{1}_{\{\vert \Delta_{\mathcal{P}_T}^n \bold{Y}^{(k)} \vert > \nu_{\mathcal{P}_T}^{(k)}, \vert\Delta_{\mathcal{P}_T}^n \bold{\bold{\tilde{Y}}}^{(k)} \vert \le 2\nu_{\mathcal{P}_T}^{(k)}\},\vert\Delta_{\mathcal{P}_T}^n \bold{U}^{(k)}\vert=0}\right) \vert \right)\\
     &=\dfrac{1}{T}\sum_{n=0}^{N_T-1} \sum_{k=1}^d\sum_{j=1}^d\mathbb{E}\left(\vert\bold{Y}_{s_n^{N_T}}^{(j)} \left(\Delta_{\mathcal{P}_T}^n \bold{\bold{D}}^{(k)}\right)\left(\bold{1}_{\{\vert \Delta_{\mathcal{P}_T}^n \bold{Y}^{(k)} \vert > \nu_{\mathcal{P}_T}^{(k)}, \vert\Delta_{\mathcal{P}_T}^n \bold{\bold{\tilde{Y}}}^{(k)} \vert \le 2\nu_{\mathcal{P}_T}^{(k)}\},\vert\Delta_{\mathcal{P}_T}^n \bold{U}^{(k)}\vert=0}\right) \vert \right).
     \end{aligned}
\end{equation*}
Following the triangle and Cauchy-Schwarz inequalities, we have
\begin{equation*}
    \begin{aligned}
        &\mathbb{E}\left(\vert\bold{Y}_{s_n^{N_T}}^{(j)} \left(\Delta_{\mathcal{P}_T}^n \bold{\bold{D}}^{(k)}\right)\left(\bold{1}_{\{\vert \Delta_{\mathcal{P}_T}^n \bold{Y}^{(k)} \vert > \nu_{\mathcal{P}_T}^{(k)}, \vert\Delta_{\mathcal{P}_T}^n \bold{\bold{\tilde{Y}}}^{(k)} \vert \le 2\nu_{\mathcal{P}_T}^{(k)}\},\vert\Delta_{\mathcal{P}_T}^n \bold{U}^{(k)}\vert=0}\right) \vert \right)\\&= \mathbb{E}\left(\vert\bold{Y}_{s_n^{N_T}}^{(j)} \left(\sum_{l=1}^d [Q_0]_{kl}\int_{s_n^{N_T}}^{s_{n+1}^{N_T}} \bold{Y}_s^{(l)} ds\right)\left(\bold{1}_{\{\vert \Delta_{\mathcal{P}_T}^n \bold{Y}^{(k)} \vert > \nu_{\mathcal{P}_T}^{(k)}, \vert\Delta_{\mathcal{P}_T}^n \bold{\bold{\tilde{Y}}}^{(k)} \vert \le 2\nu_{\mathcal{P}_T}^{(k)}\},\vert\Delta_{\mathcal{P}_T}^n \bold{U}^{(k)}\vert=0}\right) \vert \right)\\
        &\le \sum_{l=1}^d\vert [Q_0]_{kl} \vert \mathbb{E}\left(\vert\bold{Y}_{s_n^{N_T}}^{(j)} \vert\left(\int_{s_n^{N_T}}^{s_{n+1}^{N_T}} \vert\bold{Y}_s^{(l)}\vert ds\right)\left(\bold{1}_{\{\vert \Delta_{\mathcal{P}_T}^n \bold{Y}^{(k)} \vert > \nu_{\mathcal{P}_T}^{(k)}, \vert\Delta_{\mathcal{P}_T}^n \bold{\bold{\tilde{Y}}}^{(k)} \vert \le 2\nu_{\mathcal{P}_T}^{(k)}\},\vert\Delta_{\mathcal{P}_T}^n \bold{U}^{(k)}\vert=0}\right) \vert \right)\\ 
          &\le \sum_{l=1}^d\vert [Q_0]_{kl} \vert \mathbb{E}\left(\vert\bold{Y}_{s_n^{N_T}}^{(j)} \vert^2\left(\int_{s_n^{N_T}}^{s_{n+1}^{N_T}} \vert\bold{Y}_s^{(l)}\vert ds\right)^2\right)^{1/2}\\&\cdot\mathbb{P}\left(\{\vert \Delta_{\mathcal{P}_T}^n \bold{Y}^{(k)} \vert > \nu_{\mathcal{P}_T}^{(k)}, \vert\Delta_{\mathcal{P}_T}^n \bold{\bold{\tilde{Y}}}^{(k)} \vert \le 2\nu_{\mathcal{P}_T}^{(k)},\vert\Delta_{\mathcal{P}_T}^n \bold{U}^{(k)}\vert=0\}\right)^{1/2}.
    \end{aligned}
\end{equation*}
Notice that following Jensen's and Hölder's inequalities together with stationarity of $\mathbb{Y}$, we get that
\begin{equation*}
    \begin{aligned}
        \mathbb{E}\left(\vert\bold{Y}_{s_n^{N_T}}^{(j)} \vert^2\left(\int_{s_n^{N_T}}^{s_{n+1}^{N_T}} \vert\bold{Y}_s^{(l)}\vert ds\right)^2\right) &\le \mathbb{E}\left(\vert\bold{Y}_{s_n^{N_T}}^{(j)} \vert^4\right)^{1/2} \mathbb{E}\left(\left(\int_{s_n^{N_T}}^{s_{n+1}^{N_T}} \vert\bold{Y}_s^{(l)}\vert ds\right)^4\right)^{1/2}\\
        & \le \mathbb{E}\left(\vert\bold{Y}_{s_n^{N_T}}^{(j)} \vert^4 \right)^{1/2}\mathbb{E}\left(\dfrac{\Delta_{\mathcal{P}_T}^4}{s_{n+1}^{N_T}-s_n^{N_T}}\int_{s_n^{N_T}}^{s_{n+1}^{N_T}} \vert\bold{Y}_s^{(l)}\vert^4 ds\right)^{1/2}
        \\&\le\Delta_{\mathcal{P}_T}^{2}\max_{l\in\{1,\dots,d\}}\mathbb{E}\left(\vert\bold{Y}_0^{(l)}\vert^4\right).
    \end{aligned}
\end{equation*}
We note that the right hand side is also an upper bound, if we took the maximum over $j$ and $l$.
By the triangle inequality, we have
\begin{equation}\label{SetinclusionS2}
    \begin{aligned}
        &\{\vert \Delta_{\mathcal{P}_T}^n \bold{Y}^{(k)} \vert > \nu_{\mathcal{P}_T}^{(k)}, \vert\Delta_{\mathcal{P}_T}^n \bold{\bold{\tilde{Y}}}^{(k)} \vert \le 2\nu_{\mathcal{P}_T}^{(k)},\vert\Delta_{\mathcal{P}_T}^n \bold{U}^{(k)}\vert=0\}\\
        &\subset  \{\vert \Delta_{\mathcal{P}_T}^n \bold{Y}^{(k)} \vert > \nu_{\mathcal{P}_T}^{(k)},\vert\Delta_{\mathcal{P}_T}^n \bold{U}^{(k)}\vert=0\}\\
        &\subset   \{\vert \Delta_{\mathcal{P}_T}^n \bold{D}^{(k)} + \left(\Sigma^{1/2}\Delta_{\mathcal{P}_T}^n\bold{W}^{(k)}\right)+  \Delta_{\mathcal{P}_T}^n\bold{M}^{(k)}\vert > \nu_{\mathcal{P}_T}^{(k)}\}
        \\& \subset \{\vert \Delta_{\mathcal{P}_T}^n \bold{D}^{(k)} + \left(\Sigma^{1/2}\Delta_{\mathcal{P}_T}^n\bold{W}^{(k)}\right)\vert+  \vert\Delta_{\mathcal{P}_T}^n\bold{M}^{(k)}\vert > \nu_{\mathcal{P}_T}^{(k)}\}
        \\& \subset \{\vert \Delta_{\mathcal{P}_T}^n \bold{D}^{(k)} + \left(\Sigma^{1/2}\Delta_{\mathcal{P}_T}^n\bold{W}^{(k)}\right)\vert > \nu_{\mathcal{P}_T}^{(k)}/2\}\cup\{ \vert\Delta_{\mathcal{P}_T}^n\bold{M}^{(k)}\vert > \nu_{\mathcal{P}_T}^{(k)}/2\}.
    \end{aligned}
\end{equation}
We also notice, that following \Cref{Lemma:CountBound} with $\delta = \beta^*$, $l^{\prime}=2$ and $c=1/2$, it holds that
\begin{equation*}
    \begin{aligned}
    \max_{k\in\{1,\dots,d\}}\mathbb{P}\left(\vert \Delta_{\mathcal{P}_T}^n \bold{W}^{(k)}+ \Delta_{\mathcal{P}_T}^n \bold{D}^{(k)}\vert > \nu_{\mathcal{P}_T}^{(k)}/2  \right) &\le \max_{k\in\{1,\dots,d\}}\mathbb{P}\left(\vert \Delta_{\mathcal{P}_T}^n \bold{W}^{(k)}+ \Delta_{\mathcal{P}_T}^n \bold{D}^{(k)}\vert > \tfrac{1}{2}\Delta_{\mathcal{P}_T}^{\beta^*}  \right)\\&= O\left(d^{7}\Delta_{\mathcal{P}_T}^{2(1-\beta^{*})}\right),\quad \text{as } d\rightarrow \infty.
    \end{aligned}
\end{equation*}
Also, following \Cref{PropertiesOfIntroducedStochasticProcesses}, with $f(\Delta_{\mathcal{P}_T})=\tfrac{1}{2}\Delta_{\mathcal{P}_T}^{\beta^*}$, we have that
\begin{equation*}
    \begin{aligned}
         \max_{k\in\{1,\dots,d\}}\mathbb{P}\left(\vert \Delta_{\mathcal{P}_T}^n \bold{M}^{(k)}\vert > \nu_{\mathcal{P}_T}^{(k)}/2  \right) &\le \max_{k\in\{1,\dots,d\}}\mathbb{P}\left(\vert \Delta_{\mathcal{P}_T}^n \bold{M}^{(k)}\vert > \tfrac{1}{2}\Delta_{\mathcal{P}_T}^{\beta^*}  \right)\\&= O\left(d\Delta_{\mathcal{P}_T}^{1-2\beta^{*}}\right),\quad \text{as } d\rightarrow \infty.
    \end{aligned}
\end{equation*}
In both displays we used $\nu_{\mathcal{P}_T}^{(k)}=\Delta_{\mathcal{P}_T}^{\beta^{(k)}}\ge\Delta_{\mathcal{P}_T}^{\beta^*}$ (since $\beta^*=\max_k\beta^{(k)}$ and $\Delta_{\mathcal{P}_T}<1$).
We thus conclude using \eqref{SetinclusionS2}, that 
\begin{equation*}
     \begin{aligned}
     &\mathbb{E}\left(\dfrac{1}{T}\sum_{n=0}^{N_T-1} \sum_{k=1}^d\sum_{j=1}^d\vert\bold{Y}_{s_n^{N_T}}^{(j)} \left(\Delta_{\mathcal{P}_T}^n \bold{\bold{D}}^{(k)}\right)\left(\bold{1}_{\{\vert \Delta_{\mathcal{P}_T}^n \bold{Y}^{(k)} \vert > \nu_{\mathcal{P}_T}^{(k)}, \vert\Delta_{\mathcal{P}_T}^n \bold{\bold{\tilde{Y}}}^{(k)} \vert \le 2\nu_{\mathcal{P}_T}^{(k)}\},\vert\Delta_{\mathcal{P}_T}^n \bold{U}^{(k)}\vert=0}\right) \vert \right)\\
     &=\dfrac{1}{T}\sum_{n=0}^{N_T-1} \sum_{k=1}^d\sum_{j=1}^d\mathbb{E}\left(\vert\bold{Y}_{s_n^{N_T}}^{(j)} \left(\Delta_{\mathcal{P}_T}^n \bold{\bold{D}}^{(k)}\right)\left(\bold{1}_{\{\vert \Delta_{\mathcal{P}_T}^n \bold{Y}^{(k)} \vert > \nu_{\mathcal{P}_T}^{(k)}, \vert\Delta_{\mathcal{P}_T}^n \bold{\bold{\tilde{Y}}}^{(k)} \vert \le 2\nu_{\mathcal{P}_T}^{(k)}\},\vert\Delta_{\mathcal{P}_T}^n \bold{U}^{(k)}\vert=0}\right) \vert \right) \\
     &\le   \dfrac{1}{T}\sum_{n=0}^{N_T-1} \sum_{k=1}^d\sum_{j=1}^d\sum_{l=1}^d\ \vert [Q_0]_{kl} \vert \mathbb{E}\left(\vert\bold{Y}_{s_n^{N_T}}^{(j)} \vert^2\left(\int_{s_n^{N_T}}^{s_{n+1}^{N_T}} \vert\bold{Y}_s^{(l)}\vert ds\right)^2\right)^{1/2}\\&\cdot\biggl(\mathbb{P}\left(\{\vert \Delta_{\mathcal{P}_T}^n \bold{D}^{(k)} + \left(\Sigma^{1/2}\Delta_{\mathcal{P}_T}^n\bold{W}^{(k)}\right)\vert > \nu_{\mathcal{P}_T}^{(k)}/2\}\right) \\&+ \mathbb{P}\left(\{ \vert\Delta_{\mathcal{P}_T}^n\bold{M}^{(k)}\vert > \nu_{\mathcal{P}_T}^{(k)}/2\}\right)\biggr)^{1/2} 
     \\&=O\left(\left(d^{10}\Delta_{\mathcal{P}_T}^{1-\beta^*}\vee d^{7} \Delta_{\mathcal{P}_T}^{1/2-\beta^*} \right)\right),\quad \text{as } d  \rightarrow \infty. 
     \end{aligned}
\end{equation*}
By Markov's inequality, we thus have that 
\begin{equation*}
    \begin{aligned}
     &     \mathbb{P}\biggl(\dfrac{1}{T}\sum_{n=0}^{N_T-1} \sum_{k=1}^d\sum_{j=1}^d \vert\bold{Y}_{s_n^{N_T}}^{(j)}\Delta_{\mathcal{P}_T}^n \bold{\bold{D}}^{(k)}\vert\biggl(\bold{1}_{\{\vert \Delta_{\mathcal{P}_T}^n \bold{Y}^{(k)} \vert > \nu_{\mathcal{P}_T}^{(k)}, \vert\Delta_{\mathcal{P}_T}^n \bold{\bold{\tilde{Y}}}^{(k)} \vert \le 2\nu_{\mathcal{P}_T}^{(k)},\vert\Delta_{\mathcal{P}_T}^n \bold{U}^{(k)}\vert=0\}}\biggr) \vert>C\Delta_{\mathcal{P}_T}^\delta\biggr) \\&=O\left(d^{10}\Delta_{\mathcal{P}_T}^{1-\beta^*-\delta}\vee d^{7} \Delta_{\mathcal{P}_T}^{1/2-\beta^*-\delta} \right)\\&=o(1),\quad \text{as } d \rightarrow \infty,
     \end{aligned}
\end{equation*}
since $\delta$ satisfies \Cref{Ass.InfiniteAct}.
\end{proof}
\begin{lemma}\label{SecondBoundS_2W}
   Grant \Cref{Ass:Estimation,Ass.H,Ass.High-FrequencySampling,Ass.GeneralBothCases,Ass.InfiniteAct}, and let the notation be as in \Cref{LevyItoImpl}. Then for $C>0$ and $\delta$ satisfying \Cref{Ass.InfiniteAct}
    \begin{multline*}
        \mathbb{P}\biggl(\dfrac{1}{T}\sum_{n=0}^{N_T-1} \sum_{k=1}^d\sum_{j=1}^d \vert\bold{Y}_{s_n^{N_T}}^{(j)} \left(\Sigma^{1/2}\Delta_{\mathcal{P}_T}^n \bold{W}\right)^{(k)}\vert\bold{1}_{\{\vert \Delta_{\mathcal{P}_T}^n \bold{Y}^{(k)} \vert > \nu_{\mathcal{P}_T}^{(k)}, \vert\Delta_{\mathcal{P}_T}^n \bold{\bold{\tilde{Y}}}^{(k)} \vert \le 2\nu_{\mathcal{P}_T}^{(k)},\vert\Delta_{\mathcal{P}_T}^n \bold{U}^{(k)}\vert=0\}}>C\Delta_{\mathcal{P}_T}^\delta\biggr) \\= o(1),\quad \text{as } d\rightarrow \infty.
    \end{multline*}
\end{lemma}
\begin{proof}
 Notice that
\begin{equation*}
    \begin{aligned}
    &\left\{\left|\Delta_{\mathcal{P}_T}^n  \bold{Y}^{(k)}\right| > \nu_{\mathcal{P}_T}^{(k)},\left|\Delta_{\mathcal{P}_T}^n  \tilde{\bold{Y}}^{(k)}\right|\le2 \nu_{\mathcal{P}_T}^{(k)},\vert\Delta_{\mathcal{P}_T}^n \bold{U}^{(k)}\vert=0\right\} \\
    &=\bigg\{\left|\left(\Sigma^{1/2}\Delta_{\mathcal{P}_T}^n \bold{W}\right)^{(k)}+\Delta_{\mathcal{P}_T}^n \bold{D}^{(k)}+\Delta_{\mathcal{P}_T}^n \bold{M}^{(k)}\right| > \nu_{\mathcal{P}_T}^{(k)},\\&\left|\Delta_{\mathcal{P}_T}^n \bold{D}^{(k)}+\left(\Sigma^{1/2}\Delta_{\mathcal{P}_T}^n \bold{W}\right)^{(k)}\right|\le2 \nu_{\mathcal{P}_T}^{(k)},\vert\Delta_{\mathcal{P}_T}^n \bold{U}^{(k)}\vert=0\bigg\}\\
    &\subseteq \bigg\{\left|\left(\Sigma^{1/2}\Delta_{\mathcal{P}_T}^n \bold{W}\right)^{(k)}+\Delta_{\mathcal{P}_T}^n \bold{D}^{(k)}+\Delta_{\mathcal{P}_T}^n \bold{M}^{(k)}\right| > \nu_{\mathcal{P}_T}^{(k)}\bigg\} \\
     &\subseteq \bigg\{\left|\left(\Sigma^{1/2}\Delta_{\mathcal{P}_T}^n \bold{W}\right)^{(k)}+\Delta_{\mathcal{P}_T}^n \bold{D}^{(k)}\right| > \nu_{\mathcal{P}_T}^{(k)}/2\bigg\}\bigcup \bigg\{\left|\Delta_{\mathcal{P}_T}^n \bold{M}^{(k)}\right| > \nu_{\mathcal{P}_T}^{(k)}/2\bigg\}.
    \end{aligned}
\end{equation*}
As such, 
\begin{equation}\label{eqn:LemmaBrownianPart}
    \begin{aligned}
        &\mathbb{E}\biggl(\dfrac{1}{T}\sum_{n=0}^{N_T-1} \sum_{k=1}^d\sum_{j=1}^d \vert\bold{Y}_{s_n^{N_T}}^{(j)} \left(\Sigma^{1/2}\Delta_{\mathcal{P}_T}^n \bold{W}\right)^{(k)}\vert\bold{1}_{\{\vert \Delta_{\mathcal{P}_T}^n \bold{Y}^{(k)} \vert > \nu_{\mathcal{P}_T}^{(k)}, \vert\Delta_{\mathcal{P}_T}^n \bold{\bold{\tilde{Y}}}^{(k)} \vert \le 2\nu_{\mathcal{P}_T}^{(k)},\vert\Delta_{\mathcal{P}_T}^n \bold{U}^{(k)}\vert=0\}}\biggr) \\
        &\le\mathbb{E}\biggl(\dfrac{1}{T}\sum_{n=0}^{N_T-1} \sum_{k=1}^d\sum_{j=1}^d \vert\bold{Y}_{s_n^{N_T}}^{(j)} \left(\Sigma^{1/2}\Delta_{\mathcal{P}_T}^n \bold{W}\right)^{(k)}\vert\bold{1}_{\{\left|\Delta_{\mathcal{P}_T}^n \bold{W}^{(k)}+\Delta_{\mathcal{P}_T}^n \bold{D}^{(k)}\right| > \nu_{\mathcal{P}_T}^{(k)}/2\}}\biggr) \\&+\mathbb{E}\biggl(\dfrac{1}{T}\sum_{n=0}^{N_T-1} \sum_{k=1}^d\sum_{j=1}^d \vert\bold{Y}_{s_n^{N_T}}^{(j)} \left(\Sigma^{1/2}\Delta_{\mathcal{P}_T}^n \bold{W}\right)^{(k)}\vert\bold{1}_{\{\left|\Delta_{\mathcal{P}_T}^n \bold{M}^{(k)}\right| > \nu_{\mathcal{P}_T}^{(k)}/2\}}\biggr).
    \end{aligned}
\end{equation}
Using Cauchy-Schwarz and the fact that we work with forward looking increments, we have that
\begin{equation*}
    \begin{aligned}
    &\mathbb{E}\biggl(\dfrac{1}{T}\sum_{n=0}^{N_T-1} \sum_{k=1}^d\sum_{j=1}^d \vert\bold{Y}_{s_n^{N_T}}^{(j)} \left(\Sigma^{1/2}\Delta_{\mathcal{P}_T}^n \bold{W}\right)^{(k)}\vert\bold{1}_{\{\left|\Delta_{\mathcal{P}_T}^n \bold{W}^{(k)}+\Delta_{\mathcal{P}_T}^n \bold{D}^{(k)}\right| > \nu_{\mathcal{P}_T}^{(k)}/2\}}\biggr)\\ &\le \dfrac{1}{T}\sum_{n=0}^{N_T-1} \sum_{k=1}^d\sum_{j=1}^d\sum_{l=1}^d \vert \Sigma^{1/2}_{kl}\vert \mathbb{E}\left(\vert\bold{Y}_{s_n^{N_T}}^{(j)}\vert^2\right)^{1/2} \mathbb{E}\left(\vert\Delta_{\mathcal{P}_T}^n \bold{\bold{W}}^{(l)}\vert^2\right)^{1/2}\mathbb{P}\left(\left|\Delta_{\mathcal{P}_T}^n \bold{W}^{(k)}+\Delta_{\mathcal{P}_T}^n \bold{D}^{(k)}\right| > \nu_{\mathcal{P}_T}^{(k)}/2\right)^{1/2}\\
    &\le \dfrac{1}{T}N_T d^3 \max_{k,l}\vert \Sigma^{1/2}_{kl}\vert \max_{j}\mathbb{E}\left(\vert\bold{Y}_{s_n^{N_T}}^{(j)}\vert^2\right)^{1/2} O\left(\Delta_{\mathcal{P}_T}^{1/2}d^{4}\Delta_{\mathcal{P}_T}^{1-\beta^*} \right) \\
    &= O \left(d^{10}\Delta_{\mathcal{P}_T}^{1/2-\beta^*} \right), \quad \text{as } d\rightarrow \infty,
    \end{aligned}
\end{equation*}
and \Cref{Lemma:CountBound} with $l^\prime = \alpha^\prime = 2$ coupled with the fact that $\Delta_{\mathcal{P}_T}N_T = O(T)$ as $d$ tends to infinity.
As for the second term in \eqref{eqn:LemmaBrownianPart}, it holds following independence of $\mathbb{W}$ and $\mathbb{M}$ and the fact that we have forward-looking increments, that 
\begin{equation*}
    \begin{aligned}
&\mathbb{E}\left(\dfrac{1}{T}\sum_{n=0}^{N_T-1} \sum_{k=1}^d\sum_{j=1}^d\sum_{l=1}^d \vert \Sigma^{1/2}_{kl}\vert \vert\bold{Y}_{s_n^{N_T}}^{(j)} \left(\Delta_{\mathcal{P}_T}^n \bold{\bold{W}}^{(l)}\right)\left(\bold{1}_{\{\vert \Delta_{\mathcal{P}_T}^n \bold{M}^{(k)} \vert > \nu_{\mathcal{P}_T}^{(k)}\}}\right) \vert \right) \\&= \dfrac{1}{T}\sum_{n=0}^{N_T-1} \sum_{k=1}^d\sum_{j=1}^d\sum_{l=1}^d \vert \Sigma^{1/2}_{kl}\vert\mathbb{E}\left(\vert\bold{Y}_{s_n^{N_T}}^{(j)} \vert\right) \mathbb{E}\left(\vert\Delta_{\mathcal{P}_T}^n \bold{\bold{W}}^{(l)}\vert\right)\mathbb{P}\left(\vert \Delta_{\mathcal{P}_T}^n \bold{M}^{(k)} \vert > \nu_{\mathcal{P}_T}^{(k)}\right).
    \end{aligned}
\end{equation*}
Since by \eqref{eqn:IntActBoundM}
\begin{equation*}
    \max_{k}\mathbb{P}\left(\vert \Delta_{\mathcal{P}_T}^n \bold{M}^{(k)} \vert > \nu_{\mathcal{P}_T}^{(k)} \right) = O(\Delta_{\mathcal{P}_T}^{1-2\beta^*}),\quad \text{as } d\rightarrow \infty,
\end{equation*}
we have by Markov's inequality that
\begin{equation*}
    \begin{aligned}
      \mathbb{P}\biggl(\dfrac{1}{T}\sum_{n=0}^{N_T-1} \sum_{k=1}^d\sum_{j=1}^d \vert\bold{Y}_{s_n^{N_T}}^{(j)} \left(\Sigma^{1/2}\Delta_{\mathcal{P}_T}^n \bold{W}\right)^{(k)}\vert&\bold{1}_{\{\vert \Delta_{\mathcal{P}_T}^n \bold{Y}^{(k)} \vert > \nu_{\mathcal{P}_T}^{(k)},\vert\Delta_{\mathcal{P}_T}^n \bold{\bold{\tilde{Y}}}^{(k)} \vert \le 2\nu_{\mathcal{P}_T}^{(k)},\vert\Delta_{\mathcal{P}_T}^n \bold{U}^{(k)}\vert=0\}}>C\Delta_{\mathcal{P}_T}^\delta\biggr) \\&=O\left(d^{10}\Delta_{\mathcal{P}_T}^{1/2-\beta^{*}-\delta}\vee d^{6}\Delta_{\mathcal{P}_T}^{1/2-2\beta^{*}-\delta}\right)\\&=o(1), \text{as } d\rightarrow \infty,
      \end{aligned}
\end{equation*}
where we used that 
\begin{equation*}
    \begin{aligned}
    \mathbb{E}\left(\vert\Delta_{\mathcal{P}_T}^n \bold{\bold{W}}^{(l)}\vert\right) 
&\le\mathbb{E}\left(\left(\Delta_{\mathcal{P}_T}^n \bold{\bold{W}}^{(l)}\right)^2\right)^{1/2}
     =O( \Delta_{\mathcal{P}_T}^{1/2}),\quad \text{as } d\rightarrow \infty.
    \end{aligned}
\end{equation*}
\end{proof}
\begin{lemma}\label{L_1BoundInFinite}
    Grant \Cref{Ass:Estimation,Ass.H,Ass.High-FrequencySampling,Ass.GeneralBothCases,Ass.InfiniteAct}. Then for $\delta$ satisfying \Cref{Ass.InfiniteAct}
    \begin{equation}
    \Vert\left(\dfrac{1}{T}\sum_{n=0}^{N_T-1} \left( \Delta_n\bold{X}^c-\Delta_{\mathcal{P}_T}^n \bold{\bold{X}}\odot\vec{\bold{1}}(\Delta_{\mathcal{P}_T}^n \bold{\bold{X}})\right)\mathbf{X}_{s_n^{N_T}}^\top\right)\Vert_1 \le 3\Delta_{\mathcal{P}_T}^{\delta},
    \end{equation}
    with high probability under $\mathbb{P}^{{Q_0}}$.
    
\end{lemma}
\begin{proof}
We let $\tilde{\mathbb{Y}} :=\mathbb{Y}-\mathbb{M}$, i.e.\ our Lévy process with the small jumps subtracted. From \Cref{LevyItoImpl} we have that for each $t>0$,
\begin{equation*}
    \bold{\tilde{Y}}_t = \bold{Y}_0 + \bold{D}_t +\Sigma^{1/2}\bold{W}_t + \bold{U}_t.
\end{equation*}
We note that the Lévy part of this OU process is equal to a Lévy process with characteristic triplet $\left(0, \Sigma, \bar{F}=F_{\vert \{x \in \mathbb{R}^d : \Vert x \Vert_2 > 1 \}}\right)$ which is a finite activity Lévy process, since $F$ is a Lévy measure. Since $\bar{F}$ is finite and the Lévy measure associated to $\bold{U}$, we know that $\bold{U}$ has a compound Poisson representation as already argued in \Cref{LevyItoImpl}. We once again note that 
 \begin{equation*}
        \begin{aligned}
            &\Vert\left(\dfrac{1}{T}\sum_{n=0}^{N_T-1} \left( \Delta_n\bold{X}^c-\Delta_{\mathcal{P}_T}^n \bold{\bold{X}}\odot\vec{\bold{1}}(\Delta_{\mathcal{P}_T}^n \bold{\bold{X}})\right)\mathbf{X}_{s_n^{N_T}}^\top\right)\Vert_1 \\&\le \dfrac{1}{T}\sum_{n=0}^{N_T-1} \Vert\left(\left( \Delta_n\bold{X}^c-\Delta_{\mathcal{P}_T}^n \bold{\bold{X}}\odot\vec{\bold{1}}(\Delta_{\mathcal{P}_T}^n \bold{\bold{X}})\right)\mathbf{X}_{s_n^{N_T}}^\top\right)\Vert_1 \\
            &=\dfrac{1}{T}\sum_{n=0}^{N_T-1} \sum_{k=1}^d \sum_{j=1}^d \vert\bold{X}_{s_n^{N_T}}^{(j)} \left(\Delta_{\mathcal{P}_T}^n \bold{\bold{X}}^{(k)}\bold{1}_{\{\vert \Delta_{\mathcal{P}_T}^n \bold{\bold{X}}^{(k)} \vert \le \nu_{\mathcal{P}_T}^{(k)}\}}-(\Delta_{\mathcal{P}_T}^n \bold{X}^c)^{(k)} \right) \vert.
        \end{aligned}
    \end{equation*}
    As in Lemma $4.10$ of \cite{Mai2014}, we note that 
\begin{equation*}
    \begin{aligned}
   &\dfrac{1}{T}\sum_{n=0}^{N_T-1} \sum_{k=1}^d\sum_{j=1}^d \vert \vert\bold{Y}_{s_n^{N_T}}^{(j)} \left(\Delta_{\mathcal{P}_T}^n \bold{\bold{Y}}^{(k)}\bold{1}_{\{\vert \Delta_{\mathcal{P}_T}^n \bold{\bold{Y}}^{(k)} \vert \le \nu_{\mathcal{P}_T}^{(k)}\}}-(\Delta_{\mathcal{P}_T}^n \Sigma^{1/2}\bold{W})^{(k)}+\left(\int_{s_{n}^{N_T}}^{s_{n+1}^{N_T}} Q_0 \bold{Y}_sds\right)^{(k)} \right) \vert \\
   &\le \dfrac{1}{T}\sum_{n=0}^{N_T-1} \sum_{k=1}^d\sum_{j=1}^d\vert\bold{Y}_{s_n^{N_T}}^{(j)} \left(\Delta_{\mathcal{P}_T}^n \bold{\bold{\tilde{Y}}}^{(k)}\bold{1}_{\{\vert \Delta_{\mathcal{P}_T}^n \bold{\bold{\tilde{Y}}}^{(k)} \vert \le 2\nu_{\mathcal{P}_T}^{(k)}\}}-(\Delta_{\mathcal{P}_T}^n \Sigma^{1/2}\bold{W})^{(k)}+\left(\int_{s_{n}^{N_T}}^{s_{n+1}^{N_T}} Q_0 \bold{Y}_sds\right)^{(k)} \right)\vert  \\
   &+\dfrac{1}{T}\sum_{n=0}^{N_T-1} \sum_{k=1}^d\sum_{j=1}^d \vert\bold{Y}_{s_n^{N_T}}^{(j)} \Delta_{\mathcal{P}_T}^n \bold{\bold{\tilde{Y}}}^{(k)}\left(\bold{1}_{\{\vert \Delta_{\mathcal{P}_T}^n \bold{Y}^{(k)} \vert \le \nu_{\mathcal{P}_T}^{(k)}\}}-\bold{1}_{\{\vert \Delta_{\mathcal{P}_T}^n \bold{\bold{\tilde{Y}}}^{(k)} \vert \le 2\nu_{\mathcal{P}_T}^{(k)}\}}\right) \vert \\
    &+\dfrac{1}{T}\sum_{n=0}^{N_T-1} \sum_{k=1}^d\sum_{j=1}^d \vert\bold{Y}_{s_n^{N_T}}^{(j)} \Delta_{\mathcal{P}_T}^n \bold{M}^{(k)}\bold{1}_{\{\vert \Delta_{\mathcal{P}_T}^n \bold{Y}^{(k)} \vert \le \nu_{\mathcal{P}_T}^{(k)}\}}\vert  \\
    &:= S_{\Delta_{\mathcal{P}_T}}^1 + S_{\Delta_{\mathcal{P}_T}}^2 + S_{\Delta_{\mathcal{P}_T}}^3.
    \end{aligned}
\end{equation*}
We notice that all terms are positive. We are interested in bounding 
\begin{equation*}
\small
    \begin{aligned}
    &\mathbb{P}^{Q_0}\left(\Vert\left(\dfrac{1}{T}\sum_{n=0}^{N_T-1} \left( \Delta_n\bold{X}^c-\Delta_{\mathcal{P}_T}^n \bold{\bold{X}}\odot\vec{\bold{1}}(\Delta_{\mathcal{P}_T}^n \bold{\bold{X}})\right)\mathbf{X}_{s_n^{N_T}}^\top\right)\Vert_1>3 \Delta_{\mathcal{P}_T}^{\delta}\right) \\
    &= \mathbb{P}\left(\dfrac{1}{T}\sum_{n=0}^{N_T-1} \sum_{k=1}^d\sum_{j=1}^d \vert \vert\bold{Y}_{s_n^{N_T}}^{(j)} \left(\Delta_{\mathcal{P}_T}^n \bold{\bold{Y}}^{(k)}\bold{1}_{\{\vert \Delta_{\mathcal{P}_T}^n \bold{\bold{Y}}^{(k)} \vert \le \nu_{\mathcal{P}_T}^{(k)}\}}-(\Delta_{\mathcal{P}_T}^n \Sigma^{1/2}\bold{W})^{(k)}+\left(\int_{s_{n}^{N_T}}^{s_{n+1}^{N_T}} Q_0 \bold{Y}_sds\right)^{(k)} \right) \vert>3 \Delta_{\mathcal{P}_T}^{\delta}\right)
    \\&\le \mathbb{P} \left(S_{\Delta_{\mathcal{P}_T}}^1 + S_{\Delta_{\mathcal{P}_T}}^2 + S_{\Delta_{\mathcal{P}_T}}^3>3 \Delta_{\mathcal{P}_T}^{\delta}\right)\\
    &\le \mathbb{P}\left(S_{\Delta_{\mathcal{P}_T}}^1> \Delta_{\mathcal{P}_T}^\delta\right)+\mathbb{P} \left(S_{\Delta_{\mathcal{P}_T}}^2> \Delta_{\mathcal{P}_T}^\delta\right) +\mathbb{P}\left(S_{\Delta_{\mathcal{P}_T}}^3> \Delta_{\mathcal{P}_T}^{\delta}\right).
    \end{aligned}
\end{equation*}
Since $\tilde{\mathbb{Y}}$ is a finite activity Lévy process with $0$ drift which satisfies  \Cref{Ass.AssFiniteAct}, we have by repeating the arguments in \Cref{L_1BoundFInite} that $\mathbb{P} \left(S_{\Delta_{\mathcal{P}_T}}^1> \Delta_{\mathcal{P}_T}^\delta\right)$ tends to $0$ so our focus will be on the two remaining terms. We will follow the approach in \citealp{Mai2014,Courgeau2022a,Lucchese2023} but will not treat the dimension as a constant. First our focus will be on $S_{\Delta^\mathcal{P}_T}^2$. 
We note that we can bound $S_{\Delta_{\mathcal{P}_T}}^2$ as 
\begin{equation}\label{eqn:InfActSecondTermSplitup}
\begin{aligned}
S_{\Delta_{\mathcal{P}_T}}^2 &\le \dfrac{1}{T}\sum_{n=0}^{N_T-1} \sum_{k=1}^d\sum_{j=1}^d \vert\bold{Y}_{s_n^{N_T}}^{(j)} \Delta_{\mathcal{P}_T}^n \bold{\bold{\tilde{Y}}}^{(k)}\left(\bold{1}_{\{\vert \Delta_{\mathcal{P}_T}^n \bold{Y}^{(k)} \vert \le \nu_{\mathcal{P}_T}^{(k)}, \vert\Delta_{\mathcal{P}_T}^n \bold{\bold{\tilde{Y}}}^{(k)} \vert > 2\nu_{\mathcal{P}_T}^{(k)}\}}\right) \vert \\
&+ \dfrac{1}{T}\sum_{n=0}^{N_T-1} \sum_{k=1}^d\sum_{j=1}^d \vert\bold{Y}_{s_n^{N_T}}^{(j)} \Delta_{\mathcal{P}_T}^n \bold{\bold{\tilde{Y}}}^{(k)}\left(\bold{1}_{\{\vert \Delta_{\mathcal{P}_T}^n \bold{Y}^{(k)} \vert > \nu_{\mathcal{P}_T}^{(k)}, \vert\Delta_{\mathcal{P}_T}^n \bold{\bold{\tilde{Y}}}^{(k)} \vert \le 2\nu_{\mathcal{P}_T}^{(k)}\}}\right) \vert.
\end{aligned}
\end{equation}
We thus have that 
\begin{equation}\label{DecompositionS2}
    \begin{aligned}    &\mathbb{P}\left(S_{\Delta_{\mathcal{P}_T}}^2>\Delta_{\mathcal{P}_T}^\delta \right) \\&\le \mathbb{P}\biggl(  \dfrac{1}{T}\sum_{n=0}^{N_T-1} \sum_{k=1}^d\sum_{j=1}^d \vert\bold{Y}_{s_n^{N_T}}^{(j)} \Delta_{\mathcal{P}_T}^n \bold{\bold{\tilde{Y}}}^{(k)}\left(\bold{1}_{\{\vert \Delta_{\mathcal{P}_T}^n \bold{Y}^{(k)} \vert \le \nu_{\mathcal{P}_T}^{(k)}, \vert\Delta_{\mathcal{P}_T}^n \bold{\bold{\tilde{Y}}}^{(k)} \vert > 2\nu_{\mathcal{P}_T}^{(k)}\}}\right) \vert \\
&+ \dfrac{1}{T}\sum_{n=0}^{N_T-1} \sum_{k=1}^d\sum_{j=1}^d \vert\bold{Y}_{s_n^{N_T}}^{(j)} \Delta_{\mathcal{P}_T}^n \bold{\bold{\tilde{Y}}}^{(k)}\left(\bold{1}_{\{\vert \Delta_{\mathcal{P}_T}^n \bold{Y}^{(k)} \vert > \nu_{\mathcal{P}_T}^{(k)}, \vert\Delta_{\mathcal{P}_T}^n \bold{\bold{\tilde{Y}}}^{(k)} \vert \le 2\nu_{\mathcal{P}_T}^{(k)}\}}\right) \vert >\Delta_{\mathcal{P}_T}^\delta \biggr)
\\&\le \mathbb{P}\biggl(  \dfrac{1}{T}\sum_{n=0}^{N_T-1} \sum_{k=1}^d\sum_{j=1}^d \vert\bold{Y}_{s_n^{N_T}}^{(j)} \Delta_{\mathcal{P}_T}^n \bold{\bold{\tilde{Y}}}^{(k)}\left(\bold{1}_{\{\vert \Delta_{\mathcal{P}_T}^n \bold{Y}^{(k)} \vert \le \nu_{\mathcal{P}_T}^{(k)}, \vert\Delta_{\mathcal{P}_T}^n \bold{\bold{\tilde{Y}}}^{(k)} \vert > 2\nu_{\mathcal{P}_T}^{(k)}\}}\right) \vert >\Delta_{\mathcal{P}_T}^\delta/2 \biggr)\\&+
\mathbb{P}\biggl( \dfrac{1}{T}\sum_{n=0}^{N_T-1} \sum_{k=1}^d\sum_{j=1}^d \vert\bold{Y}_{s_n^{N_T}}^{(j)} \Delta_{\mathcal{P}_T}^n \bold{\bold{\tilde{Y}}}^{(k)}\left(\bold{1}_{\{\vert \Delta_{\mathcal{P}_T}^n \bold{Y}^{(k)} \vert > \nu_{\mathcal{P}_T}^{(k)}, \vert\Delta_{\mathcal{P}_T}^n \bold{\bold{\tilde{Y}}}^{(k)} \vert \le 2\nu_{\mathcal{P}_T}^{(k)}\}}\right) \vert >\Delta_{\mathcal{P}_T}^\delta/2 \biggr)\\
&= o(1) + \mathbb{P}\biggl( \dfrac{1}{T}\sum_{n=0}^{N_T-1} \sum_{k=1}^d\sum_{j=1}^d \vert\bold{Y}_{s_n^{N_T}}^{(j)} \Delta_{\mathcal{P}_T}^n \bold{\bold{\tilde{Y}}}^{(k)}\left(\bold{1}_{\{\vert \Delta_{\mathcal{P}_T}^n \bold{Y}^{(k)} \vert > \nu_{\mathcal{P}_T}^{(k)}, \vert\Delta_{\mathcal{P}_T}^n \bold{\bold{\tilde{Y}}}^{(k)} \vert \le 2\nu_{\mathcal{P}_T}^{(k)}\}}\right) \vert >\Delta_{\mathcal{P}_T}^\delta/2 \biggr), 
    \end{aligned}
\end{equation}
where we used \Cref{S_2:FirstBound} in the last equality.
We now focus on bounding the second term in \eqref{DecompositionS2}.
We have that 
\begin{equation*}
    \begin{aligned}
        &\left\{\left|\Delta_{\mathcal{P}_T}^n \bold{Y}^{(k)}\right| > \nu_{\mathcal{P}_T}^{(k)},\left|\Delta_{\mathcal{P}_T}^n \tilde{\bold{Y}}^{(k)}\right|\le2 \nu_{\mathcal{P}_T}^{(k)}\right\} \\
        &= 
        \left\{\left|\Delta_{\mathcal{P}_T}^n \bold{Y}^{(k)}\right| > \nu_{\mathcal{P}_T}^{(k)},\left|\Delta_{\mathcal{P}_T}^n \tilde{\bold{Y}}^{(k)}\right|\le2 \nu_{\mathcal{P}_T}^{(k)},\vert\Delta_{\mathcal{P}_T}^n \bold{U}^{(k)}\vert=0\right\} \\&\cup \left\{\left|\Delta_{\mathcal{P}_T}^n \bold{Y}^{(k)}\right| > \nu_{\mathcal{P}_T}^{(k)},\left|\Delta_{\mathcal{P}_T}^n \tilde{\bold{Y}}^{(k)}\right|\le2 \nu_{\mathcal{P}_T}^{(k)},\vert\Delta_{\mathcal{P}_T}^n \bold{U}^{(k)}\vert>0\right\}.
    \end{aligned}
\end{equation*}
As such, 
\begin{equation*}
    \begin{aligned}
         &\mathbb{P}\biggl( \dfrac{1}{T}\sum_{n=0}^{N_T-1} \sum_{k=1}^d\sum_{j=1}^d \vert\bold{Y}_{s_n^{N_T}}^{(j)} \Delta_{\mathcal{P}_T}^n \bold{\bold{\tilde{Y}}}^{(k)}\left(\bold{1}_{\{\vert \Delta_{\mathcal{P}_T}^n \bold{Y}^{(k)} \vert > \nu_{\mathcal{P}_T}^{(k)}, \vert\Delta_{\mathcal{P}_T}^n \bold{\bold{\tilde{Y}}}^{(k)} \vert \le 2\nu_{\mathcal{P}_T}^{(k)}\}}\right) \vert >\Delta_{\mathcal{P}_T}^\delta/2 \biggr)\\
         &\le \mathbb{P}\biggl( \dfrac{1}{T}\sum_{n=0}^{N_T-1} \sum_{k=1}^d\sum_{j=1}^d \vert\bold{Y}_{s_n^{N_T}}^{(j)} \Delta_{\mathcal{P}_T}^n \bold{\bold{\tilde{Y}}}^{(k)}\left(\bold{1}_{\{\vert \Delta_{\mathcal{P}_T}^n \bold{Y}^{(k)} \vert > \nu_{\mathcal{P}_T}^{(k)}, \vert\Delta_{\mathcal{P}_T}^n \bold{\bold{\tilde{Y}}}^{(k)} \vert \le 2\nu_{\mathcal{P}_T}^{(k)},\vert\Delta_{\mathcal{P}_T}^n \bold{U}^{(k)}\vert=0\}}\right) \vert \\&+\dfrac{1}{T}\sum_{n=0}^{N_T-1} \sum_{k=1}^d\sum_{j=1}^d \vert\bold{Y}_{s_n^{N_T}}^{(j)} \Delta_{\mathcal{P}_T}^n \bold{\bold{\tilde{Y}}}^{(k)}\left(\bold{1}_{\{ \vert \Delta_{\mathcal{P}_T}^n \bold{Y}^{(k)} \vert > \nu_{\mathcal{P}_T}^{(k)},\vert\Delta_{\mathcal{P}_T}^n \bold{\bold{\tilde{Y}}}^{(k)} \vert \le 2\nu_{\mathcal{P}_T}^{(k)},\vert\Delta_{\mathcal{P}_T}^n \bold{U}^{(k)}\vert>0\}}\right) \vert >\Delta_{\mathcal{P}_T}^\delta/2 \biggr)\\
         &\le \mathbb{P}\biggl( \dfrac{1}{T}\sum_{n=0}^{N_T-1} \sum_{k=1}^d\sum_{j=1}^d \vert\bold{Y}_{s_n^{N_T}}^{(j)} \Delta_{\mathcal{P}_T}^n \bold{\bold{\tilde{Y}}}^{(k)}\left(\bold{1}_{\{\vert \Delta_{\mathcal{P}_T}^n \bold{Y}^{(k)} \vert > \nu_{\mathcal{P}_T}^{(k)}, \vert\Delta_{\mathcal{P}_T}^n \bold{\bold{\tilde{Y}}}^{(k)} \vert \le 2\nu_{\mathcal{P}_T}^{(k)},\vert\Delta_{\mathcal{P}_T}^n \bold{U}^{(k)}\vert=0\}}\right) \vert >\Delta_{\mathcal{P}_T}^\delta/4 \biggr)\\
         &+\mathbb{P}\biggl(\dfrac{1}{T}\sum_{n=0}^{N_T-1} \sum_{k=1}^d\sum_{j=1}^d \vert\bold{Y}_{s_n^{N_T}}^{(j)} \Delta_{\mathcal{P}_T}^n \bold{\bold{\tilde{Y}}}^{(k)}\left(\bold{1}_{\{ \vert \Delta_{\mathcal{P}_T}^n \bold{Y}^{(k)} \vert > \nu_{\mathcal{P}_T}^{(k)},\vert\Delta_{\mathcal{P}_T}^n \bold{\bold{\tilde{Y}}}^{(k)} \vert \le 2\nu_{\mathcal{P}_T}^{(k)},\vert\Delta_{\mathcal{P}_T}^n \bold{U}^{(k)}\vert>0\}}\right) \vert >\Delta_{\mathcal{P}_T}^\delta/4 \biggr)\\
         &= \mathbb{P}\biggl( \dfrac{1}{T}\sum_{n=0}^{N_T-1} \sum_{k=1}^d\sum_{j=1}^d \vert\bold{Y}_{s_n^{N_T}}^{(j)} \Delta_{\mathcal{P}_T}^n \bold{\bold{\tilde{Y}}}^{(k)}\left(\bold{1}_{\{\vert \Delta_{\mathcal{P}_T}^n \bold{Y}^{(k)} \vert > \nu_{\mathcal{P}_T}^{(k)}, \vert\Delta_{\mathcal{P}_T}^n \bold{\bold{\tilde{Y}}}^{(k)} \vert \le 2\nu_{\mathcal{P}_T}^{(k)},\vert\Delta_{\mathcal{P}_T}^n \bold{U}^{(k)}\vert=0\}}\right) \vert >\Delta_{\mathcal{P}_T}^\delta/4 \biggr) \\&\quad + o(1), \quad \text{ as } d\rightarrow \infty.
    \end{aligned}
\end{equation*}
In the last equality, we used \Cref{SecondBoundFirstLemma} coupled with the fact that 
\begin{equation*}
    \begin{aligned}
        &\mathbb{P}\biggl(\dfrac{1}{T}\sum_{n=0}^{N_T-1} \sum_{k=1}^d\sum_{j=1}^d \vert\bold{Y}_{s_n^{N_T}}^{(j)} \Delta_{\mathcal{P}_T}^n \bold{\bold{\tilde{Y}}}^{(k)}\left(\bold{1}_{\{\vert \Delta_{\mathcal{P}_T}^n \bold{Y}^{(k)} \vert > \nu_{\mathcal{P}_T}^{(k)}, \vert\Delta_{\mathcal{P}_T}^n \bold{\bold{\tilde{Y}}}^{(k)} \vert \le 2\nu_{\mathcal{P}_T}^{(k)},\vert\Delta_{\mathcal{P}_T}^n \bold{U}^{(k)}\vert>0\}}\right) \vert >\Delta_{\mathcal{P}_T}^\delta/4 \biggr)\\
        &\le \mathbb{P}\biggl(\bigcup_{n=0}^{N_T-1}\bigcup_{k=1}^d \{\vert \Delta_{\mathcal{P}_T}^n \bold{Y}^{(k)} \vert > \nu_{\mathcal{P}_T}^{(k)}, \vert\Delta_{\mathcal{P}_T}^n \bold{\bold{\tilde{Y}}}^{(k)} \vert \le 2\nu_{\mathcal{P}_T}^{(k)},\vert\Delta_{\mathcal{P}_T}^n \bold{U}^{(k)}\vert>0\}\biggr).
    \end{aligned}
\end{equation*}
 We now wish to show that
\begin{equation*}
    \mathbb{P}\biggl( \dfrac{1}{T}\sum_{n=0}^{N_T-1} \sum_{k=1}^d\sum_{j=1}^d \vert\bold{Y}_{s_n^{N_T}}^{(j)} \Delta_{\mathcal{P}_T}^n \bold{\bold{\tilde{Y}}}^{(k)}\biggl(\bold{1}_{\{\vert \Delta_{\mathcal{P}_T}^n \bold{Y}^{(k)} \vert > \nu_{\mathcal{P}_T}^{(k)}, \vert\Delta_{\mathcal{P}_T}^n \bold{\bold{\tilde{Y}}}^{(k)} \vert \le 2\nu_{\mathcal{P}_T}^{(k)},\vert\Delta_{\mathcal{P}_T}^n \bold{U}^{(k)}\vert=0\}}\biggr) \vert >\Delta_{\mathcal{P}_T}^\delta/4 \biggr)
\end{equation*}
is $o(1)$ as $d$ tends to infinity. With this in mind, we note that on
$\vert\Delta_{\mathcal{P}_T}^n \bold{U}^{(k)}\vert=0$ necessarily, $\Delta_{\mathcal{P}_T}^n \bold{\bold{\tilde{Y}}}^{(k)} =\left(\Sigma^{1/2}\Delta_{\mathcal{P}_T}^n \bold{W}\right)^{(k)}+\Delta_{\mathcal{P}_T}^n \bold{\bold{D}}^{(k)}$.  Hence,
\begin{equation*}
    \begin{aligned}
    &\dfrac{1}{T}\sum_{n=0}^{N_T-1} \sum_{k=1}^d\sum_{j=1}^d \vert\bold{Y}_{s_n^{N_T}}^{(j)} \Delta_{\mathcal{P}_T}^n \bold{\bold{\tilde{Y}}}^{(k)}\biggl(\bold{1}_{\{\vert \Delta_{\mathcal{P}_T}^n \bold{Y}^{(k)} \vert > \nu_{\mathcal{P}_T}^{(k)}, \vert\Delta_{\mathcal{P}_T}^n \bold{\bold{\tilde{Y}}}^{(k)} \vert \le 2\nu_{\mathcal{P}_T}^{(k)},\vert\Delta_{\mathcal{P}_T}^n \bold{U}^{(k)}\vert=0\}}\biggr) \vert \\&= 
     \dfrac{1}{T}\sum_{n=0}^{N_T-1} \sum_{k=1}^d\sum_{j=1}^d \vert\bold{Y}_{s_n^{N_T}}^{(j)} \left(\left(\Sigma^{1/2}\Delta_{\mathcal{P}_T}^n \bold{W}\right)^{(k)}+\Delta_{\mathcal{P}_T}^n \bold{\bold{D}}^{(k)}\right)\biggl(\bold{1}_{\{\vert \Delta_{\mathcal{P}_T}^n \bold{Y}^{(k)} \vert > \nu_{\mathcal{P}_T}^{(k)}, \vert\Delta_{\mathcal{P}_T}^n \bold{\bold{\tilde{Y}}}^{(k)} \vert \le 2\nu_{\mathcal{P}_T}^{(k)},\vert\Delta_{\mathcal{P}_T}^n \bold{U}^{(k)}\vert=0\}}\biggr) \vert\\
     &\le \dfrac{1}{T}\sum_{n=0}^{N_T-1} \sum_{k=1}^d\sum_{j=1}^d \left(\vert\bold{Y}_{s_n^{N_T}}^{(j)} \left(\Sigma^{1/2}\Delta_{\mathcal{P}_T}^n \bold{W}\right)^{(k)}\vert+\vert\bold{Y}_{s_n^{N_T}}^{(j)}\Delta_{\mathcal{P}_T}^n \bold{\bold{D}}^{(k)}\vert\right)\\&\quad\biggl(\bold{1}_{\{\vert \Delta_{\mathcal{P}_T}^n \bold{Y}^{(k)} \vert > \nu_{\mathcal{P}_T}^{(k)}, \vert\Delta_{\mathcal{P}_T}^n \bold{\bold{\tilde{Y}}}^{(k)} \vert \le 2\nu_{\mathcal{P}_T}^{(k)},\vert\Delta_{\mathcal{P}_T}^n \bold{U}^{(k)}\vert=0\}}\biggr).
     \end{aligned}
\end{equation*}
We thus have that
\begin{equation*}
    \begin{aligned}
    &\mathbb{P}\biggl( \dfrac{1}{T}\sum_{n=0}^{N_T-1} \sum_{k=1}^d\sum_{j=1}^d \vert\bold{Y}_{s_n^{N_T}}^{(j)} \Delta_{\mathcal{P}_T}^n \bold{\bold{\tilde{Y}}}^{(k)}\biggl(\bold{1}_{\{\vert \Delta_{\mathcal{P}_T}^n \bold{Y}^{(k)} \vert > \nu_{\mathcal{P}_T}^{(k)}, \vert\Delta_{\mathcal{P}_T}^n \bold{\bold{\tilde{Y}}}^{(k)} \vert \le 2\nu_{\mathcal{P}_T}^{(k)},\vert\Delta_{\mathcal{P}_T}^n \bold{U}^{(k)}\vert=0\}}\biggr) \vert >\Delta_{\mathcal{P}_T}^\delta/4 \biggr)\\
    & \le \mathbb{P}\biggl(\dfrac{1}{T}\sum_{n=0}^{N_T-1} \sum_{k=1}^d\sum_{j=1}^d\biggl( \vert\bold{Y}_{s_n^{N_T}}^{(j)} \left(\Sigma^{1/2}\Delta_{\mathcal{P}_T}^n \bold{W}\right)^{(k)}\vert\\&+\vert\bold{Y}_{s_n^{N_T}}^{(j)}\Delta_{\mathcal{P}_T}^n \bold{\bold{D}}^{(k)}\vert\biggr)\bold{1}_{\{\vert \Delta_{\mathcal{P}_T}^n \bold{Y}^{(k)} \vert > \nu_{\mathcal{P}_T}^{(k)}, \vert\Delta_{\mathcal{P}_T}^n \bold{\bold{\tilde{Y}}}^{(k)} \vert \le 2\nu_{\mathcal{P}_T}^{(k)},\vert\Delta_{\mathcal{P}_T}^n \bold{U}^{(k)}\vert=0\}}>\Delta_{\mathcal{P}_T}^\delta/4\biggr)\\
        & \le \mathbb{P}\biggl(\dfrac{1}{T}\sum_{n=0}^{N_T-1} \sum_{k=1}^d\sum_{j=1}^d \vert\bold{Y}_{s_n^{N_T}}^{(j)} \left(\Sigma^{1/2}\Delta_{\mathcal{P}_T}^n \bold{W}\right)^{(k)}\vert\bold{1}_{\{\vert \Delta_{\mathcal{P}_T}^n \bold{Y}^{(k)} \vert > \nu_{\mathcal{P}_T}^{(k)}, \vert\Delta_{\mathcal{P}_T}^n \bold{\bold{\tilde{Y}}}^{(k)} \vert \le 2\nu_{\mathcal{P}_T}^{(k)},\vert\Delta_{\mathcal{P}_T}^n \bold{U}^{(k)}\vert=0\}}>\Delta_{\mathcal{P}_T}^\delta/8\biggr)\\
            & + \mathbb{P}\biggl(\dfrac{1}{T}\sum_{n=0}^{N_T-1} \sum_{k=1}^d\sum_{j=1}^d\vert\bold{Y}_{s_n^{N_T}}^{(j)}\Delta_{\mathcal{P}_T}^n \bold{\bold{D}}^{(k)}\vert\bold{1}_{\{\vert \Delta_{\mathcal{P}_T}^n \bold{Y}^{(k)} \vert > \nu_{\mathcal{P}_T}^{(k)}, \vert\Delta_{\mathcal{P}_T}^n \bold{\bold{\tilde{Y}}}^{(k)} \vert \le 2\nu_{\mathcal{P}_T}^{(k)},\vert\Delta_{\mathcal{P}_T}^n \bold{U}^{(k)}\vert=0\}}>\Delta_{\mathcal{P}_T}^\delta/8\biggr)\\
            &= o(1),\text{ as } d \rightarrow \infty,
    \end{aligned}
\end{equation*}
where we used \Cref{SecondBoundS_2D,SecondBoundS_2W} in the last equality.
We now turn our attention to 
\begin{equation}\label{eqn:SmallIncrBigInctilde}
    \begin{aligned}
    S_{\Delta_{\mathcal{P}_T}}^3 &= \dfrac{1}{T}\sum_{n=0}^{N_T-1} \sum_{k=1}^d\sum_{j=1}^d\vert\bold{Y}_{s_n^{N_T}}^{(j)} \Delta_{\mathcal{P}_T}^n \bold{M}^{(k)}\bold{1}_{\{\vert \Delta_{\mathcal{P}_T}^n \bold{Y}^{(k)} \vert \le \nu_{\mathcal{P}_T}^{(k)}\}}\vert   \\
    &\le \dfrac{1}{T}\sum_{n=0}^{N_T-1} \sum_{k=1}^d\sum_{j=1}^d \vert\bold{Y}_{s_n^{N_T}}^{(j)} \Delta_{\mathcal{P}_T}^n \bold{M}^{(k)}\vert\bold{1}_{\{\vert \Delta_{\mathcal{P}_T}^n \bold{Y}^{(k)} \vert \le \nu_{\mathcal{P}_T}^{(k)},\vert \Delta_{\mathcal{P}_T}^n \bold{\tilde{Y}}^{(k)} \vert \le 2\nu_{\mathcal{P}_T}^{(k)}\}}  \\
    &+ \dfrac{1}{T}\sum_{n=0}^{N_T-1} \sum_{k=1}^d\sum_{j=1}^d \vert\bold{Y}_{s_n^{N_T}}^{(j)} \Delta_{\mathcal{P}_T}^n \bold{M}^{(k)}\vert  \bold{1}_{\{\vert \Delta_{\mathcal{P}_T}^n \bold{Y}^{(k)} \vert \le \nu_{\mathcal{P}_T}^{(k)},\vert \Delta_{\mathcal{P}_T}^n \bold{\tilde{Y}}^{(k)} \vert > 2\nu_{\mathcal{P}_T}^{(k)}\}}.
    \end{aligned}
\end{equation}
Observe that
\begin{equation*}
    \begin{aligned}
        &\{\vert \Delta_{\mathcal{P}_T}^n \bold{Y}^{(k)} \vert \le \nu_{\mathcal{P}_T}^{(k)},\vert \Delta_{\mathcal{P}_T}^n \bold{\tilde{Y}}^{(k)} \vert \le 2\nu_{\mathcal{P}_T}^{(k)}\} \\&= \{\vert \Delta_{\mathcal{P}_T}^n \bold{Y}^{(k)} \vert \le \nu_{\mathcal{P}_T}^{(k)},\vert \Delta_{\mathcal{P}_T}^n \bold{\tilde{Y}}^{(k)} \vert \le 2\nu_{\mathcal{P}_T}^{(k)}, \vert \Delta_{\mathcal{P}_T}^n \bold{U}^{(k)} \vert =0 \} \\&\cup \{\vert \Delta_{\mathcal{P}_T}^n \bold{Y}^{(k)} \vert \le \nu_{\mathcal{P}_T}^{(k)},\vert \Delta_{\mathcal{P}_T}^n \bold{\tilde{Y}}^{(k)} \vert \le 2\nu_{\mathcal{P}_T}^{(k)}, \vert \Delta_{\mathcal{P}_T}^n \bold{U}^{(k)} \vert \ne 0 \}\\
        &\subset \{\vert \Delta_{\mathcal{P}_T}^n \bold{Y}^{(k)} \vert \le \nu_{\mathcal{P}_T}^{(k)},\vert \Delta_{\mathcal{P}_T}^n \bold{\tilde{Y}}^{(k)} \vert \le 2\nu_{\mathcal{P}_T}^{(k)}, \vert \Delta_{\mathcal{P}_T}^n \bold{U}^{(k)} \vert =0,\vert\Delta_{\mathcal{P}_T}^n \bold{M}^{(k)}  \vert> 2\nu_{\mathcal{P}_T}^{(k)} \}\\
        &\cup \{\vert \Delta_{\mathcal{P}_T}^n \bold{Y}^{(k)} \vert \le \nu_{\mathcal{P}_T}^{(k)},\vert \Delta_{\mathcal{P}_T}^n \bold{\tilde{Y}}^{(k)} \vert \le 2\nu_{\mathcal{P}_T}^{(k)}, \vert \Delta_{\mathcal{P}_T}^n \bold{U}^{(k)} \vert =0,\vert\Delta_{\mathcal{P}_T}^n \bold{M}^{(k)}  \vert\le 2\nu_{\mathcal{P}_T}^{(k)} \}\\&\cup \{\vert \Delta_{\mathcal{P}_T}^n \bold{Y}^{(k)} \vert \le \nu_{\mathcal{P}_T}^{(k)},\vert \Delta_{\mathcal{P}_T}^n \bold{\tilde{Y}}^{(k)} \vert \le 2\nu_{\mathcal{P}_T}^{(k)}, \vert \Delta_{\mathcal{P}_T}^n \bold{U}^{(k)} \vert \ne 0 \}\\
         &\subset  \{\vert \Delta_{\mathcal{P}_T}^n \bold{W}^{(k)} +\Delta_{\mathcal{P}_T}^n \bold{D}^{(k)}  \vert > \nu_{\mathcal{P}_T}^{(k)}\} \\&\cup  \{\vert\Delta_{\mathcal{P}_T}^n \bold{M}^{(k)}  \vert \le 2\nu_{\mathcal{P}_T}^{(k)}\}  \\&\cup \{\vert \Delta_{\mathcal{P}_T}^n \bold{Y}^{(k)} \vert \le \nu_{\mathcal{P}_T}^{(k)},\vert \Delta_{\mathcal{P}_T}^n \bold{\tilde{Y}}^{(k)} \vert \le 2\nu_{\mathcal{P}_T}^{(k)}, \vert \Delta_{\mathcal{P}_T}^n \bold{U}^{(k)} \vert \ne 0 \}.
    \end{aligned}
\end{equation*}
Above we used that on $\{\vert \Delta_{\mathcal{P}_T}^n \bold{U}^{(k)} \vert =0\}$, necessarily, $\Delta_{\mathcal{P}_T}^n \bold{Y}^{(k)}  = \Delta_{\mathcal{P}_T}^n \bold{W}^{(k)} +\Delta_{\mathcal{P}_T}^n \bold{D}^{(k)} +\Delta_{\mathcal{P}_T}^n \bold{M}^{(k)}$ and then we used the following set inclusion:
\begin{equation*}
    \begin{aligned}
       \{\vert \Delta_{\mathcal{P}_T}^n \bold{W}^{(k)} +\Delta_{\mathcal{P}_T}^n \bold{D}^{(k)} +\Delta_{\mathcal{P}_T}^n \bold{M}^{(k)}  \vert \le \nu_{\mathcal{P}_T}^{(k)}, \vert\Delta_{\mathcal{P}_T}^n \bold{M}^{(k)}  \vert> 2\nu_{\mathcal{P}_T}^{(k)}\} \subseteq \{\vert \Delta_{\mathcal{P}_T}^n \bold{W}^{(k)} +\Delta_{\mathcal{P}_T}^n \bold{D}^{(k)}  \vert > \nu_{\mathcal{P}_T}^{(k)}\}.
    \end{aligned}
\end{equation*}
This last inclusion follows from the reverse triangle inequality: on the event on the left-hand side,
\begin{equation*}
    \left\vert \Delta_{\mathcal{P}_T}^n \bold{W}^{(k)} +\Delta_{\mathcal{P}_T}^n \bold{D}^{(k)}  \right\vert \ge \left\vert\Delta_{\mathcal{P}_T}^n \bold{M}^{(k)}\right\vert - \left\vert \Delta_{\mathcal{P}_T}^n \bold{W}^{(k)} +\Delta_{\mathcal{P}_T}^n \bold{D}^{(k)} +\Delta_{\mathcal{P}_T}^n \bold{M}^{(k)}  \right\vert > 2\nu_{\mathcal{P}_T}^{(k)} - \nu_{\mathcal{P}_T}^{(k)} = \nu_{\mathcal{P}_T}^{(k)}.
\end{equation*}
Inserting the above in \eqref{eqn:SmallIncrBigInctilde}, we get that
\begin{equation}\label{eqn:BoundOnSn3}
    \begin{aligned}
 S_{\Delta_{\mathcal{P}_T}}^3 &\le \dfrac{1}{T}\sum_{n=0}^{N_T-1} \sum_{k=1}^d\sum_{j=1}^d\vert\bold{Y}_{s_n^{N_T}}^{(j)} \Delta_{\mathcal{P}_T}^n \bold{M}^{(k)}\vert\bold{1}_{\{\vert \Delta_{\mathcal{P}_T}^n \bold{Y}^{(k)} \vert \le \nu_{\mathcal{P}_T}^{(k)},\vert \Delta_{\mathcal{P}_T}^n \bold{\tilde{Y}}^{(k)} \vert \le 2\nu_{\mathcal{P}_T}^{(k)}\}}  \\
    &+ \dfrac{1}{T}\sum_{n=0}^{N_T-1} \sum_{k=1}^d\sum_{j=1}^d \vert\bold{Y}_{s_n^{N_T}}^{(j)} \Delta_{\mathcal{P}_T}^n \bold{M}^{(k)}\vert  \bold{1}_{\{\vert \Delta_{\mathcal{P}_T}^n \bold{Y}^{(k)} \vert \le \nu_{\mathcal{P}_T}^{(k)},\vert \Delta_{\mathcal{P}_T}^n \bold{\tilde{Y}}^{(k)} \vert > 2\nu_{\mathcal{P}_T}^{(k)}\}}\\
    &\le \dfrac{1}{T}\sum_{n=0}^{N_T-1} \sum_{k=1}^d\sum_{j=1}^d \vert\bold{Y}_{s_n^{N_T}}^{(j)} \Delta_{\mathcal{P}_T}^n \bold{M}^{(k)}\vert\bold{1}_{\{\vert \Delta_{\mathcal{P}_T}^n \bold{W}^{(k)}+\Delta_{\mathcal{P}_T}^n \bold{D}^{(k)} \vert > \nu_{\mathcal{P}_T}^{(k)}\}}  \\
    &+ \dfrac{1}{T}\sum_{n=0}^{N_T-1} \sum_{k=1}^d\sum_{j=1}^d \vert\bold{Y}_{s_n^{N_T}}^{(j)} \Delta_{\mathcal{P}_T}^n \bold{M}^{(k)}\vert\bold{1}_{\{\vert \Delta_{\mathcal{P}_T}^n \bold{M}^{(k)}\vert \le  2\nu_{\mathcal{P}_T}^{(k)}\}}\\
    &+ \dfrac{1}{T}\sum_{n=0}^{N_T-1} \sum_{k=1}^d\sum_{j=1}^d \vert\bold{Y}_{s_n^{N_T}}^{(j)} \Delta_{\mathcal{P}_T}^n \bold{M}^{(k)}\vert\bold{1}_{\{\vert \Delta_{\mathcal{P}_T}^n \bold{Y}^{(k)} \vert \le \nu_{\mathcal{P}_T}^{(k)},\vert \Delta_{\mathcal{P}_T}^n \bold{\tilde{Y}}^{(k)} \vert \le 2\nu_{\mathcal{P}_T}^{(k)},\vert \Delta_{\mathcal{P}_T}^n \bold{U}^{(k)} \vert\ne 0\}} \\
    &+ \dfrac{1}{T}\sum_{n=0}^{N_T-1} \sum_{k=1}^d\sum_{j=1}^d \vert\bold{Y}_{s_n^{N_T}}^{(j)} \Delta_{\mathcal{P}_T}^n \bold{M}^{(k)}\vert  \bold{1}_{\{\vert \Delta_{\mathcal{P}_T}^n \bold{Y}^{(k)} \vert \le \nu_{\mathcal{P}_T}^{(k)},\vert \Delta_{\mathcal{P}_T}^n \bold{\tilde{Y}}^{(k)} \vert > 2\nu_{\mathcal{P}_T}^{(k)}\}}.
    \end{aligned}
\end{equation}
As such, 
\begin{equation*}
    \begin{aligned}
        &\mathbb{P}\left( S_{\Delta_{\mathcal{P}_T}}^3 >\Delta_{\mathcal{P}_T}^\delta \right) \\&\le \mathbb{P}\left(\dfrac{1}{T}\sum_{n=0}^{N_T-1} \sum_{k=1}^d\sum_{j=1}^d \vert\bold{Y}_{s_n^{N_T}}^{(j)} \Delta_{\mathcal{P}_T}^n \bold{M}^{(k)}\vert  \bold{1}_{\{\vert \Delta_{\mathcal{P}_T}^n \bold{W}^{(k)}+\Delta_{\mathcal{P}_T}^n \bold{D}^{(k)} \vert > \nu_{\mathcal{P}_T}^{(k)}\}}> \Delta_{\mathcal{P}_T}^\delta /4 \right) \\
        &+\mathbb{P}\left(\dfrac{1}{T}\sum_{n=0}^{N_T-1} \sum_{k=1}^d\sum_{j=1}^d \vert\bold{Y}_{s_n^{N_T}}^{(j)} \Delta_{\mathcal{P}_T}^n \bold{M}^{(k)}\vert  \bold{1}_{\{\vert \Delta_{\mathcal{P}_T}^n \bold{M}^{(k)}\vert \le  2\nu_{\mathcal{P}_T}^{(k)}\}}> \Delta_{\mathcal{P}_T}^\delta /4 \right)\\
        &+\mathbb{P}\left(\dfrac{1}{T}\sum_{n=0}^{N_T-1} \sum_{k=1}^d\sum_{j=1}^d \vert\bold{Y}_{s_n^{N_T}}^{(j)} \Delta_{\mathcal{P}_T}^n \bold{M}^{(k)}\vert  \bold{1}_{\{\vert \Delta_{\mathcal{P}_T}^n \bold{Y}^{(k)} \vert \le \nu_{\mathcal{P}_T}^{(k)},\vert \Delta_{\mathcal{P}_T}^n \bold{\tilde{Y}}^{(k)} \vert \le 2\nu_{\mathcal{P}_T}^{(k)},\vert \Delta_{\mathcal{P}_T}^n \bold{U}^{(k)} \vert\ne 0\}}> \Delta_{\mathcal{P}_T}^\delta /4 \right)\\
        &+\mathbb{P}\left(\dfrac{1}{T}\sum_{n=0}^{N_T-1} \sum_{k=1}^d\sum_{j=1}^d \vert\bold{Y}_{s_n^{N_T}}^{(j)} \Delta_{\mathcal{P}_T}^n \bold{M}^{(k)}\vert  \bold{1}_{\{\vert \Delta_{\mathcal{P}_T}^n \bold{Y}^{(k)} \vert \le \nu_{\mathcal{P}_T}^{(k)},\vert \Delta_{\mathcal{P}_T}^n \bold{\tilde{Y}}^{(k)} \vert > 2\nu_{\mathcal{P}_T}^{(k)}\}} > \Delta_{\mathcal{P}_T}^\delta /4 \right)\\
        & = o(1), \quad \text{as } d\rightarrow \infty.
    \end{aligned}
\end{equation*}
As of now, the last equality is a claim, but we shall show in the following that all four terms are indeed $o(1)$. As for the first probability, we note that
\begin{equation*}
    \begin{aligned}
       &\mathbb{E}\left(\dfrac{1}{T}\sum_{n=0}^{N_T-1} \sum_{k=1}^d\sum_{j=1}^d \vert\bold{Y}_{s_n^{N_T}}^{(j)} \Delta_{\mathcal{P}_T}^n \bold{M}^{(k)}\vert  \bold{1}_{\{\vert \Delta_{\mathcal{P}_T}^n \bold{W}^{(k)}+\Delta_{\mathcal{P}_T}^n \bold{D}^{(k)} \vert > \nu_{\mathcal{P}_T}^{(k)}\}}\right)\\&\le \dfrac{1}{T}\sum_{n=0}^{N_T-1} \sum_{k=1}^d\sum_{j=1}^d \mathbb{E}\left(\vert\bold{Y}_{s_n^{N_T}}^{(j)}\vert^2 \vert\Delta_{\mathcal{P}_T}^n \bold{M}^{(k)}\vert^2\right)^{1/2}  \mathbb{P}\left({\{\vert \Delta_{\mathcal{P}_T}^n \bold{W}^{(k)}+\Delta_{\mathcal{P}_T}^n \bold{D}^{(k)} \vert > \nu_{\mathcal{P}_T}^{(k)}\}}\right)^{1/2}
       \\&= \dfrac{1}{T}\sum_{n=0}^{N_T-1} \sum_{k=1}^d\sum_{j=1}^d \mathbb{E}\left(\vert\bold{Y}_{s_n^{N_T}}^{(j)}\vert^2\right)^{1/2} \mathbb{E}\left(\vert\Delta_{\mathcal{P}_T}^n \bold{M}^{(k)}\vert^2\right)^{1/2}  \mathbb{P}\left({\{\vert \Delta_{\mathcal{P}_T}^n \bold{W}^{(k)}+\Delta_{\mathcal{P}_T}^n \bold{D}^{(k)} \vert > \nu_{\mathcal{P}_T}^{(k)}\}}\right)^{1/2}
       \\&= O\left(d^{7} \Delta_{\mathcal{P}_T}^{1/2-\beta^*} \right), \quad \text{ as } d \rightarrow \infty.
    \end{aligned}
\end{equation*}
Here we used Hölder's inequality in the first bound and the fact that we are working with forward-looking increments, $\bold{Y}_{s_n^{N_T}}^{(k)}$ and  $\Delta_{\mathcal{P}_T}^n \bold{M}^{(k)}$ are independent for any $k$ in the first equality. Finally we used \Cref{Lemma:CountBound} with $\alpha^\prime = l^\prime = 2$ and \eqref{MomentBoundOnM} in the last equality. Together with Markov's inequality, we therefore have that
\begin{equation*}
    \begin{aligned}
    &\mathbb{P}\left(\dfrac{1}{T}\sum_{n=0}^{N_T-1} \sum_{k=1}^d\sum_{j=1}^d \vert\bold{Y}_{s_n^{N_T}}^{(j)} \Delta_{\mathcal{P}_T}^n \bold{M}^{(k)}\vert  \bold{1}_{\{\vert \Delta_{\mathcal{P}_T}^n \bold{W}^{(k)}+\Delta_{\mathcal{P}_T}^n \bold{D}^{(k)} \vert > \nu_{\mathcal{P}_T}^{(k)}\}}> \Delta_{\mathcal{P}_T}^\delta /4 \right)\\&= O\left(d^{7} \Delta_{\mathcal{P}_T}^{1/2-\beta^*-\delta} \right)\\&=o(1),\quad \text{as } d\rightarrow \infty.
    \end{aligned}
\end{equation*}
As for the second term, we notice following independence of $\bold{Y}_{s_n^{N_T}}^{(k)}$ and  $\Delta_{\mathcal{P}_T}^n \bold{M}^{(k)}$ that
\begin{equation*}
    \begin{aligned}
    &\mathbb{E}\left(  \dfrac{1}{T}\sum_{n=0}^{N_T-1} \sum_{k=1}^d\sum_{j=1}^d \vert\bold{Y}_{s_n^{N_T}}^{(j)} \Delta_{\mathcal{P}_T}^n \bold{M}^{(k)}\vert\bold{1}_{\{\vert \Delta_{\mathcal{P}_T}^n \bold{M}^{(k)} \vert \le 2\nu_{\mathcal{P}_T}^{(k)}\}}\right) \\
    &=  \dfrac{1}{T}\sum_{n=0}^{N_T-1} \sum_{k=1}^d\sum_{j=1}^d \mathbb{E}\left(\vert\bold{Y}_{s_n^{N_T}}^{(j)}\vert\right) \mathbb{E}\left( \vert\Delta_{\mathcal{P}_T}^n \bold{M}^{(k)}\vert\bold{1}_{\{\vert \Delta_{\mathcal{P}_T}^n \bold{M}^{(k)} \vert \le 2\nu_{\mathcal{P}_T}^{(k)}\}}\right) \\
    &= O\left(d^3 \Delta_{\mathcal{P}_T}^\varepsilon \right),\quad \text{ as } d \rightarrow \infty.
    \end{aligned}
\end{equation*}
In turn, we find that
\begin{equation*}
    \begin{aligned}
    &\mathbb{P}\left(  \dfrac{1}{T}\sum_{n=0}^{N_T-1} \sum_{k=1}^d\sum_{j=1}^d \vert\bold{Y}_{s_n^{N_T}}^{(j)} \Delta_{\mathcal{P}_T}^n \bold{M}^{(k)}\vert\bold{1}_{\{\vert \Delta_{\mathcal{P}_T}^n \bold{M}^{(k)} \vert \le 2\nu_{\mathcal{P}_T}^{(k)}\}}>\Delta_{\mathcal{P}_T}^\delta /4\right)\\ &=  O\left(d^3 \Delta_{\mathcal{P}_T}^{\varepsilon-\delta} \right)\\&=o(1),\quad \text{ as } d \rightarrow \infty.
    \end{aligned}
\end{equation*}
For the third probability, write the event in the indicator as
\begin{equation*}
E_{n,k} := \{|\Delta_{\mathcal{P}_T}^n \bold{Y}^{(k)}|\le\nu_{\mathcal{P}_T}^{(k)},\,|\Delta_{\mathcal{P}_T}^n \bold{\tilde Y}^{(k)}|\le 2\nu_{\mathcal{P}_T}^{(k)},\,|\Delta_{\mathcal{P}_T}^n \bold{U}^{(k)}|\ne 0\}.
\end{equation*}
The weighted sum in the indicator vanishes outside $\bigcup_{n,k} E_{n,k}$, and $E_{n,k} \subseteq \{|\Delta_{\mathcal{P}_T}^n \bold{\tilde Y}^{(k)}|\le 2\nu_{\mathcal{P}_T}^{(k)},\,|\Delta_{\mathcal{P}_T}^n \bold{U}^{(k)}|>0\}$, so
\begin{equation*}
    \begin{aligned}
        &\mathbb{P}\left(\dfrac{1}{T}\sum_{n=0}^{N_T-1} \sum_{k=1}^d\sum_{j=1}^d \vert\bold{Y}_{s_n^{N_T}}^{(j)} \Delta_{\mathcal{P}_T}^n \bold{M}^{(k)}\vert  \bold{1}_{E_{n,k}} > \Delta_{\mathcal{P}_T}^\delta /4 \right)\\
        &\le \mathbb{P}\biggl(\bigcup_{n=0}^{N_T-1}\bigcup_{k=1}^d E_{n,k}\biggr) \\
        &\le \mathbb{P}\biggl(\bigcup_{n=0}^{N_T-1}\bigcup_{k=1}^d \{|\Delta_{\mathcal{P}_T}^n \bold{\tilde Y}^{(k)}|\le 2\nu_{\mathcal{P}_T}^{(k)},\,|\Delta_{\mathcal{P}_T}^n \bold{U}^{(k)}|>0\}\biggr) \\
        &= o(1), \quad \text{as } d\rightarrow \infty,
    \end{aligned}
\end{equation*}
where the last equality follows from \Cref{SecondBoundFirstLemma}.
As for the last probability, we have that 
\begin{equation*}
    \begin{aligned}
        &\mathbb{P}\left(\dfrac{1}{T}\sum_{n=0}^{N_T-1} \sum_{k=1}^d\sum_{j=1}^d \vert\bold{Y}_{s_n^{N_T}}^{(j)} \Delta_{\mathcal{P}_T}^n \bold{M}^{(k)}\vert  \bold{1}_{\{\vert \Delta_{\mathcal{P}_T}^n \bold{Y}^{(k)} \vert \le \nu_{\mathcal{P}_T}^{(k)},\vert \Delta_{\mathcal{P}_T}^n \bold{\tilde{Y}}^{(k)} \vert > 2\nu_{\mathcal{P}_T}^{(k)}\}} > \Delta_{\mathcal{P}_T}^\delta /4 \right)\\
        &\le\mathbb{P}\left(\bigcup_{n=0}^{N_T-1}\bigcup_{k=1}^d\{\vert \Delta_{\mathcal{P}_T}^n \bold{Y}^{(k)} \vert \le \nu_{\mathcal{P}_T}^{(k)},\vert \Delta_{\mathcal{P}_T}^n \bold{\tilde{Y}}^{(k)} \vert > 2\nu_{\mathcal{P}_T}^{(k)}\} \right) \\
        & = o(1),\quad \text{ as } d\rightarrow \infty,
    \end{aligned}
\end{equation*}
where we used the arguments from \eqref{eqn:InfAct5} for the last equality. In turn we get the desired result.
\end{proof}
\begin{proof}[Proof of \Cref{LikelihoodEqualityFeasibleCaseInfiniteActCase}]
    The proof is identical to the proof of \Cref{LikelihoodEqualityFeasibleCase} only now replacing \Cref{L_1BoundFInite} with \Cref{L_1BoundInFinite}.
\end{proof}
\section{Proofs: Main section}
\subsection{Helpful lemmas and propositions}
\begin{lemma}[Row-normalisation bound]\label{RowNormalisedBound}
    Let $A,B$ be $m\times n$ matrices where $A$ has no rows with all $0$ entries and where $B$ is allowed to contain rows with all $0$s. Let $A^*$ and $B^*$ be the corresponding $m\times n$ row-normalised matrices. That is, the $(i,j)$th entry of $B^*$ is $B^*_{ij}=\dfrac{B_{ij}}{\Vert B_{i\bullet}\Vert_2}$. Here we use the convention that $\dfrac{0}{0}:=0$. Then, it holds that
    \begin{equation*}
        \begin{aligned}
            \Vert A^* - B^*\Vert _F \le \dfrac{2 \Vert A - B\Vert _F}{\min_i\left(\Vert A_{i\bullet}\Vert_2\right)}\quad. 
        \end{aligned}
    \end{equation*}
\end{lemma}
\begin{proof}
    First, we note that for $i$ such that $\Vert B_{i\bullet}\Vert_2>0$ it holds that
    \begin{equation}\label{FirstBoundHelpFulLemma}
        \begin{aligned}
             \Vert A^*_{i\bullet} - B^*_{i\bullet}\Vert_2 = \dfrac{1}{\Vert A_{i\bullet}\Vert_2\Vert B_{i\bullet}\Vert_2}\Vert\Vert B_{i\bullet}\Vert_2A_{i\bullet}-B_{i\bullet}\Vert A_{i\bullet}\Vert_2\Vert_2.
        \end{aligned}
    \end{equation}
Observe that 
\begin{equation*}
    \begin{aligned}
        \Vert\Vert B_{i\bullet}\Vert_2A_{i\bullet}-B_{i\bullet}\Vert A_{i\bullet}\Vert_2\Vert_2 &=\Vert\Vert B_{i\bullet}\Vert_2\left(A_{i\bullet}-B_{i\bullet}\right)-B_{i\bullet}\left(\Vert A_{i\bullet}\Vert_2-\Vert B_{i\bullet}\Vert_2\right)\Vert_2 \\
       & \le 
       \Vert\Vert B_{i\bullet}\Vert_2\left(A_{i\bullet}-B_{i\bullet}\right)\Vert_2+\Vert B_{i\bullet}\left(\Vert A_{i\bullet}\Vert_2-\Vert B_{i\bullet}\Vert_2\right)\Vert_2 \\
       &=
        \Vert B_{i\bullet}\Vert_2\Vert A_{i\bullet}-B_{i\bullet}\Vert_2+\Vert B_{i\bullet}\Vert_2\vert \Vert A_{i\bullet}\Vert_2-\Vert B_{i\bullet}\Vert_2\vert \\
        &\le 2\Vert B_{i\bullet}\Vert_2\Vert A_{i\bullet}-B_{i\bullet}\Vert_2.
    \end{aligned}
\end{equation*}
Inserting this into \eqref{FirstBoundHelpFulLemma} yields
\begin{equation}
\Vert A^*_{i\bullet} - B^*_{i\bullet}\Vert_2      \le \dfrac{2\Vert A_{i\bullet} - B_{i\bullet}\Vert_2}{\Vert A_{i\bullet}\Vert_2}.
\end{equation}
Consider now $i$ such that $\Vert B_{i\bullet}\Vert_2 = 0$. Then $B_{i\bullet} = 0$, and by our convention $B^*_{i\bullet} = 0$ as well. Hence
\begin{equation*}
    \Vert A^*_{i\bullet} - B^*_{i\bullet}\Vert_2 = \Vert A^*_{i\bullet}\Vert_2 = 1 = \dfrac{\Vert A_{i\bullet}\Vert_2}{\Vert A_{i\bullet}\Vert_2} = \dfrac{\Vert A_{i\bullet} - B_{i\bullet}\Vert_2}{\Vert A_{i\bullet}\Vert_2},
\end{equation*}
which is even tighter than the bound just established for the case $\Vert B_{i\bullet}\Vert_2 > 0$ by a factor of $2$. We note that this bound in fact holds for any norm, not just the Euclidean norm.
We thus have, that 
\begin{equation*}
    \begin{aligned}
    \dfrac{1}{4}\Vert A^* - B^* \Vert _F^2 &\le \sum_{i=1}^m \left(\dfrac{\Vert A_{i\bullet} - B_{i\bullet}\Vert_2}{\Vert A_{i\bullet}\Vert_2}\right)^2\\
    &\le \sum_{i=1}^m \left(\dfrac{\Vert A_{i\bullet} - B_{i\bullet}\Vert_2}{\min_{i\in\{1,2,\dots, m\}}\Vert A_{i\bullet}\Vert_2}\right)^2
    \\
    &=\dfrac{1}{\min_{i\in\{1,2,\dots, m\}}\Vert A_{i\bullet}\Vert_2^2} \sum_{i=1}^m \left(\Vert A_{i\bullet} - B_{i\bullet}\Vert_2\right)^2
     \\
    &=\dfrac{\Vert A - B\Vert_F^2}{\min_{i\in\{1,2,\dots, m\}}\Vert A_{i\bullet}\Vert_2^2}, 
    \end{aligned}
\end{equation*}
which in turn concludes the proof.
\end{proof}

\begin{lemma}\label{MaximalEigenValueLaplace}
Let \( \mathcal{G} \) be an undirected, weighted graph with \( n \) vertices, and let \( \mathbf{L} = \mathbf{D}^{-1/2} \mathbf{A} \mathbf{D}^{-1/2} \) represent a slightly modified normalized Laplacian matrix, where \( \mathbf{A} \) is the adjacency matrix and \( \mathbf{D} \) is the absolute degree matrix, that it is the diagonal matrix with $(i,i)$th entry equal to $\mathbf{D}_{ii} = \sum_{j=1}^n \vert A_{ij} \vert$. Then $\Vert  \mathbf{L} \Vert \le 1$.
\end{lemma}
\begin{proof}
    First consider $ \mathbf{\tilde{L}} = \mathbf{D}^{-1} \mathbf{A}$. Clearly, for any $i\in\{1,\dots,n\}$, the disc centred at $\tilde{\bold{L}}_{ii}$ has radius $R_i$ with
    \begin{equation*}
            R_i= \dfrac{\sum_{j\ne i} \vert \mathbf{A}_{ij}\vert}{\mathbf{D}_{ii}}.
    \end{equation*}
    Following Gershgorin circle theorem, it holds that any eigenvalue of $\mathbf{\tilde{L}}$ satisfies that $\vert\lambda_j(\mathbf{\tilde{L}}) \vert \le \max_i \vert \mathbf{\tilde{L}}_{ii} \pm R_i \vert\le \vert \mathbf{\tilde{L}}_{ii}\vert + R_i = 1$. Following the proof to \Cref{Thm.1} the eigenvalues of  $\mathbf{\tilde{L}}$ and $\mathbf{L}$ are equal. Now since $\mathbf{L}$ is symmetric, the result follows.
\end{proof}
\begin{lemma}\label{ConcentrationLemma}
    Consider a random undirected and weighted graph on $n$ vertices $\mathcal{G}=(\mathcal{V}, \mathcal{E}, \mathcal{W})$. For all $i, j$ let $W_{i j}$ be a random variable which is supported on the interval $\left[\underline{w}_{i j}, \bar{w}_{i j}\right]$ where $\underline{w}_{i j}, \bar{w}_{i j}$ are real numbers such that $\underline{w}_{i j} \leq \bar{w}_{i j}$, and has mean $\mu_{i j}\ne0$ and variance $\sigma_{i j}^2$, with all $W_{ij}$ drawn independently of each other. For all edges $(i, j)$ in $\mathcal{E}$, the edge weight is then given by $w_{(i, j)}=W_{i j}$. Let $\mathbf{A}$ be the adjacency matrix of the graph; let $\mathbf{D}$ be the \emph{absolute} degree matrix, a diagonal matrix whose entries are the row (or column) sums of the absolute values in $\mathbf{A}$; and let $\mathbf{L}$ be the Laplacian, defined as $\mathbf{D}^{-1 / 2} \mathbf{A D}^{-1 / 2}$. Let $\mathcal{A}=\mathbb{E}(\mathbf{A}), \mathcal{D}=\mathbb{E}(\mathbf{D})$ and $\mathcal{L}=\mathcal{D}^{-1 / 2} \mathcal{A} \mathcal{D}^{-1 / 2}$ be their population analogues. Let $\bar{d}_i \equiv \sum_{j=1}^n \mathbb{E}\left(\vert\mathbf{A}_{ij}  \vert\right) = \sum_{j=1}^n  p_{ij} \bar{\mu}_{ij}$ where $\bar{\mu}_{ij} = \mathbb{E}\left(\vert \bold{W}_{ij}\vert\right)$. Then $\bar{d}_{\min }=\min _i \bar{d}_i$ is the minimum expected degree of $\mathcal{G}$ and $\bar{d}_{\max }=\max _i \bar{d}_i$ its maximum. Define $v \equiv 2 \max_{ij}\max\left(\vert  \bar{w}_{i j}\vert ,\vert  \underline{w}_{i j}\vert \right) +\max_{ij} \dfrac{\sigma_{ij}^2}{\vert \mu_{ij} \vert} $.
Then, for any constant $c>0$ there exists another constant $C>0$ independent of $n$ and the edge probabilities, such that if $\bar{d}_{\max }>C \log n$, then, for all $n^{-c} \leq \delta \leq 1 / 2$,
$$
\mathbb{P}\left(\|\mathbf{A}-\mathcal{A}\| \leq 4 \sqrt{\nu \bar{d}_{\text {max }} \log (2 n / \delta)}\right) \geq 1-\delta .
$$
Moreover, if $\bar{d}_{\min } \geq C \log n$, then, for all $n^{-c} \leq \delta \leq 1 / 2$,
$$
\mathbb{P}\left(\|\mathbf{L}-\mathcal{L}\| \leq 14 \sqrt{\frac{\nu \log (4 n / \delta)}{\bar{d}_{\min }}}\right) \geq 1-\delta.
$$
\end{lemma}
\begin{proof}
We reproduce the proof here for completeness, adapting Lemma~A.1 of \citet{Gudhmundsson2021} (which in turn builds on \cite{Oliveira2009}) to accommodate signed edge weights. The setup and proof scheme follow theirs closely; we flag the points where the modification for signed $W_{ij}$ enters. For all $1\le i,j \le n$, let $B_{ij}$ be independent Bernoulli variables with parameters $p_{ij}$. Let $\bold{e_1,e_2,\dots,e_n}$ be the standard basis on $\mathbb{R}^n$. For any $i,j \in \{1,2,\dots,n\}$, let 
    \begin{equation*}
        I_{ij} = \begin{cases}\boldsymbol{e}_i \boldsymbol{e}_j^{\top}+\boldsymbol{e}_j \boldsymbol{e}_i^{\top}, & i \neq j, \\ \boldsymbol{e}_i \boldsymbol{e}_i^{\top}, & i=j .\end{cases}
    \end{equation*}
We notice that $  \sum_{1 \le i \le j \le n}I_{ij}B_{ij}W_{ij}=\mathbf{A}$ whereas $ \sum_{1 \le i \le j \le n}I_{ij}p_{ij}\mu_{ij}=\mathcal{A}$. We introduce
\begin{equation*}
    \mathbf{X}_{ij} = \left( B_{ij}W_{ij} - p_{ij}\mu_{ij} \right)I_{ij}, \quad \text{for} 1\le i \le j \le n.
\end{equation*}
Clearly,
\begin{equation}\label{Equality}
    \sum_{1 \le i \le j \le n} \mathbf{X}_{ij} = \mathbf{A} - \mathcal{A}.
\end{equation}
Following, the above, it is clear that bounding the spectral norm of $\mathbf{A} - \mathcal{A}$ boils down to bounding the spectral norm of $\sum_{1 \le i \le j \le n} \mathbf{X}_{ij}$. To do this, we will use Corollay $7.1$ of \cite{Oliveira2009}. In light of this, we notice that $\Vert \mathbf{X}_{ij} \Vert= \vert B_{ij}W_{ij}  -p_{ij}\mu_{ij}\vert \Vert I_{ij} \Vert \le 2 \max\left(\vert  \bar{w}_{i j}\vert ,\vert  \underline{w}_{i j}\vert \right)$, following the triangle inequality and the fact that the eigenvalues of $I_{ij}$ are in $\{-1,0,1\}$. The two-sided support bound $\max(\vert\bar w_{ij}\vert,\vert\underline w_{ij}\vert)$ is the first place where our setting deviates from \citet{Gudhmundsson2021}: their argument uses only the upper bound $\bar w_{ij}$, which suffices when $W_{ij}\ge 0$ but not in our signed-weight setting. Following identical arguments as in the proof of Lemma $A.1$ in\citealp{Gudhmundsson2021}, it can be shown that the maximal eigenvalue of $ \sum_{1 \le i \le j \le n} \mathbb{E}\left(\mathbf{X}_{ij}^2\right) $ is given by
\begin{equation}\label{eqn:spectralNorm}
    \lambda_{\text{max}}\left( \sum_{1 \le i \le j \le n} \mathbb{E}\left(\mathbf{X}_{ij}^2\right)\right) = \max_{i\in\{1,2,\dots,n\}} \sum_{j=1}^n \mathrm{Var}\left( B_{ij}W_{ij}\right).
\end{equation}
Using the same arguments as in \cite{Gudhmundsson2021}, due to our slightly different setting, we get the following bound 
\begin{equation}\label{VarBound}
    \mathrm{Var}\left( B_{ij}W_{ij}\right) \le p_{ij}\vert \mu_{ij} \vert\left( \max_{ij} \vert \mu_{ij} \vert +\max_{ij} \dfrac{\sigma_{ij}^2}{\vert \mu_{ij} \vert} \right).
\end{equation}
Using \eqref{eqn:spectralNorm} and \eqref{VarBound}, we get that 
\begin{equation*}
    \lambda_{\text{max}}\left( \sum_{1 \le i \le j \le n} \mathbb{E}\left(\mathbf{X}_{ij}^2\right) \right) \le v \bar{d}_{\max},
\end{equation*}
where $v = \left( 2\max_{ij} \max\left(\vert  \bar{w}_{i j}\vert ,\vert  \underline{w}_{i j}\vert \right) +\max_{ij} \dfrac{\sigma_{ij}^2}{\vert \mu_{ij} \vert} \right)$. Let $\sigma^2 := \lambda_{\text{max}}\left( \sum_{1 \le i \le j \le n} \mathbb{E}\left(\mathbf{X}_{ij}^2\right) \right)$, $M=2 \max_{ij}\max\left(\vert  \bar{w}_{i j}\vert ,\vert  \underline{w}_{i j}\vert \right)\le v$ and finally $d=n$. 
Then, following Corollary $7.1$ of \cite{Oliveira2009}, for all $t \geq 0$,
\begin{equation}\label{Firstbound}
\mathbb{P}\left(\Vert \sum_{1 \le i \le j \le n} \mathbf{X}_{ij}\Vert \geq t\right) \leq 2 n e^{-\frac{t^2}{8 \sigma^2+4 M t}} \le 2 n e^{-\frac{t^2}{v(8 \bar{d}_{\max}+4 t})}.
\end{equation}
Fix $c>0$, and assume $n^{-c} \le \delta \le 1/2$. Let $t=4 \sqrt{v \bar{d}_{\max} \log(2n/\delta)}$, then we can choose a constant $C$ independent of $n$ and each $p_{ij}$ such that if $\bar{d}_{\max} \ge C \log n$, then $t \le 2 \bar{d}_{\max}$. Take for instance $C = 8v(1+c)$. If $\bar{d}_{\max} = \Omega(\log n)$ such a $C$ will always exist simply by choosing $c$ small enough. To see that this $C$ does the job, note that
\begin{equation*}
    \begin{aligned}
    \log(2n/\delta) &\le \log(2n/n^{-c}) \\
    &\le \log(2n^{1+c}) \\
    &\le \log((2n)^{1+c}) \\
    &= (1+c)\log(2n) \\
    &\le (1+c)2\log(n),
    \end{aligned}
\end{equation*}
where in the last inequality, we used that $n\ge2$, since $n^{-c}\le 1/2$. Now, clearly
\begin{equation*}
    \begin{aligned}
    t &= 4 \sqrt{v \bar{d}_{\max} \log(2n/\delta)} \\
     &= 4 \sqrt{v \bar{d}_{\max} 2(1+c)\log n} \\
     &=  \sqrt{4 \bar{d}_{\max} 8v(1+c)\log n} \\
     &=  \sqrt{4 \bar{d}_{\max} C\log n} \\
     &\le 2 \sqrt{\bar{d}^2_{\max}} \\
     &=  2 \bar{d}_{\max}.
    \end{aligned}
\end{equation*}
It follows that, for $n$ satisfying that $n^{-c}\le \delta \le 1/2$, following \eqref{Firstbound},
\begin{equation*}
    \begin{aligned}
\mathbb{P}\left(\Vert \sum_{1 \le i \le j \le n}  \mathbf{X}_{ij}\Vert \geq t\right) &\le 2 n e^{-\frac{t^2}{(8  \bar{d}_{\max}+4 t)}} \\
&= 2 n e^{-\frac{16v \bar{d}_{\max} \log(2n/\delta)}{v(8 \bar{d}_{\max}+8 \bar{d}_{\max})}} \\
&= \delta,
\end{aligned}
\end{equation*}
which proves the first assertion, following \eqref{Equality}. 
To see the second, we first define new variables: $d_i \equiv \sum_{j=1}^n \vert\mathbf{A}_{ij}  \vert=\sum_{j=1}^n  B_{ij} \vert W_{ij}  \vert$. Then $d_i - \bar{d}_i = \sum_{j=1}^n \left( B_{ij}\vert W_{ij} \vert - p_{ij}\bar{\mu}_{ij}\right)$. Again, we wish to apply Corollary $7.1$ of \cite{Oliveira2009}. We will therefore redefine $M$ and $\sigma^2$ to be consistent with the notation of this corollary.  Now on $\vert d_i - \bar{d}_i \vert = \big\vert \sum_{j=1}^n \big( B_{ij}\vert W_{ij} \vert-p_{ij}\bar{\mu}_{ij}\big) \big\vert$, i.e. treating each summand $B_{ij}\vert W_{ij} \vert-p_{ij}\bar{\mu}_{ij}$ as a $1\times 1$ random matrix. We thus notice, that $d=1$ and reiterate that we interpret the scalar as a $1\times 1$ matrix. Notice, that the only eigenvalue of this matrix/scalar is then the scalar itself. Also, $\mathbb{E}\left(\vert B_{ij}\vert W_{ij} \vert-p_{ij} \bar{\mu}_{ij} \vert^2  \right)=\mathrm{Var}\left( B_{ij} \vert W_{ij} \vert \right) \le \mathrm{Var}\left( B_{ij} W_{ij}\right)$. Using \eqref{VarBound} and that $\vert \mu_{ij}\vert \le \bar{\mu}_{ij}$, following Jensens inequality, we therefore have that  $\sigma^2 \le \bar{d}_iv $. We may bound the absolute value of each summand by the same upper bound, $M$ as before, as we following the Triangle Inequality may establish that, $\vert B_{ij}\vert W_{ij} \vert-p_{ij} \bar{\mu}_{ij}\vert \le M \le v$. Following Corollary $7.1$ of \cite{Oliveira2009}, we then have that 
\begin{equation*}
    \mathbb{P}\left(\vert d_i - \bar{d}_i \vert\ge r\right) \le 2 e^{-\dfrac{r^2}{8\sigma^2 + 4Mr}} \le 2 e^{-\dfrac{r^2}{8v\bar{d}_i + 4vr}}.
\end{equation*}
We now use that $\vert d_i - \bar{d}_i\vert = \bar{d}_i\vert \dfrac{d_i}{\bar{d}_i} - 1\vert$. Then, by taking $r=t\bar{d}_i$, we have that 
\begin{equation*}
    \mathbb{P}\left( \vert \dfrac{d_i}{\bar{d}_i} - 1\vert \ge t\right) \le 2 e^{-\dfrac{t^2\bar{d}_i^2}{8v\bar{d}_i + 4vt\bar{d}_i}}=2 e^{-\dfrac{t^2\bar{d}_i}{8v + 4vt}}\le 2 e^{-\dfrac{t^2\bar{d}_{\min}}{8v + 4vt}}.
\end{equation*}
Again, we fix $c>0$ and consider $n,\delta$ such that $n^{-c}\le \delta \le 1/2$. Again, we note that there exist a $C$ that only depends on $c$ and $W_{ij}$(in particular it is independent of $n$ and $p$) such that for $\bar{d}_{\min} \ge C\log n$, then $t=4\sqrt{v \log(4n/\delta)/\bar{d}_{\min}}\le 1/2$. Once more, for completeness, we note that $C = v(1+c)64\cdot4$ does the trick. Indeed, then, for $n \ge 2$,
\begin{equation*}
    \begin{aligned}
        4\sqrt{\dfrac{v \log(4n/\delta)}{\bar{d}_{\min}}} &= \sqrt{\dfrac{16v \log(4n/\delta)}{\bar{d}_{\min}}} \\
        &\le \sqrt{\dfrac{64v(1+c) \log(n)}{\bar{d}_{\min}}} \\&\le 
        \sqrt{\dfrac{64v(1+c) }{C}} \\
        &= \sqrt{\dfrac{64v(1+c) }{64v(1+c)4}}  \\
        &= \sqrt{\dfrac{1}{4}} = 1/2 .
    \end{aligned}
\end{equation*}
We note that clearly also $t\le 3/4 \le 2$. Hence, following Corollary $7.1$ of \cite{Oliveira2009}, we have, for all $i$, that 
\begin{equation*}
    \mathbb{P}\left( \vert \dfrac{d_i}{\bar{d}_i} - 1\vert \ge t\right) \le 2 e^{-\dfrac{16 v\log(4n/\delta)}{8v + 4vt}} \le 2 e^{-\dfrac{16 v\log(4n/\delta)}{16v}} = \dfrac{\delta}{2n}.
\end{equation*}
In turn, we conclude that with probability greater than or equal to $1-\dfrac{\delta}{2}$,
\begin{equation*}
    \forall i \in \{ 1,2,\dots,n\},\quad \vert \dfrac{d_i}{\bar{d}_i} - 1\vert  \le 4\sqrt{v \log(4n/\delta)/\bar{d}_{\min}}.
\end{equation*}
This will be important for the next step. Using identical arguments to \cite{Oliveira2009} in the proof of Theorem $3.1$ and the fact that $t$ is chosen to be less than $3/4$, it can be shown that it then also holds that, with probability $1-\dfrac{\delta}{2}$,
\begin{equation}\label{ForallI}
    \forall i \in \{ 1,2,\dots,n\},\quad \vert \sqrt{\dfrac{d_i}{\bar{d}_i}} - 1\vert  \le 4\sqrt{v \log(4n/\delta)/\bar{d}_{\min}}.
\end{equation}
We now let $\mathbf{T} = \mathbf{D}^{-1/2}$, where $\mathbf{D}$ is the diagonal matrix with $(i,i)$th entry given by $\sum_{j=1}^n B_{ij}\vert W_{ij}\vert$ which following \eqref{ForallI} is well defined with high probability if $\bar{d}_{\min}$ tends sufficiently fast to $\infty$, i.e. if $4\sqrt{v \log(4n/\delta)/\bar{d}_{\min}} \rightarrow c\text{ as } n \rightarrow \infty$ for some $c<1$. Similarly, let $\mathcal{T}=\mathcal{D}^{-1/2}$ be the population analogous. As in \cite{Gudhmundsson2021}, we introduce the intermediate operator $\mathcal{M}= \mathcal{T}\mathbf{A}\mathcal{T}$. We notice the following things. First of all, $\Vert \mathcal{T}\mathbf{T}^{-1} -I \Vert = \max_{i}\vert \sqrt{\dfrac{d_i}{\bar{d}_i}} -1 \vert$ and also, $\mathcal{M} = (\mathcal{T}\mathbf{T}^{-1})\mathbf{L}(\mathcal{T}\mathbf{T}^{-1})$, where $\mathbf{L}=\mathbf{T}\mathbf{A}\mathbf{T}$, since diagonal matrices commute. We wish to bound the distance of $\mathcal{L}-\mathbf{L}$ where $\mathcal{L} = \mathcal{T}\mathcal{A}\mathcal{T}$. To do so, we bound their respective distances to $\mathcal{M}$. Writing $E = \mathcal{T}\mathbf{T}^{-1}-I$, we have $\mathcal{M}-\mathbf{L} = E\mathbf{L}+\mathbf{L}E+E\mathbf{L}E$, so that $\Vert \mathcal{M}-\mathbf{L}\Vert \le \left(2\Vert E\Vert + \Vert E\Vert^2\right)\Vert \mathbf{L}\Vert$. By \Cref{MaximalEigenValueLaplace}, $\Vert \mathbf{L}\Vert \le 1$, and by \eqref{ForallI}, $\Vert E\Vert = \max_i\vert \sqrt{d_i/\bar{d}_i}-1\vert \le 4\sqrt{v\log(4n/\delta)/\bar{d}_{\min}} \le 1/2$; hence $\Vert E\Vert^2 \le \tfrac12\Vert E\Vert$ and $\Vert \mathcal{M}-\mathbf{L}\Vert \le \tfrac52\Vert E\Vert$. Thus, with probability $1-\dfrac{\delta}{2}$,
\begin{equation}\label{MBound1}
    \Vert \mathcal{M} - \mathbf{L} \Vert \le   10\sqrt{\dfrac{v \log(4n/\delta)}{\bar{d}_{\min}}} .
\end{equation}
Now, turning our attention to $\Vert \mathcal{M}-\mathcal{L} \Vert$, we notice that $\mathcal{M}-\mathcal{L} = \sum_{1\le i \le j \le n} \mathcal{T}\mathbf{X}_{ij}\mathcal{T}$, where once more, $\mathbf{X}_{ij} = \left(B_{ij} W_{ij} - p_{ij} \mu_{ij} \right)I_{ij}$. Define
\begin{equation*}
    \mathbf{Y}_{ij} = \mathcal{T}\mathbf{X}_{ij}\mathcal{T} = \dfrac{B_{ij} W_{ij} - p_{ij} \mu_{ij}}{\sqrt{\bar{d}_i\bar{d}_j}}I_{ij}.
\end{equation*}
Then $\mathcal{M}-\mathcal{L} = \sum_{1\le i \le j \le n} \mathbf{Y}_{ij}$. Since the eigenvalues of $I_{ij} \in \{-1,1,0 \}$, clearly the eigenvalues of $Y_{ij}$ are in $\{ \dfrac{\pm \left(B_{ij} W_{ij} - p_{ij} \mu_{ij} \right) }{\sqrt{\bar{d}_i\bar{d}_j}},\dfrac{\pm  p_{ij} \mu_{ij} }{\sqrt{\bar{d}_i\bar{d}_j}},0\}$. Hence, we have that $\Vert Y_{ij}\Vert \le \dfrac{v}{\sqrt{\bar{d}_i\bar{d}_j}}\le \dfrac{v}{\bar{d}_{\min}}$. As we will once more invoke Corollary $7.1$ of \cite{Oliveira2009}, this will serve as our $M$. It also holds that
\begin{equation*}
    \sum_{1\le i \le j \le n} \mathbb{E}\left(\mathbf{Y}_{ij}^2\right) \le \sum_{1\le i \le j \le n} \dfrac{\mathrm{Var}(B_{ij}W_{ij})}{\bar{d}_i\bar{d}_j}I_{ij}^2,
\end{equation*}
since the above matrix is diagonal, we have by the same arguments as in \cite{Gudhmundsson2021}, that $\sigma^2 \le \dfrac{v}{\bar{d}_{\min}}$. Following, Corollary $7.1$ of \cite{Oliveira2009}, with $t= 4\sqrt{v\log(4n/\delta)/\bar{d}_{\min}}$ we then have that 
\begin{equation}\label{MBound2}
    \begin{aligned}
\mathbb{P}\left(\|\mathcal{M}-\mathcal{L}\| \ge 4 \sqrt{\frac{\nu \log (4 n / \delta)}{\bar{d}_{\text {min }}}}\right) &\le 2n e^{-\dfrac{16v \log(4n/\delta)/\bar{d}_{\min}}{8\sigma^2 + 4 Mt}}\\ &\le 2n e^{-\dfrac{16v \log(4n/\delta)/\bar{d}_{\min}}{8v/\bar{d}_{\min} +8v/\bar{d}_{\min}  }} \\
&= \dfrac{\delta}{2},
    \end{aligned}
\end{equation}
where we used that $C$ is chosen so that $t\le 2$ and $\sigma^2,M \le \dfrac{v}{\bar{d}_{\min}}$. Using \eqref{MBound1} and \eqref{MBound2} together with the triangle inequality then gives the desired result. 
\end{proof}

\begin{proposition}\label{Prop.2}
    Suppose \Cref{Ass.NodeSum,Ass.2} are satisfied. Let $\Psi^s_p$ be the symmetrized population network effect associated to $\Psi$ as defined in \ref{SBOUProc}. Let $\bar{\Psi}^s_p$ be the same quantity, based on a GSBM. Introduce $\bar{\Psi}^s_{p-} = \bar{\Psi}^s_p-\mathrm{diag}(\bar{\Psi}^s_p) $. Then,  
    \begin{equation*}
        \Vert \Psi^s_p-\bar{\Psi}^s_{p-} \Vert = O\left( \dfrac{1}{\sigma_d} \right), \quad \text{ as } d\rightarrow \infty.
    \end{equation*}
\end{proposition}
\begin{proof}
Define the normalized Laplacians as $\mathcal{L} = \mathcal{D}^{-1/2}\left(\mathcal{A}^\top+\mathcal{A}\right)\mathcal{D}^{-1/2}$ and $\bar{\mathcal{L}}=\bar{\mathcal{D}}^{-1/2}\left(\bar{\mathcal{A}}^\top+\bar{\mathcal{A}}\right)\bar{\mathcal{D}}^{-1/2}$ with $\bar{L} = \bar{\mathcal{L}} - \mathrm{diag}(\bar{\mathcal{L}})$. Since $\Vert \Psi^s_p-\bar{\Psi}^s_{p-} \Vert = \psi_2 \Vert \mathcal{L} - \bar{L} \Vert$, from now on, we focus on $\Vert \mathcal{L} - \bar{L} \Vert$.
    We first note a few things. First of all, following  Proposition \ref{Prop:MinimalDegree}, it holds that $\bar{\mathcal{D}}_{ii}= \Omega\left( \log d \right)$, secondly, $\mathcal{D}_{ii}=\bar{\mathcal{D}}_{ii}- 2\bar{\mu}\theta_i^2 b_{s(i)s(i)}\le \bar{\mathcal{D}}_{ii}$. With this in mind, for $d$ large enough, clearly $\mathcal{D}_{ii}>1$. 
    Further, for $i\ne j$, it holds that $\mathcal{A}_{ij} = \bar{\mathcal{A}}_{ij}=\mu \theta_i\theta_jb_{s(i)s(j)}$. Since $\Vert \mathcal{L}-\bar{L}\Vert$ depends on $\mu$ only through $|\mu|$, we may without loss of generality assume $\mu>0$.
    We note that for $i=j$, $[\mathcal{L}-\bar{L}]_{ij} = 0$. Further,
    \begin{equation}\label{eqn:DecompDegree}
    \begin{aligned}\mathcal{D}_{ii}\mathcal{D}_{jj}&=\bar{\mathcal{D}}_{ii}\bar{\mathcal{D}}_{jj}-2\bar{\mu}\bar{\mathcal{D}}_{ii}\theta_j^2b_{s(j)s(j)}-2\bar{\mu}\bar{\mathcal{D}}_{jj}\theta_i^2b_{s(i)s(i)} + 4\bar{\mu}^2\theta_i^2b_{s(i)s(i)}\theta_j^2b_{s(j)s(j)} \\&\ge \bar{\mathcal{D}}_{ii}\bar{\mathcal{D}}_{jj}-2\bar{\mu}\bar{\mathcal{D}}_{ii}-2\bar{\mu}\bar{\mathcal{D}}_{jj} \\&= \bar{\mathcal{D}}_{ii}\bar{\mathcal{D}}_{jj}\left(1-\frac{2\bar{\mu}}{\bar{\mathcal{D}}_{jj}}-\frac{2\bar{\mu}}{\bar{\mathcal{D}}_{ii}}\right)\\
        &\ge C\bar{\mathcal{D}}_{ii}\bar{\mathcal{D}}_{jj},\quad \text{ for d large enough,}
        \end{aligned}
    \end{equation}
    where $C>0$ is independent of $i$, $j$ and $d$ using that $\min_{j}\bar{\mathcal{D}}_{jj} = \Omega(\log(d))$ and that $\bar{\mu}$ does not depend on $d$. Further, we note that for $a,b$ such that $a-b>0$, it holds that $\sqrt{a}- \sqrt{a-b} = \dfrac{b}{\sqrt{a} + \sqrt{a-b}}$. We will use this fact next with  $a=\bar{\mathcal{D}}_{ii}\bar{\mathcal{D}}_{jj}$ and $b=2\bar{\mu}\bar{\mathcal{D}}_{ii}\theta_j^2b_{s(j)s(j)}+2\bar{\mu}\bar{\mathcal{D}}_{jj}\theta_i^2b_{s(i)s(i)} - 4\bar{\mu}^2\theta_i^2b_{s(i)s(i)}\theta_j^2b_{s(j)s(j)}$.  With this in mind, consider $i\ne j$. Then,
\begin{align}
    \vert [\mathcal{L}-\bar{L}]_{ij} \vert &=  \vert  \dfrac{\mathcal{A}_{ji}+\mathcal{A}_{ij}}{\sqrt{\mathcal{D}_{ii}\mathcal{D}_{jj}}}-\dfrac{\bar{\mathcal{A}}_{ji}+\bar{\mathcal{A}}_{ij}}{\sqrt{\bar{\mathcal{D}}_{ii}\bar{\mathcal{D}}_{jj}}} \vert \notag \\
    &=\dfrac{\mathcal{A}_{ji}+\mathcal{A}_{ij}}{\sqrt{\mathcal{D}_{ii}\mathcal{D}_{jj}}} -\dfrac{\bar{\mathcal{A}}_{ji}+\bar{\mathcal{A}}_{ij}}{\sqrt{\bar{\mathcal{D}}_{ii}\bar{\mathcal{D}}_{jj}}} \notag \\
    &=\dfrac{\left(\bar{\mathcal{A}}_{ji}+\bar{\mathcal{A}}_{ij}\right)\left(\sqrt{\bar{\mathcal{D}}_{ii}\bar{\mathcal{D}}_{jj}}-\sqrt{\mathcal{D}_{ii}\mathcal{D}_{jj}}\right)}{\sqrt{\bar{\mathcal{D}}_{ii}\bar{\mathcal{D}}_{jj}}\sqrt{\mathcal{D}_{ii}\mathcal{D}_{jj}}} \label{eqn:AbsValEntryWise} \\
    &\le \dfrac{\left(\bar{\mathcal{A}}_{ji}+\bar{\mathcal{A}}_{ij}\right)\left(2\bar{\mu}\bar{\mathcal{D}}_{ii}\theta_j^2b_{s(j)s(j)}+2\bar{\mu}\bar{\mathcal{D}}_{jj}\theta_i^2b_{s(i)s(i)} - 4\bar{\mu}^2\theta_i^2b_{s(i)s(i)}\theta_j^2b_{s(j)s(j)}\right)}{\sqrt{\bar{\mathcal{D}}_{ii}\bar{\mathcal{D}}_{jj}}\sqrt{\mathcal{D}_{ii}\mathcal{D}_{jj}}\left(\sqrt{\bar{\mathcal{D}}_{ii}\bar{\mathcal{D}}_{jj}}+\sqrt{\mathcal{D}_{ii}\mathcal{D}_{jj}}\right)} \notag \\
    &\le \dfrac{\left(\bar{\mathcal{A}}_{ji}+\bar{\mathcal{A}}_{ij}\right)\left(2\bar{\mu}\bar{\mathcal{D}}_{ii}\theta_j^2b_{s(j)s(j)}+2\bar{\mu}\bar{\mathcal{D}}_{jj}\theta_i^2b_{s(i)s(i)} - 4\bar{\mu}^2\theta_i^2b_{s(i)s(i)}\theta_j^2b_{s(j)s(j)}\right)}{\sqrt{C}\bar{\mathcal{D}}_{ii}\bar{\mathcal{D}}_{jj}\sqrt{\bar{\mathcal{D}}_{ii}\bar{\mathcal{D}}_{jj}}\left(1+\sqrt{C}\right)} \notag \\
    &=\dfrac{\left(\bar{\mathcal{A}}_{ji}+\bar{\mathcal{A}}_{ij}\right)2\bar{\mu}\theta_j^2b_{s(j)s(j)}}{\sqrt{C}\bar{\mathcal{D}}_{jj}\sqrt{\bar{\mathcal{D}}_{ii}\bar{\mathcal{D}}_{jj}}\left(1+\sqrt{C}\right)} \notag \\
    &+\dfrac{\left(\bar{\mathcal{A}}_{ji}+\bar{\mathcal{A}}_{ij}\right)2\bar{\mu}\theta_i^2b_{s(i)s(i)}}{\sqrt{C}\bar{\mathcal{D}}_{ii}\sqrt{\bar{\mathcal{D}}_{ii}\bar{\mathcal{D}}_{jj}}\left(1+\sqrt{C}\right)} \notag \\
    &+\dfrac{\left(\bar{\mathcal{A}}_{ji}+\bar{\mathcal{A}}_{ij}\right)4\bar{\mu}^2\theta_i^2b_{s(i)s(i)}\theta_j^2b_{s(j)s(j)}}{\sqrt{C}\bar{\mathcal{D}}_{ii}\bar{\mathcal{D}}_{jj}\sqrt{\bar{\mathcal{D}}_{ii}\bar{\mathcal{D}}_{jj}}\left(1+\sqrt{C}\right)}. \notag
\end{align}
    Next, we will use that $\bar{\mathcal{D}}_{ii}=\bar{\mu}\sigma_d\theta_i D_{B_{s(i)s(i)}}$ following the proof of  \Cref{Lemma:2}. $D_{B_{s(i)s(i)}}$ is defined in that Supplementary material (under the proof of \Cref{Lemma:2}) and hence, it is independent of $d$. We thus have that
    \begin{equation*}
        \begin{aligned}
            \dfrac{\bar{\mathcal{D}}_{jj}}{\theta_j} = \bar{\mu}\sigma_d D_{B_{s(j)s(j)}},
        \end{aligned}
    \end{equation*}
    in turn,
    \begin{equation*}
        \begin{aligned}
            \dfrac{\left(\bar{\mathcal{A}}_{ji}+\bar{\mathcal{A}}_{ij}\right)2\bar{\mu}\theta_j^2b_{s(j)s(j)}}{\sqrt{C}\bar{\mathcal{D}}_{jj}\sqrt{\bar{\mathcal{D}}_{ii}\bar{\mathcal{D}}_{jj}}\left(1+\sqrt{C}\right)} &= \dfrac{\left(\bar{\mathcal{A}}_{ji}+\bar{\mathcal{A}}_{ij}\right)2\theta_jb_{s(j)s(j)}}{\sqrt{C}\sigma_d D_{B_{s(j)s(j)}}\sqrt{\bar{\mathcal{D}}_{ii}\bar{\mathcal{D}}_{jj}}\left(1+\sqrt{C}\right)}\\ &\le\dfrac{2\theta_j\left(\bar{\mathcal{A}}_{ji}+\bar{\mathcal{A}}_{ij}\right)}{\sqrt{C}\sigma_d D_{B_{s(j)s(j)}}\sqrt{\bar{\mathcal{D}}_{ii}\bar{\mathcal{D}}_{jj}}\left(1+\sqrt{C}\right)}\\&=\dfrac{2\mu\theta_i\theta_j^{2}\left(b_{s(j)s(i)}+b_{s(i)s(j)}\right)}{\sqrt{C}\sigma_d D_{B_{s(j)s(j)}}\sqrt{\bar{\mathcal{D}}_{ii}\bar{\mathcal{D}}_{jj}}\left(1+\sqrt{C}\right)}\\&=\dfrac{2\mu\sqrt{\theta_i\theta_j^3}\left(b_{s(j)s(i)}+b_{s(i)s(j)}\right)}{\sqrt{C}\bar{\mu}\sigma_d^2 D_{B_{s(j)s(j)}}\sqrt{D_{B_{s(i)s(i)}}D_{B_{s(j)s(j)}}}\left(1+\sqrt{C}\right)}\\
            &\le \dfrac{2\theta_j}{C^\prime\sigma_d^2},
        \end{aligned}
    \end{equation*}
    where $ C^\prime = \sqrt{C}\bar{\mu} D_B^* \left(1+\sqrt{C}\right)$ and $D_B^* = \min_{j}[D_{B}]_{s(j)s(j)}>0 $ and does not depend on $d$, since $\bold{B}+\bold{B}^\top$ is positive definite and does not depend on $d$. The above upper bound thus holds for any $i\in \{1,\dots,d\}$.
    In turn, we conclude that 
    \begin{equation*}
        \begin{aligned}
        \sum_{j\in\{1,\dots,d\}\backslash \{i\}}\dfrac{\left(\bar{\mathcal{A}}_{ji}+\bar{\mathcal{A}}_{ij}\right)2\theta_j^2b_{s(j)s(j)}}{\sqrt{C}\bar{\mathcal{D}}_{jj}\sqrt{\bar{\mathcal{D}}_{ii}\bar{\mathcal{D}}_{jj}}\left(1+\sqrt{C}\right)} & \le \sum_{j=1}^d\dfrac{2\theta_j}{C^\prime\sigma_d^2}\\
        &= \sum_{l=1}^k\sum_{j \in \mathcal{V}_l}\dfrac{2\theta_j}{C^\prime\sigma_d^2}\\
        &= \dfrac{2k\sigma_d}{C^\prime\sigma_d^2}\\
        &= O\left(\dfrac{1}{\sigma_d}\right), \quad \text{ as } d\rightarrow \infty.
        \end{aligned}
    \end{equation*}
    By symmetry, the second term in \eqref{eqn:AbsValEntryWise} shares this bound. Similarly,  
    \begin{equation*}
        \sum_{j\in\{1,\dots,d\}\backslash \{i\}}\dfrac{\left(\bar{\mathcal{A}}_{ji}+\bar{\mathcal{A}}_{ij}\right)4\bar{\mu}^2\theta_i^2b_{s(i)s(i)}\theta_j^2b_{s(j)s(j)}}{\sqrt{C}\bar{\mathcal{D}}_{ii}\bar{\mathcal{D}}_{jj}\sqrt{\bar{\mathcal{D}}_{ii}\bar{\mathcal{D}}_{jj}}\left(1+\sqrt{C}\right)} = O\left(\dfrac{1}{\sigma_d^2}\right), \quad \text{ as } d\rightarrow \infty.
    \end{equation*}
    Consequently, $\max_{i\in \{1,\dots,d\}} R_i = O\left(\dfrac{1}{\sigma_d}\right)$, as $d$ tends to infinity. Here  $R_i := \sum_{j\in\{1,\dots,d\}\backslash \{i\}}\vert [\mathcal{L}-\bar{L}]_{ij} \vert$. Let $\lambda^*$ denote the maximal eigenvalue of  $\mathcal{L}-\bar{L}$. Following the Gershgorin Circle Theorem, we conclude that 
    \begin{equation*}
        \vert \lambda^* \vert  = O\left(\dfrac{1}{\sigma_d}\right), \quad \text{ as } d\rightarrow \infty.
    \end{equation*}
    Since $\mathcal{L}-\bar{L}$ is a normal matrix we know that $\Vert \mathcal{L}-\bar{L}\Vert = \vert \lambda^* \vert $ which in turns completes the proof.
\end{proof}
\begin{proposition}\label{Prop.3}
   Assume that Assumptions \ref{Ass.NodeSum} and \ref{Ass.2} are satisfied. Then, it holds that  
    \begin{equation*}
        \Vert  \mathrm{diag}(\bar{\Psi}^s_p)  \Vert = O\left( \dfrac{1}{\sigma_d} \right),\quad \text{ as } d\rightarrow \infty.
    \end{equation*}
\end{proposition}

\begin{proof}[Proof of Proposition \ref{Prop.3}]
    Let $\bar{\mathcal{L}}$ be as in \ref{Prop.2}. Since $\mathrm{diag}(\bar{\Psi}_p
^s) = \psi_2 \mathrm{diag}(\bar{\mathcal{L}})$,we focus on $\mathrm{diag}(\bar{\mathcal{L}})$ given as in \Cref{Lemma:1}. We note that for any $i \in \{1,2,\dots,d\}$, we have that
    \begin{equation*}
        \mathrm{diag}(\bar{\mathcal{L}})_{ii} = \dfrac{2\theta_i^2\mu b_{s(i)s(i)}}{\bar{\mathcal{D}}_{ii}}= \dfrac{2\mu\theta_i b_{s(i)s(i)}}{\bar{\mu} \sigma_d D_{B_{s(i)s(i)}}} \le \dfrac{2\mu}{\bar{\mu} D_B^*\sigma_d},
    \end{equation*}
where $D_B^*$ is as defined in \Cref{Lemma:1}. If we let $\rho_d = \max_{i\in \{1,\dots,d\}} \mathrm{diag}(\bar{\mathcal{L}})_{ii}$, the above calculations tell us that $\rho_d =O\left(\dfrac{1}{ \sigma_d}\right)$. Since $\mathrm{diag}(\bar{\mathcal{L}})$ is a diagonal matrix, its entries are its eigenvalues and so $\rho_d$ is the maximal eigenvalue of $\mathrm{diag}(\bar{\mathcal{L}})$. Further, since $\mathrm{diag}(\bar{\mathcal{L}})$ is a positive, diagonal matrix, $\Vert \mathrm{diag}(\bar{\mathcal{L}}) \Vert = \rho_d$, which concludes the proof.
\end{proof}
\begin{proposition}\label{Prop.4}
Assume that \Cref{Ass.NodeSum,Ass.2} are satisfied. Let $\Psi^s$ be the symmetrized drift matrix and let $\Psi_p^s$ be the population analogous. Then, with high probability
    \begin{equation*}
        \Vert \Psi^s - \Psi_p^s \Vert = O\left(\sqrt{\dfrac{d \log d}{\sigma_d^2}} \right).
    \end{equation*}
\end{proposition}
\begin{proof}
  We notice, that $\Psi^s = \psi_2 D^{-1/2}\left(A + A^\top \right)D^{-1/2}$ and \\$\Psi^s_p = \psi_2 \mathcal{D}^{-1/2}\left(\mathcal{A} + \mathcal{A}^\top \right)\mathcal{D}^{-1/2}$. We can view $A + A^\top$ as the realisations of an undirected, weighted graph with mean given by $2\mu$, node specific probabilities given by $p_{ij}= \theta_i\theta_j(b_{s(i)s(j)} +b_{s(j)s(i)})$ for $i\ne j$ and $0$ for $j=i$. Following Proposition \ref{Prop:MinimalDegree}, we also have that $\mathcal{D}_{\min}= \Omega\left( \dfrac{\sigma_d^2}{d}\right)$ as $d$ tends to infinity. Under \Cref{Ass.2} we have $\sigma_d=\omega((d\log d)^{1/2})$, so that $\sigma_d^2/d = \omega(\log d)$ and hence $\mathcal{D}_{\min}/\log d\to\infty$. Fix any $c>0$, set $\delta=2d^{-c}$, and let $C$ be the corresponding constant from \Cref{ConcentrationLemma}; since $\mathcal{D}_{\min}/\log d\to\infty$, the hypothesis $\mathcal{D}_{\min}\ge C\log d$ holds for all $d$ large enough. With $v$ as in the Lemma, then for any $d$ satisfying $d^{-c}\le 1/4$, we have that
  \begin{equation*}
      \mathbb{P}\left( \Vert \Psi^s - \Psi^s_p \Vert \ge C^{\prime \prime} \sqrt{\dfrac{d \log d}{\sigma_d^2}} \right) \le 2d^{-c},
  \end{equation*}
  where $C^{\prime \prime}$ is another constant. This in turn finishes the proof.
\end{proof}

\begin{proposition}\label{Prop.5}
    Let $Q = Q(\psi)$ be defined as in \Cref{SBOUProc}. Then, if $\psi_1 > \psi_2$, $\lambda_1\left(Q(\psi) + Q(\psi)^\top\right) > 0$ and $Q(\psi) \in \mathcal{M}^+(\mathbb{R}^d)$.
\end{proposition}
\begin{proof}
    Consider $\bar{Q} = 2\psi_1 I + \psi_2 D^{-1}\left(A^\top + A\right)$. By the arguments in the proof of \Cref{Thm.1}, $\bar{Q}$ and $Q(\psi) + Q(\psi)^\top$ share the same eigenvalues. Applying Gershgorin's circle theorem to $\bar{Q}$, the non-diagonal row sum $R_i = \sum_{j \ne i} \psi_2 \dfrac{\vert A_{ij} +A_{ji} \vert}{D_{ii}}\le \psi_2$, while the diagonal entry is $\bar{Q}_{ii} = 2\psi_1$. Hence every eigenvalue of $\bar{Q}$ lies in $[2\psi_1 - \psi_2,\, 2\psi_1 + \psi_2]$, which is strictly positive when $\psi_1 > \psi_2$. In particular $\lambda_1\left(Q(\psi) + Q(\psi)^\top\right) > 0$, and $Q(\psi) \in \mathcal{M}^+(\mathbb{R}^d)$ follows under the assumptions of the proposition.
\end{proof}
\begin{proposition}\label{Prop.6}
     Let $\psi_1 > \psi_2$ and suppose the assumptions of \Cref{DiscreteEstimatorConsistencyBothCases} are satisfied. Then, with high probability
    \begin{equation*}
        \Vert \widehat{\Psi}^s - \Psi^s \Vert \le f(d,\mathcal{P}_T), \quad \text{ as } d \rightarrow \infty.
    \end{equation*}
\end{proposition}
\begin{proof}
The result follows from \Cref{DiscreteEstimatorConsistencyBothCases}. The assumptions of that theorem include the high-level condition \Cref{Ass.H}.4. Under the additional restriction $\psi_1>\psi_2$, however, \Cref{Prop.5} verifies that \Cref{Ass.H}.4 holds for the drift matrix of the SBOU process. The conclusion therefore follows immediately from \Cref{DiscreteEstimatorConsistencyBothCases}.
\end{proof}
\begin{proposition}\label{Prop.7}
   Grant \Cref{Ass.NodeSum,Ass.2}, the assumptions of \Cref{DiscreteEstimatorConsistencyBothCases}, and suppose $\psi_1 > \psi_2 $. Then, with high probability,
    \begin{equation*}
        \Vert \widehat{\Psi}^s - \bar{\Psi}_p^s \Vert =  O \left( f(d,\mathcal{P}_T) +\sqrt{\dfrac{d \log d}{\sigma_d^2}} \right).
    \end{equation*}
\end{proposition}
\begin{proof}
    The proof follows by the triangle inequality and \Cref{Prop.2,Prop.3,Prop.4,Prop.5,Prop.6}. Indeed, by the triangle inequality we have that
    \begin{equation*}
        \begin{aligned}
        \Vert \widehat{\Psi}^s - \bar{\Psi}_p^s \Vert &= \Vert \widehat{\Psi}^s - \Psi^s +\Psi^s -\Psi^s_p  +\Psi^s_p -\bar{\Psi}_{p}^s \Vert \\
        &= \Vert \widehat{\Psi}^s - \Psi^s +\Psi^s -\Psi^s_p  +\Psi^s_p -\bar{\Psi}_{p-}^s - \mathrm{diag}\left(\bar{\Psi}_{p}^s\right)\Vert \\
        &\le \Vert \widehat{\Psi}^s - \Psi^s \Vert +\Vert\Psi^s -\Psi^s_p\Vert  +\Vert\Psi^s_p -\bar{\Psi}_{p-}^s\Vert+\Vert \mathrm{diag}\left(\bar{\Psi}_{p}^s\right) \Vert,
        \end{aligned}
    \end{equation*}
    where we used \Cref{Prop.2,Prop.3,Prop.4,Prop.5,Prop.6} and recalled that $\bar{\Psi}^s_{p-} = \bar{\Psi}^s_p-\mathrm{diag}(\bar{\Psi}^s_p) $.
\end{proof}

\begin{proposition}\label{SetInclusion}
    Let $M =\left\{i: \|\widehat{\mathbf{C}}_{\hat{s}(i) \bullet}-\mathcal{C}_{s(i) \bullet} \|_2 \geq \sqrt{1 / 2}\right\}$. Then 
    \begin{equation*}
    \mathcal{M} \subseteq M
    \end{equation*}
\end{proposition}
\begin{proof}[Proof of \Cref{SetInclusion}]
    This follows exactly by the calculation in \cite{Rohe2011} for Theorem  $4.4$ only minimally changed to our setup. The result is also mentioned in \cite{Qin2013}. We write it out in our setting for completeness only, but do not claim any of the ideas.  We show the desired result by showing that $M^c \subseteq \mathcal{M}^c$, where $M^c =\left\{i: \|\widehat{\mathbf{C}}_{\hat{s}(i) \bullet}-\mathcal{C}_{s(i) \bullet} \|_2 < \sqrt{1 / 2}\right\}$ and $\mathcal{M}^c=\{i:  \left\|\widehat{\mathbf{C}}_{\hat{s}(i) \bullet} -\mathcal{C}_{s(i)\bullet}\right\|_2< \left\|\widehat{\mathbf{C}}_{\hat{s}(i) \bullet}-\mathcal{C}_{s(j)\bullet}\right\|_2 \forall j  \text{ s.t. } s(j)\ne s(i)\}$. With this in mind, take an $i \in M^c$. Then $ \Vert \widehat{\mathbf{C}}_{\hat{s}(i) \bullet}-\mathcal{C}_{s(i) \bullet}\Vert_2 < \sqrt{1 / 2}$. Following Lemma \ref{Lemma:2}, we have that $\mathcal{C}_{s(i)\bullet} = (\mathcal{X}^*\mathcal{O})_{i\bullet} = \mathbf{Z}_{i\bullet}U\mathcal{O} = \mathbf{Z}_{i\bullet}V$, where $V=U\mathcal{O}$ is an orthonormal rotation matrix. For $j$ such that $s(i) \ne s(j)$, i.e. nodes in different communities, we clearly have that 
    \begin{equation*}
        \Vert \mathcal{C}_{s(i)\bullet} -\mathcal{C}_{s(j)\bullet} \Vert_2 =  \Vert \mathbf{Z}_{i\bullet}V -\mathbf{Z}_{j\bullet}V \Vert_2 =  \Vert \mathbf{Z}_{i\bullet} -\mathbf{Z}_{j\bullet} \Vert_2 = \sqrt{2},
    \end{equation*}
where we used that $V$ is an orthonormal matrix. We then note, that for any $j$ and  $i$ such that $s(i) \ne s(j)$, we have by the reverse triangle inequality that
\begin{equation*}
    \begin{aligned}
    \Vert \widehat{\mathbf{C}}_{\hat{s}(i) \bullet} - \mathcal{C}_{s(j) \bullet} \Vert_2 &\ge \vert \Vert \widehat{\mathbf{C}}_{\hat{s}(i) \bullet}- \mathcal{C}_{ s(i)\bullet} \Vert_2 - \Vert \mathcal{C}_{ s(j)\bullet}-\mathcal{C}_{s(i) \bullet}   \Vert_2 \vert \\
    &=  \Vert \mathcal{C}_{s(j)\bullet}-\mathcal{C}_{s(i) \bullet}   \Vert_2 -\Vert \widehat{\mathbf{C}}_{\hat{s}(i) \bullet} - \mathcal{C}_{s(i)\bullet} \Vert_2 \\
    &> \sqrt{2} - \sqrt{1/2} \\
    &= \sqrt{2} - 1/2\sqrt{2} \\
    &= \sqrt{1/2} \\&>\Vert\widehat{\mathbf{C}}_{\hat{s}(i) \bullet} - \mathcal{C}_{s(i)\bullet} \Vert_2.
    \end{aligned}
\end{equation*}
Hence, $i \in \mathcal{M}^c$
\end{proof}

\begin{lemma}\label{EstimatorsBAndTheta}
    Assume that $\sigma_d \rightarrow \infty$ as $d\rightarrow \infty$. Let $s$ be the map from \Cref{Remark:sMapping}. Let $B_{ij}$ be independent Bernoulli random variables with $p_{ij}= \theta_i\theta_jb_{s(i)s(j)}$ for all $i,j\in \{1,\dots,d\}$ such that $i\ne j$ and $p_{ij}=0$ for $i=j$. Further, assume that $\theta_i$ can be partitioned in $k$ groups, $\left(\mathcal{V}_1,\dots,\mathcal{V}_k\right)$ such that $\sum_{j\in \mathcal{V}_l}\theta_j=\sigma_d$. Also assume that for any $l,m\in \{1,\dots,k\}$, $b_{lm}$ does not depend on $d$. For $i\in \{1,\dots,d\}$, let $d_i =\sum_{j=1}^d B_{ij}$. And finally assume that for any $s \in \{1,\dots,k\}$, $\sum_{l=1}^k b_{sl}>0$. Let $\hat{\theta}_i=\dfrac{\sum_{j=1}^d B_{ij}}{\sum_{j \in \mathcal{V}_{s(i)}} d_j}$, $\hat{b}_{lk} = \sum_{i \in \mathcal{V}_l, j \in \mathcal{V}_k} B_{ij}$ and finally let $\hat{\mathcal{A}}_{ij} = \hat{\theta}_i \hat{\theta}_j \hat{b}_{s(i)s(j)}$ denote the estimated probability of an edge between $i$ and $j$ of the $GSBM2
(\bold{Z},\bold{B},\Theta,W)$. It holds that
    \begin{equation*}
        \begin{aligned}
            \sigma_d\hat{\theta}_j &\overset{\mathbb{P}}{\rightarrow}  \theta_j,\quad \text{ as } d\rightarrow \infty. \\
            \dfrac{\hat{b}_{lk}}{\sigma_d^2} & \overset{\mathbb{P}}{\rightarrow}   b_{lk},\quad \text{ as } d\rightarrow \infty.\\
            \hat{\mathcal{A}}_{ij} & \overset{\mathbb{P}}{\rightarrow} \theta_i\theta_j b_{s(i)s(j)},\quad \text{ as } d\rightarrow \infty.
        \end{aligned}
    \end{equation*}
\end{lemma}
\begin{proof}
Let $\bar{\mathcal{A}}$ be the population adjacency matrix of a $GSBM
(\bold{Z},\bold{B},\Theta,W)$ where $W$ is the degenerate random variable $1$, i.e. there are no weights associated to the links. If a realisation of a $GSBM
(\bold{Z},\bold{B},\Theta,W)$  with weights is observed one simply sets the non-zero entries to $1$ and finds oneself in the correct setting. Then,
\begin{equation*}
    \dfrac{\sum_{j=1}^d \bar{\mathcal{A}}_{ij}}{\sum_{j\in\mathcal{V}_{s(i)}}\sum_{k=1}^d \bar{\mathcal{A}}_{jk}} = \dfrac{\sigma_d\theta_i \sum_{l=1}^k b_{s(i)l} }{\sigma_d^2 \sum_{l=1}^k b_{s(i)l}} = \dfrac{\theta_i}{\sigma_d}.
\end{equation*}
Also,
\begin{equation*}
    \sum_{i \in \mathcal{V}_l, j \in \mathcal{V}_k} \bar{\mathcal{A}}_{ij} = b_{lk}\sum_{i \in \mathcal{V}_l, j \in \mathcal{V}_k} \theta_i \theta_j = \sigma_d^2 b_{lk}.
\end{equation*}
Now, let $\mathcal{A}$ be the population adjacency matrix of a $GSBM2
(\bold{Z},\bold{B},\Theta,W)$. It holds that
\begin{equation*}
    \begin{aligned}
          \dfrac{\sum_{j=1}^d \mathcal{A}_{ij}}{\sum_{j\in\mathcal{V}_{s(i)}}\sum_{k=1}^d \mathcal{A}_{jk}} & =  \dfrac{\sum_{j=1}^d \bar{\mathcal{A}}_{ij}-\theta_i^2 b_{s(i)s(i)}}{\sum_{j\in\mathcal{V}_{s(i)}}\sum_{k=1}^d \bar{\mathcal{A}}_{jk}-\sum_{j\in\mathcal{V}_{s(i)}}\theta_j^2 b_{s(i)s(i)}}.
    \end{aligned}
\end{equation*}
We now note that
\begin{equation*}
    \dfrac{\sum_{j=1}^d \mathcal{A}_{ij}}{\sigma_d}=\dfrac{\sum_{j=1}^d \bar{\mathcal{A}}_{ij}-\theta_i^2 b_{s(i)s(i)}}{\sigma_d}\overset{d\rightarrow \infty}{\rightarrow} \theta_i \sum_{l=1}^k b_{s(i)l}.
\end{equation*}
Further,
\begin{equation}\label{StandardLimit1}
    \dfrac{\sum_{j\in\mathcal{V}_{s(i)}}\sum_{k=1}^d \mathcal{A}_{jk}}{\sigma_d^2}=\dfrac{\sum_{j\in\mathcal{V}_{s(i)}}\sum_{k=1}^d \bar{\mathcal{A}}_{jk}-\sum_{j\in\mathcal{V}_{s(i)}}\theta_j^2 b_{s(i)s(i)}}{\sigma_d^2}\overset{d\rightarrow \infty}{\rightarrow}\sum_{l=1}^k b_{s(i)l}.
\end{equation}
In turn, we conclude that
\begin{equation}\label{ConvToTheta}
     \sigma_d\dfrac{\sum_{j=1}^d \mathcal{A}_{ij}}{\sum_{j\in\mathcal{V}_{s(i)}}\sum_{k=1}^d \mathcal{A}_{jk}} \overset{d\rightarrow \infty}{\rightarrow} \theta_i.
\end{equation}
Similarly, for $l \ne k$,
\begin{equation*}
    \sum_{i \in \mathcal{V}_l, j \in \mathcal{V}_k} \mathcal{A}_{ij}  = \sum_{i \in \mathcal{V}_l, j \in \mathcal{V}_k} \bar{\mathcal{A}}_{ij}.
\end{equation*}
And, for $l=k$,
\begin{equation*}
    \sum_{i \in \mathcal{V}_l, j \in \mathcal{V}_k} \mathcal{A}_{ij}  = \sum_{i \in \mathcal{V}_l, j \in \mathcal{V}_k} \bar{\mathcal{A}}_{ij} - \sum_{i \in\mathcal{V}_l }\bar{\mathcal{A}}_{ii}= b_{lk}\sum_{i \in \mathcal{V}_l, j \in \mathcal{V}_k} \theta_i \theta_j - \sum_{i \in\mathcal{V}_l } \theta_i^2 b_{s(i)s(i)} = \sigma_d^2 b_{lk} - \sum_{i \in\mathcal{V}_l } \theta_i^2 b_{s(i)s(i)}.
\end{equation*}
We thus have that
\begin{equation}\label{BoundAWithingroupSum}
       \sigma_d^2 b_{lk} - \sigma_d \le \sum_{i \in \mathcal{V}_l, j \in \mathcal{V}_k} \mathcal{A}_{ij}  \le \sigma_d^2 b_{lk} + \sigma_d.
\end{equation}
In turn,
\begin{equation}\label{ConvOfSums1}
    \dfrac{1}{\sigma_d^2}\sum_{i \in \mathcal{V}_l, j \in \mathcal{V}_k} \mathcal{A}_{ij} \overset{d\rightarrow \infty}{\rightarrow}  b_{lk}.
\end{equation}
We underline that $\mathbb{E}\left(B_{ij}\right) = \mathcal{A}_{ij}$. Note now that, for any $i\in \{1,\dots,d\}$,
\begin{equation*}
    \begin{aligned}
    \mathrm{Var}\left(\sum_{j=1}^d B_{ij}\right) &= \sum_{j=1}^d \mathrm{Var}\left( B_{ij}\right) \\
    &=\sum_{j=1}^d p_{ij}(1-p_{ij}) \\
    &\le \sum_{j=1}^d\theta_i\theta_{j}b_{s(i)s(j)}\\
    &= \theta_i \sigma_d \sum_{l=1}^kb_{s(i)l}.
    \end{aligned}
\end{equation*}
In turn, we conclude using Markov's inequality that, for any $\epsilon>0$,
\begin{equation}\label{ConvOfSums2}
    \mathbb{P}\left(\left( \dfrac{\sum_{j=1}^dB_{ij}-\sum_{j=1}^d\mathcal{A}_{ij}}{\sigma_d}\right)^2 > \epsilon\right) \le \dfrac{\mathrm{Var}\left(\sum_{j=1}^d B_{ij}\right)}{\sigma_d^2 \epsilon} \le \dfrac{\theta_i \sum_{l=1}^k b_{s(i)l}}{\sigma_d} \overset{d\rightarrow \infty}{\rightarrow} 0.
\end{equation}
Hence, for any $\epsilon>0$,
\begin{equation}\label{LimitInProb1}
    \begin{aligned}
    \mathbb{P}\left(\left| \dfrac{\sum_{j=1}^dB_{ij}}{\sigma_d} - \theta_i \sum_{l=1}^kb_{s(i)l} \right|>\epsilon
\right) &=\mathbb{P}\left(\left| \dfrac{\sum_{j=1}^dB_{ij}}{\sigma_d} - \dfrac{\sum_{j=1}^d\mathcal{A}_{ij}}{\sigma_d} +\dfrac{\sum_{j=1}^d\mathcal{A}_{ij}}{\sigma_d} -\theta_i \sum_{l=1}^kb_{s(i)l} \right|>\epsilon
\right)\\
&\le \mathbb{P}\left(\left| \dfrac{\sum_{j=1}^dB_{ij}}{\sigma_d} - \dfrac{\sum_{j=1}^d\mathcal{A}_{ij}}{\sigma_d}  \right|>\epsilon/2
\right) \\&+ \mathbb{P}\left(\left| \dfrac{\sum_{j=1}^d\mathcal{A}_{ij}}{\sigma_d} -\theta_i \sum_{l=1}^kb_{s(i)l} \right|>\epsilon/2
\right)\overset{d\rightarrow \infty }{\rightarrow} 0,
\end{aligned}
\end{equation}
where we used \eqref{ConvOfSums1} and \eqref{ConvOfSums2} for the last assertion.
Now,
\begin{equation*}
    \begin{aligned}
        \mathrm{Var}\left(\sum_{j\in\mathcal{V}_{s(i)}}\sum_{k=1}^d B_{jk}\right) &\le \sum_{j\in\mathcal{V}_{s(i)}}\sum_{k=1}^d \mathcal{A}_{jk}\\
        &\le \sum_{j\in\mathcal{V}_{s(i)}}\sum_{k=1}^d \bar{\mathcal{A}}_{jk}\\
        &= \sigma_d^2\sum_{l=1}^k b_{s(i)l}.
    \end{aligned}
\end{equation*}
In turn, we conclude that for any $\epsilon>0$
\begin{equation}\label{ConvOfSums3}
    \begin{aligned}
    \mathbb{P}\left(\left( \dfrac{\sum_{j\in\mathcal{V}_{s(i)}}\sum_{k=1}^d B_{jk}-\sum_{j\in\mathcal{V}_{s(i)}}\sum_{k=1}^d \mathcal{A}_{jk}}{\sigma_d^2}\right)^2 > \epsilon\right) &\le \dfrac{\mathrm{Var}\left(\sum_{j\in\mathcal{V}_{s(i)}}\sum_{k=1}^d B_{jk}\right)}{\sigma_d^4 \epsilon} \\&\le \dfrac{ \sum_{l=1}^k b_{s(i)l}}{\sigma_d^2} \overset{d\rightarrow \infty}{\rightarrow} 0.
    \end{aligned}
\end{equation}
As such, using \eqref{ConvOfSums3} and \eqref{StandardLimit1} we have
\begin{equation}\label{LimitInProb2}
    \begin{aligned}
   \mathbb{P}\left( \left|\dfrac{\sum_{j\in\mathcal{V}_{s(i)}}\sum_{k=1}^d B_{jk}}{\sigma_d^2} - \sum_{l=1}^k b_{s(i)l} \right| > \epsilon\right)&\le \mathbb{P}\left( \left|\dfrac{\sum_{j\in\mathcal{V}_{s(i)}}\sum_{k=1}^d B_{jk}}{\sigma_d^2} - \sum_{j\in\mathcal{V}_{s(i)}}\sum_{k=1}^d \mathcal{A}_{jk} \right| > \epsilon/2\right) \\
   &+  \mathbb{P}\left( \left|\dfrac{\sum_{j\in\mathcal{V}_{s(i)}}\sum_{k=1}^d \mathcal{A}_{jk}}{\sigma_d^2} - \sum_{l=1}^k b_{s(i)l} \right| > \epsilon/2\right)\overset{d\rightarrow \infty}{\rightarrow} 0,
    \end{aligned}
\end{equation}
since convergence in distribution is equivalent to convergence in probability to a degenerate random variable, Slutsky's theorem, in conjunction with \eqref{LimitInProb1} and \eqref{LimitInProb2}, implies that
\begin{equation*}
      \sigma_d\hat{\theta}_j \overset{\mathbb{P}}{\rightarrow}  \theta_j\quad \text{ as } d\rightarrow \infty.
\end{equation*}
We now notice that
\begin{equation*}
    \begin{aligned}
        \mathrm{Var}\left(\sum_{i \in \mathcal{V}_l, j \in \mathcal{V}_k} B_{ij}\right) &\le\sum_{i \in \mathcal{V}_l, j \in \mathcal{V}_k} \mathcal{A}_{ij} = b_{lk}\sigma_d^2.
    \end{aligned}
\end{equation*}
Hence,
\begin{equation}\label{ConvofSum4}
    \mathbb{P}\left(\left( \dfrac{\sum_{i \in \mathcal{V}_l, j \in \mathcal{V}_k}B_{ij}-\sum_{i \in \mathcal{V}_l, j \in \mathcal{V}_k}\mathcal{A}_{ij}}{\sigma_d^2}\right)^2 > \epsilon\right) \overset{d\rightarrow \infty}{\rightarrow} 0,
\end{equation}
where we used \eqref{BoundAWithingroupSum} for the limit. In turn, we find that, for any $l,k \in \{1,\dots,k\}$,
\begin{equation}\label{LimitInProb3}
    \begin{aligned}
          \mathbb{P}\left(\left| \dfrac{\sum_{i \in \mathcal{V}_l, j \in \mathcal{V}_k}B_{ij}}{\sigma_d^2}-b_{lk}\right| > \epsilon \right)&\le \mathbb{P}\left(\left| \dfrac{\sum_{i \in \mathcal{V}_l, j \in \mathcal{V}_k}B_{ij}-\sum_{i \in \mathcal{V}_l, j \in \mathcal{V}_k}\mathcal{A}_{ij}}{\sigma_d^2}\right| > \epsilon/2 \right) \\&+\mathbb{P}\left(\left| \dfrac{\sum_{i \in \mathcal{V}_l, j \in \mathcal{V}_k}\mathcal{A}_{ij}}{\sigma_d^2}-b_{lk}\right| > \epsilon/2 \right) \overset{d\rightarrow \infty}{\rightarrow} 0,
    \end{aligned}
\end{equation}
where we used \eqref{ConvofSum4} and \eqref{ConvOfSums1} for the limit. This establishes that $   \dfrac{\hat{b}_{lk}}{\sigma_d^2}  \overset{\mathbb{P}}{\rightarrow} b_{lk}\text{, as } d\rightarrow \infty.$ Finally, by once more applying Slutsky's Theorem, we get that
\begin{equation*}
        \hat{\mathcal{A}}_{ij}  \overset{\mathbb{P}}{\rightarrow} \mathcal{A}_{ij}, \text{ as } d\rightarrow \infty,
\end{equation*}
as desired.
\end{proof}

\subsection{Proofs}\label{ProofsModelClusteringSelection}
\begin{proof}[Proof of \Cref{Thm.1}]
    For $\mathbb{Y}$ as in \eqref{eqn:1}, it is well known that if $\sigma(Q)$, the spectrum of a matrix, $Q$, satisfies that $\sigma\left( Q\right) \subseteq (0,\infty) + i \mathbb{R}$ and the L\'evy measure $F$ satisfies the moment condition $\int_{\Vert x \Vert >1} \ln \Vert x \Vert F(dx) < \infty$. Then, assuming that $\mathcal{L}(\bold{Y}_0)$ is the invariant distribution of $\mathbb{Y}$,  $\mathbb{Y}$ is strictly stationary. See e.g \cite{Masuda2004}. Hence, we only need to prove that the real part of the spectrum of $Q = Q(\psi)$ is positive. Our proof is based on the Gershgorin Circle Theorem.  We first note that $\bar{Q} = \psi_1 I_d + \psi_2 D^{-1}A^\top$ where $D$ and $A$ as in \Cref{SBOUProc} satisfies the condition. Consider a disc around $\bar{Q}_{ii}=\psi_1$, (since $A_{ii}=0$) for any $i$, and define the disc radius as $R_i = \sum_{j\ne i} \vert \bar{Q}_{ij}\vert$. Then,
    \begin{equation*}
        \begin{aligned}
            \max_{i}R_i &= \psi_2\max_{i}\dfrac{\sum_{j =1}^d\vert A\vert_{ji}}{D_{ii}}\\
            &= \psi_2\max_{i}\dfrac{\sum_{j =1}^d\vert A\vert_{ji}}{\sum_{j =1}^d\vert A\vert_{ji}+\sum_{j =1}^d\vert A\vert_{ij}}\\
            &=  \dfrac{\psi_2}{1 + \min_i\dfrac{\sum_{j =1}^d\vert A\vert_{ij}}{\sum_{j =1}^d\vert A\vert_{ji}}}
            \\&=\dfrac{\psi_2}{1+\tau},
        \end{aligned}
    \end{equation*}
     i.e. any eigenvalue $\lambda$ of $\bar{Q}$ satisfies that 
     \begin{equation*}
         \vert \lambda - \bar{Q}_{ii}\vert =\vert \lambda - \psi_1\vert \le R_i \le \dfrac{\psi_2}{1+\tau},
     \end{equation*}
     which under our assumption implies that
\[
\operatorname{Re}(\lambda)
\geq
\psi_1-\frac{\psi_2}{1+\tau}
>0.
\] Now we claim that the eigenvalues of $\bar{Q}$ and $Q=Q(\psi)$ are the same. Indeed, if we let $\lambda$ be an eigenvalue of $Q(\psi)$ and let the corresponding eigenvector be $x$ then,
    \begin{equation*}
        \begin{aligned}
            Q(\psi)x &= \lambda x\\
            \iff \left( \psi_1 I_n + \psi_2D^{-1/2}A^\top D^{-1/2}\right)x &= \lambda x
            \\
            \iff  \psi_1 x + \psi_2D^{-1/2}A^\top D^{-1/2}x &= \lambda x
            \\
            \iff  \psi_1 D^{-1/2}x + \psi_2D^{-1}A^\top D^{-1/2}x &= \lambda D^{-1/2}x
            \\
            \iff  \left(\psi_1 I_n + \psi_2D^{-1}A^\top \right)y &= \lambda y,
        \end{aligned}
    \end{equation*}
    with $y = D^{-1/2}x$. This in turn finishes the proof.
\end{proof}
\begin{proof}[Proof of \Cref{Prop:LBMinNode}]
    We first note, that 
    \begin{equation}\label{MaxCardComm}
        d = \sum_{j=1}^k \vert \mathcal{V}_j\vert  \Rightarrow \max_{j\in \{1,2,\dots,k\}}\vert \mathcal{V}_j\vert \ge \dfrac{d}{k}.
    \end{equation}
    Since,
    \begin{equation*}
        \begin{aligned}
        \sigma _d &= \sum_{j \in \mathcal{V}_l} \theta_j, \quad \forall l\in \{1,2,\dots,k\} \\ \Rightarrow 
        \min_{i}\theta_i &\le\dfrac{\sigma_d}{\vert \mathcal{V}_l \vert} \quad \forall  l\in \{1,2,\dots,k\}. 
        \end{aligned}
    \end{equation*}
    Using the above bound with \eqref{MaxCardComm} yield that 
    \begin{equation*}
         \min_{i}\theta_i \le\dfrac{k\sigma_d}{d}, \text{ for any } d\in \mathbb{N}.
    \end{equation*}
\end{proof}
\begin{proof}[Proof of \Cref{Prop:SumReg}]
 Since for any $i$, $\mathcal{D}_{ii}\le \bar{\mathcal{D}}_{ii}$, we restrict our attention to the expected degree matrix from a GSBM. For any $i$, we have that 
 \begin{equation*}
     \begin{aligned}
 \bar{\mathcal{D}}_{ii} &= \bar{\mu} \theta_i \sum_{j=1}^d \theta_j \left( b_{s(i)s(j)} +b_{s(j)s(i)} \right)\\ &=\bar{\mu} \theta_i \sum_{l=1}^k \sum_{j\in \mathcal{V}_l} \theta_j \left( b_{s(i)s(j)} +b_{s(j)s(i)} \right)\\
  &=\bar{\mu} \theta_i \sum_{l=1}^k \left( b_{s(i)l} +b_{ls(i)} \right)\sum_{j\in \mathcal{V}_l} \theta_j \\
  &= \bar{\mu} \theta_i \sigma_d \sum_{l=1}^k \left( b_{s(i)l} +b_{ls(i)} \right)\\
  & \le \bar{\mu} \sigma_d \theta_i 2k.
     \end{aligned}
 \end{equation*}
Here $\bar{\mu}:= \mathbb{E}\left(\vert W \vert \right)$, where $W$ denotes the edge distributions as specified in \Cref{SBOUProc}.
Hence, $\min_{i\in\{1,2,\dots,d\}}\bar{\mathcal{D}}_{ii} = o(\log d)$.
\end{proof}
\begin{proof}[Proof of \Cref{Prop:MinimalDegree}]
    We first notice, that 
    \begin{equation*}
        \begin{aligned}\bar{\mathcal{D}}_{ii} &= \bar{\mu} \sum_{j=1}^d\theta_i \theta_j\left( b_{s(i)s(j)}+b_{s(j)s(i)} \right) \\
        &=\bar{\mu} \theta_i \sum_{l=1}^k \left( b_{s(i)l}+b_{ls(i)} \right) \sum_{j \in \mathcal{V}_l}\theta_j \\
        &=\bar{\mu} \theta_i \sigma_d \sum_{l=1}^k \left( b_{s(i)l}+b_{ls(i)} \right),
        \end{aligned}
    \end{equation*}
    where we remember once again $\bar{\mu} = \mathbb{E}\left(\vert W \vert \right)$. The above in turn implies that $\bar{\mathcal{D}}_{\min} =\Omega \left( \dfrac{\sigma_d^2}{d}\right)=\Omega \left(\log d\right)$
    following Assumption \ref{Ass.2} and the fact that $\sum_{l=1}^k \left( b_{s(i)l}+b_{ls(i)} \right) >0$ and $\bar{\mu}$ not depending on $d$. The second assertion now follows from the first, and the fact that the $(i,i)$th entry of $\mathcal{D}$ is related to the $(i,i)$th entry of $\bar{\mathcal{D}}$ in the following manner:
\begin{equation*}
    \mathcal{D}_{ii}=\bar{\mathcal{D}}_{ii}- 2 \bar{\mu}\theta_i^2 b_{s(i)s(i)},
\end{equation*} 
where the second term is at most $2 \bar{\mu}$ which does not depend on $d$.
\end{proof}
\begin{proof}[Proof of Lemma \ref{Lemma:1}]
    We follow \cite{Qin2013} closely, adapting their arguments to our slightly different setting. First, we notice that $\bar{\Psi}^s_p = \bar{\Psi}_p + (\bar{\Psi}_p)^\top = \psi_2 \bar{\mathcal{D}}^{-1/2}\bar{\mathcal{A}}^\top\bar{\mathcal{D}}^{-1/2} +\psi_2 \bar{\mathcal{D}}^{-1/2}\bar{\mathcal{A}}\bar{\mathcal{D}}^{-1/2} = \psi_2 \bar{\mathcal{D}}^{-1/2}\left(\bar{\mathcal{A}}+\bar{\mathcal{A}}^\top\right)\bar{\mathcal{D}}^{-1/2}$. We note that $\bar{\mathcal{A}}=\mu\Theta\bold{ZB}\bold{Z}^\top \Theta$ and so $\bar{\mathcal{A}} + \bar{\mathcal{A}}^\top =\mu\Theta\bold{Z}(\bold{B}+\bold{B}^\top)\bold{Z}^\top \Theta$. We underline, that 
    \begin{equation*}
        \begin{aligned}
        \bar{\mathcal{D}}_{ii} &=\bar{\mu}\sum_{j=1}^d \theta_i \theta_j (b_{s(i)s(j)} + b_{s(j)s(i)}) \\
        &=\bar{\mu}\sum_{l=1}^k \theta_i\sum_{j \in \mathcal{V}_l} \theta_j (b_{s(i)l} + b_{ls(i)})
    \\&=\bar{\mu}\sigma_d\theta_i\sum_{l=1}^k  (b_{s(i)l} + b_{ls(i)})
    \\&=\bar{\mu}\sigma_d\theta_i [D_B]_{s(i)s(i)},
        \end{aligned}
    \end{equation*}
where by $[D_B]_{s(i)s(i)}$ we refer to the $(s(i),s(i))$th entry of $D_B$.
It thus follows that 
\begin{equation*}
    \begin{aligned}
    \bar{\Psi}^s_{p,ij} &= \psi_2\dfrac{\bar{\mathcal{A}}_{ij}+\bar{\mathcal{A}}_{ji}}{\sqrt{\bar{\mathcal{D}}_{ii}\bar{\mathcal{D}}_{jj}}} \\
    &=\mu\psi_2\dfrac{\theta_i\theta_j\left(b_{s(i)s(j)}+b_{s(j)s(i)}\right)}{\sqrt{\bar{\mu}^2\sigma_d^2\theta_i \theta_j[D_B]_{s(i)s(i)}[D_B]_{s(j)s(j)}}}\\ &=\dfrac{\mu}{\bar{\mu}}\psi_2\left(
    \sqrt{\theta_i/\sigma_d} \dfrac{b_{s(i)s(j)}+b_{s(j)s(i)}}{\sqrt{[D_B]_{s(i)s(i)}[D_B]_{s(j)s(j)}}}\sqrt{\theta_j/\sigma_d}\right).
    \end{aligned}
\end{equation*}
Hence, $\bar{\Psi}^s_{p,ij}$ equals the $(i,j)$ entry of $(\mu/\bar{\mu})\psi_2\Theta_{\sigma_d}^{1/2} Z B_L Z^\top \Theta_{\sigma_d}^{1/2}$, completing the proof.
\end{proof}
\begin{proof}[Proof of Lemma \ref{Lemma:2}]
    Given Lemma \ref{Lemma:1}, this proof follows using arguments as in \cite{Qin2013} Lemma $3.3$. Let $C = \dfrac{\mu}{\bar{\mu}}\psi_2 B_L$. Since $B_L$ is positive definite, so is $C$ if $\mu >0$ and it is negative definite if $\mu<0$. Hence, we sort the $k$ eigenvalues of $C$ according to absolute size, $\vert \lambda_1 \vert \ge \vert\lambda_2 \vert\ge \dots\ge \vert\lambda_k\vert>0$. Let $\Lambda \in \mathbb{R}^{k \times k}$ be a diagonal matrix with its $(s, s)$th element equal to $\lambda_s$. Let $U \in \mathbb{R}^{k \times k}$ be an orthogonal matrix such that $U_{\bullet s}$  is the eigenvector of $C$ corresponding to $\lambda_s, s=1, \ldots, k$. By eigen-decomposition, we have $C=U \Lambda U^T$. Define $\mathcal{X}= \Theta_{\sigma_d}^{1/2}ZU$, then we have that 
    \begin{equation*}
    \mathcal{X}^\top\mathcal{X} =  U^\top Z^\top \Theta_{\sigma_d} Z U  = U^\top U = I_k.
    \end{equation*}
    Here we used that $Z^\top \Theta_{\sigma_d} Z$ is the identity matrix. Further,
    \begin{equation*}
        \mathcal{X} \Lambda \mathcal{X}^\top = \Theta_{\sigma_d}^{1/2}Z C Z^\top \Theta_{\sigma_d}^{1/2} = \dfrac{\mu}{\bar{\mu}}\psi_2\Theta_{\sigma_d}^{1/2}Z B_L Z^\top \Theta_{\sigma_d}^{1/2} = \bar{\Psi}_{p}^{s},
    \end{equation*}
   where we used Lemma \ref{Lemma:1} in the last equality. In this case, the $k$ eigenvalues are the non-zero eigenvalues of $\bar{\Psi}^s_p$ satisfying that $\vert \lambda_1 \vert \ge \vert\lambda_2 \vert\ge \dots\ge \vert\lambda_k\vert>0$. Positive if $\mu >0$ and negative if $\mu<0$. Further, $\mathcal{X}$ contains the eigenvectors. To see this, we first show that for $s =1,\dots,k$, $\lambda_s$ and $\mathcal{X}_{\bullet s}$ are also an eigenvalue and eigenvector pair to $\bar{\Psi}_{p}^{s}$. Let $\bold{e}_{s}$ denote the s-th standard basis vector in $\mathbb{R}^k$.
\begin{equation*}
    \begin{aligned}
        \lambda_s\bold{e}_{s} &= \lambda_s \bold{e}_{s} \\
        \iff 
        \Lambda \bold{e}_{s} &= \lambda_s \bold{e}_{s} \\
        \iff 
        \Lambda \mathcal{X}^\top\mathcal{X}_{\bullet s} &= \lambda_s  \bold{e}_{s}\\
        \iff 
        \mathcal{X} \Lambda \mathcal{X}^\top \mathcal{X}_{\bullet s} &= \lambda_s\mathcal{X}\bold{e}_{s}\\
        \iff \bar{\Psi}_{p}^{s}\mathcal{X}_{\bullet s} &= \lambda_s\mathcal{X}_{\bullet s}.
    \end{aligned}
\end{equation*}
Finally, since $\mathrm{rank}(\bar{\Psi}_{p}^{s})\le \mathrm{rank}(\Lambda) = k$, these are the only non-zero eigenvalues of $\bar{\Psi}_{p}^{s}$.
   Now, notice that $\mathcal{X}_{i\bullet} = (\theta_i/\sigma_d)^{1/2}Z_{i\bullet}U$. Clearly, $\Vert \mathcal{X}_{i\bullet}\Vert_2^2 = \theta_i/\sigma_d $.  Thus, for $\mathcal{X}^*$ defined to be the matrix with $(i,j)$th entry equal to $\dfrac{\mathcal{X}_{ij}}{\Vert\mathcal{X}_{i\bullet} \Vert_2}$, we have that 
    \begin{equation*}
        \mathcal{X}^*_{i\bullet} = \dfrac{\sqrt{\theta_i/\sigma_d}Z_{i\bullet}U}{\sqrt{\theta_i/\sigma_d}} = Z_{i\bullet}U,
    \end{equation*}
    which concludes the proof. 
\end{proof}

\begin{proof}[Proof of \Cref{EigenvectorBound}]
    First, we will bound the difference between the Frobenius norm of $\widehat{\mathcal{X}}$ and $\mathcal{X}$ using Theorem $2$ in \cite{Yu2015}. Regardless of the sign on $\mu$, we sort the $d$ eigenvalues of $\bar{\Psi}^s_p$ according to size, i.e. $\lambda_1\ge \lambda_2 \dots \ge\lambda_d$. We assume first that $\mu>0$. With this in mind, we set $r=1$, $s=k$ and $p=d$ with the convention that $\lambda_0 = \infty$. We also sort the eigenvalues of $\widehat{\Psi}^s$ in decreasing order. Following Lemma \ref{Lemma:2}, there are $k$  non-zero eigenvalues of  $\bar{\Psi}_p^s$ and their corresponding eigenvectors are orthogonal. Since $\widehat{\Psi}^s$ is real and symmetric we can also choose its eigenvectors to be orthonormal by the Spectral Theorem for real matrices \citep{Axler1997}. Finally, since the $k$ largest eigenvalues of $\bar{\Psi}_p^s$ are strictly positive and the rest are $0$, the assumption $\rho_d = \min\left( \lambda_0 - \lambda_1, \lambda_k - \lambda_{k+1} \right)=\min\left( \infty - \lambda_1, \lambda_k - 0\right) = \lambda_k >0$ is satisfied, and hence, we have by Theorem $2$ of \cite{Yu2015} that
\begin{equation}\label{FirstBound}
     \Vert \widehat{\mathcal{X}}\widehat{O}_d - \mathcal{X} \Vert_F \le \dfrac{C \Vert \widehat{\Psi}^s - \bar{\Psi}^s_p\Vert }{\rho_d},
 \end{equation}
For some  $C>0$ that does not depend on $d$. Since we sorted the eigenvalues according to size, by \Cref{Lemma:2} $\rho_d$ is the smallest non-zero eigenvalue of $\bar{\Psi}^s_p$. For $\mu <0$, we proceed similarly, but now with $r = d-k+1$, and $s=p=d$. Then, using the convention that $\lambda_{d+1} = -\infty$, $\rho_d = \min \left( \lambda_{d-k}-\lambda_{d-k+1},\lambda_d - \lambda_{d+1} \right) = -\lambda_{d-k+1}= \vert \lambda_{d-k+1} \vert >0$ is equal to the absolute value of the smallest non-zero eigenvalue of $\bar{\Psi}^s_p$. We used here that $\lambda_{d-k}=0$ according to \Cref{Lemma:2}. Thus, the above bound holds regardless of the sign on $\mu$. We note that
 \begin{equation*}
      \Vert \widehat{\mathcal{X}}\widehat{O}_d - \mathcal{X} \Vert_F =  \Vert \widehat{\mathcal{X}}- \mathcal{X}\mathcal{O}_d \Vert_F,
 \end{equation*}
with $\mathcal{O}_d = \widehat{O}_d^T$. Using that $\rho_d$ is independent from $d$ by Proposition \ref{prop:eigenvaluesnindependence}, we conclude that 
\begin{equation}\label{EigenvectorBound2}
     \Vert \widehat{\mathcal{X}}- \mathcal{X}\mathcal{O}_d \Vert_F = O\left( \Vert \widehat{\Psi}^s - \bar{\Psi}^s_p\Vert \right).
\end{equation}
We finally conclude that 
\begin{equation*}
    \begin{aligned}
\Vert\widehat{\mathcal{X}^*}-\mathcal{X}^*\mathcal{O}_d \Vert_F &\le 2\dfrac{\Vert \widehat{\mathcal{X}}- \mathcal{X}\mathcal{O}_d \Vert_F}{\min_{i\in\{1,2,\dots, d\}}\Vert \mathcal{X}_{i\bullet}\Vert_2} \\
&= O\left(\sqrt{d}\Vert \widehat{\mathcal{X}}- \mathcal{X}\mathcal{O}_d\Vert_F\right),
    \end{aligned}
\end{equation*}
 where we used \Cref{RowNormalisedBound} for the upper bound and the fact that, for each $i$, $\Vert \mathcal{X}_{i\bullet}\Vert_2 = \left(\dfrac{\theta_i}{\sigma_d}\right)^{1/2}$ together with \Cref{Ass.NodeSum,Ass.2}; the $O(\sqrt d)$ scaling follows from the bullet $\min_i \theta_i = \Omega(\sigma_d/d)$ in \Cref{Ass.2}. To conclude, we bound the spectral norm $\Vert \widehat{\Psi}^s - \bar{\Psi}^s_p\Vert$ via the triangle inequality through the four-term decomposition
\begin{equation*}
\Vert \widehat{\Psi}^s - \bar{\Psi}^s_p\Vert \le \Vert \Psi^s_p - \bar{\Psi}^s_{p-}\Vert + \Vert \mathrm{diag}(\bar{\Psi}^s_p)\Vert + \Vert \Psi^s - \Psi^s_p\Vert + \Vert \widehat{\Psi}^s - \Psi^s\Vert.
\end{equation*}
The three concentration terms are bounded by \Cref{Prop.2,Prop.3,Prop.4}, contributing $O\!\left(\sqrt{d\log d / \sigma_d^2}\right)$ collectively; the estimation term is bounded by $f(d,\mathcal{P}_T)$ via \Cref{Ass.EstimationBound}. Combined with \eqref{EigenvectorBound2} and the row-normalisation bound established above, this concludes the proof.
 \end{proof}

\begin{proof}[Proof of \Cref{Thm:ConsistencyCommunitytCluster}] Let $\mathcal{M}$ be as in \Cref{Def:Miscluster} and $M$ as in \Cref{SetInclusion}, then
    $\vert \mathcal{M} \vert \le \vert M \vert$, following Proposition \ref{SetInclusion}. Then, using the arguments in \cite{Rohe2011,Gudhmundsson2021} and \Cref{EigenvectorBound} we get that, with high probability,
    \begin{equation*}
        \vert \mathcal{M} \vert \le 8 \Vert \widehat{\mathcal{X}^*}  - \mathcal{X}^*\mathcal{O}\Vert_F^2 = O \left( d f(d,\mathcal{P}_T)^2+\dfrac{d^2\log d}{\sigma_d^2} \right).
    \end{equation*}
\end{proof}

\section{Support recovery algorithm}\label{SecSupportRec}
We consider model selection for the drift matrix $\bold{Q}$ of an Ornstein-Uhlenbeck process heuristically. That is, we do not derive any theoretical results; we outline the idea behind our model-selection method and show in \Cref{Sim2} that it performs very well in finite samples.
Consider an Ornstein-Uhlenbeck diffusion model as in \eqref{eqn.Intro:1}. We think of $\bold{Q}$ as being sparse and are interested in recovering the support. We assume that we observe the OU-process on a discrete grid: As in \Cref{DiscreteSetting}, $\mathbb{X}_{\mathcal{P}_T}:=\left\{\mathbf{X}_s, s \in \mathcal{P}_T\right\}
$ where, as in \Cref{DiscreteSetting}, $N_T$ denotes the sample size (so that $N_T = T/\Delta$ for an equidistant grid). Let $[d^2]=\{1,2,\dots,d^2\}$. Let $\mathcal{P}([d^2])$ denote the power set of $[d^2]$. For any $\ell \in [d^2]$, define $i(\ell) := \lceil \ell/d \rceil$ and $j(\ell) := ((\ell-1) \bmod d) + 1$, so that $\ell \mapsto (i(\ell), j(\ell))$ is the row-major bijection between $[d^2]$ and $[d] \times [d]$. Define $\mathbb{W}\left(N \right) = \{\bold{A}\in \mathcal{M}_d(\mathbb{R}): \bold{A}_{i(\ell)\,j(\ell)}=0 \text{ for all } \ell \in N\}$ for $N \in \mathcal{P}([d^2])$. For a penalty function $\lambda: \mathbb{N}\times \mathbb{R} \rightarrow \mathbb{R}$, our model selection problem then consists in choosing
\begin{equation}\label{ModelsSelection}
    N^* = \min \argmin_{N  \in \mathcal{P}([d^2])}\left( -2\hat{L}_N+ \lambda\left(d^2-\vert N\vert, N_T\right)\right),
\end{equation}
where $\widehat{Q}_N := \argmin _{Q \in \mathbb{W}(N)}\bar{\mathcal{L}}_T^{\mathcal{P}_T}(Q)$ is the constrained MLE and $\hat{L}_N := -\bar{\mathcal{L}}_T^{\mathcal{P}_T}(\hat{Q}_N)$ is the corresponding maximized log-likelihood (recall that $\bar{\mathcal{L}}_T^{\mathcal{P}_T}$ in \eqref{FeasibleDiscLike} is the negative log-likelihood, so this sign convention puts $\hat{L}_N$ on the same footing as the standard log-likelihood in AIC/BIC). The outer minimum in \eqref{ModelsSelection} is there in the case of multiple solutions, in which the solution with the smallest cardinality is chosen.
\begin{remark}
    For $\lambda(x,y)= 2x$, \eqref{ModelsSelection} is the solution from AIC whereas when $\lambda(x,y) = x\ln(y)$  we arrive at BIC. To ensure the existence of $\widehat{Q}_N$ for any $N \in \mathcal{P}\left([d^2]\right)$, we could restrict $\mathbb{W}\left(N \right)$ to being compact which would entail bounding each entry. This is a fairly standard assumption in the literature, as in practice, parameter bounds would typically be applied during the optimization procedure and so we see no issue with it.
\end{remark}
Clearly, the solution to \eqref{ModelsSelection} is infeasible as it requires us solving $2^{d^2}$ minimization problems. In practice, we opt for the following scheme:
\begin{enumerate}
    \item Estimate drift matrix with no restrictions using \Cref{SpecificFormQEst}. Denote this estimate by $\widehat{Q}$.
    \item Sort the $d^2$ elements of $\widehat{Q}$ according to absolute size in ascending order and let \\ $\left\{(i_1,j_1),(i_2,j_2),\dots,(i_{d^2},j_{d^2})\right\}$ denote the list of corresponding row-column pairs. Let $S_{\#}:[d]\times[d] \to [d^2]$ be the row-major linear index $(i,j) \mapsto (i-1)d + j$, the inverse of the bijection $\ell \mapsto (i(\ell), j(\ell))$ defined above. Consider the list $\left\{ S_{\#}(i_1,j_1),\dots,S_{\#}(i_{d^2},j_{d^2})\right\} \subseteq [d^2]$, whose first entry is the linear index of the smallest absolute value of $\widehat{Q}$, second entry the linear index of the second smallest absolute value, and so on.
    \item From  $\left\{ S_{\#}(i_1,j_1),\dots,S_{\#}(i_{d^2},j_{d^2})\right\}$, construct $d$ sets $\left(D_j\right)_{j=1}^d$ such that $D_1\subset D_2 \subset \dots,\subset D_d$ with $$D_1 = \left\{ S_{\#}(i_1,j_1),\dots,S_{\#}(i_{1\cdot d},j_{1\cdot d})\right\},D_2 =  \left\{ S_{\#}(i_1,j_1),\dots,S_{\#}(i_{1\cdot d+1},j_{1\cdot d+1}),\dots,S_{\#}(i_{2\cdot d},j_{2\cdot d})\right\}$$ and so on. Sets increase by $d$ at a time, that is $\vert D_{j+1} \vert - \vert D_j\vert = d$.
    \item Find    \begin{equation}\label{ModelsSelectionFeasible1}
    K_{\#} = \min \argmin_{j  \in [d]}\left( -2\hat{L}_{D_j}+ \lambda\left(d^2-\vert D_{j}\vert, N_T\right)\right).
\end{equation}
Note that $K_{\#}$ is an integer, whereas $N^*$ from $\eqref{ModelsSelection}$ is a set.
\item For a set  $S$, let $S[-j]$ select all but the last $j$ elements of $S$. Find,
    \begin{equation*}
    K_{*} = \min \argmin_{j  \in [2d] }\left( -2\hat{L}_{D_{(K_{\#}+1)}[-j]}+ \lambda\left(d^2-\vert D_{(K_{\#}+1)}[-j]\vert, N_T\right)\right).
\end{equation*}
If $K_{\#} = d$, replace $(K_{\#} +1)$ by $K_{\#}$ above.
\item We then estimate $\mathrm{supp}(Q_0)=\{(i,j) \in [d]\times [d]: [Q_0]_{ij}\ne 0\}$ by $\widehat{\mathrm{supp}(Q)} = \{(i,j) \in [d]\times [d]: S_{\#}(i,j) \notin D_{(K_{\#}+1)}[-K_{*} ] \}$
\end{enumerate}
\begin{remark}
    The above procedure requires us solving $3d$ minimization problems instead of $2^{d^2}$ as in \eqref{ModelsSelection} and is indeed feasible to compute even large $d$.
\end{remark}
\begin{remark}
    We have throughout assumed that the L\'evy process $\mathbb{L}$ has L\'evy triplet $\left(0,\Sigma,F\right)$ and assumed $\Sigma$. For the unrestricted minimization problem, we showed with \Cref{SpecificFormQEst} that the assumption could be dropped. However, for the restricted minimization problems above, knowledge of $\Sigma$ is necessary for the minimization to be performed. Following \eqref{eqn.Intro:1}, clearly,
    \begin{equation*}
        \Delta_{\mathcal{P}_T}^k \mathbb{L} = \Delta_{\mathcal{P}_T}^k \mathbb{Y} + \bold{Q}\int_{s_{k}^{N_T}}^{s_{k+1}^{N_T}} \bold{Y}_s ds.
    \end{equation*}
    We note that, given an estimate of $Q$, $\widehat{Q}$, we may then approximate the L\'evy increments by
 \begin{equation*}
        \widehat{\Delta_{\mathcal{P}_T}^k \mathbb{L}} = \Delta_{\mathcal{P}_T}^k \mathbb{Y} + \widehat{\bold{Q}}\dfrac{1}{2}\left(s_{k+1}^{N_T}-s_{k}^{N_T}\right)\left(\bold{Y}_{s_{k+1}^{N_T}}+\bold{Y}_{s_{k}^{N_T}}\right).
    \end{equation*}
    We can then estimate $\Sigma$ by the method in \cite{Gegler2010} and Appendix $H.1$ in \cite{Lucchese2023}. In practice, we then replace $\Sigma$ in \eqref{FeasibleDiscLike} with this estimate in the scheme above.
\end{remark}

\section{Numerical experiments}\label{App:NumericalExperiments}
\subsection{Simulation study on the group detection algorithm of \Cref{Sec:Cons}}\label{App:SimStudyClusteringAlgo}
We are now interested in testing the finite sample performance of our community detection algorithm. In light of this, using the Euler-Maruyama scheme written out in \cite{Lucchese2023} Appendix $G$, we simulate repeated samples of a SBOU process on a $\Delta_{\mathcal{P}_T}$-grid of size $1/24$ for $T=274,1096,4384$, and thus time series dimension of size $N_T= T/\Delta_{\mathcal{P}_T}$. We interpret $T=1$ as a day, and hence the setting $T=1096$ corresponds to us having hourly data over the course of $3$ years. Further, we consider the cross-sectional dimensionality/number of nodes equal to $d = 50,75,100$. We let the Lévy-process, $\mathbb{L}$, in \eqref{eqn:1} be given by 
\begin{equation}\label{eqn:SimNoiseSpec}
\mathbb{L}=\mathbb{W}+\mathbb{J}.
\end{equation}
Here $\mathbb{W}$ is a $d$-dimensional Brownian motion with covariance matrix $I_d$ and $\mathbb{J}$ is a compound Poisson process independent of $\mathbb{W}$ with rate $\lambda=1$ such that $\mathbb{J}_t^{(j)}=\sum_{k=0}^{N_t} Z_k^{(j)}$ where $Z_k^{(j)}$ are i.i.d. jump heights that are standard normally distributed and $\mathbb{N}$ is a Poisson process with rate $\lambda$. For any $d$, we let $\sigma_d=\dfrac{d}{k} \min \left(1, C \dfrac{(d\log d)^{1/2+\epsilon}}{d} \right)$, for $\epsilon =1/100$, and $\theta_j = \theta =\min \left(1, C \dfrac{(d\log d)^{1/2+\epsilon}}{d} \right)$, where $C$ is a constant such that $\theta\in \{0.5,0.8,1\}$ when $d=100$. Clearly, the $\theta$s are decreasing in $d$, so, for $d=50, 75$, the node-specific probabilities are larger than for $d=100$. We thus underline that we are testing \Cref{Thm:ConsistencyCommunitytCluster} in finite samples in the sparsest possible setup where the share of correctly clustered nodes tend to $1$ slowest. In other words, the results displayed here are in the hardest setting where we may still expect correct clustering. Intuitively, it will be easier for our clustering algorithm to work, the larger the difference in the intra-community and inter-community probabilities. This, we already touched upon by displaying \Cref{fig:combined_signal_illustration}. That is, if the probability of a link within communities is very high paired with the probability of links between nodes in different communities being very low, then clustering will be easier. On the other hand, the closer these probabilities are, the harder clustering will be. To shed light on how much this affects the clustering algorithm, we consider different setups for the entries of the community-specific probability matrix $\mathbf{B}$ from a very clear signal to a less clear group signal. In all setups, we let its entries on the diagonal be given by $b_{ll}=0.5$, for $l\in \{1,\dots,k\}$, and then we vary the off-diagonal setups. In total, we consider four settings: $b_{lk} \in \{ 0.01, 0.05,0.1,0.2\}$ for $l\ne k$. Lastly, we let the communities be equal and of size $d/k$. We underline that these settings satisfy the assumptions of Theorem \ref{Thm:ConsistencyCommunitytCluster}. We draw $R=1000$ of such samples and display the results in \Cref{table:SimTable1}. We notice a few interesting things in these tables. Perhaps not too surprisingly, the clustering algorithm works best when $k$ is small. Indeed, for smaller $k$ the problem of clustering the nodes becomes easier. For $k=2$, by clustering at random, we would expect to cluster $50 \%$ correctly. As expected we also see that the larger the difference between the probabilities of links within and between communities, the better the algorithm works. We further notice, that as the node-specific probabilities increases the overall tendency is that the clustering algorithm works better. This is also quite intuitive and in line with the idea that the matrix $\mathbf{B}$ carries the `signal' of which nodes are in which communities; as argued, the larger the difference between the within-community and between-community link probabilities, the easier clustering is. On the other hand, the node-specific probabilities $\theta_i$ essentially play the dominating part in determining whether a link is established or not, thereby determining the strength of this signal. Lastly, for fixed $T$, it is not clear whether increasing $d$ improves the clustering percentage. For fixed $d$, increasing $T$ improves clustering. However, as Theorem \ref{Thm:ConsistencyCommunitytCluster} establishes the consistency of Algorithm  \ref{alg:VARBlockbuster} when $d$ and $T$ tend to infinity, the key pattern to note is that, for fixed $k$ and $\theta$, when increasing $d$ and $T$ together, the percentage of correctly clustered nodes increases significantly, i.e.~when one considers the diagonal elements of each sub-table. In conclusion, these results suggest that Algorithm \ref{alg:VARBlockbuster} not only works for samples tending to infinity, but actually performs very well also for finite samples.
\begin{sidewaystable}[p] 
\setlength{ \tabcolsep}{0.1cm}
\begin{center}
\caption{Share of correctly clustered nodes from applying \Cref{alg:VARBlockbuster} to observations of \Cref{SBOUProc} where the L\'evy driving noise is specified as \eqref{eqn:SimNoiseSpec} with independent noise }
\label{table:SimTable1}
\begin{footnotesize}
  \adjustbox{max width=\textheight}{%
\begin{tabular}{ccccccccccccccccccccccccccc}
    \hline
    \hline
 & &     \multicolumn{3}{c}{ k = 2} & &     \multicolumn{3}{c}{ k = 3} & &     \multicolumn{3}{c}{k = 5} & & & &    \multicolumn{3}{c}{k = 2} & &     \multicolumn{3}{c}{k = 3} & &     \multicolumn{3}{c}{k = 5} \\
    \cline{3-5} \cline{7-9} \cline{11-13} \cline{17-19} \cline{21-23} \cline{25-27}
 \\
    \multicolumn{2}{c}{Panel A}    && && && && &&& &&    \multicolumn{2}{c}{Panel B}   && && && && &&  \\
$\theta$ &$d/T$ &274&1096&4384& & 274&1096&4384 & & 274&1096&4384&&$\theta$ &$d/T$ &274&1096&4384 & & 274&1096&4384 & & 274&1096&4384 \\
    \cline{3-5} \cline{7-9} \cline{11-13} \cline{17-19} \cline{21-23} \cline{25-27}
 & 50  & 67.5\% & 92.0\% & 99.3\%&   & 55.7\% & 76.9\% & 92.8\%&   & 44.9\% & 56.3\% & 63.9\%&  &   & 50  & 64.3\% & 85.3\% & 96.5\%&   & 51.4\% & 63.4\% & 75.1\%&   & 40.9\% & 46.2\% & 48.5\% \\
$0.5$ & 75  & 61.3\% & 91.3\% & 99.7\%&   & 49.5\% & 82.4\% & 98.6\%&   & 41.2\% & 61.3\% & 80.7\%&  &  $0.5$ & 75  & 59.8\% & 83.6\% & 98.2\%&   & 46.6\% & 67.6\% & 91.3\%&   & 37.3\% & 46.3\% & 57.1\% \\
 & 100  & 57.9\% & 87.1\% & 99.7\%&   & 45.0\% & 78.6\% & 99.2\%&   & 36.7\% & 62.9\% & 92.8\%&  &   & 100  & 56.9\% & 79.0\% & 98.4\%&   & 43.5\% & 62.2\% & 94.6\%&   & 33.9\% & 45.4\% & 64.7\% \\
 \\
 & 50  & 73.5\% & 98.7\% & 100.0\%&   & 64.5\% & 98.3\% & 100.0\%&   & 54.6\% & 88.0\% & 98.3\%&  &   & 50  & 68.3\% & 95.9\% & 99.9\%&   & 55.5\% & 90.0\% & 99.3\%&   & 44.0\% & 62.5\% & 77.9\% \\
$0.8$ & 75  & 63.0\% & 96.5\% & 100.0\%&   & 52.5\% & 96.0\% & 100.0\%&   & 45.6\% & 93.1\% & 99.8\%&  &  $0.8$ & 75  & 61.0\% & 91.8\% & 99.9\%&   & 48.0\% & 85.3\% & 99.5\%&   & 39.1\% & 64.4\% & 94.7\% \\
 & 100  & 58.8\% & 92.6\% & 100.0\%&   & 46.1\% & 91.3\% & 100.0\%&   & 38.3\% & 89.3\% & 99.9\%&  &   & 100  & 57.2\% & 86.0\% & 99.7\%&   & 43.8\% & 75.9\% & 99.5\%&   & 34.5\% & 58.8\% & 97.6\% \\
 \\
 & 50  & 74.0\% & 99.2\% & 100.0\%&   & 67.0\% & 99.4\% & 100.0\%&   & 59.7\% & 97.9\% & 100.0\%&  &   & 50  & 68.5\% & 97.4\% & 100.0\%&   & 57.9\% & 95.9\% & 99.9\%&   & 45.7\% & 74.1\% & 95.8\% \\
$1.0$ & 75  & 63.2\% & 97.1\% & 100.0\%&   & 52.7\% & 97.5\% & 100.0\%&   & 46.8\% & 98.1\% & 100.0\%&  &  $1.0$ & 75  & 60.8\% & 93.6\% & 99.9\%&   & 48.5\% & 89.7\% & 99.9\%&   & 39.4\% & 74.0\% & 99.5\% \\
 & 100  & 58.6\% & 93.3\% & 100.0\%&   & 46.6\% & 93.3\% & 100.0\%&   & 38.4\% & 94.0\% & 100.0\%&  &   & 100  & 57.2\% & 87.5\% & 99.9\%&   & 44.0\% & 78.9\% & 99.8\%&   & 34.4\% & 64.7\% & 99.5\% \\
 \\
 \\
    \multicolumn{2}{c}{Panel C}    && && && && &&& &&    \multicolumn{2}{c}{Panel D}   && && && && &&  \\
$\theta$ &$d/T$ &274&1096&4384& & 274&1096&4384 & & 274&1096&4384&&$\theta$ &$d/T$ &274&1096&4384 & & 274&1096&4384 & & 274&1096&4384 \\
    \cline{3-5} \cline{7-9} \cline{11-13} \cline{17-19} \cline{21-23} \cline{25-27}
 & 50  & 61.8\% & 75.0\% & 87.8\%&   & 48.5\% & 54.0\% & 59.7\%&   & 38.1\% & 40.1\% & 41.5\%&  &   & 50  & 58.6\% & 63.0\% & 65.6\%&   & 45.3\% & 46.4\% & 47.0\%&   & 35.6\% & 36.3\% & 36.3\% \\
$0.5$ & 75  & 57.9\% & 74.3\% & 93.6\%&   & 44.8\% & 55.1\% & 71.8\%&   & 35.2\% & 38.8\% & 42.6\%&  &  $0.5$ & 75  & 56.4\% & 61.7\% & 73.1\%&   & 42.9\% & 46.1\% & 48.6\%&   & 33.1\% & 33.9\% & 34.6\% \\
 & 100  & 56.1\% & 69.1\% & 94.4\%&   & 42.4\% & 51.9\% & 79.6\%&   & 32.2\% & 37.5\% & 45.0\%&  &   & 100  & 55.4\% & 58.8\% & 74.5\%&   & 41.1\% & 44.1\% & 49.9\%&   & 30.9\% & 32.6\% & 33.7\% \\
 \\
 & 50  & 63.2\% & 89.7\% & 99.1\%&   & 50.6\% & 71.3\% & 93.7\%&   & 39.3\% & 47.2\% & 56.0\%&  &   & 50  & 59.0\% & 70.2\% & 88.6\%&   & 46.1\% & 50.1\% & 57.1\%&   & 36.1\% & 37.4\% & 39.4\% \\
$0.8$ & 75  & 58.3\% & 83.8\% & 99.1\%&   & 45.3\% & 66.8\% & 96.5\%&   & 35.4\% & 44.8\% & 65.7\%&  &  $0.8$ & 75  & 56.6\% & 65.5\% & 90.4\%&   & 42.9\% & 48.2\% & 62.8\%&   & 33.0\% & 34.8\% & 38.5\% \\
 & 100  & 56.3\% & 75.6\% & 98.6\%&   & 42.5\% & 57.7\% & 95.8\%&   & 32.5\% & 41.0\% & 71.3\%&  &   & 100  & 55.1\% & 60.9\% & 87.8\%&   & 41.0\% & 45.2\% & 62.9\%&   & 30.9\% & 33.1\% & 37.2\% \\
 \\
 & 50  & 64.4\% & 92.8\% & 99.7\%&   & 51.3\% & 80.4\% & 98.8\%&   & 39.8\% & 52.0\% & 70.4\%&  &   & 50  & 59.6\% & 74.3\% & 95.2\%&   & 46.0\% & 52.5\% & 67.7\%&   & 36.0\% & 38.1\% & 42.2\% \\
$1.0$ & 75  & 58.6\% & 85.7\% & 99.6\%&   & 45.6\% & 71.0\% & 98.7\%&   & 35.8\% & 47.3\% & 80.5\%&  &  $1.0$ & 75  & 56.5\% & 66.7\% & 93.9\%&   & 43.0\% & 48.7\% & 72.7\%&   & 33.1\% & 35.4\% & 41.1\% \\
 & 100  & 56.2\% & 76.5\% & 99.1\%&   & 42.5\% & 59.6\% & 97.9\%&   & 32.4\% & 42.5\% & 85.5\%&  &   & 100  & 55.3\% & 61.0\% & 90.8\%&   & 41.1\% & 45.3\% & 70.4\%&   & 31.0\% & 33.3\% & 39.4\% \\
 \\
    \hline
    \hline
  \end{tabular}%
}%
  \end{footnotesize}%
\smallskip
\begin{scriptsize}
\parbox{0.98\textwidth}{\emph{Note: } We simulate the process in \Cref{SBOUProc} with $\Delta_{\mathcal{P}_T=1/24}$ and interpret $T=1$ as one day so that we have hourly observations. The L\'evy driving noise is specified as \eqref{eqn:SimNoiseSpec} with independent entries. We simulate the process $R= 1,000$ times and apply \Cref{alg:VARBlockbuster} to detect groups. In this table, we present the mean of the share of correctly clustered nodes over this replications for for different group signals across varying parameters. Panels $A-B$ show results for the case where $b_{ii} = 0.5$ for all entries, while $b_{ij}$ varies across panels with values in $\{0.01, 0.05, 0.1, 0.2\}$. Specifically, Panel $C$ corresponds to the configuration where $b_{ij} = 0.1$.
}
\end{scriptsize}
\end{center}
\end{sidewaystable}
Finally, we refer to the simulation results in \Cref{SimCorrelation} that clearly highlight how clustering based on correlation metrics, such as those inspired by \cite{Mantegna1999}, mistakenly interprets noise as signal (i.e., groupings) when the noise is highly correlated. In contrast, our method is able to distinguish the \emph{true signal} from \emph{noise}.
\FloatBarrier

\subsection{Simulation study on the algorithm for detecting $k$}\label{Sim2}
In this section, we are interested in testing the finite-sample performance of the proposed method for finding $k$ mentioned in \Cref{SecModelSelection} as well as the support recovery algorithm in \Cref{SecSupportRec}.  We let the Lévy-process, $\mathbb{L}$, in \eqref{eqn:1} be given by 
\begin{equation*}
\mathbb{L}=\mathbb{W}+\mathbb{J},
\end{equation*}
where $\mathbb{W}$ is a $d$-dimensional Brownian motion with covariance matrix $I_d$ and $\mathbb{J}$ is a compound Poisson process independent of $\mathbb{W}$ with rate $\lambda=1$, given by $\mathbb{J}_t^{(j)}=\sum_{k=0}^{\mathbb{N}_t} Z_k^{(j)}$, where $\mathbb{N}$ is a Poisson process of rate $\lambda$ and the jump heights $Z_k^{(j)}$ are i.i.d.\ standard normally distributed, exactly as in \Cref{App:SimStudyClusteringAlgo}. We simulate an SBOU process as in \Cref{SBOUProc} on a $\Delta_{\mathcal{P}_T}$-grid of size $1/24$ for $T=1096$ and $T=4384$. We interpret $T=1$ as a day and these thus correspond to having hourly observations over the course of $3$ and $12$ years.  Intuitively, the larger the non-zero values of the drift matrix in the \Cref{SBOUProc}, the easier it will be to recover the support. Here, we pick $\psi_2$ in \Cref{SBOUProc} so that the median non-zero value of the network effect matrix is $0.3$. We then choose $\psi_1$ according to \Cref{Thm.1}. By comparing to the heat maps of the matrices used for support recovery in \cite{Gaiffas2019, Ciolek2020}, we observe that this represents a relatively conservative choice. We underline that in this simulation study, we hold $\theta$ fixed as we vary $d$; our interest is in seeing how the procedure performs in this regime. We display the results in \Cref{tab:SupportRecovery}. The convergence we are interested in now is as $d$ grows. We see that our support recovery algorithm performs extremely well across almost all scenarios. As specified in \Cref{tab:SupportRecovery}, we opt for using AIC for support recovery as opposed to say BIC, since our recovery rate of the non-zero elements of the drift matrix was a lot better for AIC. This is not too surprising. It is well-known that BIC prefers sparser models more than AIC, as it penalizes model complexity more than AIC. We believe the reason for BIC's somewhat poorer performance comes from the fact that the diagonal entries of the drift matrix of a \Cref{SBOUProc} tend to be quite a lot bigger than the off-diagonal entries by the normalization of these entries in \Cref{SBOUProc}. This translates to the diagonal elements in \Cref{SBOUProc} explaining the major part of the probabilistic behavior of the process. Once the diagonal entries have been added, allowing more entries (the off-diagonal entries) to be non-zero does not contribute much in terms of likelihood. The stricter penalty of BIC combined with the little contribution to the likelihood function has the consequence that BIC tends to select models that are sparser than the true drift matrix.  In \Cref{fig:EstProbabilities} we demonstrate the strong-signal regime anticipated in \Cref{SecModelSelection}: applying the plug-in estimators of \Cref{EstimatorsBAndTheta} to the partition recovered by \Cref{alg:VARBlockbuster}, the block structure of the consistently estimated drift matrix is clearly visible and the edge-probability matrix is recovered accurately.
\begin{table}[!htbp]
\centering
\caption{Accuracy of estimators from \citet{Ma2021} and support recovery algorithm in \Cref{SecSupportRec}.}
\label{tab:SupportRecovery}

% Subtable 1
\subfloat[$T = 1096$]{
\begin{footnotesize}
\begin{tabular}{@{\hspace{0.5cm}}clccc@{\hspace{0.5cm}}ccc@{\hspace{0.5cm}}ccc@{}}
\hline
\hline
& & \multicolumn{3}{c}{$k=2$} & \multicolumn{3}{c}{$k=3$} & \multicolumn{3}{c}{$k=5$} \\
\cmidrule(lr){3-5} \cmidrule(lr){6-8} \cmidrule(lr){9-11}
$\theta$ & $d$ & $1$s RR & $0$s RR & $\widehat{K}_2$ & $1$s RR & $0$s RR & $\widehat{K}_2$ & $1$s RR & $0$s RR & $\widehat{K}_2$ \\
\midrule
 & 50  & 98.0\% & 88.0\% & 47.0\% & 99.5\% & 86.5\% & 22.0\% & 100.0\% & 87.0\% & 7.0\% \\
$0.5$ & 75  & 95.5\% & 89.5\% & 65.5\% & 97.5\% & 89.5\% & 22.0\% & 99.0\% & 88.0\% & 5.0\% \\
 & 100  & 92.0\% & 90.5\% & 87.0\% & 95.5\% & 90.0\% & 13.0\% & 97.5\% & 89.5\% & 5.0\% \\
\midrule
 & 50  & 91.0\% & 91.0\% & 97.0\% & 95.5\% & 89.5\% & 65.0\% & 97.5\% & 89.5\% & 3.0\% \\
$0.8$ & 75  & 81.0\% & 92.0\% & 99.0\% & 89.0\% & 91.0\% & 90.0\% & 94.0\% & 90.0\% & 6.0\% \\
 & 100  & 70.0\% & 93.0\% & 99.0\% & 81.0\% & 92.0\% & 99.0\% & 89.0\% & 90.5\% & 22.0\% \\
\midrule
 & 50  & 86.5\% & 90.5\% & 100.0\% & 90.0\% & 91.0\% & 97.0\% & 95.0\% & 90.0\% & 9.0\% \\
$1.0$ & 75  & 71.5\% & 93.0\% & 100.0\% & 77.0\% & 92.5\% & 98.0\% & 88.0\% & 91.0\% & 60.5\% \\
 & 100  & 57.5\% & 94.5\% & 100.0\% & 65.0\% & 93.5\% & 100.0\% & 79.0\% & 92.0\% & 92.0\% \\
\bottomrule
\end{tabular}
\end{footnotesize}
}

\vspace{0.5cm}

% Subtable 2
\subfloat[$T = 4384$]{
\begin{footnotesize}
\begin{tabular}{@{\hspace{0.5cm}}clccc@{\hspace{0.5cm}}ccc@{\hspace{0.5cm}}ccc@{}}
\hline
\hline
& & \multicolumn{3}{c}{$k=2$} & \multicolumn{3}{c}{$k=3$} & \multicolumn{3}{c}{$k=5$} \\
\cmidrule(lr){3-5} \cmidrule(lr){6-8} \cmidrule(lr){9-11}
$\theta$ & $d$ & $1$s RR & $0$s RR & $\widehat{K}_2$ & $1$s RR & $0$s RR & $\widehat{K}_2$ & $1$s RR & $0$s RR & $\widehat{K}_2$ \\
\midrule
 & 50  & 99.5\% & 93.5\% & 65.0\% & 100.0\% & 91.0\% & 40.0\% & 100.0\% & 92.0\% & 12.0\% \\
$0.5$ & 75  & 98.5\% & 94.5\% & 80.0\% & 99.0\% & 93.5\% & 38.0\% & 100.0\% & 93.0\% & 15.0\% \\
 & 100  & 96.5\% & 95.0\% & 95.0\% & 98.5\% & 94.5\% & 28.0\% & 99.5\% & 93.5\% & 20.0\% \\
\midrule
 & 50  & 95.0\% & 95.5\% & 98.5\% & 97.5\% & 94.0\% & 80.0\% & 99.0\% & 93.5\% & 16.5\% \\
$0.8$ & 75  & 90.0\% & 96.5\% & 99.5\% & 93.5\% & 95.0\% & 92.0\% & 97.0\% & 94.5\% & 20.0\% \\
 & 100  & 85.0\% & 97.5\% & 99.5\% & 90.0\% & 96.5\% & 99.5\% & 95.0\% & 94.5\% & 35.0\% \\
\midrule
 & 50  & 93.5\% & 94.5\% & 100.0\% & 95.0\% & 95.5\% & 98.5\% & 98.0\% & 94.0\% & 28.0\% \\
$1.0$ & 75  & 85.0\% & 97.5\% & 100.0\% & 88.0\% & 96.0\% & 99.0\% & 94.5\% & 95.5\% & 70.0\% \\
 & 100  & 75.0\% & 98.5\% & 100.0\% & 82.0\% & 97.5\% & 100.0\% & 89.0\% & 96.0\% & 95.0\% \\

\bottomrule
\end{tabular}
\end{footnotesize}
}

\vspace{0.5cm}

\begin{scriptsize}
\parbox{0.98\textwidth}{\emph{Note: } We simulate the process from \Cref{SBOUProc} $R=200$ times for $T= 1096, 4384$ and $\Delta_{\mathcal{P}_T} = 1/24$, corresponding to hourly observation over the course of $3$ and $12$ years. The L\'evy noise equals a standard $d$-dimensional Brownian motion and with group signal with $b_{ji} = 0.5$ for $j=i$ and $b_{ji} = 0.05$ for $j \neq i$. We let $\lambda(x,y) = 2x$ in \eqref{ModelsSelection}. 

The column ``1s RR'' refers to the recovery rate (RR) of non-zero elements in the drift matrix, representing the percentage of non-zero elements in the estimated support of the drift matrix from \Cref{SecSupportRec} that are also non-zero in the actual, unobserved drift matrix. The column ``0s RR'' refers to the recovery rate of zero elements in the drift matrix, indicating the percentage of correctly estimated zero elements. Lastly, the column $``\widehat{K}_2''$ indicates the percentage where the estimator $\widehat{K}_2$ from Algorithm 1 of \citet{Ma2021} accurately estimates the true number of communities.}
\end{scriptsize}

\end{table}

\begin{figure}[htbp]
    \caption{Illustration of the ability of the support recovery algorithm in \Cref{SecSupportRec} to recover the group signal and the probability of directed edges between nodes in our SBOU process.}
    \label{fig:EstProbabilities}
    \centering
   \includegraphics[scale=0.5]{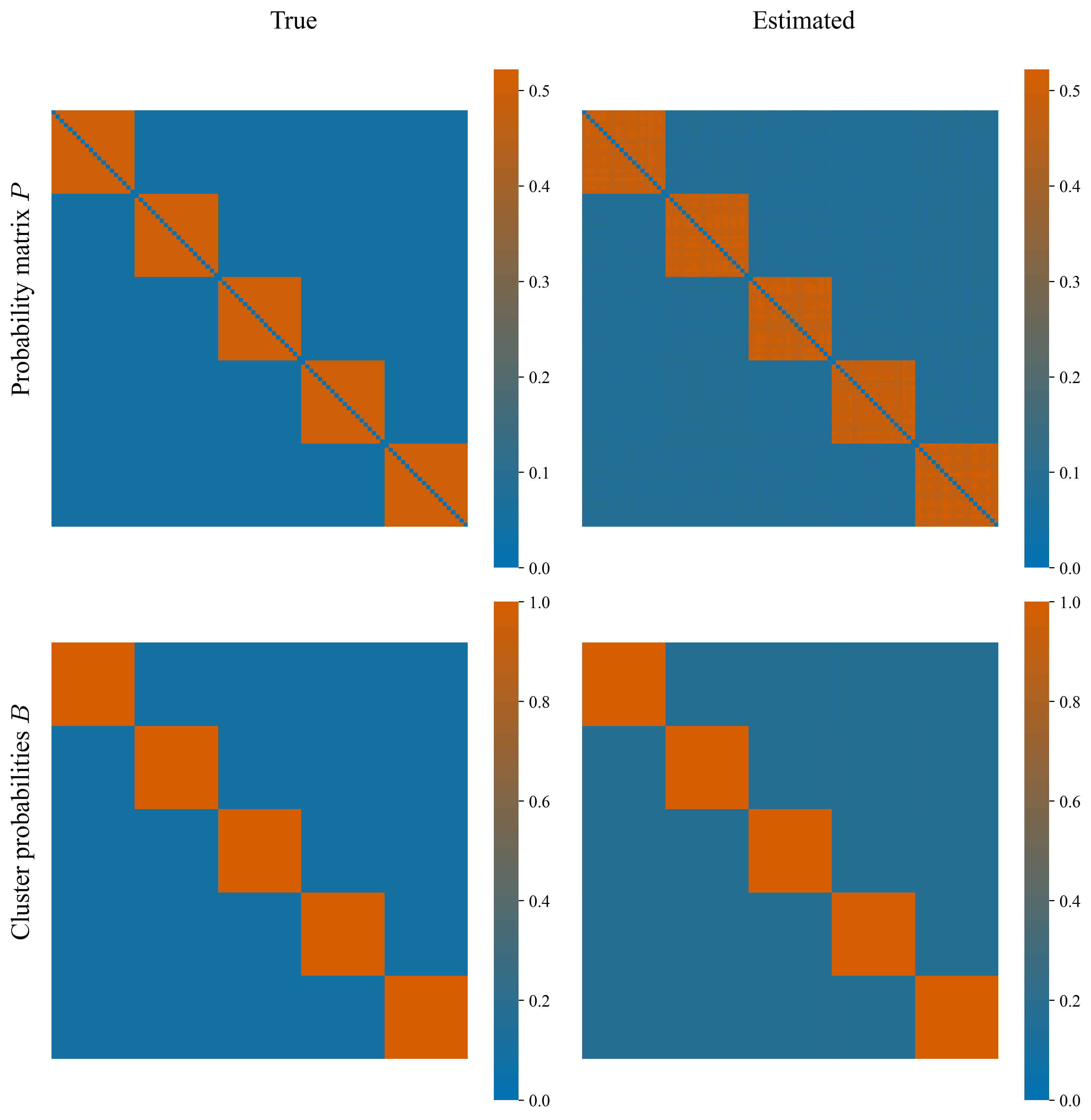}
           \begin{scriptsize}
\parbox{0.98\textwidth}{\emph{Note: } We simulate the process from \Cref{SBOUProc} $R=200$ times for $T= 4384$ and $\Delta_{\mathcal{P}_T} = 1/24$ corresponding to hourly observation over the course of $12$ years. The L\'evy noise is equal to a standard $d$-dimensional Brownian motion and with group signal with $B_{ii}=b_{ii} = 0.5$  and $B_{ij}=b_{ij} = 0.05$ for $j \neq i$. We let $\lambda(x,y) = 2x$ in \eqref{ModelsSelection}. On the left we show for a $GSBM2
(\bold{Z},\bold{B},\Theta,W)$ the heat maps of the scaled cluster probability matrix $\dfrac{\bold{B}}{\max_{ij}\bold{B}_{ij}}$ and probability matrix $\mathcal{A} = \Theta Z \mathbf{B} Z^T \Theta - \mathrm{diag}\left(\Theta Z \mathbf{B} Z^T \Theta\right)$ with $\mathcal{A}_{ij} = \theta_i\theta_jb_{s(i)s(j)}$ if $i\ne j $ and $0$ else sorted according to groups $\mathcal{V}_k$. We simulate a SBOU-process $R=200$ times with each network effect being defined by a new realisation $\mathcal{G} \sim GSBM2
(\bold{Z},\bold{B},\Theta,W)$. For each run, \Cref{alg:VARBlockbuster} is used to find $\left(\widehat{\mathcal{V}}_i\right)_{i=1}^k$. Based on these estimated groupings $\bold{B}$ and $\mathcal{A}$ are estimated using \Cref{EstimatorsBAndTheta}. $k$ is assumed known. The right hand side then shows the heatmaps of the mean matrices from these $200$ simulations.}
\end{scriptsize}
\end{figure}
\FloatBarrier

\subsection{Additional simulation results on the group detection algorithm}\label{SimCorrelation}
In this section, we investigate if our community detection method is able to discern noise from signal as motivated in the beginning. We compare it to the naive method of clustering by applying $K$-means on the correlation matrix of the SBOU-process from \Cref{SBOUProc}. We let
$$ \mathbb{L} = \Sigma^{1/2}\mathbb{W} + \mathbb{J}$$
Here, $\mathbb{W}$ is a $d$-dimensional Brownian motion with covariance matrix $\Sigma$ and $\mathbb{J}$ is a compound Poisson process independent of $\mathbb{W}$ with rate $\lambda=1$ such that $\mathbb{J}_t^{(j)}=\sum_{k=0}^{N_t} Z_k^{(j)}$, where $Z_k^{(j)}$ are i.i.d. jump heights that are standard normally distributed and $\mathbb{N}$ is a Poisson process with rate $\lambda$. In \Cref{table:SimTable1_KMeans_covar} we show the accuracy of the clustering based on the correlation matrix of the SBOU-process with $\Sigma=I_d$. We see that it detects the groups quite well, although, compared to \Cref{table:SimTable1}, it is clear that, even in the case where the noise is independent, it does not do as good a job as \Cref{alg:VARBlockbuster}. When the noise becomes correlated, the contrast, however, becomes even bigger. In \Cref{table:SimTable2} we see that, even with very correlated noise ($\Sigma_{ii}=1$ and $\Sigma_{ij} = 1/2$), \Cref{alg:VARBlockbuster} is able to discern signal from noise as opposed to the correlation-based clustering method that performs very poorly, see \Cref{table:SimTable2_KMeans_covar}.

\begin{sidewaystable}[p] 
\setlength{ \tabcolsep}{0.1cm}
\begin{center}
\caption{Accuracy of $K$-means applied directly to the correlation matrix of the SBOU process from \Cref{SBOUProc} the L\'evy driving noise is specified as \eqref{eqn:SimNoiseSpec} with independent noise}
\label{table:SimTable1_KMeans_covar}
\begin{footnotesize}
  \adjustbox{max width=\textheight}{%
\begin{tabular}{ccccccccccccccccccccccccccc}
    \hline
    \hline
 & &     \multicolumn{3}{c}{ k = 2} & &     \multicolumn{3}{c}{ k = 3} & &     \multicolumn{3}{c}{k = 5} & & & &    \multicolumn{3}{c}{k = 2} & &     \multicolumn{3}{c}{k = 3} & &     \multicolumn{3}{c}{k = 5} \\
    \cline{3-5} \cline{7-9} \cline{11-13} \cline{17-19} \cline{21-23} \cline{25-27}
 \\
    \multicolumn{2}{c}{Panel A}    && && && && &&& &&    \multicolumn{2}{c}{Panel B}   && && && && &&  \\
$\theta$ &$d/T$ &274&1096&4384& & 274&1096&4384 & & 274&1096&4384&&$\theta$ &$d/T$ &274&1096&4384 & & 274&1096&4384 & & 274&1096&4384 \\
    \cline{3-5} \cline{7-9} \cline{11-13} \cline{17-19} \cline{21-23} \cline{25-27}
 & 50  & 59.1\% & 77.3\% & 98.6\%&   & 49.1\% & 68.2\% & 92.1\%&   & 44.6\% & 60.2\% & 74.0\%&  &   & 50  & 58.4\% & 70.0\% & 92.9\%&   & 46.9\% & 58.9\% & 78.1\%&   & 40.5\% & 49.9\% & 58.8\% \\
$0.5$ & 75  & 56.3\% & 66.8\% & 98.5\%&   & 44.3\% & 60.6\% & 96.8\%&   & 37.8\% & 55.2\% & 83.5\%&  &  $0.5$ & 75  & 55.9\% & 62.1\% & 93.5\%&   & 43.0\% & 51.7\% & 83.4\%&   & 35.0\% & 43.7\% & 59.7\% \\
 & 100  & 55.1\% & 60.3\% & 96.4\%&   & 41.7\% & 51.9\% & 96.9\%&   & 33.6\% & 48.9\% & 89.4\%&  &   & 100  & 54.6\% & 57.9\% & 89.1\%&   & 41.2\% & 46.6\% & 81.9\%&   & 31.8\% & 38.4\% & 60.8\% \\
 \\
 & 50  & 57.1\% & 66.6\% & 98.8\%&   & 45.8\% & 64.8\% & 99.4\%&   & 41.8\% & 71.5\% & 97.8\%&  &   & 50  & 56.4\% & 62.3\% & 95.1\%&   & 44.7\% & 54.7\% & 93.1\%&   & 38.1\% & 51.0\% & 80.8\% \\
$0.8$ & 75  & 55.2\% & 57.9\% & 92.1\%&   & 42.2\% & 49.3\% & 96.9\%&   & 34.5\% & 51.1\% & 98.5\%&  &  $0.8$ & 75  & 54.9\% & 57.1\% & 82.8\%&   & 41.8\% & 45.4\% & 82.6\%&   & 33.2\% & 39.4\% & 75.7\% \\
 & 100  & 54.3\% & 55.6\% & 77.7\%&   & 40.5\% & 43.9\% & 87.5\%&   & 31.4\% & 39.8\% & 96.2\%&  &   & 100  & 54.2\% & 55.4\% & 67.1\%&   & 40.3\% & 42.3\% & 64.5\%&   & 30.8\% & 33.9\% & 62.5\% \\
 \\
 & 50  & 56.2\% & 60.8\% & 96.1\%&   & 44.7\% & 56.2\% & 98.9\%&   & 38.9\% & 64.7\% & 99.5\%&  &   & 50  & 56.1\% & 59.4\% & 89.2\%&   & 43.9\% & 50.5\% & 90.9\%&   & 36.7\% & 47.0\% & 86.5\% \\
$1.0$ & 75  & 54.9\% & 56.4\% & 78.8\%&   & 41.5\% & 45.3\% & 89.7\%&   & 33.6\% & 43.0\% & 97.9\%&  &  $1.0$ & 75  & 54.8\% & 56.2\% & 69.5\%&   & 41.2\% & 43.6\% & 68.8\%&   & 32.5\% & 36.5\% & 70.1\% \\
 & 100  & 54.2\% & 54.9\% & 64.2\%&   & 40.1\% & 41.9\% & 67.9\%&   & 30.8\% & 35.2\% & 88.4\%&  &   & 100  & 54.1\% & 54.5\% & 59.2\%&   & 40.1\% & 41.1\% & 51.7\%&   & 30.3\% & 32.4\% & 51.1\% \\
 \\
 \\
    \multicolumn{2}{c}{Panel C}    && && && && &&& &&    \multicolumn{2}{c}{Panel D}   && && && && &&  \\
$\theta$ &$d/T$ &274&1096&4384& & 274&1096&4384 & & 274&1096&4384&&$\theta$ &$d/T$ &274&1096&4384 & & 274&1096&4384 & & 274&1096&4384 \\
    \cline{3-5} \cline{7-9} \cline{11-13} \cline{17-19} \cline{21-23} \cline{25-27}
 & 50  & 57.4\% & 63.8\% & 82.9\%&   & 45.6\% & 51.9\% & 63.5\%&   & 37.8\% & 43.5\% & 48.4\%&  &   & 50  & 56.5\% & 59.0\% & 65.7\%&   & 43.8\% & 46.9\% & 50.0\%&   & 35.4\% & 37.8\% & 40.3\% \\
$0.5$ & 75  & 55.4\% & 59.3\% & 82.6\%&   & 42.4\% & 46.8\% & 64.6\%&   & 33.3\% & 37.5\% & 46.4\%&  &  $0.5$ & 75  & 55.1\% & 56.6\% & 63.5\%&   & 41.4\% & 42.9\% & 48.0\%&   & 32.2\% & 33.8\% & 36.8\% \\
 & 100  & 54.5\% & 56.2\% & 76.5\%&   & 40.5\% & 43.5\% & 60.8\%&   & 30.9\% & 34.0\% & 43.9\%&  &   & 100  & 54.1\% & 55.0\% & 60.0\%&   & 40.0\% & 41.2\% & 45.1\%&   & 30.2\% & 31.1\% & 34.2\% \\
 \\
 & 50  & 56.4\% & 59.9\% & 85.3\%&   & 43.7\% & 49.1\% & 74.4\%&   & 36.1\% & 42.0\% & 58.5\%&  &   & 50  & 55.9\% & 57.2\% & 66.0\%&   & 43.3\% & 44.7\% & 52.7\%&   & 34.7\% & 36.9\% & 41.8\% \\
$0.8$ & 75  & 54.7\% & 56.3\% & 70.1\%&   & 41.4\% & 43.6\% & 60.5\%&   & 32.3\% & 35.1\% & 49.3\%&  &  $0.8$ & 75  & 54.8\% & 55.6\% & 59.2\%&   & 41.1\% & 42.0\% & 45.8\%&   & 31.6\% & 32.6\% & 35.7\% \\
 & 100  & 54.1\% & 55.0\% & 60.7\%&   & 40.1\% & 41.3\% & 49.9\%&   & 30.3\% & 31.9\% & 40.4\%&  &   & 100  & 54.1\% & 54.2\% & 56.2\%&   & 40.0\% & 40.4\% & 42.5\%&   & 30.1\% & 30.4\% & 32.4\% \\
 \\
 & 50  & 56.2\% & 58.3\% & 78.0\%&   & 43.3\% & 46.7\% & 70.0\%&   & 35.6\% & 40.1\% & 60.4\%&  &   & 50  & 56.0\% & 56.8\% & 63.3\%&   & 43.1\% & 44.1\% & 50.5\%&   & 34.7\% & 36.1\% & 41.1\% \\
$1.0$ & 75  & 54.9\% & 55.4\% & 62.1\%&   & 41.0\% & 42.5\% & 52.8\%&   & 31.9\% & 33.8\% & 45.2\%&  &  $1.0$ & 75  & 54.8\% & 54.9\% & 57.4\%&   & 41.0\% & 41.6\% & 44.2\%&   & 31.5\% & 32.1\% & 34.7\% \\
 & 100  & 53.9\% & 54.4\% & 57.1\%&   & 40.2\% & 40.7\% & 45.2\%&   & 30.1\% & 31.0\% & 36.9\%&  &   & 100  & 54.1\% & 54.2\% & 55.4\%&   & 39.9\% & 40.3\% & 41.8\%&   & 30.0\% & 30.4\% & 31.7\% \\
 \\
    \hline
    \hline
  \end{tabular}%
}%
  \end{footnotesize}%
\smallskip
\begin{scriptsize}
\parbox{0.98\textwidth}{\emph{Note: } We simulate the process in \Cref{SBOUProc} with $\Delta_{\mathcal{P}_T=1/24}$ and interpret $T=1$ as one day so that we have hourly observations. The L\'evy driving noise is specified as \eqref{eqn:SimNoiseSpec} with independent entries. We simulate the process $R= 1,000$ times and apply K-means directly on the empirical correlation matrix. In this table, we present the mean of the share of correctly clustered nodes over this replications for for different group signals across varying parameters. Panels $A-B$ show results for the case where $b_{ii} = 0.5$ for all entries, while $b_{ij}$ varies across panels with values in $\{0.01, 0.05, 0.1, 0.2\}$. Specifically, Panel $C$ corresponds to the configuration where $b_{ij} = 0.1$.
}
\end{scriptsize}
\end{center}
\end{sidewaystable}

\begin{sidewaystable}[p] 
\setlength{ \tabcolsep}{0.1cm}
\begin{center}
\caption{Accuracy of $K$-means applied directly to the correlation matrix of the SBOU process from \Cref{SBOUProc} the L\'evy driving noise is specified as \eqref{eqn:SimNoiseSpec} with dependent noise}
\label{table:SimTable2_KMeans_covar}
\begin{footnotesize}
  \adjustbox{max width=\textheight}{%
\begin{tabular}{ccccccccccccccccccccccccccc}
    \hline
    \hline
 & &     \multicolumn{3}{c}{ k = 2} & &     \multicolumn{3}{c}{ k = 3} & &     \multicolumn{3}{c}{k = 5} & & & &    \multicolumn{3}{c}{k = 2} & &     \multicolumn{3}{c}{k = 3} & &     \multicolumn{3}{c}{k = 5} \\
    \cline{3-5} \cline{7-9} \cline{11-13} \cline{17-19} \cline{21-23} \cline{25-27}
 \\
    \multicolumn{2}{c}{Panel A}    && && && && &&& &&    \multicolumn{2}{c}{Panel B}   && && && && &&  \\
$\theta$ &$d/T$ &274&1096&4384& & 274&1096&4384 & & 274&1096&4384&&$\theta$ &$d/T$ &274&1096&4384 & & 274&1096&4384 & & 274&1096&4384 \\
    \cline{3-5} \cline{7-9} \cline{11-13} \cline{17-19} \cline{21-23} \cline{25-27}
 & 50  & 54.0\% & 53.5\% & 53.6\%&   & 40.7\% & 41.2\% & 42.9\%&   & 32.6\% & 35.1\% & 38.8\%&  &   & 50  & 54.0\% & 53.7\% & 53.4\%&   & 40.6\% & 41.0\% & 41.8\%&   & 32.5\% & 33.5\% & 35.5\% \\
$0.5$ & 75  & 53.5\% & 53.0\% & 52.8\%&   & 39.2\% & 39.3\% & 40.3\%&   & 29.9\% & 31.3\% & 34.2\%&  &  $0.5$ & 75  & 53.5\% & 53.0\% & 52.9\%&   & 39.5\% & 39.2\% & 39.7\%&   & 30.2\% & 30.3\% & 31.9\% \\
 & 100  & 53.1\% & 52.7\% & 52.6\%&   & 38.5\% & 38.4\% & 39.0\%&   & 28.4\% & 29.2\% & 31.2\%&  &   & 100  & 53.0\% & 52.8\% & 52.5\%&   & 38.7\% & 38.4\% & 38.5\%&   & 28.5\% & 28.5\% & 29.6\% \\
 \\
 & 50  & 53.8\% & 53.7\% & 53.3\%&   & 40.2\% & 40.3\% & 41.3\%&   & 31.7\% & 32.8\% & 36.6\%&  &   & 50  & 53.9\% & 53.4\% & 53.5\%&   & 40.5\% & 40.1\% & 40.7\%&   & 31.7\% & 32.1\% & 33.4\% \\
$0.8$ & 75  & 53.4\% & 53.1\% & 52.9\%&   & 39.1\% & 38.9\% & 39.1\%&   & 29.4\% & 29.5\% & 31.6\%&  &  $0.8$ & 75  & 53.7\% & 53.3\% & 52.9\%&   & 39.3\% & 39.0\% & 38.9\%&   & 29.7\% & 29.3\% & 29.9\% \\
 & 100  & 53.0\% & 52.7\% & 52.6\%&   & 38.5\% & 38.2\% & 38.1\%&   & 28.2\% & 27.9\% & 28.9\%&  &   & 100  & 53.0\% & 52.7\% & 52.6\%&   & 38.6\% & 38.2\% & 38.1\%&   & 28.3\% & 28.0\% & 28.2\% \\
 \\
 & 50  & 53.9\% & 53.6\% & 53.4\%&   & 40.5\% & 39.9\% & 40.4\%&   & 31.2\% & 31.5\% & 34.3\%&  &   & 50  & 54.3\% & 53.7\% & 53.4\%&   & 40.4\% & 40.1\% & 40.1\%&   & 31.5\% & 31.2\% & 31.9\% \\
$1.0$ & 75  & 53.4\% & 53.2\% & 53.0\%&   & 39.2\% & 38.8\% & 38.9\%&   & 29.2\% & 29.0\% & 30.0\%&  &  $1.0$ & 75  & 53.5\% & 53.2\% & 53.0\%&   & 39.2\% & 38.8\% & 38.6\%&   & 29.5\% & 29.0\% & 29.2\% \\
 & 100  & 53.1\% & 52.8\% & 52.6\%&   & 38.5\% & 38.1\% & 38.0\%&   & 27.9\% & 27.8\% & 28.1\%&  &   & 100  & 53.1\% & 52.8\% & 52.7\%&   & 38.6\% & 38.2\% & 38.0\%&   & 28.3\% & 27.9\% & 27.8\% \\
 \\
 \\
    \multicolumn{2}{c}{Panel C}    && && && && &&& &&    \multicolumn{2}{c}{Panel D}   && && && && &&  \\
$\theta$ &$d/T$ &274&1096&4384& & 274&1096&4384 & & 274&1096&4384&&$\theta$ &$d/T$ &274&1096&4384 & & 274&1096&4384 & & 274&1096&4384 \\
    \cline{3-5} \cline{7-9} \cline{11-13} \cline{17-19} \cline{21-23} \cline{25-27}
 & 50  & 54.3\% & 53.8\% & 53.5\%&   & 40.9\% & 40.9\% & 41.3\%&   & 32.8\% & 33.2\% & 33.9\%&  &   & 50  & 54.5\% & 54.2\% & 53.7\%&   & 41.6\% & 41.1\% & 41.1\%&   & 33.0\% & 33.1\% & 33.0\% \\
$0.5$ & 75  & 53.6\% & 53.3\% & 53.0\%&   & 39.6\% & 39.2\% & 39.4\%&   & 30.2\% & 30.1\% & 30.5\%&  &  $0.5$ & 75  & 53.8\% & 53.7\% & 53.3\%&   & 39.8\% & 39.6\% & 39.4\%&   & 30.7\% & 30.4\% & 30.3\% \\
 & 100  & 53.2\% & 52.8\% & 52.6\%&   & 38.8\% & 38.5\% & 38.4\%&   & 28.7\% & 28.6\% & 28.8\%&  &   & 100  & 53.4\% & 53.1\% & 52.8\%&   & 39.1\% & 38.9\% & 38.6\%&   & 29.3\% & 28.8\% & 28.7\% \\
 \\
 & 50  & 54.3\% & 53.8\% & 53.6\%&   & 40.6\% & 40.2\% & 40.3\%&   & 32.4\% & 31.8\% & 32.3\%&  &   & 50  & 54.6\% & 54.1\% & 53.7\%&   & 41.1\% & 40.8\% & 40.4\%&   & 32.8\% & 32.2\% & 32.1\% \\
$0.8$ & 75  & 53.6\% & 53.2\% & 53.0\%&   & 39.5\% & 39.1\% & 38.9\%&   & 30.0\% & 29.7\% & 29.5\%&  &  $0.8$ & 75  & 53.9\% & 53.5\% & 53.4\%&   & 40.1\% & 39.6\% & 39.1\%&   & 30.5\% & 30.1\% & 29.5\% \\
 & 100  & 53.2\% & 52.8\% & 52.7\%&   & 38.9\% & 38.5\% & 38.1\%&   & 28.6\% & 28.2\% & 28.1\%&  &   & 100  & 53.3\% & 53.0\% & 52.8\%&   & 39.1\% & 38.8\% & 38.4\%&   & 29.1\% & 28.7\% & 28.3\% \\
 \\
 & 50  & 54.2\% & 53.9\% & 53.5\%&   & 40.7\% & 40.3\% & 39.9\%&   & 32.2\% & 31.5\% & 31.2\%&  &   & 50  & 54.7\% & 54.2\% & 53.8\%&   & 41.3\% & 40.7\% & 40.1\%&   & 32.7\% & 32.3\% & 31.7\% \\
$1.0$ & 75  & 53.7\% & 53.3\% & 53.0\%&   & 39.4\% & 39.0\% & 38.7\%&   & 29.9\% & 29.4\% & 29.1\%&  &  $1.0$ & 75  & 53.9\% & 53.6\% & 53.2\%&   & 39.8\% & 39.6\% & 39.0\%&   & 30.5\% & 30.1\% & 29.6\% \\
 & 100  & 53.3\% & 52.9\% & 52.7\%&   & 38.9\% & 38.4\% & 38.1\%&   & 28.5\% & 28.2\% & 27.7\%&  &   & 100  & 53.3\% & 53.0\% & 52.8\%&   & 39.1\% & 38.7\% & 38.3\%&   & 29.1\% & 28.7\% & 28.2\% \\
 \\
    \hline
    \hline
  \end{tabular}%
}%
  \end{footnotesize}%
\smallskip
\begin{scriptsize}
\parbox{0.98\textwidth}{\emph{Note: } We simulate the process in \Cref{SBOUProc} with $\Delta_{\mathcal{P}_T=1/24}$ and interpret $T=1$ as one day so that we have hourly observations. The L\'evy driving noise is specified as \eqref{eqn:SimNoiseSpec} with  $\Sigma_{ii}=1$ and $\Sigma_{ij}=1/2$ for $i\ne j$. Further, $b_{ji} = 0.5$ for $j=i$ and $b_{ji} = 0.05$ for $j\ne i$. We simulate the process $R= 1,000$ times and apply K-means directly on the empirical correlation matrix. In this table, we present the mean of the share of correctly clustered nodes over this replications for for different group signals across varying parameters. Panels $A-B$ show results for the case where $b_{ii} = 0.5$ for all entries, while $b_{ij}$ varies across panels with values in $\{0.01, 0.05, 0.1, 0.2\}$. Specifically, Panel $C$ corresponds to the configuration where $b_{ij} = 0.1$.
}
\end{scriptsize}
\end{center}
\end{sidewaystable}

\begin{sidewaystable}[p] 
\setlength{ \tabcolsep}{0.1cm}
\begin{center}
\caption{Share of correctly clustered nodes from applying \Cref{alg:VARBlockbuster} to observations of \Cref{SBOUProc} where the L\'evy driving noise is specified as \eqref{eqn:SimNoiseSpec} with dependent noise }
\label{table:SimTable2}
\begin{footnotesize}
  \adjustbox{max width=\textheight}{%
\begin{tabular}{ccccccccccccccccccccccccccc}
    \hline
    \hline
 & &     \multicolumn{3}{c}{ k = 2} & &     \multicolumn{3}{c}{ k = 3} & &     \multicolumn{3}{c}{k = 5} & & & &    \multicolumn{3}{c}{k = 2} & &     \multicolumn{3}{c}{k = 3} & &     \multicolumn{3}{c}{k = 5} \\
    \cline{3-5} \cline{7-9} \cline{11-13} \cline{17-19} \cline{21-23} \cline{25-27}
 \\
    \multicolumn{2}{c}{Panel A}    && && && && &&& &&    \multicolumn{2}{c}{Panel B}   && && && && &&  \\
$\theta$ &$d/T$ &274&1096&4384& & 274&1096&4384 & & 274&1096&4384&&$\theta$ &$d/T$ &274&1096&4384 & & 274&1096&4384 & & 274&1096&4384 \\
    \cline{3-5} \cline{7-9} \cline{11-13} \cline{17-19} \cline{21-23} \cline{25-27}
 & 50  & 56.2\% & 88.6\% & 99.4\%&   & 50.3\% & 69.5\% & 90.2\%&   & 43.3\% & 52.6\% & 63.1\%&  &   & 50  & 55.8\% & 77.7\% & 96.4\%&   & 47.9\% & 58.9\% & 72.4\%&   & 39.7\% & 44.0\% & 48.3\% \\
$0.5$ & 75  & 55.1\% & 57.0\% & 99.8\%&   & 46.6\% & 61.3\% & 98.7\%&   & 40.0\% & 56.9\% & 78.1\%&  &  $0.5$ & 75  & 54.9\% & 55.2\% & 98.4\%&   & 44.7\% & 55.4\% & 88.3\%&   & 36.7\% & 44.4\% & 54.8\% \\
 & 100  & 54.5\% & 54.4\% & 99.6\%&   & 43.3\% & 59.7\% & 96.7\%&   & 35.8\% & 57.2\% & 84.4\%&  &   & 100  & 54.3\% & 54.3\% & 92.5\%&   & 42.2\% & 54.0\% & 77.2\%&   & 33.4\% & 44.5\% & 61.5\% \\
 \\
 & 50  & 56.4\% & 97.4\% & 100.0\%&   & 55.8\% & 92.8\% & 100.0\%&   & 53.8\% & 82.4\% & 98.1\%&  &   & 50  & 55.8\% & 88.2\% & 99.9\%&   & 50.7\% & 73.9\% & 99.3\%&   & 43.8\% & 59.4\% & 75.8\% \\
$0.8$ & 75  & 55.0\% & 58.1\% & 100.0\%&   & 48.5\% & 66.4\% & 100.0\%&   & 44.6\% & 80.2\% & 99.9\%&  &  $0.8$ & 75  & 54.9\% & 55.4\% & 99.9\%&   & 45.8\% & 62.1\% & 99.6\%&   & 38.3\% & 60.8\% & 88.1\% \\
 & 100  & 54.3\% & 54.2\% & 99.7\%&   & 44.1\% & 64.1\% & 99.6\%&   & 37.4\% & 76.0\% & 97.9\%&  &   & 100  & 54.3\% & 54.3\% & 95.1\%&   & 42.6\% & 58.8\% & 84.9\%&   & 33.8\% & 56.0\% & 82.6\% \\
 \\
 & 50  & 56.1\% & 98.4\% & 100.0\%&   & 57.0\% & 95.5\% & 100.0\%&   & 58.4\% & 90.5\% & 100.0\%&  &   & 50  & 55.8\% & 90.3\% & 100.0\%&   & 52.3\% & 77.9\% & 99.9\%&   & 44.9\% & 70.4\% & 94.7\% \\
$1.0$ & 75  & 55.0\% & 57.9\% & 100.0\%&   & 49.5\% & 67.4\% & 100.0\%&   & 46.1\% & 83.7\% & 100.0\%&  &  $1.0$ & 75  & 54.8\% & 55.0\% & 100.0\%&   & 46.0\% & 63.6\% & 100.0\%&   & 38.6\% & 68.2\% & 95.5\% \\
 & 100  & 54.3\% & 54.5\% & 99.9\%&   & 44.2\% & 64.6\% & 99.7\%&   & 38.0\% & 79.5\% & 99.0\%&  &   & 100  & 54.5\% & 54.5\% & 96.5\%&   & 42.8\% & 60.2\% & 86.5\%&   & 33.9\% & 60.3\% & 85.7\% \\
 \\
 \\
    \multicolumn{2}{c}{Panel C}    && && && && &&& &&    \multicolumn{2}{c}{Panel D}   && && && && &&  \\
$\theta$ &$d/T$ &274&1096&4384& & 274&1096&4384 & & 274&1096&4384&&$\theta$ &$d/T$ &274&1096&4384 & & 274&1096&4384 & & 274&1096&4384 \\
    \cline{3-5} \cline{7-9} \cline{11-13} \cline{17-19} \cline{21-23} \cline{25-27}
 & 50  & 55.9\% & 66.3\% & 87.3\%&   & 46.4\% & 51.2\% & 57.5\%&   & 37.5\% & 39.2\% & 41.2\%&  &   & 50  & 55.7\% & 57.3\% & 64.3\%&   & 44.2\% & 45.3\% & 46.9\%&   & 35.5\% & 35.8\% & 36.1\% \\
$0.5$ & 75  & 54.9\% & 54.6\% & 93.2\%&   & 43.2\% & 49.8\% & 66.3\%&   & 34.4\% & 38.1\% & 41.5\%&  &  $0.5$ & 75  & 54.8\% & 54.5\% & 64.1\%&   & 42.1\% & 44.2\% & 46.8\%&   & 33.0\% & 33.5\% & 33.9\% \\
 & 100  & 54.1\% & 54.2\% & 69.6\%&   & 41.2\% & 47.8\% & 60.0\%&   & 32.0\% & 37.0\% & 42.9\%&  &   & 100  & 54.1\% & 54.2\% & 53.8\%&   & 40.6\% & 42.5\% & 46.6\%&   & 30.6\% & 32.3\% & 32.9\% \\
 \\
 & 50  & 55.8\% & 72.1\% & 99.2\%&   & 48.0\% & 59.9\% & 91.9\%&   & 38.8\% & 45.1\% & 55.0\%&  &   & 50  & 55.6\% & 57.1\% & 87.2\%&   & 44.7\% & 48.1\% & 56.1\%&   & 35.7\% & 36.8\% & 38.7\% \\
$0.8$ & 75  & 54.8\% & 54.7\% & 99.0\%&   & 43.7\% & 55.9\% & 89.5\%&   & 35.0\% & 43.4\% & 61.8\%&  &  $0.8$ & 75  & 54.7\% & 54.6\% & 69.6\%&   & 42.1\% & 45.9\% & 55.2\%&   & 32.8\% & 34.4\% & 37.4\% \\
 & 100  & 54.3\% & 54.2\% & 70.1\%&   & 41.8\% & 52.0\% & 65.4\%&   & 32.1\% & 40.3\% & 66.3\%&  &   & 100  & 54.1\% & 54.1\% & 53.5\%&   & 40.7\% & 43.7\% & 54.0\%&   & 30.7\% & 32.8\% & 36.3\% \\
 \\
 & 50  & 55.9\% & 72.7\% & 99.8\%&   & 48.1\% & 63.1\% & 99.0\%&   & 39.2\% & 49.8\% & 68.8\%&  &   & 50  & 55.7\% & 56.4\% & 95.8\%&   & 44.9\% & 49.7\% & 63.8\%&   & 35.9\% & 37.3\% & 41.1\% \\
$1.0$ & 75  & 54.8\% & 54.7\% & 99.5\%&   & 43.9\% & 57.5\% & 92.4\%&   & 35.3\% & 46.5\% & 74.9\%&  &  $1.0$ & 75  & 54.6\% & 54.4\% & 70.7\%&   & 41.9\% & 46.4\% & 58.7\%&   & 32.9\% & 34.7\% & 40.1\% \\
 & 100  & 54.2\% & 54.0\% & 71.7\%&   & 41.5\% & 52.5\% & 67.0\%&   & 32.1\% & 42.5\% & 74.6\%&  &   & 100  & 54.0\% & 54.2\% & 53.7\%&   & 40.4\% & 43.4\% & 56.3\%&   & 30.8\% & 32.8\% & 38.4\% \\
 \\
    \hline
    \hline
  \end{tabular}%
}%
  \end{footnotesize}%
\smallskip
\begin{scriptsize}
\parbox{0.98\textwidth}{\emph{Note: } We simulate the process in \Cref{SBOUProc} with $\Delta_{\mathcal{P}_T=1/24}$ and interpret $T=1$ as one day so that we have hourly observations. The L\'evy driving noise is specified as \eqref{eqn:SimNoiseSpec} with  $\Sigma_{ii}=1$ and $\Sigma_{ij}=1/2$ for $i\ne j$. Further, $b_{ji} = 0.5$ for $j=i$ and $b_{ji} = 0.05$ for $j\ne i$. We simulate the process $R= 1,000$ times and apply K-means directly on the empirical correlation matrix. In this table, we present the mean of the share of correctly clustered nodes over this replications for for different group signals across varying parameters. Panels $A-B$ show results for the case where $b_{ii} = 0.5$ for all entries, while $b_{ij}$ varies across panels with values in $\{0.01, 0.05, 0.1, 0.2\}$. Specifically, Panel $C$ corresponds to the configuration where $b_{ij} = 0.1$.
}
\end{scriptsize}
\end{center}
\end{sidewaystable}
\FloatBarrier

\section{Implementation details for the empirical application}\label{App:EmpiricalImpl}

The empirical application in \Cref{Sec:Empirical} requires a number of implementation choices. We describe them here. The estimator in \citet{Ma2021} includes a tuning parameter that is particularly important for accurately detecting the case \(k=1\). That parameter is calibrated in simulation settings with dimensions \(d\in\{500,1000\}\). In our empirical application, however, it induces a pronounced tendency to select a single community. Further, we find that the procedure sometimes selects a relatively large number of groups, with a few of these groups containing only very few nodes, often just one. While this was not an issue in the simulation study, it becomes relevant in the empirical application and motivates imposing a minimum-group-size restriction.

To address this, we restrict attention to candidate values
\[
k\in\{2,3,\dots,K_{\max}\},
\]
so that the model-selection step is interpreted conditionally on the presence of at least two groups. Moreover, we impose a minimum-group-size restriction. This is consistent with \citet[Assumption 6]{Ma2021}, which requires the smallest group to contain more than \(\epsilon d\) observations for some \(\epsilon>0\). Since no specific value is prescribed in practice, we instead work with an integer-valued grid of admissible minimum group sizes,
\[
\mathcal G=\{g_1,\dots,g_m\}, \qquad g_1<g_2<\cdots<g_m,
\]
and in the empirical analysis we take \(m=6\) with \(\mathcal G=\{5,6,7,8,9,10\}\). For each \(g_i\in\mathcal G\), we estimate the number of groups under the constraint that every group in the resulting partition must contain at least \(g_i\) nodes. This yields a path of selected values \(\widehat K(g_1),\widehat K(g_2),\dots,\widehat K(g_m)\).

Because the same grid \(\mathcal G\) is used across datasets of different sizes, we allow the upper bound \(K_{\max}\) to depend on \(d\). Specifically, we set
\[
K_{\max}=\lfloor d/10\rfloor.
\]
The rationale is that, when \(d=50\), any candidate \(k>5\) would necessarily violate the requirement that each group contain at least 10 nodes, rendering such values redundant. Accordingly, for two-country subsets (\(d=50\)), we take \(K_{\max}=5\), while for the full sample (\(d=125\)), we take \(K_{\max}=12\).

To obtain a single estimate from the path \(\widehat K(g_i)\), we proceed as follows. For each candidate \(k\in\{2,\dots,K_{\max}\}\), let
\[
L(k)
=
\max_{1\leq a\leq b\leq m}
\Bigl\{
b-a+1 :
\widehat K(g_a)=\widehat K(g_{a+1})=\cdots=\widehat K(g_b)=k
\Bigr\},
\]
that is, \(L(k)\) is the length of the longest contiguous plateau at level \(k\) along the path of selected values as the minimum-group-size requirement varies over \(\mathcal G\). We then define
\[
\widehat K=\arg\max_{k\in\{2,\dots,K_{\max}\}} L(k).
\]
Thus, \(\widehat K\) is the value of \(k\) that remains selected over the longest contiguous range of admissible minimum-group-size requirements. We prefer this plateau-based criterion over alternatives such as the mode of \(\widehat K(g_i)\) because the grid \(\mathcal G\) is itself somewhat arbitrary, and we want to select a value of \(k\) that is insensitive to small perturbations in \(\mathcal G\). A long plateau at level \(k\) means that \(\widehat K(g_i)\) does not change as \(g_i\) moves through several adjacent grid points, which is precisely the stability we want. If multiple values of \(k\) achieve the longest plateau, we select the smallest, on grounds of parsimony.

Our final implementation choice concerns the estimator of the number of groups. We use the \(K_1\) estimator of \citet{Ma2021} rather than the \(K_2\) estimator. The latter has more appealing theoretical properties, but it introduces an additional tuning parameter by excluding candidate values of \(k\) whenever the associated likelihood-ratio statistic falls above a \(d\)-dependent threshold determined by simulation-based calibration. In our empirical application, this thresholding rule is difficult to justify, since no general guidance is available for choosing the tuning parameter outside the simulation settings considered in \citet{Ma2021}, and in a few instances none of the likelihood-ratio statistics fall below this threshold. By contrast, using the \(K_1\) estimator together with the restriction \(k\geq 2\), we avoid the introduction of any tuning parameters in the empirical analysis, which we prefer. Reassuringly, in our simulation study the \(K_1\) and \(K_2\) estimators selected the same number of groups in all the settings we consider, so this simplification does not appear to come at the cost of accuracy.

\section{Additional empirical plots}\label{app:EmpiricalClusteringSubGroups2and3}
This supplementary material collects the estimated cluster partitions for the two-country and three-country subset analyses introduced in \Cref{Sec:EmpiricalResults}. We report two complementary sets of figures. The first pair (\Cref{fig:clusters_2groups,fig:clusters_3groups}) shows the partitions obtained when the number of groups is chosen by our plateau-selection criterion, i.e.\ at the estimated $\widehat{K}$ reported in \Cref{tab:EstimatedGroupsSubsets}. The second pair (\Cref{fig:clusters_2groups_landscape,fig:clusters_3groups_landscape}) instead fixes the number of groups at the ground-truth value suggested by the country decomposition ($k = 2$ for two-country subsets and $k = 3$ for three-country subsets), providing a baseline for visual comparison.

\begin{landscape}
\begin{figure}[p]
    \centering
\includegraphics[width=1.5\textwidth]{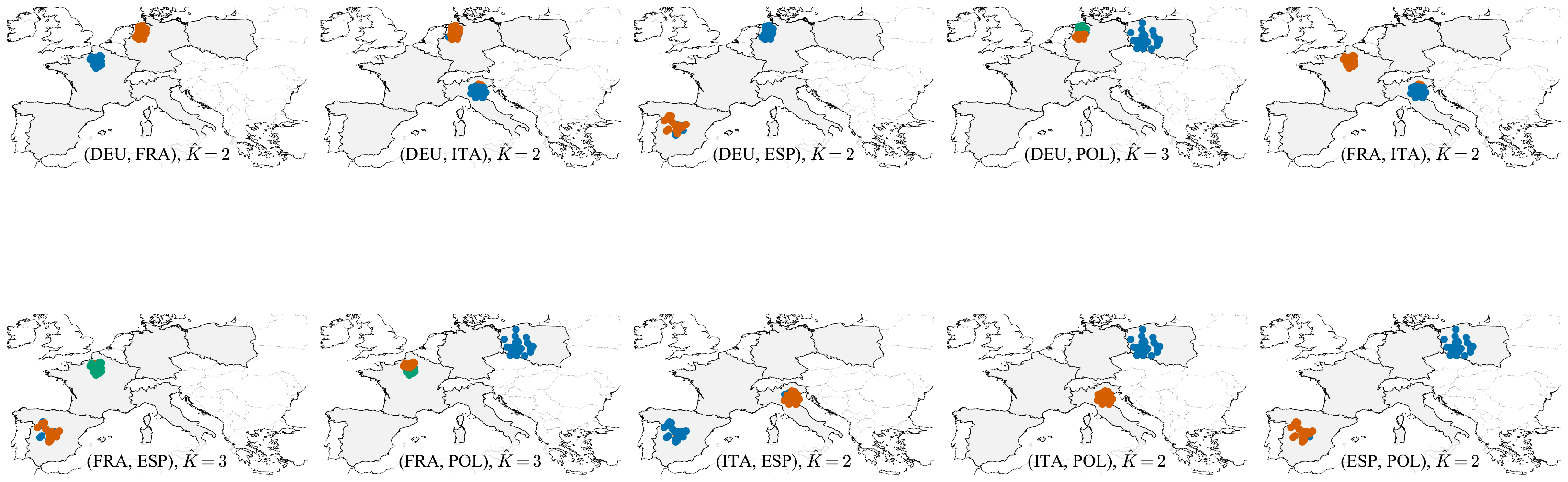}
    \caption{Estimated cluster partitions for all two-country subsets, with the number of groups selected by the plateau criterion, i.e.\ $k = \widehat{K}$.}
    \label{fig:clusters_2groups}
\end{figure}
\end{landscape}

\begin{landscape}
\begin{figure}[p]
    \centering
\includegraphics[width=1.5\textwidth]{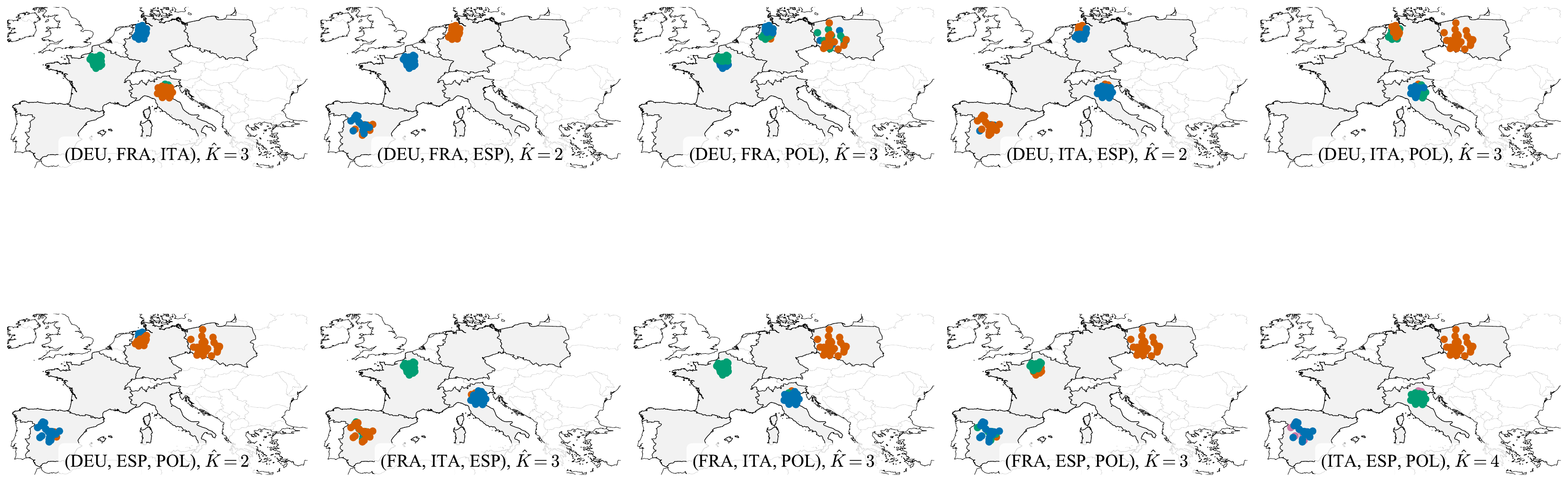}
    \caption{Estimated cluster partitions for all three-country subsets, with the number of groups selected by the plateau criterion, i.e.\ $k = \widehat{K}$.}
    \label{fig:clusters_3groups}
\end{figure}
\end{landscape}

\begin{landscape}
\begin{figure}[p]
    \centering
\includegraphics[width=1.5\textwidth]{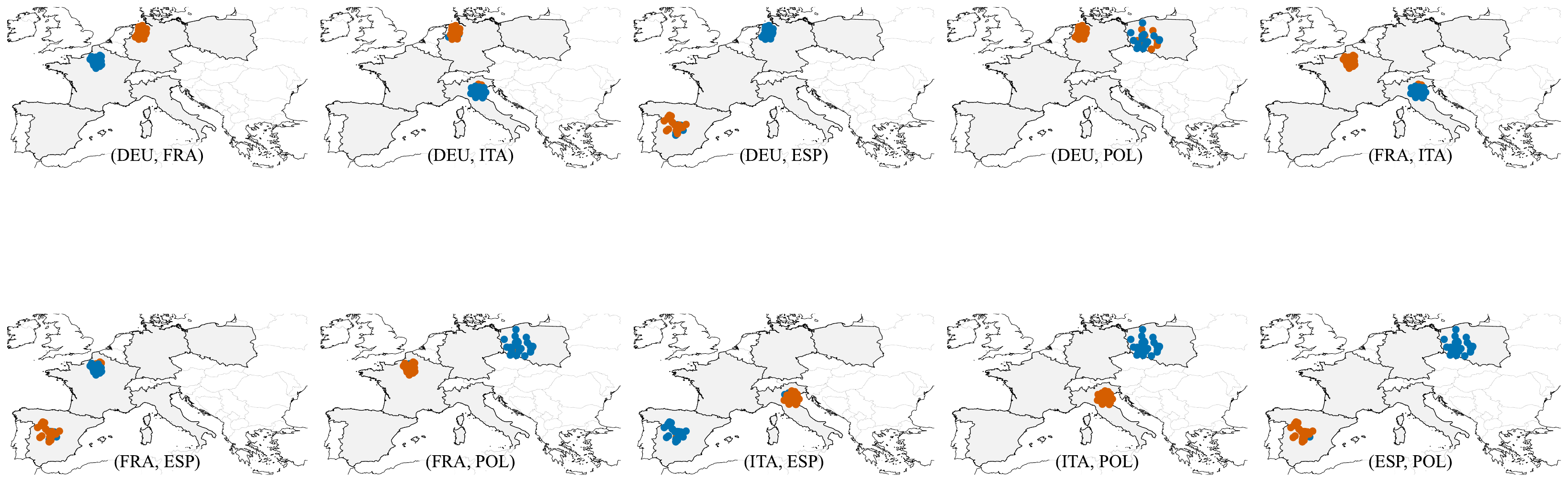}
    \caption{Cluster partitions for all two-country subsets under the ground-truth number of groups $k = 2$.}
    \label{fig:clusters_2groups_landscape}
\end{figure}
\end{landscape}

\begin{landscape}
\begin{figure}[p]
    \centering
\includegraphics[width=1.5\textwidth]{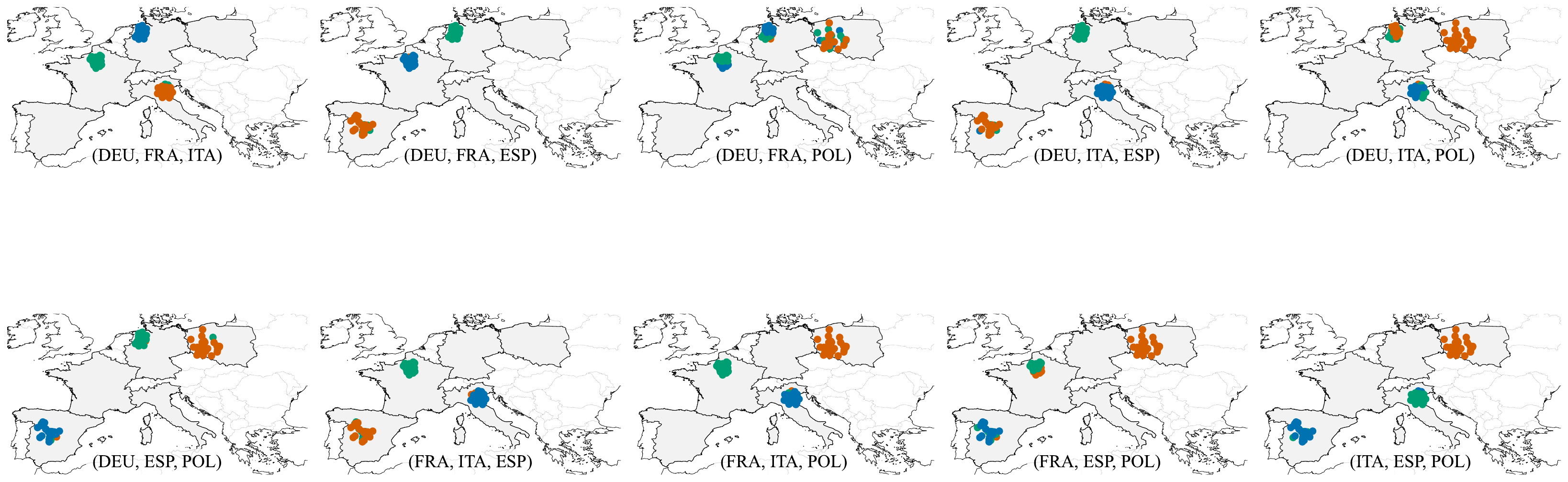}
    \caption{Cluster partitions for all three-country subsets under the ground-truth number of groups $k = 3$.}
    \label{fig:clusters_3groups_landscape}
\end{figure}
\end{landscape}

\end{document}